\documentclass{amsart}
\usepackage{amsmath, amssymb, amsthm, marginnote, ltxcmds, adjustbox, stmaryrd, amsbsy,bbold, extarrows}
\usepackage{mathrsfs}
\usepackage{mathtools}

\usepackage[usenames,dvipsnames]{xcolor}
\definecolor{cite}{HTML}{11871E}
\definecolor{url}{HTML}{698996}
\definecolor{link}{HTML}{912F1B}
\definecolor{greencolor}{rgb}{0,0.45,0}
\definecolor{greenbluecolor}{rgb}{0,0.45,0.1}
\definecolor{ucphcolor}{rgb}{0.517,0.016,0.016}
\definecolor{OliveGreen}{rgb}{0,0.6,0}
\usepackage[pdfencoding=unicode, colorlinks=true, linkcolor=ucphcolor, citecolor=greenbluecolor, urlcolor=url]{hyperref}
\usepackage[backend=biber, style=alphabetic, maxnames=10, maxalphanames=3, minalphanames=3]{biblatex}
\usepackage{tikz}
\usetikzlibrary{cd}
\usetikzlibrary{arrows, arrows.meta, positioning, calc}
\tikzcdset{arrow style=tikz, diagrams={>={Straight Barb[scale=0.8]}}}
\tikzstyle{arrow} = [-{Straight Barb[scale=0.8]}, line width=0.2mm]
\tikzset{
	math to/.tip={Glyph[glyph math command=rightarrow]},
	loop/.tip={Glyph[glyph math command=looparrowleft, swap]},
}

\usepackage{enumitem}
	{\end{enumerate}%
}

\usepackage[
letterpaper,
twoside=false,
textheight=22cm,
textwidth=14.4cm,
marginparsep=0.75cm,
marginparwidth=2.5cm,
heightrounded,
centering
]{geometry}

\usepackage{lineno}
\usepackage{stmaryrd}
\usepackage[T1]{fontenc}
\usepackage[capitalise]{cleveref}

\Crefname{prop}{Proposition}{Propositions}
\Crefname{lem}{Lemma}{Lemmas}
\Crefname{cor}{Corollary}{Corollaries}
\Crefname{thm}{Theorem}{Theorems}
\Crefname{alphThm}{Theorem}{Theorems}
\Crefname{defn}{Definition}{Definitions}
\Crefname{nota}{Notation}{Notations}
\Crefname{cons}{Construction}{Constructions}
\Crefname{rmk}{Remark}{Remarks}
\Crefname{obs}{Observation}{Observations}
\Crefname{warning}{Warning}{Warnings}
\Crefname{conj}{Conjecture}{Conjectures}
\Crefname{recollect}{Recollection}{Recollections}
\Crefname{terminology}{Terminology}{Terminologies}
\Crefname{setting}{Setting}{Settings}
\Crefname{example}{Example}{Examples}
\crefformat{equation}{(#2#1#3)}
\crefformat{section}{\S#2#1#3}
\crefmultiformat{section}{\S\S#2#1#3}{ and~#2#1#3}{, #2#1#3}{, and~#2#1#3}

\newtheorem{thm}{Theorem}[subsection]
\newtheorem{prop}[thm]{Proposition}
\newtheorem{lem}[thm]{Lemma}
\newtheorem{cor}[thm]{Corollary}
\newtheorem{alphThm}{Theorem}

\theoremstyle{definition}
\newtheorem{defn}[thm]{Definition}
\newtheorem{nota}[thm]{Notation}
\newtheorem{recollect}[thm]{Recollections}
\newtheorem{terminology}[thm]{Terminology}
\newtheorem{setting}[thm]{Setting}
\newtheorem{rmk}[thm]{Remark}
\newtheorem{obs}[thm]{Observation}
\newtheorem{example}[thm]{Example}
\newtheorem{warning}[thm]{Warning}
\newtheorem{constr}[thm]{Construction}

\newcommand{\cat}{\mathrm{Cat}}
\newcommand{\relcat}{\mathrm{RelCat}}
\newcommand{\presentable}{\mathrm{Pr}}
\newcommand{\catTopos}{\mathbf{Cat}}
\newcommand{\animaTopos}{\pmb{\mathcal{S}}}
\newcommand{\presentableTopos}{\mathbf{Pr}}
\newcommand{\cocartesianCatTopos}{\mathbf{coCart}}
\newcommand{\cartesianCatTopos}{\mathbf{Cart}}
\newcommand{\presheafTopos}{\pmb{\mathscr{P}}}
\newcommand{\funTopos}{\mathbf{Fun}}
\newcommand{\mapTopos}{\mathbf{Map}}
\newcommand{\relCatTopos}{\mathbf{RelCat}}
\newcommand{\baseTopos}{\mathscr{T}}
\newcommand{\spc}{\mathcal{S}}
\newcommand{\cofree}{\mathfrak{c}}
\newcommand{\frakTwistedArrow}{\mathfrak{T}\mathrm{w}}
\newcommand{\sphere}{\mathbb{S}}
\newcommand{\presheaf}{\mathscr{P}}

\newcommand{\scL}{\mathscr{L}} 
\newcommand{\calO}{\mathcal{O}}
\newcommand{\bbZ}{\mathbb{Z}}

\newcommand{\bbQ}{\mathbb{Q}}
\newcommand{\bbR}{\mathbb{R}}
\newcommand{\sM}{\mathcal{M}}

\newcommand{\sN}{\mathcal{N}}
\newcommand{\X}{\mathcal{X}}
\newcommand{\Y}{\mathcal{Y}}
\newcommand{\sC}{{\mathcal C}} 
\newcommand{\D}{{\mathcal D}}
\newcommand{\B}{\mathcal{B}} 
\newcommand{\Z}{\mathcal{Z}}
\newcommand{\E}{\mathcal{E}}
\newcommand{\F}{\mathcal{F}}
\newcommand{\W}{\mathcal{W}}
\newcommand{\V}{\mathcal{V}}
\newcommand{\sU}{\mathcal{U}}

\newcommand{\yoneda}{\text{\usefont{U}{min}{m}{n}\symbol{'107}}}
\DeclareFontFamily{U}{min}{}
\DeclareFontShape{U}{min}{m}{n}{<-> dmjhira}{}
 
\newcommand*\cocolon{%
        \nobreak
        \mskip6mu plus1mu
        \mathpunct{}%
        \nonscript
        \mkern-\thinmuskip
        {:}%
        \mskip2mu
        \relax
}

\newcommand{\eilenbergMacLaneCoeff}{{\mathrm{H}\bbZ}}
\newcommand{\simplex}{\mathbb{\Delta}}

\newcommand{\op}{^{\mathrm{op}}}
\newcommand{\hop}{^{\mathrm{hop}}}
\newcommand{\vop}{^{{\mathrm{vop}}}}

\DeclareMathOperator{\straighten}{\mathrm{St}}
\DeclareMathOperator{\unstraighten}{\mathrm{Un}}
\newcommand{\cartesianCategory}{\mathrm{Cart}}
\newcommand{\cocartesianCategory}{\mathrm{coCart}}
\newcommand{\spectra}{\mathrm{Sp}}
\newcommand{\stable}{\mathrm{st}}
\newcommand{\dbl}{{\mathrm{dbl}}}
\newcommand{\id}{\mathrm{id}}
\DeclareMathOperator{\map}{\mathrm{Map}}
\DeclareMathOperator{\fib}{\operatorname{fib}}
\DeclareMathOperator{\cofib}{\operatorname{cofib}}
\DeclareMathOperator{\func}{\mathrm{Fun}}
\newcommand{\module}{\mathrm{Mod}}
\newcommand{\unit}{\mathbb{1}}
\newcommand{\twistedArrow}{\mathrm{Tw}}
\newcommand{\lmodule}{\mathrm{LMod}}
\newcommand{\picardSpace}{\mathrm{Pic}}
\newcommand{\picardSpaceTopos}{\mathbf{{Pic}}}
\newcommand{\alg}{\mathrm{Alg}}
\newcommand{\calg}{\mathrm{CAlg}}
\newcommand{\leftmod}{\mathcal{LM}}
\newcommand{\canonical}{\mathrm{can}}
\newcommand{\lfib}{\mathrm{LFib}}
\newcommand{\rfib}{\mathrm{RFib}}

\newcommand{\eval}{\mathrm{ev}}
\newcommand{\constant}{\operatorname{const}}
\DeclareMathOperator{\res}{res}
\DeclareMathOperator{\nattrans}{\mathrm{Nat}}
\newcommand{\ind}{\mathrm{Ind}}
\def\colim{\qopname\relax m{colim}}

\newcommand{\udl}[1]{\underline{{#1}}}
\newcommand{\slice}[3]{{#1}\downarrow_{#2}{#3}}
\DeclareMathOperator{\lift}{lift}
\newcommand{\interior}[1]{\mathring{#1}}
\newcommand{\paramFibred}[1]{\mathfrak{#1}}

\title[{POINCAR\'E DUALITY PAIRS OF $\infty$-CATEGORIES}]{\Large\textsc{Poincar\'e Duality Pairs of $\infty$-Categories}}
\author{ANDREA BIANCHI}
\email{bianchi@mpim-bonn.mpg.de}
\address{Max Planck Institute for Mathematics, Vivatsgasse 5, 53111 Bonn, Germany}
\author{KAIF HILMAN}
\email{kaif@math.uni-bonn.de}
\address{Mathematik Zentrum der Universit\"at Bonn, Endenicher Allee 60, 53115 Bonn, Germany}
\author{DOMINIK KIRSTEIN}
\email{kirstein@math.lmu.de}
\address{Mathematisches Institut, Ludwig-Maximilians-Universität München, Theresienstraße 39, 80333 Munich, Germany}
\author{CHRISTIAN KREMER}
\email{kremer@mpim-bonn.mpg.de}
\address{Max Planck Institute for Mathematics, Vivatsgasse 5, 53111 Bonn, Germany}

\date{\today}

\begin{document}

\begin{abstract}
We introduce a notion of Poincar\'e duality for pairs of $\infty$-categories,
extending Poincar\'e--Lefschetz duality for pairs of spaces.
This categorical extension yields an efficient book-keeping device that affords, among other things, a uniform  treatment of Wall's Poincar\'e ads of spaces, iterated Poincar\'e cobordisms, and in general, diagrams of spaces parametrised by the face poset of a combinatorial manifold. 
In each of these cases, the theory reduces them to studying a single pair of $\infty$-categories and the properties of a single functor, the relative cohomology functor. Using this formalism, we prove a very general fibration theorem which, in particular, specialises to a generalisation of Klein--Qin--Su's fibration theorem for Poincar\'e triads to all ads. This theory also lays the foundation for future work by the authors on Poincar\'e cobordism categories, isovariant Poincar\'e spaces and string topology.
\end{abstract}
\def\uppercasenonmath#1{}

\subjclass[2020]{
16D90,  
55N20,  
57P10, 
57R19 
}
\keywords{Poincar\'e--Lefschetz duality, Morita theory, twisted ambidexterity}
\maketitle

\tableofcontents

\section{Introduction}
Poincar\'e spaces \cite{Browder1, Browder2, Spivak, Wall} were introduced as a homotopical counterpart to the notion of closed manifolds, and feature extensively in their classification via surgery theory. In fact, just as central in this program is the consideration of homotopical analogues for more complicated geometric objects, such as compact manifolds with boundary, cobordisms or manifold triads. Traditionally, in order to model these homotopical analogues, one considers a diagram of spaces indexed by a finite poset, which one thinks of as a family of spaces related by inclusion relations; one declares a list of data providing Poincar\'e--Lefschetz dualities relating consecutive strata; and one requires a list of suitable compatibilities among these data; see for instance \cite{Wall,Wall_scm}. As the number of strata increases, however, one is quickly burdened by combinatorial difficulties. This renders the ``obvious'' generalisations of statements from the absolute setting to the stratified setting considerably harder.  For instance, while the fibration theorem for Poincar\'e spaces was proved in \cite{Kleindual} some 25 years ago, the proof for the analogous generalisation to Poincar\'e triads only appeared relatively recently in \cite{KleinQinSu}. Moreover, the lack of a uniform treatment forces one to repeat arguments in an ad-hoc fashion for every new, but similar, situation.

In this article, we offer a new perspective on this subject by introducing the notion of Poincar\'e pairs of \textit{$\infty$-categories}, which will be seen to be a direct generalisation of Poincar\'e--Lefschetz duality for a pair of spaces. The insight here is that by merely extending the notions of interest to categories, we are also able to capture all the compatibilities and coherences between the  dualities across different strata by just considering the Poincar\'e--Lefschetz duality enjoyed by a single pair of $\infty$-categories. 

As a basic illustrative example, we prove in \cref{alphThm:classical_vs_us} that Wall's notion of Poincar\'e ads from \cite{Wall_scm} may equivalently be encoded as a property of the $\infty$-category obtained by unstraightening a cubical diagram of finitely dominated spaces. More generally, we provide in \cref{alphThm:combinatorial_manifold} a ``local'' and checkable characterisation of Poincar\'e duality for diagrams of spaces indexed over the face poset of a combinatorial manifold: taking diagrams indexed over the face poset of a hyperrectangle tesselated by hypercubes, we plan in future work to model higher Poincar\'e cobordisms categories in an efficient way \cite{pdcob}. As ingredients to prove these and more results, we provide a suite of ``cutting-and-gluing'' principles for Poincar\'e duality, some of which will be highlighted below. 

The categorical basis of our work will be the Morita theory for presheaves over $\infty$-categories. We develop aspects of this in the present work, and we distil some of the fundamental points as \cref{thm:A} below. As an extra layer of generality, we in fact developed the requisite Morita theory in the setting of categories parametrised over an arbitrary $\infty$-topos, as introduced by Martini-Wolf \cite{Martini2022Yoneda,Martini2022Cocartesian,MartiniWolf2022Presentable}. Besides paving the way for some of the future applications that we have in mind, specifically the study of equivariant Poincar\'e duality, this generality turns out to be very convenient when we want to address relative Poincar\'e duality for \textit{pairs of functors}, i.e. for pairs of $\infty$-categories lying over a common base $\infty$-category: in this article we will use this strategy to prove a fairly general fibration theorem, \cref{thm:fibred_dualising_object_factorisation}, from which we shall deduce \cref{thm:ads_fibrations} below, which in turn generalises \cite[Thm. G]{KleinQinSu} to arbitrary ads and arbitrary presentable stable $\infty$-categories of coefficients.
\subsection{Motivation}
Before stating our main results, let us mention a couple of concrete and illustrative examples of diagrams of spaces satisfying a form of stratified Poincar\'e--Lefschetz duality, which can be studied using the results and techniques of this article.
\begin{enumerate}
\item Poincar\'e pairs are a special class of pairs of spaces $(\X,\partial \X)$ designed to satisfy Poincar\'e--Lefschetz duality, thus giving a homotopical analogue of compact manifolds with boundary. They are useful, for example, to classify compact manifolds with boundary or to speak of bordisms of Poincar\'e duality spaces.
\item Given a pushout square of compact spaces
\begin{equation*}
\begin{tikzcd}[row sep=10pt]
\partial \W \ar[r] \ar[d, "p"] & \W \ar[d] \\
\X \ar[r] & \Y
\end{tikzcd}
\end{equation*}
in which $p$ is a spherical fibration, $\X$ a Poincar\'e space and $(\W, \partial \W)$ a Poincar\'e pair, it so happens that $\Y$ is a Poincar\'e space as well. We may view this diagram as a decomposition of $\Y$ into an embedded copy of $\X$, a complement to the embedding $\W$ and a normal spherical fibration $\partial \W$. This is the notion of a Poincar\'e embedding of $\X$ into $\Y$ \cite{Levitt68, KleinEmb1} and provides the homotopical analogue of a closed manifold with an embedded submanifold, relevant in the study of surgery for submanifolds \cite[Ch. 11, 12]{Wall_scm}.  
\item 
Many decompositions of Poincar\'e spaces of interest come from bordism-theoretic considerations.
Consider for example a diagram of spaces parametrised by the face poset of two adjacent squares
\[
\begin{tikzcd}
\X_{00} \arrow[d, "\V_0" description, no head] \arrow[r, "\W_{0}" description, no head] \arrow[rd, "\Z_0", phantom] & \X_{10} \arrow[r, "\W_{1}" description, no head] \arrow[d, "\V_1" description, no head] \arrow[rd, "\Z_1", phantom] & \X_{20} \arrow[d, "\V_2" description, no head] \\
\X_{10} \arrow[r, "\sU_0" description, no head]                                                                      & \X_{11} \arrow[r, "\sU_1" description, no head]                                                                      & \X_{21}                                       
\end{tikzcd}
\]
so that we have inclusions $\X_{00} \rightarrow \W_0$, $\W_0 \rightarrow \Z_0$, and so on. Diagrams of manifolds of this form occur in the definition of higher cobordism categories, and analogous diagrams of spaces, satisfying a form of Poincar\'e--Lefschetz duality at all ``depths'', ought to underly the definition of higher Poincar\'e cobordism categories. 
\end{enumerate}
All three examples have in common that they present geometrically interesting analogues of certain decompositions of compact manifolds with boundary. The notion of Poincar\`e duality for $\infty$-category pairs, developed in this article, allows one to uniformly treat all three situations above; specifically, the three $\infty$-category pairs to be considered are the following:
\begin{enumerate}
\item the unstraightening $\int_{[1]}(\partial\X\to\X)$ of the pair of spaces $(\X,\partial \X)$, considered as a functor $[1] \rightarrow \spc$, relative to its $\infty$-subcategory $\partial \X\subset\int_{[1]}(\partial\X\to\X)$;
\item the unstraightening of the diagram $\X \leftarrow \partial \W \rightarrow \W$, relative to the empty $\infty$-subcategory;
\item the unstraightening of the given diagram of spaces, relative to the $\infty$-subcategory given by the unstraightening of the restriction of the diagram of spaces on the boundary of the union of the two squares.
\end{enumerate}
Many more examples fit into our framework, including the Poincar\'e ads from \cite{Wall_scm}, which we discuss in this article, and transversality structures, see e.g. \cite{Levitt_Ranicki}, which we also plan to study in future work.

The novelty of the approach we take is its uniformity: Poincar\'e duality in a large diagram of geometric interest can be formulated as a single homological condition, instead of giving an ad-hoc decomposition in which all pieces are asked to satisfy Poincar\`e duality. To the authors, such a uniform treatment seemed necessary for at least three reasons.
\begin{enumerate}
    \item In forthcoming work of the authors on categories of Poincar\'e cobordisms and cobordism spectra \cite{pdcob}, there is an abundant need for constructions such as gluing and taking products of diagrams. Instead of giving case-by-case proofs that these constructions preserve the Poincar\'e duality property, our approach allows us to treat them all at the same time.
    \item In forthcoming work of the first named author \cite{Bianchi:stringtopology}, the approach to Poincar\'e--Lefschetz duality developed in this article will be used in identifying the Chas--Sullivan string product, constructed for oriented closed manifolds \cite{ChasSullivan}, with a purely homotopy-theoretical string product, constructed for arbitrary oriented Poincar\'e duality spaces.
    \item We simultaneously develop most of the theory in the setting of an arbitrary $\infty$-topos $\baseTopos$, in particular having applications to equivariant Poincar\'e duality spaces in mind.
    An important notion is that of an isovariant structure on an equivariant Poincar\'e space, which additionally provides nice embeddings of various fixed point spaces in each other.
    Forthcoming work of the third and fourth author \cite{isovariantpd} serves as a first step in that direction and studies the existence and uniqueness of isovariant structures on semifree Poincar\'e spaces.
\end{enumerate}

Although we develop the theory in larger generality, we restrict our attention to the base $\infty$-topos $\baseTopos = \spc$ of spaces for the rest of the introduction, to facilitate stating our main results.

\subsection{Summary of the results}
In the entire article we work with $\infty$-categories, and we use the word ``space'' to refer to an $\infty$-groupoid.
The starting point of our work is the notion of an \textit{$\infty$-category pair}, which is the datum $(\X,\partial \X)$ of an $\infty$-category $\X$ equipped with a left closed full $\infty$-subcategory $i\colon \partial \X\subseteq \X$. The archetype that will be our running example throughout the introduction is the cocartesian unstraightening $\X\coloneqq \int_{[1]}X$ associated to a functor $X\colon[1]\rightarrow\spc^{\omega}$ corresponding to a map of spaces $g\colon Y\coloneqq X(0)\rightarrow Z \coloneqq X(1)$; in this case we set $\partial\X=Y= X(0)$.

Let $\sC\in\calg(\presentable^L_{\mathrm{st}})$ be a presentably symmetric monoidal stable $\infty$-category, and let $\presheaf(\X;\sC)$ denote the $\infty$-category of $\sC$-valued presheaves on $\X$. There is a natural transformation $\X_*\rightarrow \partial \X_*i^*$ of functors $\presheaf(\X;\sC)\rightarrow\sC$ coming from the contravariant functoriality of limits.
Assume that $(\X,\partial \X)$ is $\sC$\textit{-twisted ambidextrous} (cf. \cref{defn:ambidexterityforpairs}) in the sense that   ``relative cohomology'' functor  \[(\X,\partial \X)_*\coloneqq \fib(\X_*\rightarrow \partial\X_*j^*)\colon \presheaf(\X;\sC)\longrightarrow \sC\] preserves colimits and its canonical lax $\sC$-linear structure is in fact $\sC$-linear; then the appropriate Morita theory guarantees that there is a system $\omega_{\X,\partial \X}\colon \X\rightarrow\sC$ that ``classifies'' $(\X,\partial\X)_*$. Note the change in variance: $\presheaf(\X;\sC)$ is the $\infty$-category of \textit{contravariant} functors from $\X$ to $\sC$, but $\omega_{\X,\partial\X}$ is \textit{covariant}.
We say that the $\infty$-category pair $(\X,\partial \X)$ is $\sC$\textit{-Poincar\'e} if $\omega_{\X,\partial \X}$ factors through the $\infty$-subgroupoid $\picardSpace(\sC)$ of invertible objects in $\sC$.

In the running example $\X=\int_{[1]}X$, the basic insight is that the single functor $\omega_{\X,\partial \X}\colon\X\to\sC$ carries all of the following information:
\begin{itemize}
\item the relative dualising system $D_{Z,Y}\in\sC^{Z}$, characterised by the equivalence $\fib(Z_*\to Y_*g^*)\simeq Z_!(D_{Z,Y}\otimes-)$; we have $D_{Z,Y}\simeq\omega_{\X,\partial\X}|_{Z}$;
\item the dualising system $D_{Y}\in\sC^{Y}$, characterised by the equivalence $Y_*\simeq Y_!(D_{Y}\otimes-)$; we have $\Omega D_{Y}\simeq\omega_{\X,\partial\X}|_{Y}$;
\item the connecting map $\delta\colon\Omega D_{Y}\rightarrow D_{Z,Y}|_{Y}$, coming from the canonical natural transformation $\Omega Y_*g^*\to\fib(Z_*\to Y_*g^*)$; this is recovered by assembling the values of $\omega_{\X,\partial\X}$ on the morphism of $\X$ lying over $0\to1$.
\end{itemize}

In this way, requiring that $\omega_{\X,\partial \X}$ factors through $\picardSpace(\sC)$ will precisely enforce both that the relative cohomology functor $(Z,Y)_*\colon\sC^{Z}\to\sC$ is classified by an invertible system, \textit{and} that this system restricts to the dualising system for the cohomology functor of the boundary $Y_*$. In other words, working with the pair of $\infty$-categories $(\X,\partial \X)$ allows us to simultaneously reason about both the relative cohomology of the pair of spaces $(Z,Y)$ \textit{as well as} its interaction with the cohomology of the boundary space $Y$.

\vspace{1mm}

Having sketched the general philosophy, we now state the first set of results in the form of the following omnibus Morita package that will form the categorical basis of our work.

\begin{alphThm}[Full versions in {\cref{prop:formula_with_twisted_arrow,prop:categorical_poincare_lefschetz_duality,lem:naturality_classification_C_linear_functors,cor:formula_in_terms_of_dualising_system_for_dualish}}]
\label{thm:A}
Let $\X$ and $\Y$ be small $\infty$-categories, let $\sC\in \calg(\presentable^L_{\mathrm{st}})$, and let $F\in\func^L_{\sC}(\presheaf(\X;\sC),\sC)$ and $G\in\func^L_{\sC}(\presheaf(\Y;\sC),\sC)$ be $\sC$-linear functors. Then:
\begin{enumerate}[label=(\arabic*)]
\item (Classification). Precomposition with the Yoneda embedding induces an equivalence $\yoneda^*\colon \func^L_{\sC}(\presheaf(\X;\sC),\sC)\xrightarrow{\simeq}\func(\X,\sC)$ with inverse given by sending $\zeta\in \func(\X,\sC)$ to the functor $\twistedArrow(\X)_!(t^*(-)\otimes s^*\zeta)$ where $(s,t)\colon \twistedArrow(\X)\rightarrow \X\op\times \X$ denotes the source and target projections from the twisted arrow $\infty$-category of $\X$. We will write $\omega_F\coloneqq \yoneda^*(F)$.
\item (Naturality). For a functor $f\colon \X\rightarrow \Y$, there are equivalences $\omega_{Gf_!}\simeq f^*\omega_G$ and $\omega_{Ff^*}\simeq f_!\omega_F$.\item (Groupoidality). Suppose $\omega_F$ is groupoidal, i.e. it factors through $\sC^{\simeq}\subset \sC$. Then writing $D_F\coloneqq t_!s^*\omega_F\in\presheaf(\X;\sC)$, there is a simplified equivalence $F(-)\simeq \X_!(D_F\otimes -)$.
\end{enumerate}
\end{alphThm}

It is conceptually important to note that, unlike in the case of $\infty$-groupoids, there is a ``decoupling'' that happens in the Morita theory for $\infty$-categories: for a $\sC$-linear functor $F\colon\presheaf(\X;\sC)\to\sC$, there are now \textit{two} canonical objects one can consider, namely the \textit{classifying system} $\omega_F$ and the \textit{dualising system} $D_F$. In general, the object which classifies the functor $F$ is only $\omega_F$, via the more complicated twisted arrow formula, and not $D_F$. The content of \cref{thm:A}(3) is that when $\omega_F$ is groupoidal, then $F$ is equivalently classified by $D_F$ via the simpler and more familiar formula.

\vspace{1mm}

Next, we enumerate a list of basic principles in working with the current framework.
For brevity, we only give ``slogan'' versions of our principles, without specifying all hypotheses or completely clarifying the statements; we refer to the precise statements in the body of the article.
These principles allow one to manipulate and infer Poincar\'e duality for various $\infty$-category pairs obtained from one another by cutting and pasting; such manoeuvres  form the backbone of most of the applications.
\begin{enumerate}[label=(\arabic*)]
\item \textbf{Boundary principle} (\cref{prop:boundary_principle}). \textit{Under mild finality conditions, the classifying system of the boundary is the once suspended restriction of the classifying system of the pair.}
\item \textbf{Local-to-global principle} (\cref{cor:local_to_global_principle}). \textit{If an $\infty$-category pair admits a finite complemented cover by smaller $\infty$-category pairs (cf. \cref{term:complemented_covers}), then it is a Poincar\'e $\infty$-category pair if and only if each of the covering $\infty$-category pairs is.}
\item \textbf{Localisation principle} (\cref{prop:localisation_principle}). \textit{If we have a localisation of $\infty$-category pairs  with Poincar\'e source, then the target  is Poincar\'e too.}
\item \textbf{Realisation principle} (\cref{prop:barXdXpdpair}). \textit{Under mild finality conditions, the geometric realisation of a Poincar\'e $\infty$-category pair is a Poincar\'e pair of spaces in the classical sense.}
\end{enumerate}

As a first application of the above principles, we highlight the following general form of Poincar\'e--Lefschetz duality.

\begin{alphThm}[Full version in \cref{prop:categorical_poincare_lefschetz_duality}]
\label{alphThm:PoincareLefschetz}
Let $(\X,\partial \X)$ be a groupoidally $\sC$-twisted ambidextrous $\infty$-category pair, i.e. the relative cohomology functor can be represented by means of a groupoidal dualising system $D_{\X,\partial\X}$, and denote by $i\colon\partial\X\to\X$ the inclusion functor. Under mild finality conditions, there is an equivalence of fibre sequences in $\func^L_\sC(\presheaf(\X;\sC),\sC)$:
\[
\begin{tikzcd}
\Omega\partial \X_* i^* (-) \ar[d, "\simeq"] \ar[r] &(\X,\partial \X)_*(-) \ar[r] \ar[d, "\simeq"] & \X_*(-) \ar[d, "\simeq"] \\
\partial \X_! i^*(D_{\X,\partial \X} \otimes -)\ar[r] &\X_!(D_{\X,\partial \X} \otimes -) \ar[r] & (\X,\partial \X)_!(D_{\X,\partial \X} \otimes -).
\end{tikzcd}
\]
\end{alphThm}
At this point, it behoves us to clarify the relationship of our definitions with known approaches in the literature. 
We give a full comparison in the case of most geometric interest, the case of Poincar\'e ads. 
Wall used ads to underpin his theory of surgery sets and L-groups in an essential way; they were taken up again later by Quinn \cite{quinnThesis,quinnBordism} in setting up a space-level approach to surgery theory which replaced surgery sets and groups with surgery spaces and related them to ``bordism-type'' spectra. In fact, our main motivation for writing this article precisely stems from our forthcoming work on cobordism $\infty$-categories and cobordism spectra; our approach will be based on moduli spaces of ad-like objects, and it will further Quinn's program \cite{pdcob}.

In current parlance, an $(n+1)$-ad is a diagram $X\colon [1]^n\to\spc^\omega$ indexed by a cube poset and valued in finitely dominated spaces. We call it a \textit{Poincar\'e $(n+1)$-ad} if the total category of its cocartesian unstraightening, relative to the subcategory corresponding to the subposet $[1]^{n} \setminus \{ (1,\dots 1) \}$ is a Poincar\'e duality category pair. 

Of course, when $\sC\in\{\spectra,\module_\eilenbergMacLaneCoeff\}$, there is also Wall's original  \cref{defn:classical_poincare_pairs} in terms of cap products on homology and cohomology groups. 
For more general $\sC$, one may mimic Wall's definition to get a new definition (\cref{defn:postclassical_poincare_pairs}) of Poincar\'e pairs using cap products. 
But one can also avoid cap products and employ the usual Morita theory for spaces to give a third definition (\cref{defn:neoclassical_poincare_pairs}) of Poincar\'e ads in terms of  the compatibilities of classifying systems for various pairs of spaces occurring in the ad. We then prove:

\begin{alphThm}[cf. {\cref{cor:seven_equivalent_properties,prop:wall_vs_us_ads}}]\label{alphThm:classical_vs_us}
For arbitrary $\sC\in\calg(\presentable^L_{\mathrm{st}})$, our definition of $\sC$-Poincar\'e ads coincide with the two definitions described above. For $\sC\in\{\spectra, \module_\eilenbergMacLaneCoeff\}$, these definitions also agree with Wall's classical definition of ads in the finitely dominated case.
\end{alphThm}

In light of \cref{alphThm:classical_vs_us}, one may fairly wonder what the point even is of introducing the  outwardly more complicated theory of Poincar\'e duality for $\infty$-categories. 
One argument for this is that this language cleanly packages all the local structures - of ads, say - into a global whole. And remarkably, manipulating everything at the level of these global incarnations permits us to deal with the multitude of local structures in one fell swoop by merely reasoning about a single pair of $\infty$-categories. 
For instance, we may immediately conclude that Poincar\'e punctured ads glue to yield a Poincar\'e space from the realisation principle above. The reader should try to prove this ``by hand'' to see that the combinatorial book-keeping, albeit not impossible to deal with,  escalates quite quickly. 

As a more involved application combining all the principles above together with the piecewise linear Sch\"onflies theorem, we obtain the following result, saying that glueing the spaces involved in a Poincar\'e diagram indexed by the face poset of triangulated manifolds results in a new Poincar\'e space; this generalises the case of punctured ads.
In the classical formulation, a proof of this would require intricate arguments providing a global compatibility of various local fundamental classes.
In the following theorem, Poincar\'e duality always refers to Poincar\'e duality with coefficients in $\spectra$.

\begin{alphThm}[Full and precise version in {\cref{prop:combinatorial_manifold_Poincare}}]\label{alphThm:combinatorial_manifold}
Let $I$ be the face poset of a triangulated closed manifold whose links of simplices are spheres, and let $X\colon I\to\spc^\omega$ be a functor. Then the $\infty$-category $\X\coloneqq\int_IX$, obtained by unstraightening, satisfies Poincar\'e duality if and only if for all elements $i \in I$ the pair of spaces $(X(i), \colim_{j < i} X(j))$ is a Poincar\'e pair.
In this case, the space $|\X|\simeq\colim_IX$ also satisfies Poincar\'e duality. 
\end{alphThm}

We now turn to the final application of our theory. Consider the following geometric situation:  let $p \colon E \rightarrow B$ be a fibre bundle where $B$ is a topological manifold with boundary, and the fibre $F$ is another topological manifold with boundary. Then $E$ is a topological manifold with boundary as well, and its boundary is the union of $\partial_{\mathrm{base}}E \coloneqq p^{-1}(\partial B)$ with the boundaries of all fibres. Concretely, the boundaries of all fibres assemble into a fibre bundle $\partial_{\mathrm{fib}}E \rightarrow B$,  and writing $\partial \partial E \coloneqq \partial_{\mathrm{fib}} E \cap \partial_{\mathrm{base}} E$, we have the decomposition of $\partial E$ as the union $\partial_{\mathrm{fib}} E \cup_{\partial \partial E} \partial_{\mathrm{base}}E$ of two manifolds along their common boundary. One may ask if the phenomenon above persists when passing to pure homotopy theory, i.e. after replacing topological manifolds with boundary by Poincar\'e pairs, or more generally Poincar\'e ads.
We give an affirmative answer to this question in the following result, working with general coefficients $\sC\in\calg(\presentable^L_{\mathrm{st}})$.

\begin{alphThm}[Precise version in {\cref{cor:kleinQinSu_integration}}]
\label{thm:ads_fibrations}
Let $B\colon \Box^n\rightarrow \spc^{\omega}$ be an $(m+1)$-ad of compact spaces with top space $\bar B=B(1,\dots,1)$ and suppose we have a family  $F\colon \Box^m\rightarrow \spc^\omega$ of $(n+1)$-ads of compact spaces parametrised by $\bar B$. Assume that, for all $b\in\bar\B$, the restriction $F|_{\{b\}\times\Box^n}\colon \Box^n\rightarrow\spc$ is not the empty functor. 
Then the following are equivalent:
\begin{itemize}
\item the ``total ad'' $E\in\func(\Box^m\times \Box^n,\spc)$  is a $\sC$-Poincar\'e $(m+n+1)$-ad;
\item both $B\colon \Box^m\rightarrow \spc$ and $F|_{\{b\}\times\Box^n}\colon \Box^n\rightarrow\spc$ for all $b\in\bar B$ are $\sC$-Poincar\'e ads. 
\end{itemize}
\end{alphThm}

Surprisingly, a proof of this result in the special case when the fibre and base are Poincar\'e pairs (and so the total diagram is a Poincar\'e triad) only appeared relatively recently as \cite[Thm. G]{KleinQinSu}. There, the proof comprised of a mix of homotopical and manifold methods.

In order to prove \cref{thm:fibred_dualising_object_factorisation}, of which \cref{thm:ads_fibrations} is an immediate consequence, we generalise the approach of \cite[Thm. E]{PD1}, which only considered fibrations $\mathcal{F} \rightarrow \E \xrightarrow{p} \B$ of $\sC$-Poincar\'e spaces (without boundary). 
The proof there crucially built upon the notion of ambidexterity for \textit{maps} of spaces. Cnossen \cite{Cnossen2023} defines a map $\E\rightarrow \B$ in a topos $\baseTopos$ to be \textit{$\sC$-twixted ambidextrous} if the object $\E\in\baseTopos_{/\B}$ is $\pi^*_\B\sC$-ambidextrous, where $\pi^*_\B\colon \cat_{\baseTopos}\rightarrow \cat_{\baseTopos_{/\B}}$ is the basechange functor. Out of this definition, Morita theory gives a ``fibrewise'' dualising system $D_p\in\presheaf(E;\sC)$, and in \cite{PD1} the equivalence $D_\E\simeq D_p\otimes p^*D_\B$ was established by considering how ambidextrous maps compose.

Coming back to our setting, an attempt to generalise ambidextrous maps between $\infty$-groupoids to ambidextrous maps between $\infty$-categories presents us with new complications. The essential reason is the following: while the $\infty$-category of presheaves on a groupoid is the limit (in large $\infty$-categories) of the constant diagram over the groupoid, the $\infty$-category of $\sC$-valued presheaves on a given $\infty$-category $\X$ is in general only the \textit{lax} limit of the constant diagram over $\X\op$ with value $\sC$. To overcome this issue, we will replace $\pi^*_\B\sC$ from Cnossen's theory with the ``cofree'' construction in order to handle arbitrary base $\infty$-categories. By employing Morita theory for the topos $\presheaf(\B)$, we are able to obtain appropriate classifying and dualising systems.

\subsection{Historical remarks}
As already mentioned, Poincar\'e spaces serve as the purely homotopical analogue of closed manifolds. They were introduced by Browder \cite{Browder1,Browder2} in the simply connected case and more generally by Spivak \cite{Spivak} and Wall \cite{Wall} later. Similarly, Poincar\'e  pairs serve as the purely homotopical analogue of manifolds with boundary. Spivak \cite[§3]{Spivak} gave the first definition, but the standard one is due to Wall \cite[§1]{Wall}. Wall's definitions of Poincar\'e spaces and pairs are in terms of the existence of local coefficient systems and fundamental classes inducing isomorphisms from cohomology to homology groups via cap product.
The definition of Poincar\'e  spaces via parametrised spectra first appeared in \cite{Kleindual} (see also \cite[Appendix]{LandPD} and \cite[page 93]{nineauthors} for more modern treatments), and a generalisation to Poincar\'e  pairs is given for example in \cite{LuriePoincare}.

Knowledge about relative Poincar\'e duality is abundant, but unfortunately, many statements are  ``known'' to experts and exist only either as folklore or have been proven in special cases. To improve the situation, \cite{KleinQinSu} provided the great service of comparing known definitions of Poincar\'e duality pairs in the literature as well as proving the fibration theorem for Poincar\'e triads.

\subsection{Overview and acknowledgements}
We start the article by recalling some basic facts about parametrised homotopy theory in \cref{sec:preliminaries}; our intention is to develop most of our results in the context of $\baseTopos$-categories, for a fixed $\infty$-topos $\baseTopos$. 

In \cref{part:foundations} 
we introduce and develop the  theory of ambidexterity and Poincar\'e duality for pairs of $\baseTopos$-categories, which will form the abstract cornerstone of this article. We shall mostly defer the examples, including the applications to classical Poincar\'e  ads of spaces, to \cref{part:applications}. We first set up the requisite ``Morita theory'' which will allow us to classify linear functors in terms of so-called \textit{classifying systems} in \cref{sec:morita}. We then introduce and study the basic theory of $\baseTopos$-category pairs in \cref{sec:pdcatpairs} with a special focus on those which are  \textit{twisted ambidextrous} and \textit{Poincar\'e}, the latter being a property of $\baseTopos$-category pairs defined in terms of the aforementioned classifying system. Lastly, we round off the abstract theory by bringing the cap products into the fold and explaining the miscellaneous rigidities enjoyed by the classifying system as well as our notions of ambidexterity and Poincar\'e duality.

In \cref{part:applications} we present some applications of our theory. We first deal with $\infty$-category pairs obtained by unstraightening diagrams of $\baseTopos$-spaces, usually parametrised by a plain 1-category (in fact, most of the times, by a finite poset) in \cref{sec:unstraightenedpairs}: we start by comparing the various notions of Poincar\'e pair and Poincar\'e ad, showing that our theory recovers the classical notions in the case of spaces; we then generalise to the case in which the poset parametrising our diagram of $\baseTopos$-spaces is a \textit{combinatorial manifold}, proving a local-to-global principle to recognise Poincar\'e duality in this setting; and we conclude the section by analysing a few more specific examples. We then present the fibration theorem in \cref{sec:fibredpd}, and give more applications including the Thom isomorphism for spherical fibrations.

Finally, in the appendices gathered in \cref{part:appendices},
we record some ``folklore'' results about parametrised category theory - mainly of a technical nature - which we have not been able to find explicit statements and proofs of in the literature.
\vspace{.4cm}

We are grateful to Markus Land, Wolfgang L\"uck, and Connor Malin for conversations that helped writing this paper.
KH is supported by the European Research Council (ERC) under Horizon Europe (GeoCats,
ID: 101042990 and BorSym, ID: 101163408). KH  thanks the University of Bonn and 
DK thanks the Ludwig-Maximilians-Universität München for their conducive working environments.
All authors would like to thank the Max Planck Institute for Mathematics (MPIM) in Bonn for its hospitality.

\section{Preliminaries}
\label{sec:preliminaries}
In this section we recall some material on parametrised category theory, for which suitable references are given, and we fix some notation to be used throughout the article. We postpone to the appendices in \cref{part:appendices} the development of some further tools, for which we were unable to find sources in the literature.
\subsection{Recollections on parametrised category theory}
\label{subsec:recollections}
In this article we make use of parametrised homotopy theory over a generic topos $\baseTopos$ as developed in \cite{Martini2022Yoneda,Martini2022Cocartesian,MartiniWolf2022Presentable,MartiniWolf2024Colimits}.
The reader who is only interested in the classical setting of Poincar\'e spaces can safely skip this subsection and just think of the case $\baseTopos=\spc$, where the language of (presentable, symmetric monoidal, etc) $\baseTopos$-categories collapses to the world of ordinary (presentable, symmetric monoidal, etc) $\infty$-categories.
Let us begin by recalling the basic setup of parametrised homotopy theory.
A more detailed exposition can be found in \cite[Section 2]{Cnossen2023} or \cite[Section 2]{PD1}.

\begin{description}
\item[($\baseTopos$-categories)] A \textit{$\baseTopos$-category} is a limit preserving functor $\baseTopos\op \to \cat$.
We denote by $\cat_{\baseTopos} \subseteq \func(\baseTopos\op, \cat)$ the (large) $\infty$-category of small $\baseTopos$-categories. By considering $\cat$ as a large $(\infty,2)$-category, we have that $\func(\baseTopos\op,\cat)$ is an $(\infty,2)$-category as well, and we may consider $\cat_\baseTopos$ as a full $(\infty,2)$-subcategory spanned by the aforementioned class of objects. A \textit{$\baseTopos$-functor} is a 1-morphism in $\cat_\baseTopos$, i.e. a natural transformation of functors $\baseTopos\op\to\cat$; a \textit{$\baseTopos$-natural transformation} is similarly a 2-morphism in $\cat_\baseTopos$. 
  
Passing to a larger universe and repeating the definition, we also consider the huge $\infty$-category $\widehat{\cat}_{\baseTopos} \subseteq \func(\baseTopos\op, \widehat{\cat})$ of large $\baseTopos$-categories.
One can even define a large $\baseTopos$-category of $\baseTopos$-categories, denoted $\catTopos_\baseTopos$, by setting $\catTopos_{\baseTopos}(\tau) = \cat_{\baseTopos_{/\tau}}$ for $\tau\in\baseTopos$. We then have $\catTopos_\baseTopos\in\widehat{\cat}_\baseTopos$.
 
\item[($\baseTopos$-groupoids)] The Yoneda embedding $\yoneda\colon\baseTopos \hookrightarrow \func(\baseTopos\op, \spc)$ restricts to a fully faithful functor $\baseTopos \hookrightarrow \cat_{\baseTopos}$.
We call objects in the essential image of this functor \textit{$\baseTopos$-groupoids}.
  The essential images of these functors, for varying topos among those of the form $\baseTopos_{/\tau}$ for $\tau\in\baseTopos$, assemble into the $\baseTopos$-category $\animaTopos_{\baseTopos} \subseteq \catTopos_{\baseTopos}$ of \textit{$\baseTopos$-groupoids}, given by $\animaTopos_{\baseTopos}(\tau) = \baseTopos_{/\tau} \subseteq \cat_{\baseTopos_{/\tau}}$. (This is denoted by $\Omega_{\baseTopos}$ in \cite{Martini2022Yoneda}).
  
  The inclusion $\baseTopos$-functor $\animaTopos_\baseTopos\hookrightarrow\catTopos_\baseTopos$ admits a left and a right adjoint $\baseTopos$-functors in the 2-category $\widehat{\cat}_\baseTopos$, which are denoted by $|-| $ and $(-)^\simeq\colon\catTopos_\baseTopos\to\animaTopos_\baseTopos$, referred to as the \textit{realisation functor} and \textit{core functor}, respectively. A $\baseTopos$-functor $\D\to\E$ is \textit{groupoidal} if it factors through $|\D|$, or equivalently, through $\E^\simeq$.
  \item[(Objects and morphisms)] An \textit{object} of a $\baseTopos$-category $\D$ consists of an object $t \in \baseTopos$ (sometimes called a ``level'') together with an object $x \in \D(t)$.
  Equivalently by the Yoneda lemma, it is a functor $x \colon t \to \D$ of $\baseTopos$-categories. Similarly, a morphism in $\D$ is the data of $t\in\baseTopos$ and a morphism in $\D(t)\in\cat$, or equivalently, an object in $\D^{[1]}(t)\simeq \D(t)^{[1]}\in\cat$ (cf. the point on functor $\baseTopos$-categories below for the notation $\D^{[1]}$).
  \item[(Full and wide inclusions)] Given a $\baseTopos$-functor $F\colon\D\to\E$, we say that $F$ exhibits $\D$ as a \textit{full} (resp. \textit{wide}) $\baseTopos$-subcategory of $\E$ if $F(\tau)$ is a full or wide $\infty$-subcategory inclusion, respectively, for all $\tau\in\baseTopos$.
  
\item[(Geometric basechange and global sections)] Consider a geometric morphism $f^* \colon \baseTopos \rightleftharpoons \baseTopos' \cocolon f_*$ of topoi (i.e. $f^* $ is a left adjoint and it is left exact).
There is an induced adjoint pair $f^* \colon \cat_{\baseTopos} \rightleftharpoons \cat_{\baseTopos'} \cocolon f_*$ whose right adjoint $f_*$ is given by restriction along $(f^*)\op \colon \baseTopos\op \to (\baseTopos')\op$.
For the unique geometric morphism $\constant \colon \spc \rightleftharpoons \baseTopos \cocolon \Gamma$ we obtain the adjunction $\constant \colon \cat \rightleftharpoons \cat_{\baseTopos} \cocolon \Gamma$
and refer to $\Gamma$ as the \textit{global sections} functor. When there is danger of confusion, we will write $\Gamma_{\baseTopos}\colon \cat_{\baseTopos}\rightarrow\cat$ to remind the reader to which topos we are applying the global sections functor.
  
Note that $\Gamma(\D)$ agrees with $\D(\ast)$, the evaluation of $\D$ at the terminal object of $\baseTopos$.  We often omit ``$\constant$'' and for $I\in\cat$ we write $I\in\cat_\baseTopos$ for $\constant(I)$. As important examples, we have equivalences $\Gamma\catTopos_{\baseTopos}\simeq \cat_{\baseTopos} $ and $ \Gamma\animaTopos_{\baseTopos}\simeq \spc_{\baseTopos}$ in $\widehat{\cat}$.
\item[(\'Etale basechange)]
For an $\infty$-topos $\baseTopos$ and an object $\tau\in\baseTopos$ the forgetful functor
$(\pi_\tau)_! \colon \baseTopos_{/\tau} \rightarrow \baseTopos$
has two iterated right adjoints $(\pi_\tau)_! \dashv \pi_\tau^* \dashv (\pi_\tau)_*$, so that in particular, $\pi_\tau^* \dashv (\pi_\tau)_*$ describes a geometric morphism. Geometric morphisms arising this way are called \textit{\'etale geometric morphisms}, and are characterised in \cite[Prop. 6.3.5.11.]{lurieHTT}.
For an \'etale geometric morphism, the functor $\pi_\tau^* \colon \cat_{\baseTopos} \rightarrow \cat_{\baseTopos_{/\tau}}$ has a left adjoint $(\pi_\tau)_!$ which identifies $\cat_{\baseTopos_{/\tau}} \simeq (\cat_\baseTopos)_{/\tau}$, where $\tau\in\baseTopos \subset\cat_\baseTopos$ is regarded as a $\baseTopos$-groupoid, see \cite[Sec. 3.3.]{Martini2022Yoneda}. As a principle, all statements made in internal higher category should be stable under \'etale base change. We refer the reader to \cite[Sec. 2.14.]{MartiniWolf2024Colimits} for a longer elaboration on this principle.
  
\item[(Functor $\baseTopos$-categories)] 
For $\D,\E\in\cat_\baseTopos$ we denote by $\func_{\baseTopos}(\D,\E)\in\cat$ the $\infty$-category of morphisms from $\D$ to $\E$ in $\cat_\baseTopos$, and by $\map_{\baseTopos}(\D,\E)\coloneqq\func_{\baseTopos}(\D,\E)^\simeq$ the core groupoid of the previous.
Moreover, the $\infty$-category $\cat_\baseTopos$ is cartesian closed; we denote by $\funTopos_{\baseTopos}(-, -)\colon\cat_\baseTopos\op\times\cat_\baseTopos\to\cat_\baseTopos$ the internal hom, associating to $\D,\E\in\cat_\baseTopos$ the $\baseTopos$-category of $\baseTopos$-functors; we abbreviate by $\mapTopos_\baseTopos(\D,\E)\simeq \funTopos_\baseTopos(\D,\E)^\simeq\in\spc_\baseTopos$ and we also write $\D^{\E}\in\cat_{\baseTopos}$ for $\funTopos_{\baseTopos}(\D,\E)$. There are then equivalences $\Gamma\funTopos_{\baseTopos}(\D,\E)\simeq \funTopos_\baseTopos(\D,\E)(\ast)\simeq\func_{\baseTopos}(\D,\E)\in\cat$ and $\Gamma\mapTopos_{\baseTopos}(\D,\E)\simeq \mapTopos_\baseTopos(\D,\E)(\ast)\simeq\map_{\cat_{\baseTopos}}(\D,\E)\in\spc$. For functors from constant $\infty$-categories, we have the following description: for $I\in\cat$ and $t\in\baseTopos$, $\D^I(t)\simeq \D(t)^I$.

Objects in functor $\baseTopos$-categories can be described explicitly at all levels:
for $\tau\in\baseTopos$, there is an identification $\funTopos_{\baseTopos}(\D,\E)(\tau) \simeq \Gamma_{\baseTopos_{/\tau}}(\funTopos_{\baseTopos_{/\tau}}(\pi_\tau^*\D, \pi_\tau^*\E))$.

\item[(Slices)] For $\tau\in\baseTopos$ and $x\in\D(\tau)$, i.e. for a $\baseTopos$-functor $x\colon\tau\to\D$, the slice $\baseTopos$-category $\D_{/x}$ is defined as the pullback $\D_{/x} = \tau \times_{\D} \D^{[1]}$ of the maps $x \colon \tau \to \D$ and $\eval_1 \colon \D^{[1]} \to \D$.

\item[(Adjoints and (co)limits)]
Since $\cat_\baseTopos$ is an $(\infty,2)$-category, we may consider adjunctions inside it.
Unwinding this notion, a $\baseTopos$-functor $F \colon \D \to \E$ admits a left (resp. right) adjoint if and only if $F(\tau) \colon \D(\tau) \to \E(\tau)$ admits a left (resp. right) adjoint for all $\tau\in\baseTopos$ and these satisfy the left (resp. right) Beck--Chevalley condition.
One also defines the notions of $\baseTopos$-colimits and $\baseTopos$-limits: given $\baseTopos$-categories $\X$ and $\D$, we consider the $\baseTopos$-functor $\constant_\X \colon \D \to \funTopos_{\baseTopos}(\X, \D)$, and denote by $\colim_\X,\lim_\X\colon\funTopos_\baseTopos(\X,\D)\to\D$ the left and right adjoint, if they exist, respectively. We will often abbreviate $\X^*=\constant$, $\X_!=\colim_\X$ and $\X_*=\lim_\X$. 
  
A $\baseTopos$-category $\D$ is called \textit{cocomplete} if the $\baseTopos_{/\tau}$-category $\pi_\tau^* \D$ admits colimits indexed by objects in $\cat_{\baseTopos_{/\tau}}$ for all $\tau\in\baseTopos$. Generally, existence and preservation of colimits can be checked by reducing to constant $\baseTopos$-categories and $\baseTopos$-groupoids, see \cite[Prop. 4.7.1.]{MartiniWolf2024Colimits}. For these special cases, Martini-Wolf offer explicit criteria \cite[Examples 4.1.13., 4.1.14., 4.2.6., 4.2.7.]{MartiniWolf2024Colimits}, all of which we shall freely use.
A $\baseTopos$-functor $f\colon\X\to\Y$ is \textit{final} if for all $\baseTopos$-functors $F\colon\Y\to\D$ the canonical map $\colim_\X F\circ f\to\colim_\Y F$ is an equivalence; $f$ is \textit{initial} if $f\op\colon\X\op\to\Y\op$ is final.
A functor $f \colon \sC \rightarrow \D$ of complete/cocomplete $\baseTopos$-categories is said to be limit/colimit-preserving if it preserves all $\baseTopos$-limits/colimits.

\item[(Presentability)] A (large) $\baseTopos$-category $\D \colon \baseTopos\op \to \widehat{\cat}$ is \textit{presentable} if it factors through the $\infty$-subcategory $\presentable^L \subset \widehat{\cat}$ and satisfies the following two conditions:
\begin{enumerate}
\item For any map $f \colon \tau \to \tau'$ in $\baseTopos$ the map $f^* \colon \D(\tau') \to \D(\tau)$ admits a left adjoint $f_!$.
\item For any cartesian square in $\baseTopos$ 
\[
\begin{tikzcd}[row sep=10pt]
\tau_1' \ar[r, "g'"] \ar[d, "f'"'] \ar[dr,phantom,"\lrcorner"very near start]& \tau_1 \ar[d, "f"] \\
\tau_2' \ar[r, "g"] & \tau_2,
\end{tikzcd}
\]
the Beck--Chevalley transformation $f'_! (g')^* \to g^* f_!$ of functors $\D(\tau_1) \to \D(\tau_2')$ is an equivalence.

\end{enumerate}
Denote by $\presentable_{\baseTopos}^L \subset \widehat{\cat}_{\baseTopos}$ the nonfull $(\infty,2)$-subcategory of presentable $\baseTopos$-categories, $\baseTopos$-left adjoint functors, and all $\baseTopos$-natural transformations between them.
By the adjoint functor theorem for $\baseTopos$-categories, a $\baseTopos$-functor $F \colon \D \to \E$ of presentable $\baseTopos$-categories is a left adjoint if and only if $F(\tau) \colon \D(\tau) \to \E(\tau)$ preserves colimits for all $\tau\in\baseTopos$ and for any map $f \colon \tau\to\tau'$ in $\baseTopos$ the Beck--Chevalley transformation $f_! F(\tau) \to F(\tau') f_!$ is an equivalence.
The $\infty$-categories $\presentable_{\baseTopos_{/\tau}}^L$ for varying $\tau\in\baseTopos$ again assemble into the $\baseTopos$-category of presentable $\baseTopos$-categories, denoted $\presentableTopos_\baseTopos^L$.
We will also write $\presentableTopos^L_{\baseTopos,\stable}\subseteq \presentableTopos^L_{\baseTopos}$ for the full $\baseTopos$-subcategory spanned at level $\tau\in\baseTopos$ by those presentable $\baseTopos_{/\tau}$-categories which are levelwise stable.
The $\baseTopos$-category $\presentableTopos^L_{\baseTopos}$ has a symmetric monoidal structure \cite[Section 2.6]{MartiniWolf2022Presentable} given by a parametrised version of Lurie's tensor product \cite[Section 4.8]{lurieHA}; we may therefore regard $\presentableTopos^L_{\baseTopos}$ as an object in $\calg(\widehat{\cat}_\baseTopos)$. 

\item[(Simplicial $\baseTopos$-objects)]
We denote by $\simplex$ the category of finite, nonempty totally ordered sets and nondecreasing maps; for $n\ge0$ we let $[n]=\{0\prec\dots\prec n\}\in\simplex$. We sometimes regard $\cat$ as the full $\infty$-subcategory of $s\spc\coloneqq\func(\simplex\op,\spc)$ spanned by complete Segal objects. Here $\spc\subset\cat$ denotes the $\infty$-category of small $\infty$-groupoids, also known as ``spaces''.

For $\E\in\cat_\baseTopos$ and $I\in\cat$ we have identifications of $\infty$-categories $\func_\baseTopos(\constant(I),\E)(\tau)\simeq\func(I,\E(\tau))$. Taking $I=\simplex\op$, we obtain the $\baseTopos$-category $s\E\coloneqq\funTopos_\baseTopos(\simplex\op,\E)\simeq\E^{\simplex\op}$. 
Passing to a larger universe we may moreover set $\E=\animaTopos_\baseTopos\in\widehat{\cat}_\baseTopos$, thus getting the large $\baseTopos$-category $s\animaTopos_\baseTopos\in\widehat{\cat}_\baseTopos$. For $\tau\in\baseTopos$, the objects of the $\infty$-category $s\animaTopos_\baseTopos(\tau)$ are limit preserving functors $\baseTopos_{/\tau}\op\to s\spc$.
We may then regard $\catTopos_\baseTopos$ as a full $\baseTopos$-subcategory of $s\animaTopos_\baseTopos$, having at level $\tau\in\baseTopos$ the full $\infty$-subcategory of $\func(\baseTopos_{/\tau},s\spc)$ spanned by functors $\baseTopos_{/\tau}\to s\spc$ that are limit preserving and attain complete Segal spaces as values. In particular, for $\tau=\ast$ being terminal in $\baseTopos$, we obtain a fully faithful inclusion of $\cat_\baseTopos$ inside $s\baseTopos$; concretely, we associate with $\D\in\cat_\baseTopos$ the simplicial object $[n]\mapsto\D_n\coloneqq\mapTopos_\baseTopos(\constant([n]),\D)\in\baseTopos$.
  
\item[(Straightening)]
For a $\baseTopos$-category $\D$, we will consider the nonfull, large $\infty$-subcategories $\cocartesianCategory^{\baseTopos}_{/\D}, \lfib^{\baseTopos}_{/\D} \subset (\cat_{\baseTopos})_{/\D}$ of $\baseTopos$-cocartesian fibrations (cf. \cite[Def. 3.1.1]{Martini2022Cocartesian} or \cref{rec:parametrised_cocartesian_fibrations}) and $\baseTopos$-left fibrations (c.f \cite[Def. 4.1.1]{Martini2022Yoneda}) over $\D$, and maps of $\baseTopos$-cocartesian of $\baseTopos$-left fibrations, respectively. 
They  participate in the following straigthening-unstraightening equivalence of $\infty$-categories:
\begin{equation}\label{eq:straightening_cocart}
 \straighten \colon \cocartesianCategory^{\baseTopos}_{/\D} \simeq \func_{\baseTopos}(\D, \catTopos_{\baseTopos}) \cocolon \unstraighten,
 \end{equation}
restricting to the following equivalence, see \cite[Thm. 4.5.1]{Martini2022Yoneda}:
\begin{equation}\label{eq:straightening_lfib}
\straighten \colon \lfib^{\baseTopos}_{/\D} \simeq \func_{\baseTopos}(\D, \animaTopos_{\baseTopos}) \cocolon \unstraighten.
\end{equation}
We often denote by $\int_\D X\in\cat_\baseTopos$ the source of a $\baseTopos$-cocartesian fibration corresponding to the $\baseTopos$-functor $X\colon \D\to\catTopos_\baseTopos$.

\item[(Twisted arrow construction)]
For $\X\in\cat_\baseTopos$, we denote by $\twistedArrow(\X)\in\cat_\baseTopos$ the twisted arrow $\baseTopos$-category of $\X$ (e.g. constructed in \cite[Def. 4.2.4]{Martini2022Yoneda}).
Objects at level $\tau\in\baseTopos$ are given by morphisms in $\X(\tau)$; a morphism in $\twistedArrow(\X)(\tau)$ from $f \colon x \to y$ to $f' \colon x' \to y'$ is given by a commutative square in $\X(\tau)$ of the form
\[
\begin{tikzcd}[row sep=10pt]
  x\ar[d,"f"']
  & x' \ar[d, "f'"]\lar
  \\
  y  \rar
  &y'.
\end{tikzcd}
\]
Formally, the underlying simplicial $\baseTopos$-space of $\twistedArrow(\X)$ is the composite functor $\simplex\op\xrightarrow{[-]\ast[-]\op}\simplex\op\xrightarrow{\X}\baseTopos$, where $[-]\op\colon\simplex\to\simplex$ is the non-trivial involution, and where $\ast$ denotes the concatenation monoidal product on $\simplex$. In particular, for $t\in\baseTopos$, we have $(\twistedArrow(\X))(t)\simeq\twistedArrow(\X(t))$.

Taking sources and targets of arrows gives rise to a $\baseTopos$-left fibration $(s,t)\colon\twistedArrow(\X)\to \X\op\times \X$, which classifies a  $\baseTopos$-functor $\mapTopos_{\X} \colon \X\op \times \X \to \animaTopos_{\baseTopos}$ under the straightening equivalence \eqref{eq:straightening_lfib}.

We may consider $\mapTopos_{\X}$ also as an object in the $\infty$-category $\presheaf_\baseTopos(\X\times\X\op)\coloneqq\func_\baseTopos(\X\op\times\X,\animaTopos_\baseTopos)$ of $\animaTopos_\baseTopos$-valued presheaves. In this case we have $\mapTopos_{\X}\simeq(s,t)_!(\ast)$, i.e. $\mapTopos_{\X}$ is the left Kan extension of the terminal presheaf$\ast$ along $(s,t)$.

\item[(Presheaves and Yoneda embedding)]
By currying $\mapTopos_{\X}$, we get the \textit{Yoneda embedding} $\yoneda \colon \X \hookrightarrow \presheafTopos_{\baseTopos}(\X)$, see \cite[Thm. 4.7.8]{Martini2022Yoneda}, where $\presheafTopos_{\baseTopos}(\X)\coloneqq\funTopos_\baseTopos(\X\op,\animaTopos_\baseTopos)\in\widehat{\cat}_\baseTopos$ denotes the $\baseTopos$-category of $\animaTopos_\baseTopos$-valued presheaves. When the topos is clear, we will often abbreviate this to $\presheafTopos(\X)$.
  The $\baseTopos$-Yoneda lemma states that there is an equivalence between the $\baseTopos$-functors $\eval \simeq \mapTopos_{\presheafTopos(\X)} \circ (\yoneda \times \id) \colon \X\op \times \presheafTopos(\X) \to \animaTopos_{\baseTopos}$.
  
  For a small $\baseTopos$-category $\X$, the large $\baseTopos$-category $\presheafTopos(\X)$ is presentable and restriction along $\yoneda$ induces for any cocomplete $\baseTopos$-category $\D$ an equivalence of $\baseTopos$-categories
  \begin{equation}\label{eq:presheaves_free}
    \yoneda^* \colon \funTopos^L_\baseTopos(\presheafTopos(\X), \D) \xrightarrow{\simeq} \funTopos_\baseTopos(\X, \D)
  \end{equation}
  by \cite[Theorem 7.1.1]{MartiniWolf2024Colimits}. Here $\funTopos_\baseTopos^L(\presheafTopos(\X),\D)$ is defined as the full $\baseTopos$-subcategory of $\funTopos_\baseTopos(\presheafTopos(\X),\D)$ spanned levelwise by $\baseTopos$-colimit-preserving functors.
The inverse equivalence is given by left Kan extension $\yoneda_! \colon \funTopos_{\baseTopos}(\X, \D) \to  \funTopos_{\baseTopos}(\presheafTopos(\X), \D)$ along the Yoneda embedding, which factors through the $\baseTopos$-subcategory $\funTopos^L_\baseTopos(\presheafTopos(\X), \D) \subseteq \funTopos_\baseTopos(\presheafTopos(\X), \D)$. Hence, the Yoneda embedding exhibits $\presheafTopos(\X)$ as the free cocompletion of $\X$.

  More generally, for $\D\in\widehat{\cat}_{\baseTopos}$, we write $\presheafTopos_{\baseTopos}(\X;\D) = \funTopos_\baseTopos(\X\op,\D)\in\cat_{\baseTopos}$ for the $\baseTopos$-categories of $\D$-valued presheaves. For ease of notation, we will also abbreviate this to $\presheafTopos(\X;\D)$ when the base topos $\baseTopos$ is clear. In particular, we have by definition that $\presheafTopos(\X) = \presheafTopos(\X; \spc_{\baseTopos})$.

\end{description}

\begin{example}[Presheaf topoi]
\label{ex:presheaftopoi}
For a small $\infty$-category $\B\in\cat$ we may consider the presheaf topos $\baseTopos \coloneqq\presheaf(\B) = \func(\B\op,\spc)$.
Then restriction along the Yoneda embedding $\yoneda\colon\B \to\presheaf(\B)$ induces an equivalence $\cat_{\presheaf(\B)}\simeq \func(\B\op, \cat)$.
In particular, for $\B =*$ we obtain $\cat_\spc \simeq \cat$.
\end{example}

\begin{rmk}[Checking equivalences on objects]
Given a natural transformation $(\alpha\colon F\rightarrow G)\colon \sC\rightarrow \D^{[1]}$ of $\baseTopos$-functors, we have that $\alpha$ is a natural equivalence if and only if for all objects in $\sC$ (i.e. for all $\tau\in \baseTopos$ and $x\in\sC(\tau)$, as recalled in the list of recollections above), the morphism $\alpha_x\colon F(x)\rightarrow G(x)$ in the ordinary $\infty$-category $\D(\tau)\in\cat$ is an equivalence. In other words, a natural transformation of $\baseTopos$-functors may be checked to be a natural equivalence ``objectwise'', where objects are to be interpreted in the preceding sense. 
\end{rmk}

\subsection{Summary of notations}
We collect here the important notational conventions we adopt throughout the article for the convenience of the reader. Let $\baseTopos$ be a fixed base topos.

\begin{enumerate}
\item We will denote by $\spc_{\baseTopos}$, $\cat_{\baseTopos}$ and $ \presentable^L_{\baseTopos}$ the ordinary $\infty$-categories of $\baseTopos$-spaces, of $\baseTopos$-categories, and of presentable $\baseTopos$-categories and left adjoint $\baseTopos$-functors as morphisms, respectively. Of course, $\spc_\baseTopos$ is just different notation for $\baseTopos$, chosen to confirm with the internal enhancement below.
All of $\spc_{\baseTopos}$,  $\cat_{\baseTopos}$ and $\presentable^L_{\baseTopos}$ belong to $\widehat{\cat}$, the (huge) $\infty$-category of large $\infty$-categories. 
In contradistinction to this, we shall denote by $\animaTopos_{\baseTopos}$, $\catTopos_{\baseTopos}$ and $ \presentableTopos^L_{\baseTopos}$ their respective $\baseTopos$-categorical enhancements, i.e. these are objects in the huge $\infty$-category of large $\baseTopos$-categories $\widehat{\cat}_{\baseTopos}$. 
In particular, we have equivalences $\Gamma\animaTopos_{\baseTopos}\simeq \spc_{\baseTopos}$, $ \Gamma\catTopos_{\baseTopos}\simeq \cat_{\baseTopos}$ and $ \Gamma\presentableTopos^L_{\baseTopos}\simeq \presentable^L_{\baseTopos}$. 
    
\item For $I\in\cat_{\baseTopos}$, we will write $\cocartesianCategory^{\baseTopos}_{/I}$ (resp. $ \cartesianCategory^{\baseTopos}_{/I}$) $\in\widehat{\cat}$ for the nonfull ordinary $\infty$-subcategory of $\cat^{\baseTopos}_{/I}$ spanned by cocartesian (resp. cartesian) fibrations and cocartesian (resp. cartesian) functors over $I$. Similarly, we will write $\cocartesianCatTopos^{\baseTopos}_{/I}, \cartesianCatTopos^{\baseTopos}_{/I}\subset \catTopos_{\baseTopos}$ for the respective nonfull $\baseTopos$-subcategories.
\item Let $\X, \E\in\cat_{\baseTopos}$, $\D\in\widehat{\cat}_{\baseTopos}$. We will write $\func_{\baseTopos}(\X,\E)\in\cat$ for the ordinary $\infty$-category of $\baseTopos$-functors from $\X$ to $\E$ and $\funTopos_{\baseTopos}(\X,\E)\in\cat_{\baseTopos}$ for the $\baseTopos$-category of $\baseTopos$-functors from $\X$ to $\E$. We also abbreviate $\E^\X\coloneqq\funTopos_{\baseTopos}(\X,\E)$, so that we have $\func_{\baseTopos}(\X,\E)\simeq\Gamma\E^\X$.
Similarly, we will write $\presheaf_{\baseTopos}(\X;\D)\coloneqq \func_{\baseTopos}(\X\op,\D)\in\cat$ for the ordinary $\infty$-category of $\baseTopos$-presheaves, and $\presheafTopos_{\baseTopos}(\X;\D)\coloneqq \funTopos_{\baseTopos}(\X\op,\D)$ for the $\baseTopos$-category of $\baseTopos$-presheaves. Whenever the topos $\baseTopos$ is clear from the context, we abbreviate $\presheafTopos_\baseTopos$ by $\presheafTopos$ and $\funTopos_{\baseTopos}$ by $\funTopos$.
\item For a $\baseTopos$-functor $f\colon \X \to \Y$ we denote by $f^*\colon\D^\Y\to\D^\X$, respectively $f^*\colon\presheafTopos(\Y;\D)\to\presheafTopos(\X;\D)$ the associated restriction $\baseTopos$-functor, and we denote by $f_!$ and $f_*$ the left and right adjoints of $f^*$, if they exist. When $\Y \simeq \ast$, we will often write the map $f$ as $\X\colon \X\rightarrow\ast$, so that  $\X^*\colon\D\to\D^\X$ and $\X^*\colon\D\to\presheafTopos(\X;\D)$ denote the respective diagonal $\baseTopos$-functors. Correspondingly,  $\X_!$ and $\X_*$ will denote the respective left and right adjoints, if they exist.
\item We usually denote by $\sC$ a presentably symmetric monoidal $\baseTopos$-category, i.e. $\sC\in\calg(\presentable^L_\baseTopos)$. From \cref{sec:pdcatpairs} onwards, with the exception of the appendices, we assume that $\sC \in \calg(\presentable^L_{\baseTopos,\stable})$ is also levelwise stable.
\end{enumerate}

\part{Foundations}\label{part:foundations}
\section{Morita theory}
\label{sec:morita}
Morita theory in higher algebra says that for $S,S' \in \alg(\spectra)$, evaluation at $S$ gives rise to an equivalence $\func^L(\lmodule_S,\lmodule_S') \xrightarrow{\simeq}  \mathrm{BiMod}_{S,S'}$, so that a bimodule ``classifies'' a linear functor.
This phenomenon can be generalised to the situation in which module $\infty$-categories are replaced by $\infty$-categories of local systems over a space, taking values in some $\sC \in \calg(\presentable^L)$: this is implemented in the setting of parametrised category theory in  \cite{Cnossen2023}.
Morita theory is the foundation upon which the implementation of Poincar\'e duality in the context of parametrised homotopy is built. For example, if $\X\in\spc^\omega$ is a compact space, the equivalence $\func^L(\spectra^\X,\spectra) \simeq \spectra^\X$ allows us to assign a dualising object $D_\X\in\spectra^\X$ to the colimit-preserving limit functor $\X_* \in \func^L(\spectra^\X,\spectra)$; following \cite{Kleindual} we then say that $\X$ is a Poincar\'e duality space if $D_\X$ takes values in $\picardSpace(\spectra)$, the full subgroupoid of $\spectra^\simeq$ spanned by invertible spectra.

In \cref{subsec:Morita_theory}, we provide a version of Morita theory for local systems over \textit{$\infty$-categories} instead of spaces. To accommodate potential examples in the study of Poincar\'e duality in more general contexts (e.g. for the study of equivariant Poincar\'e duality), we develop the theory in the generality of local systems on a $\baseTopos$-category $\X$ with coefficients in a presentably symmetric monoidal $\baseTopos$-category $\sC$.
The goal is to classify colimit-preserving and $\sC$-linear $\baseTopos$-functors $F\colon \presheafTopos(\X;\sC) \rightarrow \sC$ in terms of \textit{copresheaves}, i.e. covariant $\baseTopos$-functors $\omega_F\colon\X\to\sC$, which we call the \textit{classifying system}. This is achieved in \cref{prop:classification_of_linear_functors}. As complements to this, we also record various naturalities enjoyed by the Morita equivalence in \cref{subsec:Morita_theory} which we will use repeatedly throughout the article.

As hinted at above, a new subtlety appears in passing from spaces to $\infty$-categories: the classifying system is no longer a presheaf but a copresheaf. This precludes a simplistic Morita classification, as the following simple example memorably illustrates. 

\begin{example}
\label{ex:interval_has_no_dualising_system}
Let 
$\X=[1]=\{0\to1\}$. Then $[1]_!\colon\presheafTopos([1];\sC) \rightarrow \sC$ is evaluation at $0$, while $[1]_*$ is evaluation at $1$ (in particular, it is $\sC$-linear and colimit preserving). Unless $\sC=0$, there is no $\xi \in \func([1]\op,\sC)$ for which there is an equivalence $[1]_* \simeq [1]_!(-\otimes \xi)$. 
\end{example}

Nevertheless, there is still a formula re-expressing a linear functor in terms of its classifying system in the form of a \textit{coend}, which we provide in \cref{prop:formula_with_twisted_arrow}. We will have use of such a formula for many of our applications in \cref{part:applications} where we need some control over the classifying systems. Even more importantly, we show in \cref{subsec:classifying_vs_dualising} that the situation improves significantly when the classifying system $\X\rightarrow\sC$ is \textit{groupoidal}, in that it factors through the groupoid $|\X|$. In this case, we do, for instance, have the expected formula $\X_*(-)\simeq \X_!(D_\X\otimes-)$ for some essentially unique \textit{presheaf} $D_{\X}\in\presheaf_{\baseTopos}(\X;\sC)$ which we term as the \textit{dualising system}.

\subsection{Classification of linear functors}
\label{subsec:Morita_theory}
Given $\sC \in \calg(\presentable^L_{\baseTopos})$, we consider the symmetric monoidal $(\infty,2)$-category $\module_{\sC}(\presentable^L_\baseTopos)$ whose objects are referred to as \textit{$\sC$-linear $\baseTopos$-categories}. For $\D,\E\in\module_{\sC}(\presentable^L_\baseTopos)$ we denote by $\func^L_{\baseTopos,\sC}(\D,\E)$ the corresponding morphism $\infty$-category. The symmetric monoidal structure of $\module_{\sC}(\presentable^L_\baseTopos)$ is closed, and we denote by 
\[
\funTopos^L_{\baseTopos,\sC}(-,-)\colon \module_{\sC}(\presentable^L_\baseTopos)
\op\times \module_{\sC}(\presentable^L_\baseTopos)\to \module_{\sC}(\presentable^L_\baseTopos),
\]
the internal hom, so that $\Gamma \funTopos^L_{\baseTopos,\sC}(\D,\E)\simeq \func^L_{\baseTopos,\sC}(\D,\E)$. As in \cref{sec:preliminaries}, we will often abbreviate this to $\funTopos^L_{\sC}$ when the base topos $\baseTopos$ is clear.

There is a symmetric monoidal base change functor $- \otimes \sC \colon \presentable^L_\baseTopos \rightarrow \module_\sC(\presentable^L_\baseTopos)$. In particular, since for any $\X\in\cat_\baseTopos$, the pointwise symmetric monoidal structure refines $\animaTopos^{\X}$ to an object in $\calg(\presentable^L_{\baseTopos})$,  $\sC^\X \simeq \animaTopos^\X \otimes \sC\in\module_\sC(\presentable^L_\baseTopos)$ also refines to an object in $\calg(\module_{\sC}(\presentable^L_\baseTopos))$ with the pointwise symmetric monoidal structure; we will write $\yoneda\colon \X\xhookrightarrow{\yoneda}\animaTopos^{\X}\rightarrow \sC^\X$ for the Yoneda functor (which need not be fully faithful anymore). Moreover, given  $\D \in \module_{\sC}(\presentable^L_\baseTopos)$, the $\baseTopos$-category $\D^\X \simeq \sC^\X \otimes_\sC \D$ carries a canonical $\sC$-linear structure.

\begin{prop}\label{prop:classification_of_linear_functors}
For $\D, \E \in \module_{\sC}(\presentable^L_{\baseTopos})$ and $\X\in\cat_\baseTopos$ there is an equivalence 
\[
\funTopos^L_{\sC}(\presheafTopos(\X;\D), \E) \simeq \funTopos^L_{\sC}(\D, \E^\X)\in \module_{\sC}(\presentable^L_\baseTopos),
\]
natural in $\D$ and $\E$, which we term as \textit{Morita equivalence}. For $\D = \sC$ this specialises to
\[
\yoneda^*\colon \funTopos^L_{\sC}(\presheafTopos(\X;\sC), \E) \xrightarrow{\simeq} \E^\X\in \module_{\sC}(\presentable^L_\baseTopos).
\]
\end{prop}

We will deduce the general case of \cref{prop:classification_of_linear_functors} from the case $\sC \simeq \D \simeq \E = \animaTopos$. As we recalled in \cref{eq:presheaves_free}, restricting along the Yoneda embedding defines an equivalence $\yoneda^*\colon\funTopos^L(\presheafTopos(\X), \animaTopos) \xrightarrow{\simeq} \animaTopos^\X$ exhibiting $\presheafTopos(\X)$ as the free cocompletion of $\X$. For better adaptibility to the various module structures, we rephrase this in a multiplicative fashion.

\begin{lem}\label{lem:presheaves_dualisable}
Let $\X\in\cat_\baseTopos$; then
$\animaTopos^\X \in \presentable^L_{\baseTopos}$ is dualisable with dual $\presheafTopos(\X)$ and evaluation $\eval \colon \animaTopos^\X\otimes\presheafTopos(\X)\to \animaTopos$ adjoint to the inverse of $\yoneda^*\colon\funTopos^L(\presheafTopos(\X), \animaTopos) \xrightarrow{\simeq} \animaTopos^\X$.
\end{lem}
\begin{proof}
To check that the evaluation $\eval \colon  \animaTopos^\X\otimes\presheafTopos(\X)  \to \animaTopos$ exhibits a duality between $\animaTopos^{\X}$ and $\presheafTopos(\X)$, we have to check that the induced map $\funTopos^L(\D,\E \otimes\animaTopos^\X )\to \funTopos^L(\D \otimes \presheafTopos(\X),\E)$ is an equivalence for all $\D,\E \in \presentable^L_\baseTopos$. 
We can factor the map as a chain of equivalences as follows, 
\[
\funTopos^L(\D,\E\otimes \animaTopos^\X)\simeq \funTopos^L(\D,\E^\X)\overset{(\yoneda^*)^{-1}} \simeq \funTopos^L(\D,\funTopos^L(\presheafTopos(\X),\E))\simeq\funTopos^L(\D\otimes \presheafTopos(\X), \E).
\]
where the middle equivalence uses that $\presheafTopos(\X)$ is the free cocompletion of $\X$.
\end{proof}

\begin{proof}[Proof of \cref{prop:classification_of_linear_functors}]
By symmetric monoidality of the functor $- \otimes \sC \colon \presentable^L_{\baseTopos} \to \module_\sC(\presentable^L_\baseTopos)$, the $\sC$-linear $\baseTopos$-category $\sC^\X \simeq \animaTopos^\X \otimes \sC$ is dualisable with dual $\presheafTopos(\X) \otimes \sC \simeq \presheafTopos(\X;\sC)$.
We thus obtain the claimed equivalence as the composite equivalence of internal homs
\[
\funTopos^L_{\sC}(\presheafTopos(\X;\D), \E) \simeq \funTopos^L_{\sC}(\D \otimes_\sC \presheafTopos(\X;\sC), \E) \simeq \funTopos^L_{\sC}(\D, \sC^\X \otimes_{\sC}\E) \simeq \funTopos^L_{\sC}(\D, \E^\X).
\]
In the case $\D=\sC$ we additionally have the equivalence $\funTopos^L_{\sC}(\sC, \E^\X) \simeq \E^\X$.
\end{proof}

The following lemma records some functoriality of the Morita  equivalence of \cref{prop:classification_of_linear_functors} that will be relevant in the sequel.

\begin{lem}\label{lem:naturality_classification_C_linear_functors}
Let $\sC\in\calg(\presentable^L_{\baseTopos})$, let $\D \in \module_{\sC}(\presentable^L_{\baseTopos})$ and let $f \colon \X \to \Y$ be a $\baseTopos$-functor of small $\baseTopos$-categories. Then the following hold:
\begin{enumerate}[label=(\arabic*)]
\item The $\sC$-linear functor $f^* \colon \D^\Y \to \D^\X$ admits a $\sC$-linear left adjoint $f_!$;
\item The following squares in $\module_\sC(\Pr^L_\baseTopos)$ commute: 
\[
\begin{tikzcd}
\funTopos^L_{\sC}(\presheafTopos(\Y;\sC),\D) \ar[r, "\simeq"] \ar[d, "(f_!)^*"]
&\D^\Y \ar[d, "f^*"]\\
\funTopos^L_{\sC}(\presheafTopos(\X;\sC),\D) \ar[r, "\simeq"]
& \D^\X;
\end{tikzcd}
\hspace{1cm}
\begin{tikzcd}
\funTopos^L_{\sC}(\presheafTopos(\X;\sC), \D) \ar[r, "\simeq"] \ar[d, "(f^*)^*"]
& \D^\X \ar[d, "f_!"]\\
\funTopos^L_{\sC}(\presheafTopos(\Y;\sC),\D) \ar[r, "\simeq"]
& \D^\Y.
\end{tikzcd}
\]
Moreover, the horizontal equivalences are compatible with the adjunctions $(f^*)^*\dashv (f_!)^*$ between the left vertical  functors and $f_!\dashv f^*$ between the right vertical  functors.
\end{enumerate}
\end{lem}
\begin{proof}
For (1), it suffices to consider the case $\sC = \D = \animaTopos$. The general case then follows by base change along $- \otimes \D \colon \presentable^L_{\baseTopos} \to \module_\sC(\presentable^L_{\baseTopos})$. Now by \cite[Thm. 6.3.5]{MartiniWolf2024Colimits}, the functor $f^* \colon \animaTopos^\Y \to \animaTopos^\X$ admits a (colimit preserving) left adjoint $f_!$ given by left Kan extension, since $\animaTopos$ is a cocomplete $\baseTopos$-category.

For (2), it suffices to prove commutativity of the left square, as the right square is obtained by passing to left adjoints. This would moreover also prove the compatibility of the adjunctions.
We can also reduce to the case $\sC = \animaTopos$ by basechange.
Let us first prove that the following square of $\baseTopos$-categories commutes; the proof is the same as in the nonparametrised version of this statement \cite[Prop. 5.2.6.3]{lurieHTT}, and we spell out some details:
\begin{equation}\label{eq:naturality_yoneda_embedding}
\begin{tikzcd}[row sep=10pt]
\X \ar[d, "f"] \ar[r, "\yoneda_\X"] & \presheafTopos(\X) \ar[d, "{f_!}"] \\
\Y \ar[r, "\yoneda_\Y"] & \presheafTopos(\Y).
\end{tikzcd}
\end{equation}
To show that \eqref{eq:naturality_yoneda_embedding} indeed commutes, consider the following diagram of $\baseTopos$-categories, where the commutativity of the triangles are provided by Yoneda's lemma \cite[Thm. 4.7.8]{Martini2022Yoneda}, and the commutativity of the square is given by the definition of $f^*\colon\presheafTopos(\Y)\to\presheafTopos(\X)$:
\begin{equation}
\begin{tikzcd}
& \X\op \times \presheafTopos(\X) \ar[dr, "\eval"']\ar[rr, "(\yoneda_\X)\op \times \id"] & &\presheafTopos(\X)\op \times \presheafTopos(\X) \ar[dl, "\hom_{\presheafTopos(\X)}"] \\
\X\op \times \presheafTopos(\Y) \ar[dr, "f \times \id"'] \ar[ur, "\id \times f^*"]  & & \animaTopos_\baseTopos \\
& \Y\op \times \presheafTopos(\Y) \ar[rr, "(\yoneda_\Y)\op \times \id"] \ar[ur, "\eval"] & & \presheafTopos(\Y)\op \times \presheafTopos(\Y). \ar[ul, "\hom_{\presheafTopos(\Y)}"']
\end{tikzcd}
\end{equation}
The adjunction $f_! \dashv f^*$ \cite[Prop. 3.3.4]{MartiniWolf2024Colimits} provides a further commutative diagram:
\begin{equation}
\begin{tikzcd}[column sep=40pt]
\X\op \times \presheafTopos(\Y) \ar[r, "(\yoneda_\X)\op \times \id"] & \presheafTopos(\X)\op \times \presheafTopos(\Y) \ar[r, "\id \times f^*"] \ar[d, "f_! \times \id"] & \presheafTopos(\X)\op \times \presheafTopos(\X) \ar[d, "\hom_{\presheafTopos(\X)}"]\\
& \presheafTopos(\Y)\op \times \presheafTopos(\Y) \ar[r, "\hom_{\presheafTopos(\Y)}"] & \animaTopos_\baseTopos.
\end{tikzcd}
\end{equation}
Putting together the two previous diagrams, we get an equivalence $\hom_{\presheafTopos(\Y)}(\yoneda_\Y\circ f, \id_{\presheaf(\Y)}) \simeq \hom_{\presheafTopos(\Y)}(f_! \circ\yoneda_\X,\id_{\presheaf(\Y)})$ of $\baseTopos$-functors $\X\op \times \presheafTopos(\Y)\to\animaTopos_\baseTopos$.
Now let $g \colon \X \rightarrow \presheafTopos(\Y)$ be an arbitrary $\baseTopos$-functor. Precomposing with $\twistedArrow(\X)\xrightarrow{(s,t)}\X\op\times\X\xrightarrow{\id_{\X\op}\times g}\X\op\times\presheafTopos(\Y)$,
we get an equivalence of $\baseTopos$-functors $\hom_{\presheaf(\Y)}(f_! \circ\yoneda_\X,g) \simeq \hom_{\presheaf(\Y)}(\yoneda_\Y\circ f, g) \colon \twistedArrow(\X) \rightarrow \animaTopos_\baseTopos$; taking $\baseTopos$-limits over $\twistedArrow(\X)$ and then global sections, we finally get that an equivalence of spaces $\map_\baseTopos(f_! \circ\yoneda_\X,g)\simeq\map_\baseTopos(\yoneda_\Y\circ f,g)$, naturally in $g\in\func_\baseTopos(\X,\presheafTopos(\Y))$. By the Yoneda lemma, we get $f_! \circ\yoneda_\X \simeq \yoneda_\Y\circ f$. 
  
The left square in (2) is obtained by applying $\funTopos(-, \D)\colon\widehat{\cat}_\baseTopos\op\to\module_\sC(\presentable^L_\baseTopos)$ to \cref{eq:naturality_yoneda_embedding}.
\end{proof}

\begin{obs}[Naturality in $\sC$]
\label{obs:naturality_in_C}
The equivalence given in \cref{prop:classification_of_linear_functors} is natural in $\sC\in \calg(\presentable^L_{\baseTopos})$; in particular, if $\sC\to\sC'$ is a colimit-preserving symmetric monoidal $\baseTopos$-functor, then for $\sC$-linear $\baseTopos$-categories $\D,\E$ and for $\X\in\cat_\baseTopos$, we have a commutative square of  $\sC'$-linear $\baseTopos$-categories
\[
\begin{tikzcd}
\sC'\otimes_\sC\funTopos^L_{\sC}(\presheafTopos(\X;\D), \E) \ar[r,"\simeq"]\ar[d]& \sC'\otimes_\sC\funTopos^L_{\sC}(\D, \E^\X)\ar[d]\\
\funTopos^L_{\sC'}(\presheafTopos(\X;\sC'\otimes_\sC\D), \sC'\otimes_\sC\E) \ar[r,"\simeq"]& \funTopos^L_{\sC'}(\sC'\otimes_\sC\D, (\sC'\otimes_\sC\E)^\X).
\end{tikzcd}
\]
If $\D \in \module_\sC(\presentable^L_\baseTopos)$ is dualisable, so is $\presheafTopos(\X;\D)$, and in this case the vertical maps are equivalences.
In particular, the specialised equivalence $\funTopos^L_{\sC'}(\presheafTopos(\X,\sC'), \sC'\otimes_\sC\E) \simeq (\sC'\otimes_\sC\E)^\X$ can be obtained by applying $\sC'\otimes_\sC-$ to the specialised equivalence $\funTopos^L_{\sC}(\presheafTopos(\X,\sC), \E) \simeq \E^\X$.

\end{obs}

\begin{defn}
\label{defn:classifying_system_of_C_linear_functor}
Given a $\sC$-linear colimit preserving $\baseTopos$-functor $F \colon \presheafTopos(\X;\sC) \rightarrow \sC$, we let $\omega_F \colon \X \rightarrow \sC$ denote the \textit{classifying system of $F$} under \cref{prop:classification_of_linear_functors}.
\end{defn}

\begin{obs}
\label{obs:concrete_formula_for_classifying_system}
Concretely, $\omega_F$ is given as the composite
$\X \xrightarrow{\yoneda_\X} \presheafTopos(\X;\sC) \xrightarrow{F} \sC$ of the Yoneda embedding with $F$.
\end{obs}

The following is a useful characterisation of $\sC$-linear functors.

\begin{recollect}[Lax $\sC$-linear functors and internal adjoints]
\label{rec:recognition_c_linear_functors}
Consider $\sC \in \calg(\presentable^L_{\baseTopos})$ and $\sM, \sN \in \module_{\sC}(\presentable^L_{\baseTopos})$.
The symmetric monoidal structure on $\presentable^L_{\baseTopos}$ is defined as a suboperad $(\presentable^L_{\baseTopos})^\otimes \hookrightarrow \widehat{\cat}_\baseTopos^\times$.
Therefore, the forgetful functor $\func^L_{\baseTopos, \sC}(\sM, \sN) \to \func_{\baseTopos, \sC}(\sM, \sN)$, whose target denotes the mapping $\infty$-category in the $(\infty, 2)$-category $\module_{\sC}(\widehat{\cat}_\baseTopos)$, identifies the source as the full subcategory spanned by colimit preserving functors.

Recall the notion of a lax $\sC$-linear functor of $\sC$-linear $\baseTopos$-categories from \cref{subsec:lax_linearity}.
In the situation above, let $F \colon \sM \to \sN$ be a lax $\sC$-linear functor; then either $F$ admits no refinement to a $\sC$-linear functor, i.e. to a morphism in $\module_{\sC}(\presentable^L_{\baseTopos})$, or $F$ admits an essentially unique such refinement.
The refinement exists if and only if $F$ is colimit preserving and the lax $\sC$-linear structure on $F$ is strict.

Now let $F \colon \sM \to \sN$ in $\module_{\sC}(\presentable^L_{\baseTopos})$ and assume that $F$ admits a right adjoint $G$ in $\widehat{\cat}_\baseTopos$.
By \cref{prop:right_adjoint_to_C_linear_functor_is_lax_C_linear}, $G$ admits a preferred lax $\sC$-linear structure. Furthermore, \cref{lem:adjunction_unit_lax_linear} shows that $F$ is an internal left adjoint in $\module_{\sC}(\widehat{\cat}_\baseTopos)$ if and only if this lax $\sC$-linear structure on $G$ is strict.It follows that $F$ is an internal left adjoint in $\module_{\sC}(\presentable^L_{\baseTopos})$ if and only if the lax $\sC$-linear structure on $G$ is strict and $G$ is a functor in $\presentable_\baseTopos^L$, i.e. it preserves colimits.
\end{recollect}

In the next two lemmas we compare different $\infty$-topoi via basechange, and we reflect each time the base $\infty$-topos in the notation to avoid confusion. 

\begin{lem}
\label{lem:geometric_basechange_classifying_systems}
Let $f^* \colon \baseTopos \rightleftharpoons \baseTopos' \cocolon f_*$ be a geometric morphism, let $\X\in\cat_\baseTopos$, let $\D\in\calg(\presentable^L_{\baseTopos'})$, and let 
$F\colon\presheafTopos_{\baseTopos'}(f^*\X;\D)\to \D$ be a lax $\D$-linear $\baseTopos$-functor. Consider the lax $f_*\D$-linear functor $f_*F\colon\presheafTopos_\baseTopos(\X;f_*\D)\simeq f_*\presheafTopos_{\baseTopos'}(f^*\X;\D)\to f_*\D$, where we use the equivalence from \cref{lem:base_change_formula_functor_category}.
Then $F$ is colimit preserving and $\D$-linear if and only if $f_*F$ is colimit preserving and $f_*\D$-linear.
In that case, the classifying systems for $F$ and $f_*F$ can be computed from one another by the composites
\[
\omega_F\colon f^*\X \xrightarrow{f^*\omega_{f_*F}} f^* f_* \D \xrightarrow{\epsilon_{f_*}} \D \quad\text{ and }\quad \omega_{f_*F}\colon\X \xrightarrow{\eta_{f_*}} f_* f^* \X \xrightarrow{f_* \omega_F} f_* \D.
\]
\end{lem}
\begin{proof}
Recall from \cite[Lemma 2.1.31]{PD1} that the functor $f_* \colon \presentable^L_{\baseTopos'} \to \presentable^L_{\baseTopos}$ has a lax symmetric monoidal enhancement. 
This allows us to consider $f_* F$ as a lax $f_* \D$-linear functor as stated.
We have moreover a commutative diagram
\begin{equation}
\label{eq:omegaF_vs_omegafF}
\begin{tikzcd}[row sep=10pt, column sep=50pt]
\func^L_{\baseTopos', \D}(\presheafTopos_{\baseTopos'}(f^*\X; \D), \D) \ar[r, "\simeq"',"(\yoneda_{f*\X})^*"] \ar[d, "f_*"]&
\func_{\baseTopos'}(f^* \X, \D) \ar[d, "f_*"] \\
\func^L_{\baseTopos, f_*\D}(f_*\presheafTopos_{\baseTopos'}(f^*\X; \D),f_*\D) \ar[r, "(\yoneda_{f*\X})^*"]\ar[d,"\simeq"]
&  \func_{\baseTopos} (f_*f^*\X, f_* \D)\ar[d,"(\eta_{f_*})^*"]\\
\func^L_{\baseTopos, f_*\D}(\presheafTopos_{\baseTopos}(\X; f_*\D), f_*\D) \ar[r, "(\yoneda_{f*\X})^*","\simeq"']
&  \func_{\baseTopos} (\X, f_* \D),
\end{tikzcd}
\end{equation}
in which the bottom left vertical functor, as well as the right vertical composition, are equivalences by Lemma \ref{lem:base_change_formula_functor_category}. It follows that left vertical composite in \cref{eq:omegaF_vs_omegafF} is also an equivalence; as this functor is restricted from the analogous functor between $\infty$-categories of lax linear functors, we obtain as desired that $F$ is colimit preserving and $\D$-linear if and only if $f_*F$ is colimit preserving and $f_*\D$-linear. The computation of $\omega_{f_*F}$ in terms of $\omega_F$ is given by chasing \cref{eq:omegaF_vs_omegafF} starting at $F$ in the top left corner; the given formula for $\omega_F$ in terms of $\omega_{f_*F}$ then holds true because of the triangle identities.
\end{proof}
For the next lemma, recall that given an \'etale geometric morphism $\pi^*_\tau \colon \baseTopos \rightarrow \baseTopos_{/\tau}$ one gets an induced symmetric monoidal left adjoint functor $\pi_\tau^* \colon \presentable^L_\baseTopos \rightarrow \presentable^L_{\baseTopos_{/\tau}}$ which in fact induces for $\sC \in \calg(\presentable^L_\baseTopos)$ an $(\infty,2)$-functor $\pi_\tau^* \colon \module_{\sC}(\presentable^L_\baseTopos) \rightarrow \module_{\pi_\tau^* \sC}(\presentable^L_{\baseTopos_{/\tau}})$.

\begin{lem}
\label{lem:etale_basechange_classifying_systems}
Let $\sC\in\calg(\presentable^L_\baseTopos)$, let $F\colon \presheafTopos_{\baseTopos}(\X;\sC) \rightarrow \sC$ be a lax $\sC$-linear functor, and let $\tau\in\baseTopos$.
Then the following hold.
\begin{enumerate}
\item If $F$ is colimit preserving and $\sC$-linear, then also
\[
\pi_\tau^* F \colon \presheafTopos_{\baseTopos_{/\tau}}(\pi_\tau^* \X;\pi_\tau^* \sC)  \simeq \pi^*_\baseTopos \presheafTopos_{\baseTopos}(\X;\sC) \longrightarrow \pi_\tau^*\sC
\]
is colimit preserving and its lax $\pi_\tau^*\sC$-linear structure is strict; in this case, the classifying system of $\pi_\tau^*F$ agrees with $\pi_\tau^*\omega_F\colon\pi_\tau^*\X\to\pi_\tau^*\sC$.
\item If $\tau\to*$ is an effective epimorphism in $\baseTopos$, and if $\pi_\tau^*F$ is colimit preserving and $\pi_\tau^*\sC$-linear, then also $F$ is colimit preserving and $\sC$-linear.
\end{enumerate}
\end{lem}

\begin{proof}
For (1), we use that $\pi_\tau^*$ is a symmetric monoidal 2-functor $\presentable^L_{\baseTopos}\to\presentable^L_{\baseTopos_{/\tau}}$. By \cref{obs:concrete_formula_for_classifying_system}, $\pi_\tau^* \omega_F$ is given as the composite
\[
\pi_\tau^* \X \xrightarrow{\pi_\tau^* \yoneda_\X} \pi_\tau^* \presheafTopos_{\baseTopos}(\X;\sC) \xrightarrow{\pi_\tau^* F} \pi_\tau^* \sC
\]
and the claim follows from identifying the composite $\pi_\tau^* \X \rightarrow \pi_\tau^* \presheafTopos_\baseTopos(\X;\sC) \simeq \presheafTopos_{\baseTopos_{/\tau}}(\pi_\tau^* \X;\pi_\tau^* \sC)$ with the Yoneda embedding, using \cite[Lem. 4.7.14.]{Martini2022Yoneda}.

For (2), suppose that $\tau\to*$ is an effective epimorphism and that $\pi_\tau^* F$ is colimit preserving and $\pi_\tau^*\sC$-linear.
Then for all $[n]\in\simplex$ we have that $\pi_{\tau^{n+1}}^*F$ is colimit preserving and $\pi_{\tau^{n+1}}^*\sC$-linear, i.e. it is a morphism in $\module_{\pi_{\tau^{n+1}}^*\sC}(\presentable^L_{\baseTopos_{/\tau^{n+1}}})$. 
We have a symmetric monoidal equivalence of large $\infty$-categories $\presentable^L_\baseTopos\simeq\lim_{[n]\in\simplex}\presentable^L_{\baseTopos_{/\tau^{n+1}}}$, embedding lax symmetric monoidally and non-fully into the symmetric monoidal equivalence of huge $\infty$-categories $\widehat{\cat}_\baseTopos\simeq\lim_{[n]\in\simplex}\widehat{\cat}_{\baseTopos_{/\tau^{n+1}}}$. Since $F$ corresponds in the latter equivalence to the system of functors $(\pi_{\tau^{n+1}}^*F)_{[n]\in\simplex}$, it follows that $F$ is colimit preserving (between presentables), i.e. it lies in the $\infty$-subcategory $\presentable^L_\baseTopos$. 
Similarly, the lax $\sC$-linear structure on $F$ is
the limit of the strict $\pi_{\tau^{n+1}}^*\sC$-linear structures on $\pi_{\tau^{n+1}}^*F$, so it is strict.
\end{proof}

\subsection{A coend formula}
Let $\X\in\cat_{\baseTopos}$ and let $\sC\in\calg(\presentable^L_{\baseTopos})$. In this subsection we give an explicit formula for the inverse $\sC^\X\xrightarrow{\simeq}\funTopos^L_{\sC}(\presheafTopos(\X;\sC), \sC)$
of the equivalence from
\cref{prop:classification_of_linear_functors}. 
Since $\funTopos^L_{\sC}(-,-)$ is the internal hom of the symmetric monoidal $\baseTopos$-category $\module_\sC(\presentable^L_\baseTopos)$, in order to describe a $\sC$-linear $\baseTopos$-functor $\sC^\X\xrightarrow{\simeq}\funTopos^L_{\sC}(\presheafTopos(\X;\sC), \sC)$ we may equivalently describe a $\sC$-linear $\baseTopos$-functor $\sC^\X\otimes_\sC\presheafTopos(\X;\sC)\to\sC$.

For the next result, recall the twisted arrow $\baseTopos$-category $\twistedArrow(\X)\in\cat_{\baseTopos}$ from \cref{subsec:recollections} associated to $\X\in\cat_{\baseTopos}$, equipped with its source and target $\baseTopos$-left fibration $(s,t)\colon \twistedArrow(\X)\rightarrow\X\op\times\X$. 

\begin{prop}
\label{prop:formula_with_twisted_arrow}
The morphism in $\module_{\sC}(\presentable^L_\baseTopos)$  given by the composition
$
\sC^\X\otimes_\sC\presheafTopos(\X;\sC)\simeq \presheafTopos(\X\op\times\X;\sC)\xrightarrow{(s,t)^*}\presheafTopos(\twistedArrow(\X);\sC)\xrightarrow{\twistedArrow(\X)_!}\sC$ 
is adjoint to the equivalence $(\yoneda^*)^{-1}\colon \sC^\X\xrightarrow{\simeq}\funTopos^L_{\sC}(\presheafTopos(\X;\sC), \sC)$ from \cref{prop:classification_of_linear_functors}.
In particular, an object $\xi\in\func_\baseTopos(\X,\sC)$ corresponds to the colimit preserving $\sC$-linear $\baseTopos$-functor
\[
\twistedArrow(\X)_!(s^*\xi\otimes t^*(-))\colon \presheafTopos(\X;\sC) \xrightarrow{t^*} \presheafTopos(\twistedArrow(\X);\sC) \xrightarrow{s^*\xi\otimes-} \presheafTopos(\twistedArrow(\X);\sC) \xrightarrow{\twistedArrow(\X)_!} \sC.  \]
\end{prop}
\begin{proof}
For uniform notation, we write $\sC^\X$ as $\presheafTopos(\X\op;\sC)$ in the proof.
The equivalence $\funTopos^L_{\baseTopos,\sC}(\presheafTopos(\X;\sC),\sC)\xrightarrow{\simeq}\sC^\X$ from \cref{prop:classification_of_linear_functors} is given by restriction along the $\sC$-linear Yoneda embedding $\yoneda_{\X,\sC}\colon\X\to\presheafTopos(\X;\sC)$.
We will construct a commutative diagram as follows:
\[
\begin{tikzcd}[column sep=15pt]
\funTopos^L_{\baseTopos,\sC}(\presheafTopos(\X;\sC), \sC)\otimes_\sC\presheafTopos(\X;\sC)\ar[d,"\yoneda_{\X,\sC}^*\otimes\id_{\presheafTopos(\X;\sC)}","\simeq"']& &\funTopos^L_{\baseTopos,\sC}(\presheafTopos(\X;\sC), \sC)
\times\X\ar[d,"\yoneda_{\X,\sC}^*","\simeq"'] \ar[ll,"\id_{\funTopos^L_{\baseTopos,\sC}}\times\yoneda_{\X,\sC}"]\\
\presheafTopos(\X\op;\sC)\otimes_\sC\presheafTopos(\X;\sC) \ar[d,"\simeq"']
& &\presheafTopos(\X\op;\sC)\times\X\ar[ll,"\id_{\presheafTopos(\X\op;\sC)}\times\yoneda_{\X,\sC}"']\ar[d,"\eval"]\\
\presheafTopos(\X\op\times\X;\sC)\ar[r,"{(s,t)^*}"] &
\presheafTopos(\twistedArrow(\X);\sC)\ar[r,"{\twistedArrow(\X)_!}"]&\sC.
\end{tikzcd}
\]
We first assume to have such a commutative diagram and conclude the proof from there. We need to prove that the left-bottom composite is equivalent to the evaluation $\baseTopos$-functor $\eval_\sC\colon\funTopos^L_{\baseTopos,\sC}(\presheafTopos(\X;\sC), \sC)\otimes_\sC\presheafTopos(\X;\sC)\to\sC$, as this would show that the adjoint of the left-bottom composite is the identity of $\funTopos^L_{\baseTopos,\sC}(\presheafTopos(\X;\sC), \sC)$. Since $\presheafTopos(\X;\sC)$ is the free $\sC$-linear cocompletion of $\X$, it suffices to check that the top-left-bottom composite is equivalent to $\eval_\sC\circ (\id_{\funTopos^L_{\baseTopos,\sC}}\times\yoneda_{\X,\sC})$, but the latter is evidently equivalent to the right composite in the diagram, whence the commutativity of the diagram lets us conclude.

The top square in the diagram commutes in an evident way, so it is only left to make the bottom pentagon commutative.
We claim that the $\baseTopos$-functor $\yoneda_{\X,\sC}$ is adjoint to the $\baseTopos$-functor $\X\op\times\X\to\sC$ given by $h\coloneqq((s,t)\op)_!(\unit)\in\presheaf_{\baseTopos}(\X\times\X\op;\sC)$,
where $\unit\in\presheaf_{\baseTopos}(\twistedArrow(\X)\op;\sC)$ denotes the monoidal unit.
By base change, it suffices to check this for $\sC = \animaTopos_\baseTopos$.
Note that, in this case, $\yoneda_{\X,\animaTopos}$ is adjoint to the functor $\mapTopos_\X \colon \X\op \times \X \to \animaTopos_\baseTopos$, which is classified by the left fibration $(s,t)\colon \twistedArrow(\X)\rightarrow\X\op\times\X$ under the straightening-unstraightening equivalence.
This proves the claim.
Using this, and adjoining the above pentagon, it suffices to prove that the top composite in the following diagram is equivalent to the identity of $\presheafTopos(\X\op;\sC)$, where for simplicity we remove ``$\sC$'' from the notation:
\[
\begin{tikzcd}[column sep=25pt]
&\presheafTopos(\X\op\times\X\times\X\op)\ar[r,"{((s,t)\times\id)^*}"] &\presheafTopos(\twistedArrow(\X)\times\X\op)\ar[dr,"{(\pi_2)_!}"]\\
\presheafTopos(\X\op)\ar[dr,"s^*"']\ar[ur,"{\pi_1^*\otimes(\pi_{2,3})^*h}"]\ar[r,"{\pi_1^*}"]&\presheafTopos(\X\op\times\twistedArrow(\X)\op)\ar[u,"{(\id\times(s,t))_!}"']\ar[r,"{(s\pi_1,\pi_2)^*}"]&\presheafTopos(\twistedArrow(\X)\times_\X\twistedArrow(\X)\op)\ar[u,"{(\pi_1,t\pi_2)_!}"]\ar[d,"c_!"]\ar[r,"{(t\pi_2)_!}"] &\presheafTopos(\X\op)\\
 &\presheafTopos((\X^{[1]})\op)\ar[r,equal]\ar[ur,"c^*"]&\presheafTopos((\X^{[1]})\op).\ar[ur,"t_!"']
\end{tikzcd}
\]
In the previous diagram we have denoted by $\pi_k$ the projection onto the $k$\textsuperscript{th} factor(s); the fibre product $\twistedArrow(\X)\times_\X\twistedArrow(\X)\op$ is taken along the target and source functors; and we denote by $c\colon\twistedArrow(\X)\times_\X\twistedArrow(\X)\op\to(\X^{[1]})\op$ the concatenation functor. To explain why the diagram is commutative, we observe in particular the following.
\begin{itemize}
\item The upper rectangle commutes because $(s,t)\colon\twistedArrow(\X)\op\to\X\times\X\op$ is a $\baseTopos$-cartesian fibration, and the two $\baseTopos$-functors $(\id\times(s,t))$ and $(\pi_1,t\pi_2)$ are pullbacks of this $\baseTopos$-cartesian fibration, so the Beck--Chevalley map is an equivalence by basechange.

\item The bottom triangle commutes because $c$ is a localisation. To see this, note that $t\pi_2\colon\twistedArrow(\X)\times_\X\twistedArrow(\X)\op\to\X\op$ is a $\baseTopos$-cartesian fibration;
its straightening is the $\baseTopos$-functor $\theta\colon\X\to\catTopos_\baseTopos$ given by the formula $\twistedArrow(\X)\times_\X\X_{/(-)}$. Similarly, the $\baseTopos$-cartesian fibration $t\colon(\X^{[1]})\op\to\X\op$ is classified by the $\baseTopos$-functor $(\X_{/(-)})\op$. 
Then $c$ refines to a $\baseTopos$-natural transformation $c'\colon \theta\to(\X_{/(-)})\op$ of $\baseTopos$-functors $\X\to\catTopos_\baseTopos$; in fact $c'\colon\X\to\funTopos_\baseTopos([1],\catTopos_\baseTopos)$ takes values into the full $\baseTopos$-subcategory spanned at level $\tau\in\baseTopos$ by $\baseTopos_{/\tau}$-cocartesian fibrations with weakly contractible fibres,
so by \cref{lem:cartesian_fibrations_with_contractible_fibres} below, we may consider the wide $\baseTopos$-subfunctor $\theta'\coloneqq(\X\op)^\simeq\times_{\X\op}\twistedArrow(\X)\times_\X\X_{/(-)}$ and identify $(\X_{/(-)})\simeq\scL_\baseTopos\circ(\theta'\to \theta)$; we then apply \cref{lem:localisation_of_fibration} and conclude that $c$ is a localisation.
\end{itemize}
We conclude by arguing that the composite $\presheafTopos(\X\op;\sC)\xrightarrow{s^*}\presheafTopos((\X^{[1]})\op;\sC)\xrightarrow{t_!}\presheafTopos(\X\op)$ is equivalent to the identity of $\presheafTopos(\X\op;\sC)$: indeed $s\colon(\X^{[1]})\op\to\X\op$ is left adjoint to the fully faithful inclusion $\id_{(-)}\colon\X\op\to(\X^{[1]})\op$, so we may identify $s^*$ with $(\id_{(-)})_!$ on presheaves.
\end{proof}

\begin{example}
\label{ex:interval_continued}
We continue with \cref{ex:interval_has_no_dualising_system}, so let
again $\baseTopos = \spc$ and $\X = [1]$. Then $\twistedArrow([1]) \simeq \{ 00 \leftarrow 01 \rightarrow 11 \}$. Let $\omega \colon [1] \rightarrow \sC$ denote the presheaf sending $0\mapsto\unit\in\sC$ and $1\mapsto0$. Then we compute, naturally in $\xi \in \presheaf_{\baseTopos}([1];\sC)$, the following chain of equivalences:
\[
\twistedArrow_!(s^*\omega \otimes t^* \xi) \simeq \twistedArrow_!(0 \rightarrow \xi(1) \leftarrow 0) \simeq \xi(1) \simeq [1]_*(\xi). 
\]
This shows that $\omega$ is in fact the classifying system of $[1]$. 
Note that $\omega$ is not a groupoidal system, since it does not send the non-identity arrow in $[1]$ to an equivalence. The next subsection will simplify \cref{prop:formula_with_twisted_arrow} in the case when $\omega$ happens to be groupoidal.
\end{example}

\subsection{Classifying vs dualising systems}\label{subsec:classifying_vs_dualising}

In this section, we will start with a general functor $F \in \func^L_{\baseTopos,\sC}(\presheafTopos(\X;\sC),\sC)$ and ask the general question about the existence of an equivalence $\X_!(\xi\otimes -)$ for some object $\xi \in \presheaf_{\baseTopos}(\X;\sC)$. Our answer to this will be \cref{cor:formula_in_terms_of_dualising_system_for_dualish}.

\begin{defn}
\label{defn:DandbarD}
For $\X\in\cat_{\baseTopos}$ and $\sC\in\calg(\presentable^L_{\baseTopos})$, we denote by $D\colon\sC^\X\simeq\presheafTopos(\X\op;\sC)\to\presheafTopos(\X;\sC)$ the $\baseTopos$-functor $t_!s^*$ associated with the $\baseTopos$-functors $\X\op\overset{s}{\leftarrow}\twistedArrow(\X)\xrightarrow{t}\X$, and by $\bar D=s_*t^*$ its right adjoint.
\end{defn}

\begin{defn}
For $F\in\func^L_{\baseTopos,\sC}(\presheafTopos(\X;\sC),\sC)$, we denote by $\omega_F\in\func_\baseTopos(\X,\sC)$ the associated object under the equivalence from \cref{prop:classification_of_linear_functors}. We call $\omega_F$ the \textit{classifying system of $F$} and $D_F\coloneqq D(\omega_F)$ the \textit{(candidate) dualising system of $F$.}
\end{defn}

\begin{warning}
Note that, away from the setting in which $\X$ is a $\baseTopos$-groupoid, $D_F$ does \textit{not} classify the functor $F$ in any way. Only the object $\omega_F$ corresponds uniquely to the functor $F$, under the equivalence from \cref{prop:classification_of_linear_functors}, whence its name ``classifying system''. 
\end{warning}

In some cases, however,  $D(\omega_F)$ can also be used to classify $F$, whence the name ``(candidate) dualising system''. More precisely,
the projection formula for $t_!$ allows us to write a natural transformation of $\sC$-linear $\baseTopos$-functors $\presheafTopos(\X;\sC)\to\sC$ as follows, naturally in $\omega\in\func_\baseTopos(\X,\sC)$:
\begin{equation}
\label{eq:yonedatoD}
\twistedArrow(\X)_!(s^*(\omega)\otimes t^*(-))\simeq \X_!t_!(s^*(\omega)\otimes t^*(-))\to \X_!(t_!s^*(\omega)\otimes -)\simeq \X_!(D(\omega)\otimes-).
\end{equation}
The $\baseTopos$-natural transformation \cref{eq:yonedatoD} is \textit{not} an equivalence in general, but when it is, it allows us to employ $D(\omega)$ and represent as $\X_!(D(\omega)\otimes-)$ the $\sC$-linear $\baseTopos$-functor corresponding to $\omega$. In \cref{prop:dualish},
we will provide a suitable worth of objects $\omega$ for which \cref{eq:yonedatoD} is an equivalence, including all groupoidal objects.

To this end, we first study the adjunction $D\dashv \bar D$ from \ref{defn:DandbarD}.
\begin{prop}
\label{prop:Dequivalence}
Let $\X\in\cat_{\baseTopos}$ and let $\sC\in\calg(\presentable^L_{\baseTopos})$. Then the following hold.
\begin{enumerate}[label=(\arabic*)]
\item The adjunction $D\dashv\bar D$ from \cref{defn:DandbarD} restricts to a pair of inverse equivalences
\[
D\colon\sC^{\vert \X\vert}\simeq\presheafTopos(\vert \X\vert;\sC)\cocolon\bar D
\]
when passing to the full $\sC$-linear $\baseTopos$-subcategories of groupoidal functors.
\item We can identify the restricted equivalence $D\colon\sC^{\vert\X\vert}\xrightarrow{\simeq}\presheafTopos(\vert\X\vert;\sC)$ with the composite equivalence $\sC^{\vert\X\vert}=\funTopos_\baseTopos(\vert\X\vert,\sC)\simeq\funTopos_\baseTopos(\vert\X\vert\op,\sC)=\presheafTopos(\vert\X\vert;\sC)$ induced by the equivalence $\vert\X\vert\simeq\vert\X\vert\op$.
\end{enumerate}
\end{prop}
We need some preliminary results to prove  \cref{prop:Dequivalence}.
\begin{lem}
\label{lem:fibration_weakly_contractible_fibres}
Let $p \colon \X \rightarrow \Y$ be a cocartesian or cartesian fibration of $\baseTopos$-categories; assume that all fibres of $p$ have contractible classifying $\baseTopos$-space. Then $p$ exhibits $\Y$ as the localisation $\scL(\X,p^{-1}(\Y^\simeq))$ of $\X$ at the morphisms $p^{-1}(\Y^{\simeq})$.
\end{lem}
\begin{proof}
Assume first that $p$ is cocartesian; then by \cref{lem:localisation_of_fibration} we may compute the localisation $\scL(\X,p^{-1}(\Y^\simeq))$ as the $\baseTopos$-unstraightening of the $\baseTopos$-functor $\vert-\vert\circ \theta\colon \Y\to\catTopos_\baseTopos$, where $\theta\colon \Y\to\catTopos_\baseTopos$ is a straightening of $p$. The hypothesis that $\vert-\vert\circ \theta\simeq\ast$ ensures that $\scL(\X,p^{-1}(\Y^\simeq))\simeq \Y$.
If $p$ is cartesian, we argue similarly with $p\op\colon \X\op\to \Y\op$, using \cref{obs:localisation_op} twice.
\end{proof}

\begin{lem}
\label{lem:cartesian_fibrations_with_contractible_fibres}
In the hypotheses of \cref{lem:fibration_weakly_contractible_fibres}, the following hold.
\begin{enumerate}[label=(\arabic*)]
\item For any cocomplete $\baseTopos$-category $\D$, the counit of the adjunction $p_!\colon \D^\X\rightleftharpoons\D^\Y\colon p^*$ is an equivalence.
\item For any cocomplete $\baseTopos$-category $\D$, the adjunction $p_!\colon \D^\X\rightleftharpoons\D^\Y\colon p^*$ restricts to a pair of inverse equivalences $p_!\colon \D^{\vert \X\vert}\simeq\D^{\vert \Y\vert}\colon p^*$ between full $\baseTopos$-subcategories of groupoidal functors.
\item Dually, for a complete $\baseTopos$-category $\D$, the unit of the adjunction $p^*\dashv p_*$ is an equivalence, and the adjunction restricts to a pair of inverse equivalences $\D^{\vert \Y\vert}\simeq\D^{\vert \X\vert}$.
\end{enumerate}
\end{lem}
\begin{proof}
\begin{enumerate}[label=(\arabic*)]
\item By \cref{lem:fibration_weakly_contractible_fibres} the $\baseTopos$-functor $p^*\colon\D^\Y\to\D^\X$ is fully faithful; it follows that the counit $\epsilon_{p_!}\colon p_!p^*\to\id_{\D^\Y}$ is an equivalence.
\item The unit $\eta_{p_!}\colon p^*p_!\to\id_{\D^\X}$ restricts to an equivalence on the essential image of $p^*$, which contains $\D^{\vert \X\vert}\simeq p^*(\D^{\vert \Y\vert})$, where this equivalence is since $|X|\xrightarrow[\simeq]{p}|\Y|$.
\item The argument is analogous to (2).\qedhere
\end{enumerate}
\end{proof}

\begin{cor}
\label{cor:Beck_Chevalley_for_twisted_arrows}
Let $\X\in\cat_\baseTopos$ and consider the following diagram of $\baseTopos$-categories:
\begin{equation}
\begin{tikzcd}[row sep=15pt]
\X\op \ar[d,"r\op"] & \ar[l, "s"']\twistedArrow(\X) \ar[r,"t"] \ar[d,"r'"] & \X \ar[d, "r"] \\
\vert \X \vert\op  &\ar[l, "s'"',"\simeq"]\twistedArrow(\vert \X \vert) \ar[r, "t'","\simeq"'] & \vert \X \vert
\end{tikzcd}
\end{equation}
Then the Beck--Chevalley transformation $t_! (r')^* \rightarrow r^*(t')_!$ is an equivalence of $\baseTopos$-functors $\presheafTopos(\twistedArrow(\vert \X\vert);\sC)\to\presheafTopos(\X;\sC)$, and the Beck--Chevalley transformation $(r\op)^*s'_*\to s_*(r')^*$ is an equivalence of $\baseTopos$-functors $\presheafTopos(\twistedArrow(\vert \X\vert);\sC)\to\presheafTopos(\X\op;\sC)$.
\end{cor}
\begin{proof}
The first Beck--Chevalley transformation is the composite
\[
t_!(r')^*\to t_!(r')^*(t')^*(t')_!\simeq t_!t^*r^*(t')_!\to r^*(t')_!
\]
where the first arrow is an equivalence because $t'$ is an equivalence, and the second arrow is an equivalence because $t\op\colon\twistedArrow(\X)\op\to \X\op$ is a cartesian fibration with weakly contracible fibres,
so that we can apply \cref{lem:cartesian_fibrations_with_contractible_fibres}. The argument for the second Beck--Chevalley transformation is analogous.
\end{proof}

\begin{proof}[Proof of \cref{prop:Dequivalence}]
For (1), by \cref{cor:Beck_Chevalley_for_twisted_arrows}
we obtain an equivalence $t_!s^*(r\op)^*\simeq r^*(t')_!(s')^*$ of $\baseTopos$-functors $\presheafTopos(\vert \X\vert\op;\sC)\to\presheafTopos(\X;\sC)$, and an equivalence $(r\op)^*(s')_*(t')^*\simeq s_*t^*r^*$ of $\baseTopos$-functors $\presheafTopos(\vert \X\vert;\sC)\to\presheafTopos(\X\op;\sC)$. We also observe that $(r\op)^*\colon\presheafTopos(\vert \X\vert\op;\sC)\hookrightarrow\presheafTopos(\X\op;\sC)$ and $r^*\colon\presheafTopos(\vert \X\vert;\sC)\hookrightarrow\presheafTopos(\X;\sC)$ are the fully faithful inclusions of the $\baseTopos$-subcategories of groupoidal presheaves, so that we may identify the restrictions of $D$ and $\bar D$ with $(t')_!(s')^*$ and with $(s')_*(t')^*$, respectively, and the latter are inverse equivalences.

For (2), first recall that the canonical equivalence $|\X|\simeq |\X|\op$ is explicitly implemented by $s'\circ (t')^{-1}\colon |\X|\rightarrow|\X|\op$, where we have used the notations from \cref{cor:Beck_Chevalley_for_twisted_arrows}. Furthermore, since $(t')_!$ is the inverse of $(t')^*$, we thus see that $((t')^{-1})^*\simeq (t')_!$. On the other hand, by the previous paragraph, restricting $D=t_!s^*$ along the inclusion $(r\op)^*\colon \presheafTopos(\vert \X\vert\op;\sC)\subseteq \presheafTopos(\X\op;\sC)$ coincides with the functor $\presheafTopos(\vert \X\vert\op;\sC)\xrightarrow{(t')_!(s')^*}\presheafTopos(\vert \X\vert;\sC)\xhookrightarrow{r^*}\presheafTopos(\X;\sC)$. Combining these points then yields the desired conclusion.
\end{proof}

We next turn to the problem of finding suitable conditions on $\omega\in\func_\baseTopos(\X,\sC)$ under which \eqref{eq:yonedatoD} is an equivalence.
\begin{prop}
\label{prop:dualish}
Let $\X\in\cat_\baseTopos$; then the following hold.
\begin{enumerate}
\item If $\omega\in\func_\baseTopos(\X,\sC)$ satisfies that $s^*(\omega)\xrightarrow{\simeq} t^*t_!s^*(\omega)$ is an equivalence, the $\baseTopos$-natural transformation \cref{eq:yonedatoD} is an equivalence.
\item If $\omega \in \func_\baseTopos(\X,\sC)$ is groupoidal, then $s^*(\omega) \xrightarrow{\simeq} t^*t_!s^*(\omega)$ is an equivalence.
\end{enumerate}
\end{prop}
\begin{proof}
Let $\omega\in\func_\baseTopos(\X,\sC)$; then the following composite $\baseTopos$-natural transformation of $\baseTopos$-functors $\presheafTopos(\X;\sC)\to\sC$ agrees with the middle arrow in \cref{eq:yonedatoD}:
\begin{equation}
\label{eq:decomposition_projection_formula}
\X_!t_!(s^*(\omega)\otimes t^*(-))\xrightarrow{\eta_{t_!}} \X_!t_!(t^*t_!s^*(\omega)\otimes t^*(-))\simeq \X_!t_!t^*(t_!s^*(\omega)\otimes-)\xrightarrow{\epsilon_{t_!}} \X_!(t_!s^*(\omega)\otimes -).
\end{equation}
The second arrow in \cref{eq:decomposition_projection_formula} is an equivalence for all $\omega\in\func_\baseTopos(\X,\sC)$: it involves the counit of the adjunction $t_!\dashv t^*$, which is an equivalence by \cref{lem:cartesian_fibrations_with_contractible_fibres} (1). The first arrow in \cref{eq:decomposition_projection_formula} involves the unit of the adjunction $t_!\colon\func_\baseTopos(\twistedArrow(\X)\op,\sC)\rightleftharpoons\func_\baseTopos(\X\op,\sC)\cocolon t^*$ computed at $s^*(\omega)\in\func_\baseTopos(\twistedArrow(\X)\op,\sC)$, whence (1) follows.
For (2), we observe that $s^*(\omega)$ is a groupoidal presheaf on $\twistedArrow(\X)$, so we may apply \cref{lem:cartesian_fibrations_with_contractible_fibres} (2), by which $s^*(\omega)\xrightarrow{\simeq} t^*t_!s^*(\omega)$ is an equivalence.
\end{proof}

\begin{cor}\label{cor:formula_in_terms_of_dualising_system_for_dualish}
Let $\X\in\cat_\baseTopos$ and let $F \in \func^L_{\baseTopos,\sC}(\presheafTopos(\X;\sC),\sC)$ be a $\sC$-linear $\baseTopos$-functor. Then there exists a groupoidal $\zeta\in\presheaf_{\baseTopos}(\X;\sC)$ participating in a $\sC$-linear equivalence $F(-)\simeq \X_!(-\otimes\zeta)$ if and only if $\omega_F$ is groupoidal. In this case, there is an equivalence $D_F\simeq \zeta$.
\end{cor}
\begin{proof}
Given such a $\zeta$, its image under the composite functor $\presheaf(\X;\sC)\xrightarrow{\bar{D}} \func_{\baseTopos}(\X,\sC) \xrightarrow[\simeq]{(\yoneda^*)^{-1}}\func^L_{\baseTopos,\sC}(\presheafTopos(\X;\sC),\sC)$    is equivalent to $F$: this is because $\bar{D}(\zeta)\in\func_{\baseTopos}(\X,\sC)$ is yet again groupoidal by \cref{prop:Dequivalence}, and so by \cref{prop:dualish}, the image of $\zeta$ under the composite is precisely given by $\X_!(-\otimes\zeta)$. Hence, via the equivalence $\yoneda^*$, we see that $\bar{D}(\zeta)\simeq \omega_F$, whence the groupoidality of $\omega_F$. Additionally, since the adjunction $D\dashv \bar{D}$ restricts to an equivalence on groupoidal systems by \cref{prop:Dequivalence}, we also get that $\zeta\simeq D\bar{D}(\zeta)\simeq D(\omega_F)=D_F$. Conversely, if $\omega_F$ is groupoidal, then by  \cref{prop:dualish}, we see that $F(-)\simeq \X_!(-\otimes D_F)$, and so $\zeta=D_F$ works.
\end{proof}

The following lemma will be used much later in \cref{sec:fibredpd}, but its content fits well in this section.
\begin{lem}\label{lem:projection_formula_for_presheaves}
Let $\sC\in\calg(\presentable^L_{\baseTopos})$, let $f\colon \X\rightarrow\Y$ be
a $\baseTopos$-cartesian fibration, and let $\zeta\in\presheaf_{\baseTopos}(\Y;\sC)$. Then the projection formula map $f_!(-\otimes f^*\zeta)\rightarrow f_!(-)\otimes \zeta$ is an equivalence of $\baseTopos$-functors $\presheafTopos(\X;\sC)\to\presheafTopos(\Y;\sC)$.
\end{lem}
\begin{proof}
First, suppose that $f$ is a $\baseTopos$-cartesian fibration.
As the map $i^* \colon \presheafTopos(\Y;\sC) \to \presheafTopos(\Y^\simeq;\sC)$ induced by the inclusion $i \colon \Y^\simeq \to \Y$ of the core $\baseTopos$-groupoid is conservative,
it suffices to prove that $i^*f_!(-\otimes f^*\zeta)\to i^*(f_!(-)\otimes \zeta)\simeq i^*f_!(-)\otimes i^*\zeta$ is an equivalence; by smooth basechange, letting $f'\colon\X':=\Y^\simeq\times_\Y\X\to\Y^\simeq$ denote the basechange of $f$, we may identify the previous with the projection formula map for $f'$ associated with $i^*\zeta\in\presheaf(\Y^\simeq;\sC)$.

We may therefore assume that $\Y\simeq\tau\in\baseTopos$ is a $\baseTopos$-space; in this case we may consider $\X\in(\cat_\baseTopos)_{/\tau}\simeq\cat_{\baseTopos_{/\tau}}$ and regard $\zeta$ as an object in the $\infty$-category $\Gamma(\pi^*_\tau\sC)$; doing this we reduce to checking that the projection formula map $\X_!(-\otimes\X^*\zeta)\to \X_!(-)\otimes\zeta$ is an equivalence of $\baseTopos_{/\tau}$-functors $\presheafTopos_{\baseTopos_{/\tau}}(\X;\pi^*_\tau\sC)\to\pi^*_\tau\sC$, and this is the fact that $\X_!$ is $\pi^*_\tau\sC$-linear.
\end{proof}

\begin{obs}[Formulas for the classifying and dualising systems]
\label{obs:formula_for_dualising_object}
Let $\X\in\cat_\baseTopos$ and let $F\in\func^L_{\baseTopos,\sC}(\presheafTopos(\X;\sC),\sC)$.
Then $\omega_F\colon\X\to\sC$ agrees \textit{by definition} with the composite $\X\xrightarrow{\yoneda_\X}\presheafTopos(\X;\sC)\xrightarrow{F}\sC$.
For $\baseTopos=\spc$ and $x\in \X$, this recovers Klein's formula $\omega_F(x)\simeq F(x_!(\unit_\sC))$, since we have $\yoneda_\X(x)\simeq x_!(\unit_\sC)$.

We can also express $\omega_F$ in terms of the twisted arrow $\baseTopos$-category. Under the equivalence $\func_\baseTopos(\X,\presheafTopos(\X;\sC))\simeq\presheafTopos(\X\op\times\X;\sC)$, the Yoneda embedding $\yoneda_\X$ corresponds to the presheaf $\mapTopos_\X=(s,t)_!\twistedArrow(\X)^*(\unit_\sC)$, where we denote by $\unit_\sC\in\Gamma\sC\simeq\presheaf(*;\sC)$ the monoidal unit. We may therefore identify $\omega_F$ also with the image of $\unit_\sC$ along the composite functor
\[
\Gamma\sC\xrightarrow{\twistedArrow(\X)^*}
\presheaf(\twistedArrow(\X);\sC)\xrightarrow{(s,t)_!}\presheaf(\X\op\times\X;\sC)\simeq\presheaf(\X\op;\sC)\otimes_{\Gamma\sC}\presheaf(\X;\sC)\xrightarrow{\id\otimes \Gamma F}\presheaf(\X\op;\sC).
\]
If $\omega_F$ is moreover groupoidal, then we may also identify $D_F\in\presheaf_\baseTopos(\X;\sC)$ with the the image of $\unit_\sC$ along the composite functor
\[
\Gamma\sC\xrightarrow{\X^*}\presheaf(\X;\sC)\xrightarrow{\Delta_!}\presheaf(\X\times\X;\sC)\simeq\presheaf(\X;\sC)\otimes_{\Gamma\sC}\presheaf(\X;\sC)\xrightarrow{\id\otimes \Gamma F}\presheaf(\X;\sC).
\]
For $\baseTopos=\spc$ and $x\in\X$, we then have $D_F(x)\simeq\omega_F(x)\simeq F(x_!(\unit_{\Gamma\sC}))$ as before.
\end{obs}

\section{Poincar\'e duality category pairs}
\label{sec:pdcatpairs}
As explained in the introduction, Wall's notion of Poincar\'e pairs and ads of spaces involves various compatibilities between the local coefficient systems on the interior space (witnessing Poincar\'e--Lefschetz duality) with those on the boundaries, corners, etc. In this section, we propose to package all these structures in one stroke: we introduce the notion of \textit{Poincar\'e $\baseTopos$-category pairs} and we develop some basic theory about these objects.

The basic substrate of this approach will be the notion of \textit{$\baseTopos$-category pairs} $(\X,\partial\X)$, i.e. the datum of a $\baseTopos$-category $\X$ equipped with the datum of a left closed subcategory $\partial\X\subseteq \X$. Introducing these will be the purpose of \cref{subsec:left_closed_subcategories}. 
In \cref{subsec:pdcatpairs_definition} we introduce Poincar\'e duality as a property of $\baseTopos$-category pairs. This will be defined in terms of the relative cohomology $\baseTopos$-functor $(\X,\partial\X)_*\colon \presheafTopos(\X;\sC) \to \sC$, depending on the pair $(\X,\partial\X)$. 
In fact, Poincar\'e duality will be the strictest of three properties of $\baseTopos$-category pairs, the other two being \textit{twisted ambidexterity}, which is a finiteness/compactness condition ensuring that the $\baseTopos$-functor $(\X,\partial\X)_*$ admits a classifying system $\omega_{\X,\partial\X}$ via Morita theory, and \textit{groupoidal twisted ambidexterity}, which demands additionally that the classifying system is groupoidal; the last property ensures in particular that the classification formula in terms of coends from \cref{prop:formula_with_twisted_arrow} simplifies via \cref{cor:formula_in_terms_of_dualising_system_for_dualish} to one of the form $(\X,\partial\X)_*(-)\simeq \X_!(-\otimes D_{\X,\partial\X})$ for some $D_{\X,\partial\X}\in\presheaf(|\X|;\sC)$. 
As expected, Poincar\'e duality then imposes the further requirement that $\omega_{\X,\partial\X}$ (or equivalently, $D_{\X,\partial\X}$) takes values in invertible objects in $\sC^{\simeq}$. In the short \cref{subsec:connecting_map_of_classifying_systems}, we explain how the classifying system $\omega_{\X,\partial\X}$ captures the compatibilities between the classifying system of the pair with that of the boundary via certain \textit{connecting morphisms}.

We suggest again the following motivating example to keep in mind, which will be discussed in detail in \cref{sec:unstraightenedpairs}: for a functor $X\colon [1]\rightarrow\spc$ encoding a map $g\colon X(0)\rightarrow X(1)$ of compact spaces, we may consider the $\infty$-category pair $(\X,\partial\X)=(\int_{[1]}X,X(0))$ which will turn out to be twisted ambidextrous. The condition that the classifying system takes pointwise values in invertible objects is the usual condition of Poincar\'e duality; the additional demand that $\omega_{\X,\partial\X}$ be groupoidal precisely encodes, via the aforementioned connecting morphisms, that the classifying system of the pair restricts to the desuspension of the classifying system of the boundary space $\partial\X$. 

Next, we provide a suite of ``cutting and pasting'' techniques in \cref{subsec:cut_and_paste} to check the Poincar\'e property as well as obtain new Poincar\'e $\baseTopos$-category pairs from old ones.
The results of \cref{subsec:cut_and_paste} should justify our focus on left closed $\baseTopos$-subcategory inclusions $\partial\X\to\X$: this assumption ensures a certain amount of ``directionality'' which, for instance, allows us to compute certain left and right Kan extensions as being extensions-by-zero. Note that such manoeuvres are for instance not available if one only considers a map of spaces $g\colon X(0)\rightarrow X(1)$; one has to pass to the associated left closed inclusion $X(0)=\partial\X \subseteq \X$. Lastly, we will prove two versions of the Poincar\'e--Lefschetz sequence for Poincar\'e $\baseTopos$-category pairs in \cref{subsec:poincare_lefschetz_sequence}, the second of which establishes the compatibility of Poincar\'e--Lefschetz duality with Mayer--Vietoris sequences. This computationally powerful result is well known for homology with discrete coefficients, and we generalise it to arbitrary coefficients.

To facilitate our comparison with classical definitions in the following section, we prove in  \cref{subsec:Spivak}  that, under suitable conditions, all natural transformations from cohomology to homology are implemented by cap product with a fundamental class. Finally, in \cref{subsec:dualisable_results}, we generalise  \cite[Thm. B]{KleinQinSu} by way of a ``purely algebraic'' proof, answering \cite[Remark 4.2]{KleinQinSu}. Along the way, we prove some useful generalities about local systems on $\baseTopos$-categories: for example, we show in \cref{prop:dualisables_are_groupoidal} that dualisable local systems are automatically groupoidal, and that dualisability of classifying systems already implies Poincar\'e duality in \cref{prop:omega_dualisable_iff_invertible}.

\subsection{Left closed subcategories and category pairs}
\label{subsec:left_closed_subcategories}
We bring to attention here the notion of left closed $\baseTopos$-subcategories and define $\baseTopos$-category pairs. We will also record some elementary properties of these that will be relevant to us. Given a $\baseTopos$-category pair, the relationship between the ambient $\baseTopos$-category, the boundary, and the complement of the boundary is conveniently organised via a stable recollement which we record as \cref{constr:stable_recollement_for_pairs}.

\begin{defn}
\label{defn:categorypair}
Let $\X$ be a $\baseTopos$-category.
We call a full $\baseTopos$-subcategory $\Y\subseteq \X$ \textit{left closed} in $\X$ if there exists a $\baseTopos$-functor $f\colon\X\to[1]$ such that $\Y=f^{-1}(0)=\X\times_{[1]}\{0\}$.
Similarly, we say that $\Y$ is \textit{right closed} in $\X$ if there exists a $\baseTopos$-functor $f\colon\X\to[1]$ such that $\Y=f^{-1}(1)=\X\times_{[1]}\{1\}$.

A $\baseTopos$-\textit{category pair} is a pair $(\X,\partial\X)$ of a $\baseTopos$-category and a left closed subcategory $\partial\X\subseteq \X$. A morphism of $\baseTopos$-category pairs $(\X,\partial\X)\rightarrow(\Y,\partial\Y)$ is a morphism in $(\cat_{\baseTopos})_{/[1]}$.
\end{defn}

\begin{nota}\label{nota:X_dX_interiorX}
    Given a $\baseTopos$-category pair $(\X,\partial\X)$, we write $\interior{\X}\coloneqq \X\backslash \partial\X$ for the complement right-closed subcategory which we think of as the ``interior'' of $\X$.
    We further denote by $i\colon\partial\X\to\X$ and $j\colon\interior{\X}\to\X$ the inclusions.
\end{nota}
The next lemma shows that being left closed is a \textit{property} of full $\baseTopos$-subcategories.

\begin{lem}
\label{lem:uniqueness_functor_to_1}
Let $\Y\subseteq\X$ be a full $\baseTopos$-subcategory inclusion; then the subspace $\map_{\cat_{\baseTopos}}(\X,[1])_{\Y\mathrm{lc}}\subseteq\map_{\cat_{\baseTopos}}(\X,[1])$ of $\baseTopos$-functors that exhibit $\Y$ as a left closed $\baseTopos$-subcategory of $\X$ is either empty or contractible. Similarly, the choice of a $\baseTopos$-functor $\X\to[1]$ exhibiting $\Y$ as right closed is either empty or contractible. 
\end{lem}
\begin{proof}
We focus on the left closed case.
We may regard $\X$ and $[1]$ as objects in $\cat_{\baseTopos}\subseteq s\baseTopos$ as in \cref{subsec:recollections}; then we have $[1]_n\simeq\udl{n+2}\in\baseTopos$, where $\udl{n+2}$ denotes the coproduct of $n+2$ copies of the terminal object in $\baseTopos$. In particular the spaces $\map_{\baseTopos}(\X_n,[1]_n)$ are discrete, and hence so  are the spaces $\map_{\cat_{\baseTopos}}(\X,[1])$ and its subspace $\map_{\cat_{\baseTopos}}(\X,[1])_{\Y\mathrm{lc}}$.

Now assume $\map_{\cat_{\baseTopos}}(\X,[1])_{\Y\mathrm{lc}}\neq\emptyset$ and let $(f\colon\X\to[1])\in \map_{\cat_{\baseTopos}}(\X,[1])_{\Y\mathrm{lc}}$.
The map $f_0\colon\X_0\to[1]_0\simeq\{0,1\}$ exhibits $\X_0$ as a coproduct $\Y_0\amalg\Z_0$, for some ``coproduct complement'' $\Z_0\in\baseTopos$. For any $g\in\map_{\cat_{\baseTopos}}(\X,[1])_{\Y\mathrm{lc}}$ we then have that the map $g_0\colon\X_0\to\{0,1\}$ must restrict to $\Y_0\to\{0\}$ and to $\Z_0\to\{1\}$ if we want to have $g^{-1}(0)=\Y_0$, i.e. $f_0=g_0$.
Since the maps $[1]_n\to([1]_0)^{n+1}$ are monomorphisms, we then have $f_n\simeq g_n$ for all $n\ge1$, and this implies $f=g$. 
\end{proof}

We collect a couple of elementary observations, focusing on the left closed case.
\begin{obs}
\label{obs:leftclosed}
\begin{enumerate}[label=(\arabic*)]
\item \label{rmk:left_closed_fibration} Let $F\colon \X\to \X'$ be a functor, let $\Y'\subseteq \X'$ be a full subcategory and denote $\Y:=F^{-1}(\Y')\subseteq \X$. If $\Y'$ is left (resp. right) closed in $\X'$, then so is $\Y$ in $\X$.
We combine this with the following additional observation: if $F$ is a cartesian (resp. cocartesian) fibration and $\Y'$ is initial (resp. final) in $\X'$, then so is $\Y$ in $\X$.
\item \label{rmk:left_closed_left_fibration} The inclusion $\Y \subseteq \X$ of a left closed subcategory is a right fibration: it is the pullback of the right fibration $\{0\}\to[1]$ along the essentially unique functor $\X\to[1]$ classifying $\Y$ as in
\cref{defn:categorypair}.
\item If $\Z\subseteq\Y$ and $\Y\subseteq\X$ are left closed $\baseTopos$-subcategory inclusions, then $\Z\subseteq\X$ is also left closed, as the $\baseTopos$-functor $\Y\to[1]$ classifying $\Z$ uniquely extends to $\X$.
\item \label{rmk:colimits_left_closed_subcategories} We may consider a $\baseTopos$-category pair as an object in $(\cat_\baseTopos)_{/[1]}$. We have two functors $(\cat_\baseTopos)_{/[1]}\to\cat_\baseTopos$, sending $(\X,\partial\X)\mapsto \X$ and $(\X,\partial\X)\mapsto\partial\X$. Both functors preserve colimits: for the first we observe that it is the source functor of an overcategory; for the second we observe that it admits a right adjoint $(-)^{\triangleright}\colon\cat_\baseTopos\to(\cat_\baseTopos)_{/[1]}$, sending $\D$ to the $\baseTopos$-category $\D^{\triangleright}$ obtained from $\D$ by ``freely adjoining a terminal object'': concretely, $\D^\triangleright$ is the unstraightening of the functor $[1]\to\cat_\baseTopos$ given by the terminal $\baseTopos$-functor $\D\to\ast$.
\end{enumerate}
\end{obs}

\begin{example}
If $\baseTopos=\presheaf(\B)$ as in \cref{ex:presheaftopoi}, then a $\baseTopos$-category $\X$ can be described as a functor $\X\colon\B\op\to\cat$; a left closed $\baseTopos$-subcategory $\Y\subseteq\X$ is then the same as a family of full subcategories $\Y(b)\subseteq\X(b)$, for $b\in\B$, satisfying the following:
\begin{itemize}
\item for all $b\in \B$ and all morphism $x\to y$ in $\X(b)$ with $y\in\Y(b)$ we have $x\in\Y(b)$;
\item for all $f\colon b\to b'$ in $\B$ we have $\Y(b')=(f^*)^{-1}(\Y(b))$.
\end{itemize}
\end{example}

\begin{defn}
\label{defn:complement_union_intersection}
Given left closed $\baseTopos$-subcategories $\Y=f^{-1}(0)$ and $\Y'=g^{-1}(0)$ for $\baseTopos$-functors $f,g\colon\X\to[1]$ we define the following:
\begin{itemize}
\item we let $\X\setminus\Y\subseteq\X$ denote the right closed $\baseTopos$-subcategory $f^{-1}(1)$; this definition does not depend on $f$ by \cref{lem:uniqueness_functor_to_1};
\item the intersection $\Y\cap\Y'$ is defined as the pullback $\Y\times_\X\Y'$; this coincides with the left closed $\baseTopos$-subcategory of $\X$ corresponding to the composite $\baseTopos$-functor $\X\xrightarrow{f\times g} [1]^2\xrightarrow{\max}[1]$: indeed $\Y\times_\X\Y'$ is the fibre at $(0,0)$ of $f\times g\colon\X\to[1]^2$;
\item the union $\Y\cup\Y'$ is the left closed $\baseTopos$-subcategory corresponding to the composite $\baseTopos$-functor $\X\xrightarrow{f\times g} [1]^2\xrightarrow{\min}[1]$.
\end{itemize}
We similarly define complements, intersections and unions of right closed $\baseTopos$-subcategories.
\end{defn}

\begin{rmk}
\label{rmk:left_closed_covers}
Recall from Subsection \ref{subsec:recollections} that a $\baseTopos$-category $\D$ may be regarded as a complete Segal object in $s\baseTopos$, with object of $n$-simplices given by $\D_n\coloneqq \mapTopos_{\baseTopos}([n],\D) \in \baseTopos$.

Let $\X$ be a $\baseTopos$-category, let $\Y,\Y' \subset \X$ be left closed $\baseTopos$-subcategories, and denote  $\Z\coloneqq\Y\cap\Y'$ and $\W=\Y\cup\Y'$. Then we may compute the pushout $\Y\amalg_\Z\Y'$ in $s\baseTopos$ as the $\baseTopos$-functor $\simplex\op\to\baseTopos$ sending $[n]\mapsto \Y_n \amalg_{\Z_n} \Y'_n$.
Using the left closedness assumption, for all $\tau\in\baseTopos$ and $[n]\in\simplex$ we have $\Y_n(\tau)\amalg_{\Z_n(\tau)}\Y'_n(\tau)\xrightarrow{\simeq}\W_n(\tau)$, in particular the pushout $\Y\amalg_\Z\Y'$ in $s\baseTopos$ is computed levelwise and it is isomorphic to $\W$; since $\W$ is a $\baseTopos$-category, we conclude that $\W$ is also the pushout in $\cat_\baseTopos$ of the span $Y\leftarrow\Z\to\Y'$.
\end{rmk}

\begin{constr}\label{cons:integral_categories}
Given a $\baseTopos$-category pair $(I,\partial I)$ and a functor $p \colon \X \to I$, we get a new $\baseTopos$-category pair $(\X,\partial\X)$ by setting $\partial\X = p^{-1}(\partial I)$. In particular, this gives the following construction device for $\baseTopos$-categories. Given a $\baseTopos$-category pair $(I,\partial I)$ and a $\baseTopos$-functor $X \colon I \to \cat_\baseTopos$ we get by means of the cocartesian unstraightening the \textit{unstraightened $\baseTopos$-category} $\X = \int_I X \to I$. The functor down to $I$ enhances $\X$ to a $\baseTopos$-category pair with $\partial\X = \int_{\partial I} X$.
\end{constr}

\begin{example}
\label{ex:mappingcylinder}
Suppose $X\colon I \rightarrow\baseTopos\subseteq \cat_{\baseTopos}$. If $(I,\partial I)=(*,\emptyset)$, then $\int_I X\simeq X(*)$ is a $\baseTopos$-space, 
and $\int_{\partial I}X\simeq\emptyset$. 
If $(I,\partial I)=([1],\{0\})$, then $\int_I X$ is  the ``mapping cylinder'' $(X(0)\times[1])\amalg_{X(0)\times1}X(1)$ of the morphism $g\colon X(0)\to X(1)$ in $\baseTopos$ corresponding to $X$, and $\int_{\partial I}X\simeq X(0)\times\{0\}\simeq X(0)$ is again a $\baseTopos$-space. 
\end{example}

We end the subsection by establishing a stable recollement for $\baseTopos$-category pairs.
\begin{rmk}
\label{rmk:left_closed_kan_extension}
Let $i\colon \Y \subset \X$ be a left closed $\baseTopos$-subcategory inclusion, and let $\D$ be a complete $\baseTopos$-category. Then, for $\zeta\colon \Y \to \D$, the right Kan extension $i_* (\zeta) \colon \X \to \D$ restricts to the terminal functor on the complement $\Z\coloneqq\X \setminus \Y$. To see this, we observe that the pullback $\Z\times_\X\Y$ is the ``empty'' (initial) $\baseTopos$-category $\emptyset$, as a consequence of the left closed assumption; moreover, by \cref{obs:leftclosed}\ref{rmk:left_closed_left_fibration}, the vertical maps in the pullback square
\[
\begin{tikzcd}[row sep=10pt]
\emptyset\ar[r,"\tilde j"]\ar[d,"\tilde i"']\ar[dr,phantom,"\lrcorner"very near start]&\Y\ar[d,"i"]\\
\Z\ar[r,"j"]&\X
\end{tikzcd}
\]
are right fibrations. Hence, using \cref{prop:proper_smooth_base_change}, we may compute $j^*i_*(\zeta)\simeq \tilde i_*\tilde j^*(\zeta)$, and the right Kan extension of the (unique) functor $\emptyset\to\D$ along $\tilde i$ yields the terminal functor $\Z\to\D$. This argument can be repeated levelwise in $\baseTopos$ to conclude that the $\baseTopos$-functor $j^*i_*\colon\D^\Y\to\D^\Z$ is equivalent to the composite $\D^\Y\to\ast\xrightarrow{\mathrm{ter}}\D\xrightarrow{\Z^*}\D^\Z$, where $\mathrm{ter}$ picks the terminal object.

Analogously, if $j\colon\Z \subset \X$ is right closed and $\D$ is cocomplete, then the left Kan extension of $\zeta\colon\Z\to\D$ to $\X$ restricts to the initial functor on $\Y\coloneqq\X\setminus\Z$; and $i^*j_!\colon\D^\Z\to\D^\Y$ agrees with the composite $\D^\Z\to\ast\xrightarrow{\mathrm{ini}}\D\xrightarrow{\Y^*}\D^\Y$.

If we work with presheaves as opposed to $\baseTopos$-categories of covariant $\baseTopos$-functors, we precompose with $(-)\op$, which interchanges the notion of left and right closed $\baseTopos$-subcategories. Hence the composite $\presheafTopos(\Z;\D) \xrightarrow{j_*}\presheafTopos(\X;\D) \xrightarrow{i^*} \presheafTopos(\Y;\D)$ is constant at the terminal object, and $\presheafTopos(\Y;\D) \xrightarrow{i_!} \presheafTopos(\X;\D) \xrightarrow{j^*} \presheafTopos(\Z;\D)$ is constant at the initial object.
\end{rmk}

\begin{constr}[Stable recollement for pairs]\label{constr:stable_recollement_for_pairs}
Let $(\X,\partial\X)$ be a $\baseTopos$-category pair as in \cref{nota:X_dX_interiorX}.
The fully faithful inclusion $j_* \colon \presheafTopos(\interior{\X};\sC) \hookrightarrow \presheafTopos(\X;\sC)$ may be identified with the canonical inclusion $\fib(\presheafTopos(\X;\sC)\xrightarrow{i^*}\presheafTopos(\partial\X;\sC))\subseteq \presheafTopos(\X;\sC)$. Indeed, $j^* j_* \simeq \id$ and $i^*j_* \simeq 0$ by \cref{rmk:left_closed_kan_extension}, as a consequence of $\interior{\X} \subset \X$ being right closed. 

Moreover the fibre $\fib(\presheafTopos(\X;\sC)\xrightarrow{i^*}\presheafTopos(\partial\X;\sC))\subseteq \presheafTopos(\X;\sC)$ can be computed levelwise in $\baseTopos$, and for all $\tau\in\baseTopos$ we have that $j_*(\tau)$ is fully faithful with essential image $\fib(\func_{\baseTopos_{/\tau}}(\pi_\tau^*\X\op,\pi_\tau^*\sC)\xrightarrow{\pi_\tau^*(i)}\func_{\baseTopos_{/\tau}}(\pi_\tau^*\partial\X\op,\pi_\tau^*\sC))$.
Since $i^*$ admits both adjoints which happen to be fully faithful, by \cite[Sec. A.2]{nineauthorsII} we get the stable recollement
\[
\begin{tikzcd}
\presheafTopos(\interior{\X};\sC) \rar[hook, "j_*"] & \presheafTopos(\X;\sC) \rar[two heads, "i^*"]\lar[two heads, bend right = 45, "j^*"']\lar[two heads, bend right = -40, "j^{\#}"] & \presheafTopos(\partial\X;\sC), \lar[hook, bend right = 45, "i_!"']\lar[hook, bend right = -40, "i_*"]
\end{tikzcd}
\]
i.e. the datum of three bifibre sequences in $\presentable^L_{\baseTopos,\stable}$ where each bifibre sequence is left adjoint to the one immediately below it. By general properties of a stable recollement, we may describe the $\baseTopos$-functors $j^{*} $ and $j^{\#}$ as
\[
j^{\#}\simeq j^*j_*j^{\#}\simeq \fib\big(j^*\to j^*i_*i^*\big)\quad\text{    and    }\quad  j_*j^*\simeq \cofib\big(i_!i^*\to \id_{\presheafTopos(\X;\sC)}\big).
\]
\end{constr}

The functors in \cref{constr:stable_recollement_for_pairs} are all essential for our study of Poincar\'e duality pairs, and we will often refer to them in what follows.

\subsection{Poincar\'e duality category pairs}
\label{subsec:pdcatpairs_definition}
\begin{nota}
For the rest of this article, except for the appendices, we will assume that $\sC \in \calg(\presentable^L_{\baseTopos,\stable})$ is a presentably symmetric monoidal $\baseTopos$-category that is levelwise stable. 
\end{nota}

Let $(\X,\partial \X)$ be a $\baseTopos$-category pair as in \cref{nota:X_dX_interiorX}. We have $\baseTopos$-natural transformations $\X_* \rightarrow \X_* i_* i^* \simeq \partial \X _* i^*$ and $\partial\X_!i^*\simeq \X_!i_!i^*\rightarrow\X_! $ in $\func_{\baseTopos}(\presheafTopos(\X;\sC),\sC)$ coming from the adjunction (co)units $\id\rightarrow i_*i^*$ and $i_!i^*\rightarrow \id$.
\begin{defn}
\label{defn:relative_cohomology}
We define the \textit{relative cohomology} and \textit{relative homology} $\baseTopos$-functors  as
\[ 
(\X,\partial \X)_* \coloneqq \fib(\X_* \to \partial \X_* i^*)\colon\presheafTopos(\X;\sC)\rightarrow\sC,\quad  (\X,\partial\X)_!\coloneqq\cofib(\partial\X_!i^*\to\X_!)\colon\presheafTopos(\X;\sC)\to\sC.
\]
\end{defn}
\begin{obs}\label{obs:ambidextral_interpretation}
Using \cref{constr:stable_recollement_for_pairs}, we see that the relative cohomology functor $(\X,\partial\X)_*$ is equivalent to $\interior{\X}_*j^{\#}$, and hence it 
is right adjoint to $j_*\interior{\X}^*\simeq\cofib(i_!i^*\X^*\to \X^*)$; the latter functor $\sC\to\presheafTopos(\X;\sC)$, informally, sends an object of $\sC$ to the presheaf on $\X$ which is constant at that object on $\interior{\X}$, and vanishes on $\partial\X$.
\end{obs}

Our definition of Poincar\'e duality for the pair $(\X,\partial \X)$ will be a property of the functor $(\X,\partial\X)_*$.
Namely, we want to assign a dualising system to this functor via Morita theory from \cref{sec:morita}, and define $(\X,\partial \X)$ to be a Poincar\'e duality pair provided that the dualising system is invertible. 
To do so, we need that $(\X,\partial \X)$ refines to a colimit preserving $\sC$-linear functor.
This is encapsulated in the next definition, following the terminology of \cite{Cnossen2023}.

\begin{defn}
\label{defn:ambidexterityforpairs}
A $\baseTopos$-category $\X$ is \textit{$\sC$-twisted ambidextrous} if the $\baseTopos$-functor $\X_* \colon \presheafTopos(\X;\sC) \to \sC$ with its lax $\sC$-linear structure provided by \cref{prop:right_adjoint_to_C_linear_functor_is_lax_C_linear} is in fact $\sC$-linear and preserves colimits. A $\baseTopos$-category pair $(\X,\partial \X)$ is \textit{$\sC$-twisted ambidextrous} if both $\X$ and $\partial \X$ are so.
\end{defn}

The condition on a $\baseTopos$-category to be $\sC$-twisted ambidextrous is rather strong. Philosophically, it should be thought of as a ``smallness/finiteness/compactness'' condition on $\X$ from the eyes of the levelwise stable $\sC\in\calg(\presentable^L_{\baseTopos,\stable})$. For example, in \cref{cor:ambidexterity_and_duality} we will see that if $\X$ is $\sC$-twisted ambidextrous, then $\X_!(\unit) \in \sC$ is dualisable.

\begin{lem}
\label{lem:C_linear_structure_on_relative_cohomology}
Let $(\X,\partial \X)$ be a $\sC$-twisted ambidextrous pair. Then the $\baseTopos$-functor 
$(\X,\partial \X)_*$ commutes with colimits and carries a $\sC$-linear structure so that $(\X,\partial \X)_* \rightarrow \X_* \rightarrow \partial \X_*i^*$ is a fibre sequence of $\sC$-linear $\baseTopos$-functors.
\end{lem}

\begin{proof}
The adjunction unit $\X_* \to \X_* i_* i^* \simeq \partial \X i^*$ is a map of (lax) $\sC$-linear $\baseTopos$-functors by \cref{lem:adjunction_unit_lax_linear}, so it is in fact a map in $\func_{\baseTopos,\sC}^L(\presheafTopos(\X;\sC),\sC) \simeq \func_\baseTopos(\X,\sC)$, which is a category that admits fibres.
We conclude by showing that both forgetful functors in the composite
\[ \func_{\baseTopos,\sC}^L(\presheafTopos(\X;\sC),\sC) \to \func^L_{\baseTopos}(\presheafTopos(\X;\sC),\sC) \to  \func_{\baseTopos}(\presheafTopos(\X;\sC),\sC) \]
preserve fibres. By levelwise stability of $\sC$, it suffices to check that both functors preserve colimits. This is clear for the second functor, whereas the first functor identifies as
\[ 
\func_{\baseTopos,\sC}^L(\presheafTopos(\X;\sC),\sC) \simeq \func_{\baseTopos}(\X,\sC) \to  \func_{\baseTopos}(\X,\funTopos^L_{\baseTopos}(\sC,\sC)) \simeq \func^L_{\baseTopos}(\presheafTopos(\X;\sC),\sC) \]
where the middle functor is induced by the colimit preserving $\baseTopos$-functor $\sC \to \funTopos^L_\baseTopos(\sC,\sC)$ adjoint to the tensor product $\sC\otimes\sC\to\sC$ morphism in $\presentable^L_{\baseTopos}$.
\end{proof}

\begin{obs}
Let $(\X,\partial \X)$ be a $\baseTopos$-category pair as in \cref{nota:X_dX_interiorX}. 
The adjunction $j_*\interior{\X}^* \dashv (\X,\partial \X)_*$ from \cref{obs:ambidextral_interpretation} always equips $(\X,\partial \X)_*$ with a lax $\sC$-linear structure. Indeed, the lax $\sC$-linear structure on $j_*$ as a right adjoint to $j^*$ is in fact $\sC$-linear, as it is an ``extension by zero'', see \cref{obs:ambidextral_interpretation}. Hence $j_* \interior{X}^*$ attains a $\sC$-linear structure as a composite of $\sC$-linear functors, and so its right adjoint carries a lax $\sC$-linear structure.
If both $\partial \X$ and $\X$ are $\sC$-twisted ambidextrous, the latter lax $\sC$-linear structure agrees with the one from \cref{lem:C_linear_structure_on_relative_cohomology}, hence it is in fact a $\sC$-linear structure.
\end{obs}

\begin{defn}
Given a $\sC$-twisted ambidextrous $\baseTopos$-category pair $(\X,\partial\X)$, we define the \textit{classifying system} $\omega_{\X,\partial \X} \in \func_{\baseTopos}(\X,\sC)$ to be the image of $(\X,\partial \X)_*$ under the equivalence of categories $\yoneda^*\colon \func_{\baseTopos,\sC}^L(\presheafTopos(\X;\sC),\sC) \xrightarrow{\simeq} \func_{\baseTopos}(\X,\sC)$ from \cref{prop:classification_of_linear_functors}, so that $\omega_{\X,\partial\X}\simeq (\X,\partial\X)_*\circ\yoneda$. When $\partial\X=\emptyset$, we write $\omega_\X\coloneqq \omega_{\X,\emptyset}\in\func_{\baseTopos}(\X,\sC)$  corresponding to $\X_*$.
\end{defn}

We now come to the main definition of this article.
\begin{defn}[$\sC$-Poincar\'e $\baseTopos$-category pair]
\label{defn:Poincarecatpair}
Let $\sC\in\calg(\presentable^L_{\baseTopos,\stable})$ be a levelwise stable presentably symmetric monoidal $\baseTopos$-category. A $\baseTopos$-category pair $(\X,\partial \X)$ is \textit{$\sC$-Poincar\'e} if it is $\sC$-twisted ambidextrous and, moreover, $\omega_{\X,\partial \X}$ takes values in the $\baseTopos$-subgroupoid $\picardSpaceTopos(\sC) \subset \sC^{\simeq}$ of invertible objects in $\sC$. We say that $\X\in\cat_{\baseTopos}$ is $\sC$-Poincar\'e if the pair $(\X,\emptyset)$ is $\sC$-Poincar\'e.
\end{defn}

The following definition is weaker than the requirement of satisfying $\sC$-Poincar\'e duality, but is a very useful intermediate notion that facilitating many theoretical considerations. 
\begin{defn}[Groupoidal ambidexterity]
    \label{defn:C_groupoidal_ambidexterity}
    A $\sC$-twisted ambidextrous $\baseTopos$-category pair $(\X,\partial \X)$ is \textit{groupoidally $\sC$-twisted ambidextrous} if $\omega_{\X,\partial \X} \in \sC^{\X}$ takes values in $\sC^{\simeq}$.
\end{defn}

\begin{prop}
\label{obs:omega_natural_in_C}
    Let $f \colon \sC \rightarrow \sC'$ be a morphism in $\calg(\presentable^L_\baseTopos)$.
    Then any $\sC$-twisted ambidextrous $\baseTopos$-category pair $(\X, \partial \X)$ is also $\sC'$-twisted ambidextrous, classified by the composite $\X\xrightarrow{\omega_{\X,\partial\X}}\sC\to\sC'$.
    In particular, if $(\X, \partial \X)$ is groupoidally $\sC$-twisted ambidextrous (resp. Poincar\'e), then $(\X, \partial \X)$ is also groupoidally $\sC'$-twisted ambidextrous (resp. Poincar\'e).
\end{prop}
\begin{proof}
    The functor $f$ induces a functor of $2$-categories $\ind_f=\sC'\otimes_\sC- \colon \module_{\sC}(\presentable^L_\baseTopos) \rightarrow \module_{\sC'}(\presentable^L_\baseTopos)$. 
    Now, if $\X$ is $\sC$-twisted ambidextrous, then $\X^* \colon \sC \rightarrow \presheafTopos(\X;\sC)$ is an internal left adjoint in the 2-category $\module_{\sC}(\presentable^L_\baseTopos)$.
    Its image $\X^* \colon \sC' \rightarrow \presheafTopos(\X;\sC')$ under the 2-functor $\ind_f$ is an interal left adjoint in $\module_{\sC}(\presentable^L_\baseTopos)$, showing that $\X$ is $\sC'$-twisted ambidextrous.
    The same will hold for pairs $(\X, \partial \X)$.

    Moreover, by \cref{obs:naturality_in_C}, we obtain after applying $\Gamma$ that the composite $\baseTopos$-functor $\X\xrightarrow{\omega_{\X,\partial\X}}\sC\to\sC'$ is precisely the $\baseTopos$-functor classifying the $\sC'$-linear functor $\sC'\otimes_\sC(\X,\partial\X)_*$. This implies that if $(\X,\partial\X)$ is groupoidally $\sC$-ambidextrous (resp. $\sC$-Poincar\'e), then it is also groupoidally $\sC'$-ambidextrous (resp. $\sC'$-Poincar\'e).
\end{proof}

\begin{obs}
Recall \cref{obs:formula_for_dualising_object}, and let $(\X,\partial\X)$ be a $\sC$-twisted ambidextrous $\baseTopos$-category pair. Then for $x\in\Gamma\X$, considered as a $\baseTopos$-functor $x\colon\ast\to\X$, we have that $\Gamma\omega_{\X,\partial\X}(x)\in\Gamma\sC$ agrees with the image of $\unit_{\Gamma\sC}$ along the composite functor $\Gamma\sC\simeq\presheaf(*;\sC)\xrightarrow{x_!}\presheaf(\X;\sC)\xrightarrow{(\X,\partial\X)_*}\Gamma\sC$. If $(\X,\partial\X)$ is \textit{groupoidally} $\sC$-twisted ambidextrous, then the same description holds for $\Gamma D_{\X,\partial\X}(x)$. 
This extends the validity of Klein's formula for the dualising spectrum in \cite{Kleindual}.
\end{obs}

As a motivating example, we immediately give our definition of Poincar\'e pair of $\baseTopos$-spaces. We will compare this notion to Wall's classical classical one in \cref{subsec:pdpairs}.

\begin{defn}
\label{defn:pdpairs}
Recall \cref{ex:mappingcylinder},
let $X\colon [1]\to\baseTopos$ be a functor, and let $(\X,\partial\X)\coloneqq(\int_{[1]}X,\int_{\{0\}}X)$ denote the associated unstraightened $\baseTopos$-category pair, with $\partial\X\simeq X(0)\in\baseTopos$.
Let $\sC\in\calg(\presentable^L_{\baseTopos,\stable})$. We say that $X$ is a \textit{$\sC$-Poincar\'e pair of $\baseTopos$-spaces} if the classifying system $\omega_{\X,\partial\X}\in\func_{\baseTopos}(\X,\sC)$
factors through $\picardSpaceTopos(\sC)$. If $X(0)\simeq\emptyset$, we say that $X(1)\simeq\X$ is a \textit{$\sC$-Poincar\'e $\baseTopos$-space}, which recovers \cite[Definition 3.2.9]{PD1}.
\end{defn}

We end this subsection with some easy examples of $\sC$-twisted ambidextrous $\baseTopos$-category pairs coming from the following two closure properties.

\begin{rmk}[Closure under retracts]
\label{rmk:ambidexterity_closed_under_retracts}
The class of $\sC$-twisted ambidextrous $\baseTopos$-categories  is closed under retracts. Indeed, if $\X \xrightarrow{s} \Y \xrightarrow{r} \X$ is a retraction in $\cat_\baseTopos$, then we get a retraction $\X_* \rightarrow \Y_* r^* \rightarrow \X_* s^* r^* \simeq \X_*(rs)^* \simeq \X_*$
of lax $\sC$-linear functors $\presheaf(\X;\sC) \rightarrow \sC$. Since $r^*$ is a $\sC$-linear functor and since $\sC$-linear functors are closed under composition and retracts, we see that $\X$ is $\sC$-linear provided $\Y$ is. 
\end{rmk}

\begin{rmk}[Closure under finite colimits]
\label{rmk:ambidexterity_closed_under_finite_colimits}
If $\sC\in\calg(\presentable^L_{\baseTopos,\stable})$ is levelwise stable, then the class of $\sC$-twisted ambidextrous $\baseTopos$-categories is closed under finite fibrewise colimits. Indeed, it contains the empty $\baseTopos$-category and is closed under pushouts. To see the latter, let 
\[
\begin{tikzcd}[row sep=12pt]
\X \ar[r, "f'"] \ar[d, "g'"]\ar[dr,phantom,"\ulcorner"very near end] & \X' \ar[d, "g"]\\
\Y \ar[r, "f"] & \Y'
\end{tikzcd}
\]
be a pushout in $\cat_\baseTopos$ and set $h= gf'$. We get an equivalence $\Y'_* \simeq (\Y_* f^*) \times_{\X_* h^*} (\X'_* g^*)$.
The stability assumption on $\sC$ implies that a pullback of colimit preserving $\baseTopos$-functors is again colimit preserving, and a pullback of $\sC$-linear functors, when endowed with the induced lax $\sC$-linear structure, is $\sC$-linear; these remarks apply to the latter functor.
\end{rmk}

\begin{lem}
\label{lem:compact_is_ambidextrous}
Let $\sC\in\calg(\presentable^L_{\baseTopos,\stable})$. Then every $\infty$-category pair $(\X,\partial\X)$ with $\X$ a compact $\infty$-category, when considered as constant a $\baseTopos$-category pair, is $\sC$-twisted ambidextrous. In particular every compact $\infty$-category, when considered as a $\baseTopos$-category, is $\sC$-ambidextrous.
\end{lem}
\begin{proof}
The presentable $\infty$-category $(\cat_\infty)_{/[1]}$ is generated under colimits by the three objects $[1]\xrightarrow{0}[1]$, $[1]\xrightarrow{\id}[1]$ and $[1]\xrightarrow{1}[1]$, i.e. by the three $\infty$-category pairs $([1],\emptyset)$, $([1],\{0\})$, and $([1],[1])$; its compact objects are generated under retracts and finite colimits by the same three objects. Using \cref{obs:leftclosed}\ref{rmk:colimits_left_closed_subcategories}, we obtain that an $\infty$-category pair $(\X,\partial\X)$ is a compact object in $(\cat_\infty)_{/[1]}$ if and only if $\X$ is compact in $\cat_\infty$; and if this is the case, then also $\partial\X$ is compact.

Since the functor $\constant\colon\cat\to\cat_\baseTopos$ preserves colimits, it suffices to check that the three generating compact $\infty$-category pairs are $\sC$-twisted ambidextrous when considered as $\baseTopos$-category pairs. For this, we just identify $([1],\emptyset)_*\simeq\eval_1$, $([1],\{0\})_*\simeq\fib(\eval_1\to\eval_0)$, and $([1],[1])_*\simeq0$, so that all these three lax $\sC$-linear $\baseTopos$-functors $\presheafTopos([1];\sC)\to\sC$ are in fact $\sC$-linear and colimit preserving.
\end{proof}

\subsection{The connecting map of classifying systems}
Next, we explain how our categorified approach to Poincar\'e duality pairs automatically builds in the structure of \textit{connecting maps} which witness the lax compatibilities between local coefficients of the ambient $\baseTopos$-category $\X$ and its boundary $\partial\X$. For the convenience of our discussions later, we first introduce the following definition generalising relative (co)homology as introduced in the previous subsection.

\begin{defn}
\label{defn:generalised_relative_cohomology}
Let $u\colon \Y\to\Z$ be a $\baseTopos$-functor between small $\baseTopos$-categories. We denote by 
\[
(\Z,\Y,u)_*\coloneqq\fib(\Z_*\xrightarrow{\eta_{u_*}}\Y_*u^*)\quad \text{ and } \quad (\Z,\Y,u)_!\coloneqq\cofib(\Y_!u^*\xrightarrow{\epsilon_{u_!}}\Z_!)\colon\presheafTopos(\Z;\sC)\longrightarrow \sC
\]
the \textit{generalised relative cohomology} and \textit{generalised relative homology} $\baseTopos$-functors, respectively. If both $\Y$ and $\Z$ are twisted $\sC$-ambidextrous, then $(\Z,\Y,u)_*$ is colimit preserving and $\sC$-linear (compare with \cref{lem:C_linear_structure_on_relative_cohomology}) and we let 
$\omega_{\Z,\Y,u}\in\func_\baseTopos(\Z,\sC)$ denote its classifying system.
\end{defn}

For $u$ a left closed inclusion, \cref{defn:generalised_relative_cohomology} recovers \cref{defn:relative_cohomology}. 

\label{subsec:connecting_map_of_classifying_systems}
\begin{defn}[Connecting map of classifying systems]\label{defn:connecting_map_of_classifying_systems}
Using the notation from \cref{defn:generalised_relative_cohomology},
we denote by $\delta\colon \Omega \omega_\Y \longrightarrow j^*\omega_{\Z,\Y,u}$, leaving $\Y$, $\Z$ and $u$ implicit, the morphism in $\func_\baseTopos(\Y,\sC)$ corresponding via \cref{lem:naturality_classification_C_linear_functors} to the following composite $\sC$-linear $\baseTopos$-natural transformation, in which 
$\canonical\colon\Omega\Y_*u^*\to(\Z,\Y,u)_*$ is the canonical $\baseTopos$-natural transformation coming from the definition of $(\Z,\Y,u)_*$ as a fibre:
\[
\begin{tikzcd}
\Omega \Y_* \xlongrightarrow{\eta_{u_!}}\Omega \Y_*u^*u_! \xlongrightarrow{\canonical} (\Z,\Y,u)_*u_!.
\end{tikzcd}
\]
\end{defn}

\begin{nota}\label{nota:left_adjoint_to_j}
Let $(\X,\partial\X)$ be a $\baseTopos$-category pair and assume that the $\baseTopos$-functor $f\colon\X\to[1]$ classifying $\partial\X$ as in 
\cref{defn:categorypair} is a $\baseTopos$-cocartesian fibration. 
In this setting, we usually denote by $(\partial\X \xrightarrow{g} \interior{\X}) \in \func([1],\cat_\baseTopos)$ the straightening of $f$, and by $\ell\colon\X\to\interior{\X}$ the left adjoint to the inclusion $j$; note that $g\simeq \ell i$.
\end{nota}

\begin{prop}
\label{obs:connecting_map_of_classifying_systems}
Consider a category pair $(\X,\partial\X)$ arising from a cocartesian fibration $\X \to [1]$.
Then the map $i^*\omega_{\X,\partial\X}\to g^*j^*\omega_{\X,\partial\X}$ induced by the adjunction unit $\id \to l^* j^*$ identifies with the connecting map $\delta\colon\Omega\omega_{\partial\X}\to g^*\omega_{\interior{\X},\partial\X,g}$ from \cref{defn:connecting_map_of_classifying_systems}, where $g$ is as in  \cref{nota:left_adjoint_to_j}.
\end{prop}
\begin{proof}
We have a commutative triangle of $\baseTopos$-functors $\presheafTopos(\interior{\X};\sC)\to\sC$
\[
\begin{tikzcd}[row sep=2pt,column sep=50pt]
&\X_*\ell^*\ar[dr,"\eta_{i_*}"]\\
\interior{\X}_*\ar[ur,"{\eta_{\ell_*}, \simeq}"]\ar[rr,"\eta_{g_*}"']& &\partial\X_* f^*,
\end{tikzcd}
\]
giving rise to an identification $(\X,\partial\X)_*\ell^*\simeq(\interior{\X},\partial\X,g)_*$. This gives in particular a formula of generalised relative cohomology from \cref{defn:generalised_relative_cohomology} in terms of relative cohomology from \cref{defn:relative_cohomology}.
Moreover, if $(\X,\partial\X)_*$ is colimit preserving and $\sC$-linear, then so is $(\interior{\X},\partial\X)_*$ and we obtain an equivalence $j^*\omega_{\X,\partial\X}\simeq \ell_!\omega_{\X,\partial\X}\simeq\omega_{\interior{\X},\partial\X,g}$. 

We also have a commutative diagram as follows, in which the arrow labeled $\eta_{\ell_*}$ is an equivalence because $\ell^*$ is fully faithful, and the bent arrow is an equivalence by \cref{prop:boundary_principle}:
\[
\begin{tikzcd}
\Omega\partial\X_*\ar[rr,bend left=10,"\simeq"]\ar[r,"\eta_{i_!}"']\ar[d,equal]&\Omega\partial\X_*i^*i_!\ar[d,"\eta_{\ell_!}"]\ar[r,"\canonical"']&(\X,\partial\X)_*i_!\ar[d,"\eta_{\ell_!}"]\ar[r]&\X_*i_!\ar[d,"\eta_{\ell_!}"]\ar[r,"\eta_{i_*}"]&\partial\X_*i^*i_!\ar[d,"\eta_{\ell_!}"]\\
\Omega\partial\X_*\ar[r,"\eta_{g_!}"]&\Omega\partial\X_*g^*g_!\ar[r,"\canonical"] &(\X,\partial\X)_*\ell^*g_!\ar[r]&\X_*\ell^*g_!\ar[r,"\eta_{i_*}"]&\partial\X_*g^*g_!\\
\Omega\partial\X_*\ar[u,equal]\ar[r,"\eta_{g_!}"]&\Omega\partial\X_*g^*g_!\ar[u,equal]\ar[r,"\canonical"] &(\interior{\X},\partial\X,g)_*g_!\ar[r],\ar[u,"\simeq"]&\interior{\X}_*g_!\ar[u,"\eta_{\ell_*}"',"\simeq"]\ar[r,"\eta_{g_*}"] &\partial\X_*g^*_!.\ar[u,equal]
\end{tikzcd}
\]
The diagram allows us to identify the connecting natural transformation $\Omega\partial\X_*\to(\interior{\X},\partial\X,g)_*g_!$ with the natural transformation labeled $\eta_{\ell_!}$ in the top middle, and hence it allows us to identify the connecting map $\delta\colon\Omega\omega_{\partial\X}\to g^*\omega_{\interior{\X},\partial\X,g}$ with the map $i^*\omega_{\X,\partial\X}\to g^*\ell_!\omega_{\X,\partial\X}$ also induced by $\eta_{\ell_!}$, thanks to \cref{lem:naturality_classification_C_linear_functors}.

We may then identify the adjunction $\ell_!\colon\func_\baseTopos(\X,\sC)\rightleftharpoons\func_\baseTopos(\interior{\X},\sC)\cocolon \ell^*$ with the adjunction $j^*\dashv \ell^*$ induced by the adjunction $b\dashv j$, as the two adjunctions share the right adjoint. 
And thus we may further identify $\delta$ with the map $i^*\omega_{\X,\partial\X}\to g^*j^*\omega_{\X,\partial\X}$ induced by the $\baseTopos$-natural transformation $\eta_{b\dashv j}i\colon i\to jbi\simeq jg$.
Note that the latter $\baseTopos$-natural transformation adjoints to a map $\partial\X\times[1]\to\X$ of $\baseTopos$-cocartesian fibrations over $[1]$, whose straightening is the following commutative square, seen as a morphism in $\func([1],\cat_\baseTopos)$ between vertical arrows:
\[
\begin{tikzcd}[row sep=8pt]
\partial\X\ar[d,equal]\ar[r,equal]&\partial\X\ar[d,"g"]\\
\partial\X\ar[r,"g"]&\interior{\X},
\end{tikzcd}
\]
\end{proof}

The following is a direct consequence of \cref{obs:connecting_map_of_classifying_systems}.
\begin{cor}
\label{cor:towards_neoclassical_equivalence}
Let $(\X,\partial\X)$ be a $\sC$-twisted ambidextrous $\baseTopos$-category pair arising from a cocartesian fibration $\X \to [1]$.
Then $(\X,\partial\X)$ is groupoidally $\sC$-twisted ambidextrous (respectively, $\sC$-Poincar\'e) if and only if both $\omega_{\partial\X}$ and $\omega_{\interior{\X},\partial\X,g}$ factor through $\sC^\simeq$ (respectively, through $\picardSpaceTopos(\sC)$) and the connecting map $\delta\colon\Omega\omega_{\partial\X}\to g^*\omega_{\interior{\X},\partial\X,g}$ from \cref{defn:connecting_map_of_classifying_systems} is an equivalence.
\end{cor}

\subsection{Cutting and pasting principles}
\label{subsec:cut_and_paste}
We now develop  various devices  to check $\sC$-Poincar\'e duality via local-global principles and to infer $\sC$-Poincar\'e duality on certain localisations of a $\baseTopos$-category pairs, which in some cases can be interpreted as a technique to glue $\sC$-Poincar\'e $\baseTopos$-category pairs. 
As usual, we assume throughout that $\sC$ is a levelwise stable presentably symmetric monoidal $\baseTopos$-category.
We begin by recording the following easy observation. 
\begin{lem}\label{lem:vanishing_limit_extension_by_zero}
Let $(\X,\partial\X)$ be a $\baseTopos$-category pair and assume that the inclusion of the complement $j\colon \interior{\X}\subseteq\X$ is final.
Then $\X_*i_! \simeq 0$ in $\func_{\baseTopos}(\presheafTopos(\X;\sC),\sC)$.
\end{lem}
\begin{proof}
Since $j$ is final, $j\op\colon \interior{\X}\op \to \X\op$ is initial, and in particular the unit of the adjunction $\X_* \to \X_*j_*j^*\simeq \interior{\X}_* j^*$ is an equivalence. So to show $\X_* i_! \simeq 0$ it suffices to prove $j^* i_! \simeq 0$, and this follows from \cref{rmk:left_closed_kan_extension}.
\end{proof}

This result is the reason why we shall often assume that the inclusion $\interior{\X}\subseteq \X$ to be final for various pasting results. Let us package this in a definition.
\begin{defn}
\label{defn:tame_pair}
We call a $\baseTopos$-category pair $(\X,\partial\X)$ tame if the inclusion $\interior{\X}\subseteq \X$ is final.
\end{defn}

The next proposition allows us to compute the classifying system of the boundary of a $\baseTopos$-category pair in terms of that of the pair itself.
\begin{prop}[Boundary principle]
\label{prop:boundary_principle}
Let $(\X,\partial\X)$ be a $\sC$-twisted ambidextrous tame $\baseTopos$-category pair. Then the connecting map $\delta \colon \Omega \omega_{\partial\X} \to i^* \omega_{\X,\partial\X}$ from \cref{defn:connecting_map_of_classifying_systems} is an equivalence.
\end{prop}
\begin{proof}
Composing the defining fibre sequence for $(\X, \partial\X)_*$ with  the $\sC$-linear $\baseTopos$-functor $i_! \colon \presheafTopos(\partial\X;\sC) \to \presheafTopos(\X;\sC)$, we obtain the fibre sequence
\begin{equation}\label{eq:fiber_sequence_dualising_sheaf_boundary}
(\X, \partial\X)_* i_! \to \X_* i_! \to \partial\X_* i^* i_!
\end{equation}
of $\sC$-linear $\baseTopos$-functors $\presheafTopos(\partial\X;\sC)\to\sC$.
As $i$ is fully faithul, the adjunction unit $\id_{\presheafTopos(\partial\X;\sC)} \to i^* i_!$ is an equivalence of $\sC$-linear $\baseTopos$-functors; moreover \cref{lem:vanishing_limit_extension_by_zero} shows that $\X_* i_! \simeq 0$.
This means that the connecting map $\Omega \partial\X_* \to (\X, \partial\X)_* i_!$ is an equivalence.
Under \cref{lem:naturality_classification_C_linear_functors}, this corresponds to the connecting map $\delta \colon \Omega \omega_{\partial\X} \to i^* \omega_{\X,\partial\X}$ in the statement.
\end{proof}
\begin{cor}
Suppose that $(\X,\partial\X)$ is a $\sC$-Poincar\'e tame $\baseTopos$-category pair. Then $\partial\X$ also is a $\sC$-Poincar\'e $\baseTopos$-category.
\end{cor}

The next proposition allows us to restrict the classifying system from a larger to a smaller $\baseTopos$-category pair.
\begin{prop}[Complementation principle]
\label{prop:complementation_principle}
Let $(\X,\partial\X)$ be a $\baseTopos$-category pair and assume that $\X$ is written as a union $\X_1\cup \X_2$ of two left closed subcategories as in \cref{defn:complement_union_intersection}. Denote $\W\coloneqq\X_1\cap \X_2$, $\partial\X_k=\partial\X\cap \X_k$ for $k=1,2$, and $\partial\W=\partial\X\cap \W$.
\begin{enumerate}[label=(\arabic*)]
\item Assume that all $\baseTopos$-category pairs $(\X_k,\partial\X_k)$ and $(\W,\partial\W)$ are $\sC$-twisted ambidextrous; then also $(\X,\partial\X)$ is $\sC$-twisted ambidextrous.
\item  Assume further that both of the inclusions $\X_2\setminus \W\subseteq \X_2$ and $\partial\X_2\setminus\partial\W\subseteq \partial\X_2$ are final. Then  $\omega_{\X_1,\partial\X_1\cup \W}\simeq\omega_{\X,\partial\X}|_{\X_1}$.
\end{enumerate}
\end{prop}
\begin{proof}
Denote by $i_k \colon \X_k \hookrightarrow \X$, $j_k \colon \W \hookrightarrow \X_k$ and $i \colon \W \hookrightarrow X$ the respective inclusions.
By \cref{rmk:left_closed_covers}, $\X$ and $\partial\X$ are the pushouts in $\cat_{\baseTopos}$ of the span diagrams  $\X_1 \hookleftarrow \W\hookrightarrow \X_2$ and $\partial\X_1\hookleftarrow\partial\W\hookrightarrow\partial\X_2$, respectively. As a consequence $(\X,\partial \X)$ is $\sC$-twisted ambidextrous; moreover 
the canonical $\sC$-linear map is an equivalence:
\begin{equation}
\label{eq:limit_over_union}
(\X, \partial\X)_* \xrightarrow{\simeq} (\X_1, \partial\X_1)_* i_1^* \times_{(\W, \partial\W)_*i^*} (\X_2, \partial\X_2)_* i_2^*. 
\end{equation}

For point (2), consider the following cartesian square in $\cat_{\baseTopos}$:
\[
\begin{tikzcd}[row sep=10pt, column sep=40pt]
\W\op \ar[r, "j_1\op"] \ar[d, "j_2\op"]\ar[dr,phantom,"\lrcorner"very near start] & \X_1\op \ar[d, "i_1\op"] \\
\X_2\op \ar[r, "i_2\op"] & \X\op.
\end{tikzcd}
\]
Recall from 
\cref{obs:leftclosed}\ref{rmk:left_closed_left_fibration} that $i_2\op$ is a right fibration.
It follows from \cref{prop:proper_smooth_base_change} that the Beck--Chevalley transformation $(j_2)_! (j_1)^* \to (i_2)^* (i_1)_!$ is an equivalence.
We then obtain that $(\X_2)_* (i_2)^* (i_1)_! \simeq (\X_2)_* (j_2)_! (j_1)^* \simeq 0$ from \cref{lem:vanishing_limit_extension_by_zero}.
With a similar argument, now using that $\partial\X_2 \setminus \partial\W \subseteq \partial\X_2$ is final, one can show that $(\partial\X_2)_* (i_2)^* (i_1)_! \simeq 0$.
Together this gives $(\X_2, \partial\X_2)_* (i_2)^* (i_1)_! \simeq 0$.
Precomposing the equivalence from  \cref{eq:limit_over_union} with $(i_1)_!$, and using also that the unit of the adjunction $(i_1)_!\dashv (i_1)^*$ is an equivalence, we obtain the following $\sC$-linear equivalence
\[
(\X, \partial\X)_* (i_1)_! \simeq \fib((\X_1, \partial\X_1)_* \to (\W, \partial\W)_*j_1^*) \simeq (\X_1, \partial\X_1 \cup \W)_*.
\]
By \cref{lem:naturality_classification_C_linear_functors}, the left side is classified by the system $i_1^*(\omega_{X, \partial\X})$, from which we obtain the desired equivalence $\omega_{\X_1, \partial\X_1 \cup \W} \simeq i_1^* (\omega_{\X, \partial\X})$.
\end{proof}
Motivated by \cref{prop:complementation_principle} we give the following definition.
\begin{defn}\label{defn:complements_of_subcategory_pairs}
Let $(\X,\partial\X)$ be a $\baseTopos$-category pair. A pair of left closed full $\baseTopos$-subcategories $\partial\Y \subseteq\Y\subseteq\X$ is called \textit{complemented} in $(\X,\partial\X)$ if there exists a pair of left closed $\baseTopos$-subcategories $\partial\Z\subseteq\Z\subseteq\X$ such that the following hold:
\begin{enumerate}[label=(\arabic*)]
\item $\X=\Y\cup\Z$;
\item writing $\W\coloneqq\Y\cap\Z$ and $\partial\W\coloneqq\partial\X\cap\W$ we have $\partial\Y\simeq(\partial\X \cap\Y)\cup\W$ and $\partial\Z\simeq (\partial\X\cap\Z)\cup\W$;
\item the inclusions $\Z\setminus\W \subseteq\Z$ and $\partial\Z\setminus \partial\W\subseteq\partial\Z$ are final $\baseTopos$-functors.
\end{enumerate}
In this case, we call $(\Z,\partial\Z)$ a \textit{complement} for $(\Y,\partial\Y)$.
\end{defn}

\begin{terminology}[Complemented covers]\label{term:complemented_covers}
For a $\baseTopos$-category pair $(\X,\partial\X)$, we say that a finite collection $\{(\Y_i,\partial\Y_i)\}$ of left closed $\baseTopos$-subcategories $\Y_i\subseteq \X$ equipped with further left closed $\baseTopos$-subcategories $\partial\Y_i\subseteq \Y_i$ form a \textit{complemented cover} of $(\X,\partial\X)$ if each $(\Y_i,\partial\Y_i)$ is complemented in $(\X,\partial\X)$ as in \cref{defn:complements_of_subcategory_pairs}, and $\X=\bigcup_i\Y_i$ as in \cref{defn:complement_union_intersection}.
\end{terminology}

\begin{cor}[Local-to-global principle]\label{cor:local_to_global_principle}
Let $(\X,\partial\X)$ be a $\sC$-twisted ambidextrous $\baseTopos$-category pair together with a complemented cover $\{(\Y_i,\partial\Y_i)\}_{i \in I}$ by $\sC$-twisted ambidextrous $\baseTopos$-category pairs.
Then $(\X,\partial\X)$ is groupoidally $\sC$-twisted ambidextrous (resp. $\sC$-Poincar\'e) if and only if for each $i \in I$ the pair $(\Y_i, \partial\Y_i)$ is groupoidally $\sC$-twisted ambidextrous (resp. $\sC$-Poincar\'e).
\end{cor}
\begin{proof}
From \cref{prop:complementation_principle} it follows that $\omega_{\Y_i,\partial\Y_i}$ is the restriction of $\omega_{\X,\partial\X}$; so if $(\X,\partial\X)$ is groupoidally $\sC$-twisted ambidextrous (resp. $\sC$-Poincar\'e), then so is $(\Y_i,\partial\Y_i)$ for all $i$. 

Conversely, assume that each $(\Y_i,\partial\Y_i)$ is groupoidally $\sC$-twisted ambidextrous (resp. $\sC$-Poincar\'e).
Note that the functor $j \colon \coprod_{i \in I} \Y_i \to \X$ is an effective epimorphism, meaning that $\X$ is equivalent to the realisation of the associated \v{C}ech nerve.
By the resulting equivalence 
\[
\funTopos(\X, \sC) \xrightarrow{\simeq} \lim_{[n] \in \Delta} \funTopos\left(\left(\coprod\Y_i\right)^{\times_\X n+1}, \sC\right)
\]
it suffices to show the restriction of $\omega_{\X,\partial\X}$ to each $\Y_i$ factors through $\sC^\simeq$ (resp. $\picardSpaceTopos(\sC)$), which is true as it identifies with $\omega_{\Y_i,\partial\Y_i}$, again by \cref{prop:complementation_principle}.
\end{proof}

\begin{example}
\label{ex:gluingXY}
In the setting of \cref{prop:complementation_principle}, assume that both pairs of inclusions $\X_k\setminus \W\subseteq \X_k$ and $\partial\X_k\setminus\partial\W\subseteq \partial\X_k$ are final, for $k=1,2$. Then  $\{(\X_k,\partial\X_k\cup\W)\}_{k=1}^2$ is a complemented cover of $(\X,\partial\X)$, hence $(\X,\partial\X)$ is groupoidally $\sC$-twisted ambidextrous (resp. $\sC$-Poincar\'e) if and only if both $\baseTopos$-category pairs $(\X_k,\partial\X_k\cup \W)$, for $k=1,2$, are so.
\end{example}

\begin{example}[Doubling principle]
\label{ex:doubling_principle}
Let $(\Y,\partial\Y)$ be a $\sC$-twisted ambidextrous tame $\baseTopos$-category pair.
We may then construct the $\baseTopos$-category $\X=\Y\amalg_{\partial\Y}\Y$, by ``doubling'' $\Y$ along $\partial\Y$. The hypothesis that $(\Y,\partial\Y)$ is tame ensures that $\X$ admits a complemented cover by two copies of $(\Y,\partial\Y)$; it follows from \cref{cor:local_to_global_principle} that $\X$ is a $\sC$-Poincar\'e $\baseTopos$-category if and only if $(\Y,\partial\Y)$ is a $\sC$-Poincar\'e $\baseTopos$-category pair. We will use variations of this observation in some proofs, thus reducing from the ``relative case'' to the ``absolute case'' with empty boundary.
\end{example}

The next lemma allows us to transfer complemented covers along cocartesian fibrations.
\begin{lem}\label{lem:pulling_back_complemented_pairs}
Let $(\X,\partial\X)$ be a $\baseTopos$-category pair and let $p \colon \E \to \X$ be a $\baseTopos$-cocartesian fibration. Set $\partial\E = p^{-1}(\partial\X)$. Assume that $\partial\Y\subseteq\Y \subseteq\X$ is a complemented $\baseTopos$-subcategory pair. Then $p^{-1}(\partial\Y) \subseteq p^{-1}(\Y) \subseteq\E$ is also a complemented $\baseTopos$-subcategory pair. In particular, if $\{(\Y_i,\partial\Y_i)\}_{i \in I}$ is a complemented cover of $(\X,\partial\X)$, then $\{(p^{-1}(\Y_i),p^{-1}(\partial\Y_i))\}_{i\in I}$ is a complemented cover of $(\E,\partial\E)$.
\end{lem}
\begin{proof}
Let $(\Z,\partial\Z)$ be a complement for $(\Y,\partial\Y)$, and set $\W=\Y \cap\Z$ and $\partial\W = \partial\X \cap\W$. Then $(p^{-1}(\Z),p^{-1}(\Y))$ is a complement for $(p^{-1}(\Y),p^{-1}(\partial\Y))$ by \cref{obs:leftclosed}\ref{rmk:left_closed_fibration}.
\end{proof}

The next proposition allows us to identify the classifying system of the localisation of a $\baseTopos$-category pair.
\begin{prop}[Localisation principle]
\label{prop:localisation_principle}
Let $f \colon (\X,\partial\X)\to(\Y,\partial\Y)$ be a map of $\baseTopos$-category pairs. Assume that both $\baseTopos$-functors $f^*\colon\presheafTopos(\Y;\sC)\to\presheafTopos(\X;\sC)$ and $\partial f^*\colon\presheafTopos(\partial\Y;\sC)\to\presheafTopos(\partial\X;\sC)$ are fully faithful. Also suppose that $(\X, \partial \X)$ is $\sC$-twisted ambidextrous. Then:
\begin{enumerate}
    \item the pair $(\Y,\partial\Y)$ is $\sC$-twisted ambidextrous, and $\omega_{\Y,\partial\Y}\simeq f_! (\omega_{\X,\partial\X})$;
    \item if, additionally, $\omega_{\X,\partial\X}\simeq f^*\xi$ for some $\xi \in \func_{\baseTopos}(\Y,\sC)$, then $\omega_{\Y,\partial\Y}\simeq \xi$;
    \item if $f\colon\X\rightarrow \Y$ were a localisation, and $(\X,\partial\X)$ is groupoidally $\sC$-twisted ambidextrous (resp. $\sC$-Poincar\'e), then so is $(\Y,\partial\Y)$. 
\end{enumerate}
\end{prop}
\begin{proof}
For point (1), the assumption implies that the units for the adjunctions $f^* \dashv f_*$ and $\partial f^* \dashv \partial f_*$ are equivalences.
Now the resulting equivalences $\Y_* \simeq \X_* f^*$ and $\partial \Y_* \simeq \partial \X_* f^*$ show that $(\Y,\partial \Y)$ is $\sC$-twisted ambidextrous provided $(\X,\partial \X)$ is. 
In this case, denoting by $i\colon\partial\X\to \X$ and $j\colon\partial\Y\to \Y$ the inclusions, we have a chain of equivalences of $\sC$-linear functors
\begin{equation}\label{eq:limit_localisation_lemma}
(\Y,\partial\Y)_*\simeq \fib(\Y_*\to Y_*j_*j^*) \xrightarrow{\simeq} \fib(\Y_*f_*f^*\to \Y_*j_*\partial f_*\partial f^* j^*) \simeq (\X,\partial\X)_* f^*.
\end{equation}
The right side in \cref{eq:limit_localisation_lemma} is classified by the object $\omega_{\Y,\partial\Y} \simeq f_! (\omega_{\X,\partial\X})$ by \cref{lem:naturality_classification_C_linear_functors}. 

Point (2) now is a simple consequence of point (1), since $\omega_{\Y,\partial\Y}\simeq f_!f^*\xi\xrightarrow[\simeq]{\epsilon}\xi$. 
Finally, point (3) is immediate since $(\X,\partial\X)$ being groupoidally $\sC$-twisted ambidextrous and $f$ being a localisation implies that we have a factorisation $\omega_{\X,\partial\X}\colon \X\xrightarrow{f} \Y\rightarrow |\X|\rightarrow \sC$, and so we are in the situation of point (2).
\end{proof}

\begin{example}[Absolute realisation principle]
\label{ex:absolute_realisation_principle}
Let $\X$ be a groupoidally $\sC$-ambidextrous $\baseTopos$-category with classifying system $\omega_{\X}\in\func_{\baseTopos}(\X,\sC)$. Then $|\X|\in\spc_\baseTopos$ is also $\sC$-ambidextrous  with classifying system the essentially unique $\omega_{|\X|}\in\func(|\X|,\sC)$ factoring $\omega_\X$ through $|\X|$. In particular, if $\X$ is a $\sC$-Poincar\'e $\baseTopos$-category, then $|\X|$ is a $\sC$-Poincar\'e $\baseTopos$-space.
\end{example}

\cref{ex:absolute_realisation_principle} generalises as follows for $\baseTopos$-Poincar\'e pairs of $\baseTopos$-spaces.
\begin{prop}[Relative realisation principle]
\label{prop:barXdXpdpair}
Let $(\X,\partial\X)$ be a $\sC$-Poincar\'e $\baseTopos$-category pair arising from a cocartesian fibration $\X \to [1]$.
Then the functor $[1]\to\baseTopos$ corresponding to the arrow $|i|\colon|\partial\X|\to|\X| \simeq \lvert \interior{\X} \rvert$ is a $\sC$-Poincar\'e pair of $\baseTopos$-spaces
in the sense of \cref{defn:pdpairs}.

\end{prop}
\begin{proof}
Let $\W\coloneqq\partial\X\amalg\interior{\X}\simeq[1]^\simeq\times_{[1]}\X$, where the pullback is taken over a $\baseTopos$-functor $f\colon\X\to[1]$ classifying $\partial\X$ as in
\cref{defn:categorypair}, and let $\X'\coloneqq\scL(\X,\W)$. By \cref{lem:localisation_of_fibration} we have that the $\baseTopos$-functor $\X'\to[1]$ induced on the localisation is a left fibration corresponding to the functor $(|g|\colon|\partial\X|\to|\interior{\X}|)\colon[1]\to\baseTopos$, where $g \colon \partial \X \to \interior{\X}$ is the arrow classified by $\X$. By \cref{prop:localisation_principle} we have that $|g|$ is a $\sC$-Poincar\'e pair of $\baseTopos$-spaces; and since $\interior{\X}\subset\X$ is final, we may identify $|\interior{\X}|$ with $|\X|$ and $|g|$ with $|i|$. 
\end{proof}

\begin{rmk}
The hypothesis that $\X\to[1]$ is a cocartesian fibration in \cref{prop:barXdXpdpair} allows us to write a concise proof, but it is probably not optimal; for instance, we expect the statement of \cref{prop:barXdXpdpair} to hold under the weaker assumption that the inclusion $\interior{\X}\subset\X$ be final. Since all applications we have in mind in our future work satisfy the stronger cocartesian hypothesis, we have refrained from trying to relax it.
\end{rmk}

\subsection{The categorical Poincar\'e--Lefschetz sequence}
\label{subsec:poincare_lefschetz_sequence}
We now record a version of the Poincar\'e--Lefschetz sequence. We first give the dual homological version of \cref{defn:relative_cohomology}. To this end, first observe that, using the notation of \cref{constr:stable_recollement_for_pairs}, we may also express $(\X,\partial\X)_!$ as $\X_!j_*j^*$.

\begin{prop}[Categorical Poincar\'e--Lefschetz sequence]\label{prop:categorical_poincare_lefschetz_duality}
Let $(\X,\partial\X)$ be a groupoidally $\sC$-twisted ambidextrous (e.g. $\sC$-Poincar\'e) tame $\baseTopos$-category pair. Then there is a commuting diagram of fibre sequences in $\func^L_{\baseTopos,\sC}(\presheafTopos(\X;\sC),\sC)$ as follows:
\begin{equation}
\label{eq:categorical_poincare_lefschetz_duality}
\begin{tikzcd}
\Omega\partial \X_* i^* (-) \ar[d, "\simeq"] \ar[r] &(\X,\partial \X)_*(-) \ar[r] \ar[d, "\simeq"] & \X_*(-) \ar[d, "\simeq"] \\
 \partial \X_! i^*(D_{\X,\partial \X} \otimes -)\ar[r] &\X_!(D_{\X,\partial \X} \otimes -) \ar[r] & (\X,\partial \X)_!(D_{\X,\partial \X} \otimes -).
\end{tikzcd}
\end{equation}
\end{prop}
\begin{proof}
By \cref{lem:naturality_classification_C_linear_functors}, the top left horizontal functor $\Omega\partial\X_*i^*\to(\X,\partial\X)_*$, coming from the definition of $(\X,\partial\X)_*$ as a fibre, corresponds to a map $\Omega i_!\omega_{\partial\X}\to\omega_{\X,\partial\X}$; this map enters the definition of $\delta$, see
\cref{defn:connecting_map_of_classifying_systems} (up to a mild application of \cref{lem:naturality_classification_C_linear_functors}), so that
\cref{prop:boundary_principle} really identifies the above map $\Omega i_!\omega_{\partial\X}\to\omega_{\X,\partial\X}$ with the adjunction counit $i_!i^*\omega_{\X,\partial\X}\to\omega_{\X,\partial\X}$. By \cref{lem:naturality_classification_C_linear_functors} and \cref{prop:Dequivalence} we obtain the  left square in \cref{eq:categorical_poincare_lefschetz_duality}, from which the rest of the diagram is obtained by taking horizontal cofibres.
\end{proof}

Based on this more ``standard'' formulation of the Poincar\'e--Lefschetz sequence, we may also obtain the following ``Mayer--Vietoris'' incarnation thereof, usually attributed to Browder \cite[Proof of Thm. 2.1.]{Wall}.

\begin{cor}[Categorical Poincar\'e--Lefschetz sequence, Mayer--Vietoris version]
Let $\X$ be a groupoidally $\sC$-twisted ambidextrous $\baseTopos$-category written as a union $\X_1\cup\X_2$ of left closed $\baseTopos$-subcategories, let $\W\coloneqq\X_1\cap\X_2$ and assume that for $k=1,2$ the inclusion $\interior{\X}_k\coloneqq\X_k\setminus\W\to\X_k$ is final. Denote by $i_k \colon \X_k \hookrightarrow \X$, $j_k \colon \W \hookrightarrow \X_k$ and $i \colon \W \hookrightarrow X$ the respective inclusions.
Then we have two equivalences of cartesian squares in $\func^L_{\baseTopos,\sC}(\presheafTopos(\X;\sC),\sC)$ as follows:
\[
\left(\begin{tikzcd}
\Omega\W_*i^*\ar[r]\ar[d]&(\X_1,\W)_*i_1^*\ar[d]\\
(\X_2,\W)_*i_2^*\ar[r]&\X_*
\end{tikzcd}\right)\simeq
\left(\begin{tikzcd}
\W_!i^*(D_\X\otimes-)\ar[r]\ar[d]&(\X_1)_!i_1^*(D_\X\otimes-)\ar[d]\\
(\X_2)_!i_2^*(D_\X\otimes-)\ar[r]&\X_!(D_\X\otimes-)
\end{tikzcd}\right);
\]
\[
\left(\begin{tikzcd}
\X_*\ar[r]\ar[d]&(\X_1)_*i_1^*\ar[d]\\
(\X_2)_*i_2^*\ar[r]&\W_*i^*
\end{tikzcd}\right)\simeq
\left(\begin{tikzcd}
\X_!(D_\X\otimes-)\ar[r]\ar[d]&(\X_1,\W)_!i_1^*(D_\X\otimes-)\ar[d]\\
(\X_2,\W)_!i_2^*(D_\X\otimes-)\ar[r]&\Sigma\W_!i^*(D_\X\otimes-)
\end{tikzcd}\right).
\]
\end{cor}
\begin{proof}
We focus on the first equivalence, the second being analogous. 
By \cref{prop:boundary_principle}, \cref{prop:complementation_principle} and \ref{prop:dualish}, the restriction of $D_\X$ on $\X_k$ and $\W$ is $D_{\X_k,\W}$ and $\Omega D_{\W}$, respectively. 
By \cref{prop:categorical_poincare_lefschetz_duality} we then have an equivalence between the spans in $\func_{\baseTopos,\sC}^L(\presheafTopos(\X;\sC),\sC)$
\[
\left(\begin{tikzcd}[row sep=5pt, column sep=-8pt]
&\Omega\W_*i^*\ar[dl]\ar[dr]\\
(\X_1,\W)_*i_1^*& &
(\X_2,\W)_*i_2^*
\end{tikzcd}\right)
\simeq
\left(\begin{tikzcd}[row sep=5pt, column sep=-8pt]
&\W_!i^*(D_\X\otimes-)\ar[dl]\ar[dr]\\
(\X_1)_!i_1^*(D_\X\otimes-)& &
(\X_2)_!i_2^*(D_\X\otimes-)
\end{tikzcd}\right).
\]
The pushout of the right span is immediately identified with $\X_!(D_\X\otimes-)$, whereas the pushout of the left cospan can be computed as the suspension of the iterated pullback coming from the following diagram, and can thus be identified with $\X_*$:
\[
\begin{tikzcd}[row sep=7pt]
\Omega\X^*\ar[r]\ar[d]\ar[dr,phantom,"\lrcorner"very near start] &\Omega(\X_1)_*i_1^*\ar[r]\ar[d]\ar[dr,phantom,"\lrcorner"very near start]&0\ar[d]\\
\Omega(\X_1)_*i_1^*\ar[r]\ar[d]\ar[dr,phantom,"\lrcorner"very near start]&\Omega\W_*i^*\ar[r]\ar[d]&(\X_1,\W)_*i_1^*\\
0\ar[r]&(\X_2,\W)_*i_2^*
\end{tikzcd}
\]
We thus obtain an identification of pushouts $\X_*\simeq\X_!(D_\X\otimes-)$. We now observe that in the argument so far we could have \textit{defined} $D_\X$ as the functor $\X\to\sC$ restricting to $D_{\X_k,\W}$ on $\X_k$ for $k=1,2$, and identifying both restrictions $j_k^*D_{\X_k,\W}$ with $\Omega D_\W$ by \cref{prop:boundary_principle}; the obtained equivalence $\X_*\simeq\X_!(D_\X\otimes-)$ therefore reproves one arrow of \cref{prop:complementation_principle}, exhibiting $D_\X$ also as the dualising system of $\X$.
\end{proof}

\subsection{Cap product and Spivak data}
\label{subsec:Spivak}
Next, we turn to the problem of expressing the equivalence between (generalised) relative cohomology and (suitably twisted) homology as induced by cap product with a fundamental class.

\begin{defn}
Let $\Y\in\cat_\baseTopos$ be $\sC$-twisted ambidextrous.
The cap product $\cap\colon\Y_!(-)\otimes\Y_*(-)\to\Y_!(-\otimes-)$ is defined as the $\sC$-linear natural transformation of $\sC$-linear $\baseTopos$-functors $\presheafTopos(\Y;\sC)\otimes_\sC\presheafTopos(\Y;\sC)\to\sC$ given by the following composite:
\[
\Y_!(-)\otimes\Y_*(-)\overset{\simeq}{\leftarrow}\Y_!(-\otimes\Y^*\Y_*(-))\xrightarrow{\epsilon_{\Y_*}}\Y_!(-\otimes-).
\]
We used that the projection formula holds for $\Y_!$, i.e. the following composite is an equivalence:
\[
\Y_!(-\otimes\Y^*(-))\xrightarrow{\eta_{\Y_!}}\Y_!(\Y^*\Y_!(-)\otimes\Y^*(-))\simeq\Y_!\Y^*(\Y_!(-)\otimes-)\xrightarrow{\epsilon_{\Y_!}}\Y_!(-)\otimes-.
\]
\end{defn}

\begin{nota}
Given two 1-morphisms $F,G\colon\D\to\E$ in $\module_{\sC}(\presentable^L_{\baseTopos})$, that is, two objects in $\func^L_{\baseTopos,\sC}(\D,\E)$, we denote by $\nattrans_\sC(F,G)\in\Gamma\sC$ the hom object of $\sC$-linear natural transformations from $F$ to $G$.
\end{nota}

\begin{lem}
\label{lem:absolute_classification_natural_transformations}
Let $q\colon\Y\to\Z$ be a $\baseTopos$-cartesian fibration of $\sC$-twisted ambidextrous $\baseTopos$-categories, and let $\zeta\in\presheaf(\Z;\sC)$. Then the composite $\Y_!q^*\zeta\otimes\Y_*q^*(-)\xrightarrow{\cap}\Y_!q^*(\zeta\otimes-)\xrightarrow{\epsilon_{q_!}}\Z_!(\zeta\otimes-)$ is adjoint to an equivalence $\Y_!q^*\zeta\xrightarrow{\simeq}\nattrans_\sC(\Y_*q^*,\Z_!(\zeta\otimes-))$ in $\Gamma\sC$.
\end{lem}
\begin{proof}
The $\sC$-linear adjunction $q_!\Y^*\colon\sC\rightleftharpoons\presheafTopos(\Z;\sC)\cocolon\Y_*q^*$ gives rise to a $\Gamma\sC$-linear adjunction $(-\circ\Y_*q^*)\colon\func^L_{\baseTopos,\sC}(\sC,\sC)\rightleftharpoons\func^L_{\baseTopos,\sC}(\presheafTopos(\Z;\sC),\sC)\cocolon(-\circ q_!\Y^*)$; it suffices therefore to check that the following
composite is an equivalence in $\Gamma\sC$, as the composite of all maps but the first is known to be an equivalence:
\begin{equation}
\label{eq:absolute_classification_natural_transformations}
\begin{tikzcd}[row sep=10pt]
\Y_!q^*\zeta\ar[r]&\nattrans_\sC(\Y_*q^*,\Z_!(\zeta\otimes -))\ar[r,"-\circ q_!\Y^*"]&\nattrans_\sC(\Y_*q^*q_!\Y^*,\Z_!(\zeta\otimes q_!\Y^*(-)))
\ar[d,"\eta_{\Y_*q^*}"]\\
& &\nattrans_\sC(\id_{\sC},\Z_!(\zeta\otimes q_!\Y^*(-)))
\end{tikzcd}
\end{equation}
The adjoint of the previous composite is the following composite in
$\func^L_{\baseTopos,\sC}(\sC,\sC)$:
\begin{equation}
\label{eq:absolute_classification_natural_transformations_adjoined}
\Y_!q^*\zeta\otimes-\xrightarrow{\eta_{\Y_*q^*}}\Y_!q^*\zeta\otimes\Y_*q^*q_!\Y^*(-)\xrightarrow{\cap}\Y_!q^*(\zeta\otimes q_!\Y^*(-))\xrightarrow{\epsilon_{q_!}}\Z_!(\zeta\otimes q_!\Y^*(-))
\end{equation}
and using the definition of cap product and the triangle identities of the adjunctions $\Y^*\dashv\Y_*$ and $q_!\dashv q^*$, we may identify \cref{eq:absolute_classification_natural_transformations_adjoined} with the equivalence
\[
\Y_!q^*\zeta\otimes-\simeq\Y_!(q^*\zeta\otimes\Y^*(-))\simeq\Z_!q_!(q^*\zeta\otimes\Y^*(-))\simeq\Z_!(\zeta\otimes q_!\Y^*(-))
\]
given by the projection formulae for $\Y$ and $q$. It follows that \cref{eq:absolute_classification_natural_transformations} is also homotopic to the following composite, which is an equivalence:
\[
\Y_!q^*\zeta\simeq\nattrans_\sC(\id_\sC,\Y_!q^*\zeta\otimes-)\simeq\nattrans_\sC(\id_\sC,\Z_!(\zeta\otimes\Y^*(-)). \qedhere
\]
\end{proof}
\begin{constr}[Relative cap product]
\label{constr:relative_cap_product}
Let $u\colon\Y\to\Z$ be a $\baseTopos$-functor betwee $\sC$-twisted ambidextrous $\baseTopos$-categories and recall \cref{defn:generalised_relative_cohomology}. We define a relative cap product $(\Z,\Y,u)_!(-)\otimes\Z_*(-)\to(\Z,\Y,u)_!(-\otimes-)$ by considering the region in the following commutative diagram, coming from naturality of cap product, and by taking horizontal cofibres:
\begin{equation}
\label{eq:first_relative_cap_product}
\begin{tikzcd}[row sep=10pt]
\Y_!u^*(-)\otimes\Z_*(-)\ar[r,"\epsilon_{u_!}"]\ar[d,"\eta_{u_*}"]&\Z_!(-)\otimes\Z_*(-)\ar[dd,"\cap"]\\
\Y_!u^*(-)\otimes\Y_*u^*(-)\ar[d,"\cap"]\\
\Y_!u^*(-\otimes-)\ar[r,"\epsilon_{u_!}"]&\Z_!(-\otimes-).\\
\end{tikzcd}
\end{equation}
If we take horizontal fibres in the above left commutative square, we obtain the left-bottom region in the following commutative diagram, whose horizontal cofibres give rise to a relative cap product $(\Z,\Y,u)_!(-)\otimes(\Z,\Y,u)_*(-)\to\Z_!(-\otimes-)$:
\[
\begin{tikzcd}[row sep=10pt]
\Omega(\Z,\Y,u)_!(-)\otimes\Z_*(-)\ar[dd,"\cap"]\ar[dr]\ar[rr,"\eta_{u_*}"]& &\Omega(\Z,\Y,u)_!(-)\otimes\Y_*u^*(-)\ar[d]\\
&\Y_!u^*(-)\otimes\Z_*(-)\ar[r,"\eta_{u_*}"]&\Y_!u^*(-)\otimes\Y_*u^*(-)\ar[d,"\cap"]\\
\Omega(\Z,\Y,u)_!(-\otimes-)\ar[rr]& &\Y_!u^*(-\otimes-)
\end{tikzcd}
\]
\end{constr}

\begin{cor}
\label{cor:relative_classification_natural_transformations}
Let $q\colon\Y\to\Z$ be a $\baseTopos$-cartesian fibration of $\sC$-twisted ambidextrous $\baseTopos$-categories, and let $\zeta\in\presheaf_\baseTopos(\Z;\sC)$. Then the relative cap product $\cap\colon(\Z,\Y,q)_!\zeta\otimes (\Z,\Y,q)_*\to\Z_!(\zeta\otimes-)$ adjoins to an equivalence $(\Z,\Y,q)_!\zeta\xrightarrow{\simeq}\nattrans_\sC((\Z,\Y,q)_*,\Z_!(\zeta\otimes-))$ in $\Gamma\sC$.
\end{cor}
\begin{proof}
We adjoin \cref{eq:first_relative_cap_product} to the following commutative diagram:
\begin{equation}
\label{eq:relative_cap_product}
\begin{tikzcd}[row sep=10pt]
\Y_!u^*\zeta\ar[r,"\epsilon_{u_!}"]\ar[d]& \Z_!\zeta\ar[dd]\\
\nattrans_\sC(\Y_*u^*,\Y_!u^*(\zeta\otimes-))\ar[d,"\epsilon_{u_!}"]\\
\nattrans_\sC(\Y_*u^*,\Z_!(\zeta\otimes-)\ar[r,"\eta_{u_*}"] & \nattrans_\sC(\Z_*,\Z_!(\zeta\otimes-));
\end{tikzcd}
\end{equation}
By \cref{lem:absolute_classification_natural_transformations} applied to $q$ and to $\id_\Z$, the vertical composites in \cref{eq:relative_cap_product} are equivalences, hence the map induced on cofibres is an equivalence in $\Gamma\sC$ as well.
\end{proof}

Classically, Poincar\'e duality or Poincar\'e--Lefschetz duality is formulated by specifying a local system and a fundamental class in the respective homology. This line of thought led in \cite[Definition 3.1.1]{PD1} to the notion of a \textit{Spivak datum}, which we generalise here.
\begin{defn}
Let $(\X,\partial\X)$ be a $\sC$-twisted ambidextrous $\baseTopos$-category pair, let $\zeta\in\presheaf_\baseTopos(\X;\sC)$, and let $c\colon\unit_{\Gamma\sC}\to(\X,\partial\X)_!\zeta$ be a $\Gamma\sC$-linear map: we then say that the pair $(\zeta,c)$ is a \textit{candidate Spivak datum}. We say that $(\zeta,c)$ is a \textit{Spivak datum} if $\zeta$ is groupoidal and the natural transformation $(\X,\partial\X)_*\to\X_!(\zeta\otimes-)$ induced by $c$ under cap product is an equivalence.
\end{defn}
We observe that for a $\baseTopos$-category pair $(\X,\partial\X)$, the moduli space of Spivak data is empty or contractible, depending on whether or not $(\X,\partial\X)$ is groupoidally $\sC$-twisted ambidextrous.

\subsection{Rigidity of classifying systems and duality}
\label{subsec:dualisable_results}
We fix a $\sC$-twisted ambidextrous $\baseTopos$-category pair $(\X,\partial\X)$ as in \cref{nota:X_dX_interiorX} throughout the subsection. Our goal is to investigate various questions about dualisable objects in $\func_\baseTopos(\X,\sC)$, which is considered as a $\sC$-linear symmetric monoidal $\baseTopos$-category. First, we recall the definition of dualisability for the reader's convenience.

\begin{defn}
\label{defn:dualisable}
An object $x$ in a symmetric monoidal $\infty$-category $\D$ is \textit{dualisable} if there exists another object $x^\vee$ and morphisms $e\colon x\otimes x^\vee\to\unit$ and $u\colon\unit\to x^\vee\otimes x$ such that there are equivalences $(e\otimes\id_x)\circ(\id_x\otimes u)\simeq\id_x$ and $(\id_{x^\vee}\otimes e)\circ(u\otimes \id_{x^\vee})\simeq\id_{x^\vee}$.

For a symmetric monoidal $\infty$-category $\D$ we let $\D^\dbl\subseteq\D$ denote the full $\infty$-subcategory spanned by dualisable objects. For a symmetric monoidal $\baseTopos$-category $\E$ we let similarly $\E^\dbl\subseteq\E$ denote the $\baseTopos$-subcategory spanned levelwise by dualisable objects. There is a duality $\baseTopos$-functor $(-)^\vee\colon(\E^\dbl)\op \to\E^\dbl$,
induced by taking duals.
\end{defn}

\begin{prop}\label{prop:dualisables_are_groupoidal}
Let $J\in\cat_{\baseTopos}$ and $\sC \in \calg(\cat_{\baseTopos})$. Then the fully faithful restriction along $J \rightarrow |J|$ induces an equivalence of $\baseTopos$-categories $(\sC^{\vert J \vert} )^\dbl \xrightarrow{\simeq} (\sC^{J} )^\dbl$.
\end{prop}
\begin{proof}
We first consider the special case $\baseTopos=\spc$ and $J=[1]$: we claim that an object $(x \xrightarrow{f} y)$ in $\sC^{[1]}$ is dualisable if and only if $f$ is an equivalence and $x$ and $y$ are dualisable. Assume first that $f$ is an equivalence and $x$ and $y$ are dualisable. Then $\id_x\in\sC^{[1]}$ is dualisable, as it is the image of $x$ along the symmetric monoidal functor $\sC \rightarrow \sC^{[1]}$ 
given by restriction along $[1] \rightarrow *$; moreover we have an equivalence $f\simeq \id_x\in\sC^{[1]}$, hence also $f$ is dualisable.
    
Conversely, suppose that $f$ is dualisable; denote by $f' \colon x' \to y'$ its dual, and
let $e_f \colon f \otimes f' \to \id_\unit$ and $u_f \colon \id_\unit \to f' \otimes f$ denote the evaluation and coevaluation, respectively.
Restricting to $0$ and $1$, this shows that $x$ and $y$ are dualisable; more precisely, we obtain equivalences $(x')^{\vee}\simeq x, (y')^{\vee}\simeq y$, and so we may identify $(f')^{\vee}$ as a morphism $y\rightarrow x$.
We claim that $f$ is invertible with inverse $(f')^\vee$.
For this, consider the following commutative diagrams:
\[
\begin{tikzcd}[row sep=10pt]
x \otimes x' \ar[r, "e_x"] \ar[d, "f \otimes f'"'] & \unit \ar[d, "\id_{\unit}"]\\
y \otimes y' \ar[r, "e_y"] & \unit;
\end{tikzcd}
\hspace{1cm}
\begin{tikzcd}[column sep=30pt,row sep=10pt]
x \ar[r, "f"] \ar[d, "x \otimes u_x"']
& y \ar[r, "{(f')^\vee}"] \ar[d, "y \otimes u_x"]
& x\\
x \otimes x' \otimes x \ar[r, "{f \otimes x' \otimes x}"] & y \otimes x' \otimes x \ar[r, "{y \otimes f' \otimes x}"] & y \otimes y' \otimes x. \ar[u, "{e_y \otimes x}"']
\end{tikzcd}
\]
The commutativity of the square on left identifies the bottom-right composite $x \otimes x' \otimes x \to x$ in the right rectangle with $e_x \otimes x$ and the triangle identity shows that the top composite $x \to x$ in the right rectangle is equivalent to the identity.
An analogous argument yields the equivalence $f \circ(f')^\vee\simeq\id_y$. This completes the proof in the special case $\baseTopos=\spc$ and $J=[1]$.

In the general case, let $\tau\in \baseTopos$ and let $\zeta \in (\sC^J)(\tau)$ be dualisable. We want to show that $\zeta$ is in the image of $(\sC^{\vert J \vert})(\tau) \rightarrow (\sC^J)(\tau)$. Viewing $\zeta$ as a $\baseTopos$-functor $\tau\times J \rightarrow \sC$, we want to show that for each $\tau' \in \baseTopos$ and each morphism $f \colon \tau'\times [1] \rightarrow \tau \times J$  in $\cat_\baseTopos$, the restriction of $\zeta$ along $f$ is in the essential image of the functor $\sC(\tau')^\dbl \rightarrow (\sC^{[1]})(\tau')\simeq\sC(\tau')^{[1]}$ given by restriction along $[1]\to*$. Since $f^*\colon (\sC^J)(\tau) \rightarrow (\sC^{[1]})(\tau') \simeq (\sC(\tau'))^{[1]}$ is symmetric monoidal, we have  that $f^*(\zeta)$ is dualisable; since the target $(\sC(\tau'))^{[1]}$ carries the pointwise symmetric monoidal structure, the preceding special case shows that $f^*(\zeta)$ must be restricted from $\sC(\tau')^\dbl$ as desired.
\end{proof}

\begin{cor}[Ambidexterity and duality]
\label{cor:ambidexterity_and_duality}
If $\X\in\cat_{\baseTopos}$ is  $\sC$-twisted ambidextrous, the $\baseTopos$-functor $\X_! \colon \presheafTopos(\X;\sC) \rightarrow \sC$ restricts to a $\baseTopos$-functor $\X_! \colon \presheafTopos(\X;\sC)^\dbl \rightarrow \sC^\dbl$. In fact, for $\xi\in\presheaf(\X;\sC)^\dbl$, there is an equivalence $\X_!(\xi)^\vee \simeq \X_*(\xi^\vee)\in\Gamma\sC$, natural in $\xi$.
\end{cor}

\begin{proof}
Recall that for $\D \in \calg(\presentable^L)$, an object $x \in \D$ is dualisable if and only if the right adjoint to the $\D$-linear functor $x \otimes - \colon \D \rightarrow \D$ is itself a $\D$-linear left adjoint. On the one hand, if $x$ is dualisable, then the mentioned right adjoint is given by $x^\vee\otimes-$; on the other hand, if the right adjoint $\hom_\D(x,-)$ of $x\otimes-$ is $\D$-linear, then by Morita theory we have $\hom_\D(x,-) \simeq y \otimes -$ for some $y \in \D$, and in fact $y$ is dual to $x$, the duality data being provided by the unit and counit of the adjunction.

We next prove the second statement.
Let $\xi \in \presheaf_\baseTopos(\X;\sC)$; the $\sC$-linearity of $\X_!$ then gives an equivalence
$- \otimes \X_!(\xi) \simeq \X_!(\X^*(-) \otimes \xi)$. If $\xi$ is dualisable, the previous equivalence exhibits $-\otimes\X_!(\xi)$ as a composite of left adjoints in $\module_\sC(\presentable^L_\baseTopos)$ having $\sC$-linear right adjoints; it follows that $-\otimes\X_!(\xi)$ is itself a left adjoint with a $\sC$-linear right adjoint, establishing dualisability of $\X_!(\xi)$. More precisely, the right adjoint of $\X_!(\X^*(-) \otimes \xi)$ is $\X_*(\X^*(-) \otimes \xi^\vee)\simeq-\otimes\X_*(\xi^\vee)$, where we use $\sC$-linearity of $\X_*$ in the last formula.

The first statement follows by \'etale basechange: for $\tau\in\baseTopos$ the $\baseTopos$-functor $\X_!$ restricts at level $\tau$ to the functor $\X_!\colon\presheaf_{\baseTopos_{/\tau}}(\pi_\tau^*\X;\pi_\tau^*\sC)\to\pi_\tau^*\sC$, which by the previous point restricts to a functor $\X_!\colon\presheaf_{\baseTopos_{/\tau}}(\pi_\tau^*\X;\pi_\tau^*\sC)^\dbl\to(\pi_\tau^*\sC)^\dbl$
\end{proof}

\begin{prop}
\label{prop:omega_dualisable_iff_invertible}
Assume that $(\X, \partial \X)$ is a tame $\sC$-twisted ambidextrous $\baseTopos$-category pair such that $\omega_{\X,\partial\X}\in\func_\baseTopos(\X,\sC)$ is dualisable. Then $(\X,\partial\X)$ is $\sC$-Poincar\'e.
\end{prop}
\begin{proof}
The classifying system $\omega_{\X, \partial \X}$ is groupoidal by \cref{prop:dualisables_are_groupoidal}.
We abbreviate $\bar\omega\coloneqq D(\omega_{\X,\partial\X})$ and let $\zeta\in\func_\baseTopos(|\X\op|,\sC)\subset\func_\baseTopos(\X\op,\sC)$ be an arbitrary groupoidal object. Then we obtain a sequence of equivalences in $\Gamma\sC$, also leveraging \cref{cor:relative_classification_natural_transformations}:
\[
\begin{split}
\hom_{\Gamma\sC^\X}(\bar\omega\otimes\bar\omega^\vee,\zeta)&\simeq\hom_{\Gamma\sC^\X}(\bar\omega,\bar\omega\otimes\zeta)\simeq\nattrans_\sC((\X,\partial\X)_*,\X_!(\bar\omega\otimes\zeta\otimes-))\\
&\simeq(\X,\partial\X)_!(\bar\omega\otimes\zeta)
\simeq\cofib(\partial\X_!i^*(\bar\omega\otimes\zeta)\to\X_!(\bar\omega\otimes\zeta))\\
&\simeq\cofib(\Omega\partial\X_*i^*\zeta\to(\X,\partial\X)_*\zeta)\simeq\X_*\zeta\simeq\hom_{\Gamma\sC^\X}(\X^*\unit,\zeta).
\end{split}
\]
The previous chain of equivalences is natural in $\zeta\in\func_\baseTopos(|\X\op|,\sC)$, in particular in $\zeta\in\func_\baseTopos(\X\op,\sC)^\dbl$, so the Yoneda lemma provides an identification $\bar\omega\otimes\bar\omega^\vee\simeq\X^*\unit$, thus proving that $\bar\omega$ and hence $\omega$ are invertible.
\end{proof}

The following recovers \cite[Theorem B]{KleinQinSu} in the setting $\baseTopos=\spc$ and $\sC=\module_\eilenbergMacLaneCoeff$, using that $\hom_\eilenbergMacLaneCoeff(-,\eilenbergMacLaneCoeff)\colon\module_\eilenbergMacLaneCoeff\op\to\module_\eilenbergMacLaneCoeff$ is conservative, a fact that we learned from Markus Land. We thus give a ``purely algebraic'' proof of the mentioned theorem, answering \cite[Remark 4.2]{KleinQinSu}.

\begin{prop}
\label{prop:KQS_theorem_B}
Let $(\X,\partial \X)$ be a compact $\infty$-category pair considered as a constant $\baseTopos$-category pair, and let $(\zeta,c)$ be a Spivak datum for $(\X,\partial\X)$ with $\zeta$ invertible. Then the following two conditions are equivalent:
\begin{itemize}
\item[(1)] cap product with $c$ induces an equivalence $(\X,\partial\X)_*\xrightarrow{\simeq}\X_!(\zeta\otimes-)$ of $\baseTopos$-functors $\presheafTopos(\X;\sC)^\dbl\to\sC$;
\item[(2)] cap product with $c$
induces an equivalence $\X_*\xrightarrow{\simeq}(\X,\partial\X)_!(\zeta\otimes-)$ of $\baseTopos$-functors $\presheafTopos(\X;\sC)^\dbl\to\sC$.
\end{itemize}
Let moreover $\unit\in\Gamma\sC$ denote the monoidal unit, and assume that the $\baseTopos$-functor $\hom_\sC(-;\unit)\colon\sC\op\to\sC$ is conservative. Then the following two conditions are equivalent:
\begin{itemize}
\item[(1)'] cap product with $c$ induces an equivalence $(\X,\partial\X)_*\xrightarrow{\simeq}\X_!(\zeta\otimes-)$ of $\baseTopos$-functors $\presheafTopos(\X;\sC^\simeq)\to\sC$;
\item[(2)'] cap product with $c$
induces an equivalence $\X_*\xrightarrow{\simeq}(\X,\partial\X)_!(\zeta\otimes-)$ of $\baseTopos$-functors $\presheafTopos(\X;\sC^\simeq)\to\sC$.
\end{itemize}
\end{prop}
\begin{proof}
For simplicity, we abbreviate by $(-)^\vee\colon\sC\op\to\sC$ the $\baseTopos$-functor $\hom_\sC(-;\unit)$; note that the restriction of $\hom_\sC(-;\unit)$ on $\sC^\dbl$ agrees with the functor $(-)^\vee$ from \cref{defn:dualisable}, and note also that the latter restricted functor is always conservative.
We similarly abbreviate by $(-)^\vee\colon\presheafTopos(\X;\sC)\op\to\presheafTopos(\X;\sC)$ the internal hom $\baseTopos$-functor $\hom_\X(-;\X^*\unit)$.
Note that on the subcategory $\presheafTopos(\X, \sC^\simeq)$ of groupoidal systems it is given by the pointwise internal hom as the composite
\begin{equation*}
\presheafTopos(\X;\sC^\simeq) \xrightarrow{\hom_{\sC}(-, \unit)} \funTopos(\X, \sC^\simeq) \simeq \presheafTopos(\X, \sC^\simeq).
\end{equation*}
We consider the following commutative diagram of $\baseTopos$-functors $\presheafTopos(\X;\sC)\to\sC$ and $\baseTopos$-natural transformations, in which we abbreviate by ``$\X_!(\unit)$'' the composite $\baseTopos$-functor $\presheafTopos(\X;\sC)^\simeq\to\ast\xrightarrow{\X_!(\unit)}\sC$, and similarly for $\unit$, and we input twice the same variable:
\[
\begin{tikzcd}[column sep=0pt]
&\X_!((-)^\vee)\otimes\X_*(-)\ar[dr,"-\cap-"]\\
(\X,\partial\X)_*(\zeta^{-1}\otimes(-)^\vee)\otimes\X_*(-)\ar[ur,"(c\cap-)\otimes\id"]\ar[dr,"\id\otimes(c\cap-)"']\ar[r,"-\cup-"]&(\X,\partial\X)_*(\zeta^{-1}\otimes(-)^\vee\otimes-)\ar[r,"c\cap-"] &\X_!((-)^\vee\otimes-)\ar[d,"\eval"]\\
&(\X,\partial\X)_*(\zeta^{-1}\otimes(-)^\vee)\otimes(\X,\partial\X)_!(-\otimes\zeta)\ar[ur,"c\cap-"']&\X_!(\unit)\to\unit.
\end{tikzcd}
\]
The previous diagram provides an equivalence between two morphisms in $\func_{\baseTopos,\sC}(\presheafTopos(\X;\sC),\sC)$ of the form $(\X,\partial\X)_*(\zeta^{-1}\otimes(-)^\vee)\otimes\X_*(-)\to\unit$. 
We may adjoin it to the following commutative diagram in $\func_{\baseTopos,\sC}(\presheafTopos(\X;\sC)^\simeq,\sC)$
\begin{equation}\label{diag:thm_b}
\begin{tikzcd}
(\X,\partial\X)_*(\zeta^{-1}\otimes(-)^\vee)\ar[r]\ar[d,"c\cap-"]&((\X,\partial\X)_!(-\otimes\zeta))^\vee\ar[d,"(c\cap-)^\vee"]\\
\X_!((-)^\vee)\ar[r]&(\X_*(-))^\vee.
\end{tikzcd}
\end{equation}
We now claim that the top horizontal arrow is always an equivalence while the bottom horizontal arrow is an equivalence when restricted to $\presheafTopos(\X;\sC^\simeq)$.

For the top map, we claim that the analogous $\baseTopos$-natural transformation $\X_* ((-)^\vee)\xrightarrow{\simeq}(\X_!(-))^\vee$ is an equivalence of $\baseTopos$-functors $\presheafTopos(\X;\sC)\rightarrow  \sC\op$ is an equivalence, from which it is easy to deduce the desired conclusion for relative (co)homology. To see this, note that the composite $\baseTopos$-functor $(\X_!(-))^\vee \colon \presheafTopos(\X;\sC) \xrightarrow{\X_!} \sC \xrightarrow{\hom_{\sC}(-, \unit)} \sC\op$ is left adjoint to the composite $\sC\op \xrightarrow{\hom_{\sC}(-, \unit)} \sC \xrightarrow{\X^*} \presheafTopos(\X;\sC)$.
This right adjoint identifies with the composite $\sC\op \xrightarrow{\X^*} \presheafTopos(\X, \sC)\op \xrightarrow{\hom_\X(-, \X^* \unit)} \presheafTopos(\X, \sC)$, since $\X^*\colon\sC\to \presheafTopos(\X, \sC)$ factors through $\presheafTopos(\X;\sC^{\simeq})$; the left adjoint to the latter composite is then evidently
$\X_* ((-)^\vee)$ (note the role of opposite $\baseTopos$-categories).

We now turn to the bottom arrow, which is natural in $\X$, and is clearly an equivalence for $\X = *$.
As both source and target are compatible with finite colimits in $\X$ by stability of $\sC$, and as compact spaces are obtained by finite colimits and retracts from $*$, we obtain that the bottom horizontal transformation is an equivalence whenever $\X$ is a compact space.
For a general $\baseTopos$-category $\X$, and in particular for a constant one, we have $\presheafTopos(\X; \sC^\simeq) \simeq \presheafTopos(\lvert \X \rvert; \sC)$; as the classifying space of a compact $\infty$-category is a compact space, this shows that the bottom horizontal map is an equivalence when restricted to groupoidal systems.

The left vertical $\baseTopos$-natural transformation in diagram \cref{diag:thm_b} agrees with (1)' precomposed by the invertible $\baseTopos$-functor $\zeta^{-1}\otimes-$, and restricted to the core $\baseTopos$-groupoid $\presheafTopos(\X;\sC)^\simeq$. The right vertical $\baseTopos$-natural transformation agrees instead with the restriction of (2)' to the core $\baseTopos$-groupoid $\presheafTopos(\X;\sC)^\simeq$, postcomposed by the $\baseTopos$-functor $(-)^\vee$. 

Assuming that $(-)^\vee$ is conservative, 
we thus obtain that (2)' is invertible given that (1)' is invertible.
An analogous argument identifies the dual of the transformation (1)' with (2)' precomposed with $(-)^\vee$, showing that (1)' is an equivalence if (2)' is one.

To show that (1) and (2) are equivalent, we can make similar arguments again using the commutative square \cref{diag:thm_b} by noting that $(-)^\vee$ is a self duality on both $\presheafTopos(\X, \sC)^\dbl$ and $\sC^\dbl$ as well as $\X_* \xi$ and $\X_! \xi$ being dualisable for $\xi \in \presheafTopos(\X, \sC)^\dbl$.
\end{proof}

\part{Applications}\label{part:applications}
\section{Examples from geometric topology}
\label{sec:unstraightenedpairs}
We apply the results from \cref{part:foundations}, most notably \cref{prop:boundary_principle,prop:complementation_principle,prop:localisation_principle} together with their corollaries, to study Poincar\'e $\baseTopos$-category pairs of the form $(\int_IX,\int_{\partial I}X)$ as in \cref{cons:integral_categories}. When $(I,\partial I)$ is a pair of a finite poset and a left closed subposet, we may think of $\int_IX$ as a stratified topological space, and of $\int_{\partial I}X$ as a closed stratified subspace which is a union of strata.
The main cases we will cover are those of Poincar\'e pairs 
and Poincar\'e ads, which we study in 
\cref{subsec:pdpairs} and \cref{subsec:pdads}: in particular, as a sanity check, we show that our theory recovers the classical definitions due to Wall, at least under finite domination assumptions. We then generalise in \cref{subsec:triangulated_manifolds} to the study of $\baseTopos$-category pairs obtained by unstraightening a diagram of $\baseTopos$-spaces parametrised by a \textit{combinatorial manifold}: the latter notion refers to a finite poset enjoying special properties that allow one to efficiently use the local-to-global principle from \cref{cor:local_to_global_principle}. 
We end this section with additional examples in \cref{subsec:more_examples}..
Throughout the section we fix $\sC\in\calg(\presentable^L_{\baseTopos,\stable})$.

We start with a lemma guaranteeing 
twisted ambidexterity for $\baseTopos$-categories of the form $\int_IX$ as above.
\begin{lem}\label{lem:unstraightening_twisted_ambidextrous}
Let $I$ be a compact $\infty$-category and let $X \colon I \to\baseTopos$ be a functor valued in $\sC$-twisted ambidextrous $\baseTopos$-spaces.
Then the $\baseTopos$-category $\int_I X$ is $\sC$-twisted ambidextrous. 
\end{lem}
\begin{proof}
Recall that $\cat^\omega \subseteq \cat$ is the smallest $\infty$-subcategory containing $\emptyset$ and $[1]$ and being closed under retracts and pushouts.
If $I \xrightarrow{i} J \xrightarrow{r} I$ exhibits the $\infty$-category $I$ as a retract of $J$, then $\int_I X$ is a retract of $\int_J r^*X$ and by \cref{rmk:ambidexterity_closed_under_retracts} $\int_I X$ is $\sC$-twisted ambidextrous if $\int_J r^* X$ is.
Similarly, if $I\simeq I_1 \amalg_{I_0} I_2$ is a pushout of $\infty$-categories, we have a pushout square of $\baseTopos$-categories as follows:
\[ 
\begin{tikzcd}[row sep=10pt]
\int_{I_0}X \ar[d]\ar[r]\ar[dr,phantom,"\ulcorner"very near end]& 
\int_{I_1} X\ar[d]\\
\int_{I_2}X\ar[r]
& \int_{I} X,
\end{tikzcd}
\]
and by \cref{rmk:ambidexterity_closed_under_finite_colimits}, $\int_{I} X$ is $\sC$-twisted ambidextrous if $\int_{I_0} X$, $\int_{I_1} X$ and $\int_{I_2} X$ are.
   
Consequently, we only need to check the statement for $I = \emptyset, [1]$.
The first case vacuously holds; in the second case we have $\int_{[1]} X \simeq X(0) \times [1] \amalg_{X(0) \times \{1\}} X(1)$.
Again using stability of $\sC$-twisted ambidextrous $\baseTopos$-categories under pushouts, we only need to argue that $X(0) \times [1]$ is $\sC$-twisted ambidextrous.
For this, we identify $(X(0) \times [1])_* \colon \presheafTopos(X(0) \times [1], \sC) \to \sC$ with the composite $\presheafTopos(X(0) \times [1], \sC) \xrightarrow{(i_1)^*} \presheafTopos(X(0), \sC) \to \xrightarrow{X(0)_*} \sC$, where $i_1 \colon X(0) \to X(0) \times [1]$ is induced by the inclusion $\{1\} \to [1]$;
both of these functors are colimit preserving $\sC$-linear, showing that $X(0) \times [1]$ is $\sC$-twisted ambidextrous.
\end{proof}

\subsection{Poincar\'e pairs}
\label{subsec:pdpairs}
In this subsection we consider the case $(I,\partial I)=([1],\{0\})$. We shall compare four notions of Poincar\'e duality pairs. The first is the notion stemming from our theory, and has already been given in \cref{defn:pdpairs}.
The second definition is the classical one due to Wall,
restricted to the setting $\baseTopos=\spc$ and $\sC=\module_\eilenbergMacLaneCoeff$. The third is a ``postclassical'' adaptation of the previous to the $\baseTopos$-parametrised and $\sC$-linear setting; both the classical and postclassical notions rely on the existence of fundamental classes inducing isomorphisms via cap product.
The fourth is an intermediate, ``neoclassical'' notion, in which 
the focus is still on $\sC$-valued presheaves on $\baseTopos$-spaces rather than on $\baseTopos$-categories, but no explicit reference to cap product is made any more.

\begin{defn}[Wall's classical Poincar\'e pairs]
The following definition is taken from \cite[p. 215]{Wall}.
\label{defn:classical_poincare_pairs}
Let $X\colon[1]\to\spc$ be a functor, and denote by $g\colon Y = X(0)\to Z = X(1)$ the corresponding morphism of spaces. We say that $X$ is a  \textit{classical Poincar\'e  pair of formal dimension $d\ge0$} if there is a local coefficient system $\calO$ over $Z$, valued in infinite cyclic groups, together with a relative fundamental class $[Z, Y]\in H_d(Z, Y;\calO)$ satisfying the following properties.
\begin{enumerate}[label=(\arabic*)]
\item For every local coefficient system of abelian groups $\Lambda_1$ over $Z$, cap product with $[Z,Y]$ induces an isomorphism of abelian groups
    \[
    [Z,Y]\cap-\colon H^*(Z,Y;\Lambda_1)\overset{\cong}{\to}H_{d-*}(Z;\calO\otimes\Lambda_1).
    \]
\item Let $[Y]=\partial[Z,Y]\in H_{d-1}(Y;g^*\calO)$ denote the image along the connecting homomorphism; then for any coefficient system of abelian groups $\Lambda_0$ over $Y$, cap product with $[Y]$ induces an isomorphism of abelian groups
    \[
    [Y]\cap-\colon H^*(Y;\Lambda_0)\overset{\cong}{\to}H_{d-1-*}(Y;g^*\calO\otimes\Lambda_0).
    \]
\end{enumerate}
If $Z$ is not connected, we say that $X$ i a classical Poincar\'e pair if each component of $Z$, relative the respective components of $Y$ that it contains, is a classical Poincar\'e pair of some dimension.
\end{defn}
\begin{defn}[Postclassical Poincar\'e pairs]
\label{defn:postclassical_poincare_pairs}
The following definition appears in the setting $\baseTopos=\spc$ and $\sC=\spectra$ as \cite[Def. 6]{LuriePoincare}.
Let $X\colon[1]\to\baseTopos$ be a functor, and denote by $g\colon Y\to Z$ the corresponding morphism of $\baseTopos$-spaces. We say that $X$ is a \textit{postclassical $\sC$-Poincar\'e pair} if there is a system $\zeta\in\presheaf_\baseTopos(Z;\picardSpaceTopos(\sC))$ and a fundamental class $c\colon\unit\to(Z,Y,g)_!\zeta$, i.e. a morphism in $\Gamma\sC\simeq\presheaf_\baseTopos(*;\sC)$, satisfying the following properties.
\begin{enumerate}[label=(\arabic*)]
\item Both $Z$ and $Y$ are $\sC$-twisted ambidextrous $\baseTopos$-spaces (consequently also $(Z,Y,g)_*\colon\presheafTopos(Z;\sC)\to\sC$ is colimit preserving and $\sC$-linear).
\item Cap product with $c$ induces an equivalence $(Z,Y,g)_*\xrightarrow{\simeq}Z_!(\zeta\otimes-)$.
\item Let $\partial c$ denote the composite $\unit\xrightarrow{c}(Z,Y,g)_!\zeta\xrightarrow{\canonical}\Sigma Y_!g^*\zeta$ in $\Gamma\sC$; then cap product with $\partial c$ induces an equivalence $\Omega Y_*\xrightarrow{\simeq} Y_!(g^*\zeta\otimes-)$ of $\sC$-linear $\baseTopos$-functors $\presheafTopos(Y;\sC)\to\sC$.
\end{enumerate}
\end{defn}

\begin{defn}[Neoclassical Poincar\'e pairs]\label{defn:neoclassical_poincare_pairs}
Recall \cref{defn:connecting_map_of_classifying_systems}.
Let $X\colon [1]\to\baseTopos$ be a functor and
let $g\colon Y\to Z$ denote the corresponding morphism of $\baseTopos$-spaces. We say that $X$ is a \textit{neoclassical $\sC$-Poincar\'e pair} if the following conditions hold.
\begin{enumerate}[label=(\arabic*)]
\item Both $Y$ and $Z$ are $\sC$-twisted ambidextrous $\baseTopos$-spaces.
\item The classifying system $\omega_{Z,Y,g}\colon Z\to\sC$ factors through $\picardSpaceTopos(\sC)$;
\item The connecting map $\delta\colon\Omega \omega_{Y}\to g^*\omega_{Z,Y,g}$ is an equivalence in $\func_\baseTopos(Y,\sC)$.
\end{enumerate}
\end{defn}
We start by proving the equivalence between \cref{defn:pdpairs,defn:postclassical_poincare_pairs,defn:neoclassical_poincare_pairs}.

\begin{prop}
\label{prop:pdpair_equals_neoclassical}
Let $X\colon[1]\to\baseTopos$ be a functor. Then $X$ is a $\sC$-Poincar\'e pair if and only it is a neoclassical $\sC$-Poincar\'e pair, and if and only if it is a postclassical $\sC$-Poincar\'e pair.
\end{prop}
\begin{proof}
We use the notation from \cref{defn:pdpairs}
and \cref{nota:left_adjoint_to_j}.
First, note that $\int[1] X$ is $\sC$-twisted ambidextrous by \cref{lem:unstraightening_twisted_ambidextrous}. 
As a direct consequence of \cref{cor:towards_neoclassical_equivalence}, $X$ is a $\sC$-Poincar\'e pair if and only if $X$ is a neoclassical $\sC$-Poincar\'e pair.

For the implication ``postclassical implies neoclassical'', let $(c,\zeta)$ be as in \cref{defn:postclassical_poincare_pairs}.
We can specialise \cref{eq:relative_cap_product} to the morphism $u\coloneqq g$; after taking  horizontal cofibres twice we obtain the following commutative diagram in $\Gamma\sC$, some of whose objects are endowed with a morphism from $\unit$:
\begin{equation}
\label{eq:neoclassical_postclassical}
\begin{tikzcd}[column sep=0pt, row sep=12pt]
\unit\xrightarrow{c}(Z,Y,g)_!\zeta\ar[r]\ar[ddd]&\Sigma Y_!g^*\zeta\overset{\partial c}{\leftarrow}\unit\ar[d]\\
&\unit\xrightarrow{\partial c\cap-}\nattrans_\sC(\Omega Y_*,Y_!(g^*\zeta\otimes-)\ar[d,"-\circ g^*"]\\
&\unit\xrightarrow{(\partial c\cap-)g^*}\nattrans_\sC(\Omega Y_*g^*,Y_!g^*(\zeta\otimes-))\ar[d,"\epsilon_{g_!}"]\\
\unit\xrightarrow{c\cap-}\nattrans_\sC((Z,Y,g)_*,Z_!(\zeta\otimes-))\ar[r]&\nattrans_\sC(\Omega Y_*g^*,Z_!(\zeta\otimes-)).
\end{tikzcd}
\end{equation}
Recall \cref{defn:DandbarD} and let $\bar\zeta\coloneqq \bar D(\zeta)\in\func_{\baseTopos}(Z,\picardSpaceTopos(\sC))$.
By \cref{prop:classification_of_linear_functors}, the $\baseTopos$-natural equivalence $c\cap-\colon(Z,Y,g)_*\xrightarrow{\simeq}Z_!(\zeta\otimes-)$ induced by $c$ corresponds to an equivalence $\hat c\colon\omega_{Z,Y,g}\xrightarrow{\simeq}\bar\zeta$, using also \cref{prop:dualish} to identify $\bar\zeta$ as the classifying system of the $\sC$-linear $\baseTopos$-functor $Z_!(\zeta\otimes-)$.
Similarly, the $\baseTopos$-natural equivalence $\partial c\cap-\colon\Omega Y_*\xrightarrow{\simeq} Y_!(g^*\zeta\otimes-)$ induced by $\partial c$ corresponds to an equivalence $\widehat{\partial c}\colon\Omega\omega_{Y}\xrightarrow{\simeq} g^*\bar\zeta$.
By \cref{lem:naturality_classification_C_linear_functors}, the $\baseTopos$-natural transformation $(\partial c\cap-)g^*$ corresponds to the morphism $g_!(\widehat{\partial c})\colon \Omega g_!\omega_{Y}\xrightarrow{\simeq}g_!g^*\bar\zeta$, which is therefore an equivalence. Finally, postcomposing the natural equivalence $(\partial c\cap-)g^*$ with the natural transformation $Y_!\epsilon_{g_!}(\zeta\otimes-)$ corresponds, again by \cref{lem:naturality_classification_C_linear_functors}, to postcomposing $g_!(\widehat{\partial c})$ by the map $g_!g^*\bar\zeta\xrightarrow{\epsilon_{g_!}}\bar\zeta$, whereas precomposing $(c\cap-)$ with the canonical natural transformation $\canonical\colon\Omega Y_*\to(Z,Y,g)_*$ corresponds to the map $\Omega g_!\omega_{Y}\to\omega_{Z,Y,g}$ adjoint to $\delta$. Commutativity of \cref{eq:neoclassical_postclassical} implies that the following square in $\func_\baseTopos(Z,\sC)$ on left commutes; applying $g^*$ to it and then precomposing by $\eta_{g_!}$ we obtain the commutative diagram on right.
\begin{equation}
\label{eq:neoclassical_postclassical_2}
\begin{tikzcd}
\Omega g_!\omega_{Y}\ar[r,"g_!(\widehat{\partial c})"]\ar[d]&g_!g^*\bar\zeta\ar[d,"\epsilon_{g_!}"]\\
\omega_{Z,Y,g}\ar[r,"\hat c"]&\bar\zeta;
\end{tikzcd}
\hspace{1cm}
\begin{tikzcd}
\Omega\omega_{Y}\ar[r,"\widehat{\partial c}"]\ar[d,"\eta_{g_!}"]\ar[dd,bend right=80,"\delta"']&g^*\bar\zeta\ar[d,"\eta_{g_!}"]\ar[dd,bend left=60, equal]\\
\Omega g^*g_!\omega_{Y}\ar[r,"g^*g_!(\widehat{\partial c})"]\ar[d]&g^*g_!g^*\bar\zeta\ar[d,"\epsilon_{g_!}"]\\
g^*\omega_{Z,Y,g}\ar[r,"g^*(\hat c)"]&g^*\bar\zeta.
\end{tikzcd}
\end{equation}
The assumption that $\hat c$ and $\widehat{\partial c}$ are equivalences implies that the horizontal arrows in the right diagram of \cref{eq:neoclassical_postclassical_2} are equivalences; hence also $\delta$ is an equivalence. Moreover, since $\bar\zeta$ factors through $\picardSpaceTopos(\sC)$, so do $\omega_{Z,Y,g}$, $g^*\bar\zeta$ and $\omega_{Y}$.

For the implication ``neoclassical implies postclassical'', we let $\zeta\coloneqq\omega_{Z,Y,g}\colon Z\to\picardSpaceTopos(\sC)$. We observe that $g$, being a morphism between $\baseTopos$-spaces, is a cartesian fibration; by \cref{cor:relative_classification_natural_transformations}, we obtain that the equivalence $(Z,Y,g)_*\simeq Z_!(\zeta\otimes-)$ is induced by cap product with some class $c\colon\unit\to(Z,Y,g)_!\zeta$. Letting $\partial c$ denote the composition of $c$ with the canonical map $(Z,Y,g)_!\zeta\to\Sigma Y_!g^*\zeta$, commutativity of \cref{eq:neoclassical_postclassical} implies again the existence of the commutative squares in \cref{eq:neoclassical_postclassical_2}; in particular, the hypothesis that $\delta$ is an equivalence implies that the left, right and bottom arrows of the right diagram in \cref{eq:neoclassical_postclassical_2} are equivalences, and thus also the top arrow is an equivalence: this means that $\partial c$ also induces an equivalence $Y_*\xrightarrow{\simeq} Y_!(g^*\zeta\otimes-)$ as desired.
\end{proof}

In the rest of the subsection we let $\baseTopos=\spc$, and we focus on the cases in which $\sC$ is either $\spectra$ or $\module_\eilenbergMacLaneCoeff=\module_\eilenbergMacLaneCoeff(\spectra)$. We shall prove the following proposition. 
\begin{prop}
\label{prop:equivalence_classical_and_neoclassical}
For a functor $X\colon[1]\to\spc$ corresponding to a map of spaces $g \colon Y \to Z$ the following hold. 
\begin{enumerate}
\item If $X$ is a postclassical $\module_\eilenbergMacLaneCoeff$-Poincar\'e pair, then it is also a classical Poincar\'e pair. The converse holds assuming that $Y$ and $Z$ are compact spaces.
\item If $X$ is a $\spectra$-Poincar\'e pair, then it is also a $\module_\eilenbergMacLaneCoeff$-Poincar\'e pair. The converse holds assuming that $Y$ and $Z$ are compact spaces.
\end{enumerate}

\end{prop}
Before proceeding, we mention the following corollary, which is an immediate consequence of 
\cref{prop:pdpair_equals_neoclassical} and
\cref{prop:equivalence_classical_and_neoclassical}.
\begin{cor}
\label{cor:seven_equivalent_properties}
Let $X\colon[1]\to\spc^\omega$. Then the following seven properties are equivalent for $X$: classical Poincar\'e duality pair, postclassical $\sC$-Poincar\'e duality pair for either $\sC=\spectra$ or $\sC=\module_\eilenbergMacLaneCoeff$, neoclassical $\sC$-Poincar\'e duality pair for either $\sC=\spectra$ or $\sC=\module_\eilenbergMacLaneCoeff$, and  $\sC$-Poincar\'e duality pair for either $\sC=\spectra$ or $\sC=\module_\eilenbergMacLaneCoeff$.
\end{cor}
The next lemma highlights how the hypothesis of compactness enters (2) in \cref{prop:equivalence_classical_and_neoclassical}.
\begin{lem}
\label{lem:compact_omega_bounded_below}
Let $\X$ be a compact $\infty$-category and let $\omega_\X\colon\X\to\spectra$ be the classifying object of the functor $\X_*\colon\presheaf(\X;\spectra)\to\spectra$. Then $\omega_\X$ is uniformly bounded below, i.e. there is some $k\in\bbZ$ such that $\omega_\X$ factors through $\spectra_{\ge k}\subset\spectra$. Similarly, if $(\X,\partial\X)$ is an $\infty$-category pair with $\X$ (and hence also $\partial\X$)
compact, then $\omega_{\X,\partial\X}$ is uniformly bounded below.
\end{lem}
\begin{proof}
For the first part, we argue that the class of $\infty$-categories $\Y$ that are $\spectra$-twisted ambidextrous and for which $\omega_\Y$ is uniformly bounded below is closed under retracts and finite colimits, and it contains $[1]$. As a consequence this class will contain all compact $\infty$-categories as desired.

If $\X$ is a retract of $\Y$ as in \cref{rmk:ambidexterity_closed_under_retracts} and $\omega_\Y$ is uniformly bounded below, then the left Kan extension $r_!\omega_\Y$ is also  uniformly bounded below; moreover, since the functor $\X_*$ is a retract of $\Y_*r^*$, the system $\omega_\X$ is a retract of $r_!\omega_\Y$, hence $\omega_\X$ is uniformly bounded below. 

To check stability under finite colimits, we observe that $\omega_\emptyset$ is obviously uniformly bounded below; given a pushout as in \cref{rmk:ambidexterity_closed_under_finite_colimits}, and assuming that all of $\omega_\X,\omega_\Y,\omega_{\X'}$ are uniformly bounded below, the equivalence $\Y'_* \simeq (\Y_* f^*) \times_{\X_* h^*} (\X'_* g^*)$ gives rise to an equivalence $\omega_Y\simeq f_!\omega_\Y\times_{h_!\omega_{\X}}g_!\omega_{\X'}$, exhibiting $\omega_{Y'}$ as a pullback of three uniformly bounded below systems, so that $\omega_{\Y'}$ is uniformly bounded below as well. 

We finally observe that $\omega_{[1]}\colon[1]\to\spectra$ is the arrow $0\to\sphere$, which is uniformly bounded below.
This completes the argument for the first part.

To give an argument for the second part, recall from the proof of \cref{lem:compact_is_ambidextrous} that if $\X$ is a compact $\infty$-category and $\partial \X \subset \X$ is left closed, then $\partial \X$ is compact as well.
Using \cref{nota:X_dX_interiorX}, we argue that $\omega_{\X,\partial\X}\simeq\fib(\omega_\X\to i_!\omega_{\partial\X})$ is the fibre of a map of uniformly bounded below systems, hence it is uniformly bounded below.
\end{proof}

\begin{proof}[Proof of \cref{prop:equivalence_classical_and_neoclassical}]
For (1), we may assume that $Z$ is connected and that both $Y$ and $Z$ are $\module_\eilenbergMacLaneCoeff$-twisted ambidextrous. 
For the implication ``postclassical implies classical'', let $(c,\zeta)$ be a pair as in \cref{defn:postclassical_poincare_pairs}. Since $\picardSpace(\module_\eilenbergMacLaneCoeff)\simeq\coprod_{\bbZ}BC_2$ and $BC_2 = B GL_1(\bbZ)$, there is a number $d \in \bbZ$ and a local coefficient system $\calO$ of infinite cyclic groups over $Z$ such that $\zeta \simeq \calO[-d]$.
The map $c\colon\bbZ\to (Z,Y,g)_!\zeta$ then gives a homology class $[Z,Y]$ in $H_d(Z,Y;\calO)$, where relative homology is computed with respect to $g$; the image of this class along the connecting homomorphism gives us a class $[Y]\in H_{d-1}(Y;g^*\calO)$, which is represented by $\partial c\colon\eilenbergMacLaneCoeff\to \Sigma Y_!g^*\zeta$.
We may regard every coefficient system of abelian groups $\Lambda_0$ on $Y$ or $\Lambda_1$ on $Z$ as an object in $\presheaf(Y;\module_\eilenbergMacLaneCoeff)$ or $\presheaf(Z;\module_\eilenbergMacLaneCoeff)$, namely a presheaf with values in the heart $\module_\eilenbergMacLaneCoeff^\heartsuit$. 
The equivalences $c\cap-\colon (Z,Y,g)_*\Lambda_1\xrightarrow{\simeq}Z_!(\zeta\otimes\Lambda_1)$ and $\partial c\cap-\colon (Z,Y,g)_*\Lambda_0\xrightarrow{\simeq}Z_!(\zeta\otimes\Lambda_0)$ in $\module_\eilenbergMacLaneCoeff$ then give the desired equivalences between cohomology and homology groups after passing to homotopy groups.
We can finally deduce that $d\ge0$ from the equivalence $H^d(Z,Y;\calO)\simeq H_0(Z)$, where the latter is nonzero because $Z$, being connected, is non-empty.

For the implication ``classical implies postclassical'' we start with a pair $(\calO,[Z,Y])$ as in \cref{defn:classical_poincare_pairs}, and let $d\ge0$ denote the formal dimension. We may interpret $\calO$ as a map $\zeta\colon Z\to\picardSpace(\module_\eilenbergMacLaneCoeff)$ with image in the $-d$\textsuperscript{th} component, so that $[Z,Y]$ is represented by a map $c\colon\eilenbergMacLaneCoeff\to (Z,Y,g)_!\zeta$. Letting $\partial c$ denote the composite $\eilenbergMacLaneCoeff\xrightarrow{c}(Z,Y,g)_!\zeta\xrightarrow{\canonical}\Sigma Y_!g^*\zeta$, the hypothesis implies that the natural transformations induced by cap product with $c$ and $\partial c$
\begin{equation}
\label{eq:classical_postclassical}
\begin{split}
(Z,Y,g)_*\to Z_!(\zeta\otimes-)&\text{ in }\func^L_{\module_\eilenbergMacLaneCoeff}(\presheaf(Z;\module_\eilenbergMacLaneCoeff),\module_\eilenbergMacLaneCoeff);\\
\Omega Y_*g^*\to Y_!(g^*\zeta\otimes-)&\text{ in }\func^L_{\module_\eilenbergMacLaneCoeff}(\presheaf(Y;\module_\eilenbergMacLaneCoeff),\module_\eilenbergMacLaneCoeff),
\end{split}
\end{equation}
are equivalences at least when restricted to presheaves with values in $\module_\eilenbergMacLaneCoeff^\heartsuit$, as can be checked on homotopy groups. Since such presheaves generate $\presheaf(X(k);\module_\eilenbergMacLaneCoeff)$ for $k=0,1$ under colimits and limits
and since all functors involved in \cref{eq:classical_postclassical} preserve colimits and
limits by our compactness assumptions, we obtain that $c\cap-$ and $\partial c\cap-$ are equivalences.

The first part of point (2) is a direct consequence of \cref{obs:omega_natural_in_C}. For the second part of (2) we let $\omega_{\X,\partial\X}\colon\X\to\spectra$ denote the classifying system for the functor $(\X,\partial\X)_*\colon\presheaf(\X;\spectra)\to\spectra$, using the notation from \cref{defn:pdpairs}; the hypothesis and \cref{obs:omega_natural_in_C} imply that $\omega_{\X,\partial\X}\otimes\eilenbergMacLaneCoeff\colon\X\to\module_\eilenbergMacLaneCoeff$ factors through $\picardSpace(\module_\eilenbergMacLaneCoeff)$; by 
\cref{lem:compact_omega_bounded_below} and the Hurewicz theorem, also $\omega_{\X,\partial\X}$ factors through $\picardSpace(\spectra)$.
\end{proof}

\subsection{Poincar\'e ads}
\label{subsec:pdads}
In this subsection we present our definition of Poincar\'e duality for \textit{ads}, and prove its equivalence with Wall's original definition.

\begin{nota}[$n$-cubes]\label{nota:n_cubes}
For $n\ge0$ we let refer to the poset $\Box^n\coloneqq[1]^{\times n}$ as the \textit{$n$-cube}, and we let $\partial \Box^n\subseteq \Box^n$ denote the punctured cube $\Box^n\setminus\{\udl{1}\}$, where the maximum element has been taken out.
For an object $\udl{s}\in\Box^n$ we write $\#\udl{s}=\sum_is_i\ge0$, so for example, $\#\udl{1}=1+\dots+1=n$. We observe that $\Box^n_{/\udl{s}}$ is a left closed subposet of $\Box^n$, which is isomorphic to $\Box^{\#\udl{s}}$. We denote $\partial \Box^n_{/\udl{s}}\coloneqq \Box^n_{/\udl{s}}\setminus\{\udl{s}\}$, which under any identification $\Box^n_{/\udl{s}}\simeq\Box^{\#\udl{s}}$ corresponds to the left closed subposet $\partial\Box^{\#\udl{s}}$.
\end{nota}

\begin{obs}[Complements in cubes]\label{obs:complements_in_cubes}
Write $S\coloneqq \{\udl{s}\in \Box^n \:| \:\#\udl{s}=n-1\}$. Then the collection of category pairs
$\{(\Box^n_{/\udl{s}},\partial \Box^n_{/\udl{s}})\}_{\udl{s}\in S}$
form a complemented cover of $(\partial\Box^n,\emptyset)$. That they cover $\partial \Box^n$ is clear, so we are left to argue that $(\Box^n_{/\udl{s}},\partial \Box^n_{/\udl{s}})$ are complemented inside $\partial \Box^n$.
For a fixed $\udl{s}\in S$, let $Z\coloneqq\bigcup_{\udl{t}\in S\setminus\{\udl{s}\}}\Box^n_{/\udl{t}}$, and let $\partial Z\coloneqq Z\cap\partial \Box^n_{/\udl{s}}$; we claim that $(Z,\partial Z)$ forms a valid complement of $(\Box^n_{/\udl{s}},\partial \Box^n_{/\udl{s}})$ inside $\partial \Box^n$. Conditions (1) and (2) in  \cref{defn:complements_of_subcategory_pairs} are immediate, and condition (3) requires us to check that $Z\setminus\partial Z$ is final in $Z$.
Without loss of generality, we may further assume that $\udl{s}=(0,1,1,\ldots,1)$; then we have $Z=[1]\times\partial\Box^{n-1}$ and $\partial Z=\{0\}\times\partial\Box^{n-1}$, so that $Z\setminus\partial Z=\{1\}\times\partial\Box^{n-1}$ is indeed final in $Z$.
\end{obs}

\begin{terminology}
For $n\ge0$, an \textit{$(n+1)$-ad of $\baseTopos$-spaces} is a functor $X\colon \Box^n\rightarrow \baseTopos$.
\end{terminology}

\begin{rmk}
A $(0+1)$-ad is an object in $\baseTopos$. A $(1+1)$-ad is a map $X_0 \rightarrow X_1$ in $\baseTopos$, i.e. a pair of $\baseTopos$-spaces. A $(2+1)$-ad is what is classically called a \textit{triad}, i.e. a  square in $\baseTopos$
\[
\begin{tikzcd}[row sep=10pt]
X_{00} \ar[r] \ar[d] & X_{01} \ar[d]\\
X_{10} \ar[r] & X_{11}
\end{tikzcd}
\]
which one classically would represent by a triple of CW complexes $(X_{11},X_{01},X_{10})$ with $X_{01},X_{10}$ being subcomplexes of $X_{11}$, and with $X_{00}$ \textit{defined} as $X_{01}\cap X_{10}$. Compare the notion of a manifold triad.
Similarly, an $(n+1)$-ad is classically represented by an $(n+1)$-tuple of CW complexes, of which the last $n$ are subcomplexes of the first; the rest of the functor $\Box^n\to\mathrm{Top}$ is then defined by taking iterated intersections.
\end{rmk}

We give two definitions of Poincar\'e $(n+1)$-ads. The first is a direct generalisation of Wall's definition \cite[§2]{Wall_scm} in the case of $\sC=\module_{\eilenbergMacLaneCoeff}$ by \cref{cor:seven_equivalent_properties},
reducing to the notion of Poincar\'e pair; the second is the one directly stemming from our theory. We then prove the equivalence between the two definitions.

\begin{defn}[Poincar\'e $(n+1)$-ads]\label{defn:poincare_ads}
Let $X\colon \Box^n\rightarrow \baseTopos$ be an $(n+1)$-ad. Then we say that $X$ is:
\begin{itemize}
\item a \textit{classical $\sC$-Poincar\'e $(n+1)$-ad} if for every $\udl{s}\in\Box^n$ the following canonical morphism of $\baseTopos$-spaces, regarded as a functor $[1]\to\baseTopos$, is a $\sC$-Poincar\'e pair:
\[
\colim_{\partial\Box^n_{/\udl{s}}}X\rightarrow X(\udl{s}).
\]

\item a \textit{$\sC$-Poincar\'e $(n+1)$-ad} if the $\baseTopos$-category pair $(\X,\partial\X)\coloneqq(\int_{\Box^n}X,\int_{\partial\Box^n}X)$ is $\sC$-Poincar\'e in the sense of \cref{defn:categorypair}.
\end{itemize}
\end{defn}

\begin{prop}[Characterisations of Poincar\'e $(n+1)$-ads]\label{prop:wall_vs_us_ads}
Let $X\colon \Box^n\rightarrow\baseTopos$ be an $(n+1)$-ad. Then $X$ is a classical $\sC$-Poincar\'e $(n+1)$-ad if and only if it is a $\sC$-Poincar\'e $(n+1)$-ad.
\end{prop}
\begin{proof}
We prove the coimplication by induction on $n$. For $n=0,1$ the definitions of classical $\sC$-Poincar\'e $(n+1)$-ad and $\sC$-Poincar\'e $(n+1)$-ad coincide, so let us assume $n\ge2$. We denote $(\X,\partial\X)\coloneqq(\int_{\Box^n}X,\int_{\partial\Box^n}X)$, and use \cref{nota:X_dX_interiorX,nota:left_adjoint_to_j}: in particular we observe that $\interior{\X}$ is final in $\X$, as $\{\udl{1}\}$ is final in $\Box^n$.

We first assume that $X$ is a $\sC$-Poincar\'e $(n+1)$-ad.
By \cref{prop:boundary_principle} we then have that $\partial\X$ is a $\sC$-Poincar\'e $\baseTopos$-category, and by \cref{cor:local_to_global_principle}, which is applicable thanks to \ref{obs:complements_in_cubes}, we have that for all $\udl{s}\in\Box^n$ with $\#\udl{s}=n-1$, the $\baseTopos$-category pair $(\int_{\Box^n_{/\udl{s}}}X,\int_{\partial\Box^n_{/\udl{s}}})$ is $\sC$-Poincar\'e. By inductive hypothesis we have that for all $\udl{t}\in\Box^n$ with $\#\udl{t}\le n-1$, the canonical morphism of $\baseTopos$-spaces $\colim_{\partial\Box^n_{/\udl{t}}}X\to X(\udl{t})$ is a $\sC$-Poincar\'e pair. It is therefore only left to prove that the morphism of $\baseTopos$-spaces $\colim_{\partial\Box^n}X\to X(\udl{1})$ is a $\sC$-Poincar\'e pair. For this we observe that the functor $\min\colon\Box^n\to[1]$, sending all elements to $0$ except $\udl{1}\mapsto 1$, is a cocartesian fibration. Composing $\X\to\Box^n\xrightarrow{\min}[1]$ we obtain a $\baseTopos$-cocartesian fibration, corresponding to the $\baseTopos$-category pair $(\X, \partial \X)$.
By \cref{prop:barXdXpdpair} we obtain that the functor $[1] \to \baseTopos$ corresponding to $\lvert \partial \X \rvert \to \lvert \X \rvert$, which is precisely the map $\colim_{\partial\Box^n}X\to X(\udl{1})$, is a $\sC$-Poincar\'e pair.

Assume now that $X$ is a classical $\sC$-Poincar\'e $(n+1)$-ad. 
Applying \cref{prop:localisation_principle} to the pair $\colim_{\partial\Box^n_{/\udl{s}}}X\rightarrow X(\udl{s})$ we obtain that $X(\udl{s})$ is $\sC$-twisted ambidextrous for all $\udl{s}\in\Box^n$, i.e. $X$ takes values in $\sC$-twisted ambidextrous $\baseTopos$-spaces.
Consequently, $(\X, \partial \X)$ is $\sC$-twisted ambidextrous by \cref{lem:unstraightening_twisted_ambidextrous}. 
Furthermore, for all $\udl{s}\in\Box^n$ with $\#\udl{s}=n-1$ we have that $X|_{\Box^n_{/\udl{s}}}$ is a classical $\sC$-Poincar\'e $((n-1)+1)$-ad, hence, by inductive hypothesis, a $\sC$-Poincar\'e $((n-1)+1)$-ad. 
Again by \cref{cor:local_to_global_principle} and \cref{obs:complements_in_cubes} we obtain that $\partial\X$ is a $\sC$-Poincar\'e $\baseTopos$-category. Moreover, since $|\partial\X|\to\interior{\X}$ is a $\sC$-Poincar\'e pair, we get that $\interior{\X}\simeq X(\udl1)$ is $\sC$-twisted ambidextrous, and hence $\X$ is $\sC$-twisted ambidextrous.
This implies that $\omega_{\X,\partial\X}$ is defined, and the previous remarks show that $\omega_{\X,\partial\X}\colon\X\to\sC$ factors through the localisation of $\X$ at the wide subcategory $\partial\X\amalg\interior{\X}$. The argument from the previous paragraph identifies this localisation with the unstraightening of $|g|$, and hence the hypothesis that $|g|$ is a $\sC$-Poincar\'e pair, together with \cref{prop:localisation_principle}, imply that $(\X,\partial\X)$ is a $\sC$-Poincar\'e $\baseTopos$-category pair.
\end{proof}

\begin{prop}[Neoclassical Poincar\'e--Lefschetz sequence]
Let $X\colon[1]\to\baseTopos$ be a $\sC$-Poincar\'e pair as in \cref{defn:pdpairs}, let $\X=\int_{[1]}X$ and use \cref{nota:X_dX_interiorX,nota:left_adjoint_to_j}. Let $D\coloneqq j^* D_{\X,\partial\X}\in\presheafTopos(\interior{\X};\sC)$. Then there is a commuting diagram of fibre sequences in $\func_{\sC,\baseTopos}^L(\presheafTopos(\interior{\X};\sC),\sC)$
\[
\begin{tikzcd}
(\interior{\X},\partial \X,g)_*\ar[r] \ar[d, "\simeq"] & \interior{\X}_* \ar[r] \ar[d, "\simeq"] & \partial\X_* g^* \ar[d, "\simeq"] \\
\interior{\X}_!(D \otimes -) \ar[r] & (\interior{\X},\partial \X,g)_!(D \otimes -) \ar[r] & \Sigma \partial \X_! g^*(D \otimes \xi)
\end{tikzcd}
\]
\end{prop}
\begin{proof}
We rotate the diagram \cref{eq:categorical_poincare_lefschetz_duality} and precompose it by $j_!\simeq \ell^*$.
\end{proof}

\subsection{Triangulations of manifolds}
\label{subsec:triangulated_manifolds}
In this subsection we study Poincar\'e $\baseTopos$-category pairs obtained from diagrams of $\baseTopos$-spaces parametrised by the face poset of a triangulated manifold, and establish a local-to-global principle for these in the form of \cref{prop:combinatorial_manifold_Poincare}. 
We use the following definition of a \textit{combinatorial manifold}.
\begin{defn}
\label{defn:combinatorial_manifold}
A \textit{combinatorial manifold with boundary of dimension $d\ge0$} is a finite poset $I$ with a left closed subposet $\partial I$ for which there is a ``rank'' functor $\rho\colon I\to[d]$ satisfying the following properties, where we consider the classical geometric realisation of simplicial sets obtained e.g. as colimit in $\mathrm{Top}$ of a diagram of simplices:
\begin{itemize}
\item the set $I_{\max}$ of maximal elements of $I$ is contained in the right closed subposet $\interior{I}\coloneqq I\setminus\partial I$, and $\rho$ attains constant value $d$ on $I_{\max}$;
\item for all $j<i$ in $I$, the geometric realisation of the subposet $I_{j<<i}\subset I$ of elements strictly between $j$ and $i$ is homeomorphic to the sphere $S^{\rho(i)-\rho(j)-2}$;
\item for all $j\in \interior{I}$ the geometric realisation $|I_{>j}|$ of the subposet $I_{>j}\subset I$ 
of elements strictly larger than $j$ is homeomorphic to the sphere $S^{d-\rho(j)-1}$;
\item for all $j\in\partial I$, we have a homeomorphism of pairs $(|I_{>j}|, |\partial I_{>j}|) \cong (D^{d-\rho(j)-1}, S^{d-\rho(j)-2})$.
\end{itemize}
\end{defn}

For a more classical definition of combinatorial manifolds in the case of the face poset of a finite simplicial complex, we invite the reader to compare \cite[Ch. 1, p.26]{Hudson}. We stress here that we do not require the rank functor $\rho$ to be essentially surjective; in fact the essential image of $\rho$ will be a subposet of $[d]$ of the form $[d]_{\ge k}$ for some $k\in [d]$.

\begin{example}
Let $M$ be a compact $d$-manifold with boundary $\partial M$, endowed with a polyhedral decomposition such that links of facets contained in $\interior{M}$ are spheres, and links of facets contained in $\partial M$ are discs. We obtain a combinatorial manifold as follows: we let $I$ be the poset of facets of $M$, we let $\partial I$ be the subposet of facets of $\partial M$, and we let $\rho$ be the dimension function on facets. Note that in this case the following additional property holds: for all $j\in I$, the geometric realisation $|I_{<j}|$ of the subposet $I_{<j}\subset I$ of elements strictly smaller than $j$ is homeomorphic to the sphere $S^{\rho(j)-1}$.
\end{example}
\begin{example}
\label{ex:cubes_combinatorial_manifolds}
For $n\ge1$ the pair of posets $(\Box^n,\partial\Box^n)$, with the rank function $\rho=\#$ from \cref{nota:n_cubes}, is a combinatorial manifold. We may think of it as the poset of strata of the non-compact manifold with boundary and corners $(\mathbb{R}_{\ge0})^n$.
\end{example}

\begin{lem}
\label{lem:locally_flat_in_combinatorial_manifolds}
Let $I$ be a combinatorial manifold as in \cref{defn:combinatorial_manifold} with $\partial I=\emptyset$, let $i\in I_{\max}$ and let $j\in I_{<i}$. Then the embedding of spheres $|I_{j<<i}|\cong S^{d-\rho(j)-2}\hookrightarrow |I_{>j}|\cong S^{d-\rho(j)-1}$ is locally flat, hence standard by the piecewise linear Sch\"onflies theorem.
\end{lem}
\begin{proof}
We proceed by induction on $d-\rho(j)\ge1$. For $d-\rho(j)=1$ we have that the unique embedding of $S^{-1}\cong\emptyset$ in $S^0$ is locally flat. Assuming $d-\rho(j)\ge2$, we need to check that for all $j'\in I_{j<<i}$ the embedding $|I_{j<<i}|\hookrightarrow |I_{>j}|$ is locally flat near $j'$. We observe that the subspaces $|I_{j<\le j'}|\times|I_{j'\le<i}|\subseteq|I_{j<<i}|$ are closed discs whose interiors give a cover of the sphere $|I_{j<<i}|$; it thus suffices to check that the following inclusion of discs is standard: 
\[
|I_{j<\le j'}|\times|I_{j'\le<i}|\simeq D^{\rho(j')-\rho(j)-1}\times D^{d-\rho(j')-1}\hookrightarrow |I_{j<\le j'}|\times |I_{\ge j'}|\simeq D^{\rho(j')-\rho(j)-1}\times D^{d-\rho(j')},
\]
and by the piecewise linear Sch\"onflies theorem this reduces to checking that the same inclusion of discs is locally flat; the latter condition is now implied by the inductive assumption that the inclusion $|I_{j'<<i}|\cong S^{d-\rho(j')-2}\hookrightarrow |I_{>j'}|\cong S^{d-\rho(j')-1}$ is locally flat.
\end{proof}

\begin{lem}
\label{lem:complements_in_combinatorial_manifolds}
Let $I$ be a combinatorial manifold as in \cref{defn:combinatorial_manifold} with $\partial I=\emptyset$. Then the collection of subcategory pairs $\{(I_{\le i},I_{<i})\}_{i\in I_{\max}}$ forms a complemented cover of $I$ in the sense of \cref{term:complemented_covers}.
\end{lem}
\begin{proof}
It is clear that $I=\bigcup_{i\in I_{\max}}I_{\le i}$. It suffices therefore to prove that for all $i\in I_{\max}$ the pair of left closed subposets $I_{<i}\subset I_{\le i}$ of $I$ is complemented by the pair $I_{<i}\subset I\setminus\{i\}$.

Conditions (1) and (2) of \cref{defn:complements_of_subcategory_pairs} are immediate; to check (3), i.e. that the inclusion $I\setminus I_{\le i}=(I\setminus\{i\})\setminus I_{<i}\to I\setminus\{i\}$ is final, by Quillen's Theorem A it suffices to pick $j\in I_{<i}$ and show that the overcategory $(I\setminus I_{\le i})_{j/}\simeq I_{>j}\setminus I_{j<\le i}$ is weakly contractible.

The geometric realisation $|I_{>j}\setminus I_{j<\le i}|$ is homotopy equivalent to the (open) complement of $|I_{j<\le i}|$ in $|I_{>j}|$: this is best seen by regarding $I_{>j}$ as the geometric realisation of the finite simplicial complex whose $p$-simplices are chains of length $p$ in $I_{>j}$, and by identifying $I_{j<\le i}$ as a full subsimplicial complex of the previous.
By \cref{lem:locally_flat_in_combinatorial_manifolds} the topological space $|I_{>j}|\setminus|I_{j<\le i}|$ is homeomorphic to an open disc of dimension $d-\rho(j)-1$, which in particular is contractible as desired.
\end{proof}
\begin{obs}
The statement of \cref{lem:complements_in_combinatorial_manifolds} is wrong in general if we do not assume $\partial I=\emptyset$. For instance, let $I$ be the poset of facets of the triangulation of a pentagon; then $I$ has three maximal elements $i_1,i_2,i_3$, corresponding to the three triangles, and precisely one minimal element $j\in I$ lying below all maximal elements. Assuming that the triangle $i_2$ lies between $i_1$ and $i_3$ in the triangulation, we see that $(I\setminus I_{\le i_2})_{j/}$ is isomorphic to $[1]\amalg[1]$, in particular it is not weakly contractible. For this reason, we shall use a doubling argument similar to \cref{ex:doubling_principle} in the proof of \cref{prop:combinatorial_manifold_Poincare} to cover combinatorial manifolds with boundary as well.
\end{obs}

\begin{thm}
\label{prop:combinatorial_manifold_Poincare}
Let $(I,\partial I)$ be a combinatorial manifold as in \cref{defn:combinatorial_manifold}.
Let $X\colon I\to\cat_\baseTopos$ be a functor, and assume that $X$ takes values in $\sC$-twisted ambidextrous $\baseTopos$-categories. Let $(\X,\partial\X)\coloneqq(\int_IX,\int_{\partial I}X)$, and for $j\in I$ let $(\X_j,\partial\X_j)\coloneqq(\int_{I_{\le j}}X,\int_{I_{<j}}X)$. 
Then for any $j\in I$ the restriction of $\omega_{\X,\partial\X}$ on $\X_j$ is equivalent to $\Omega^{d-\rho(j)}\omega_{\X_j,\partial\X_j}$. In particular, the following conditions are equivalent:
\begin{enumerate}[label=(\arabic*)]
\item $(\X,\partial\X)$ is $\sC$-Poincar\'e;
\item for all $j\in I$, the $\baseTopos$-category pair $(\X_j,\partial\X_j)$ is $\sC$-Poincar\'e;
\item for all $i\in I_{\max}$, the $\baseTopos$-category pair $(\X_i,\partial\X_i)$ is $\sC$-Poincar\'e;
\end{enumerate}
If these conditions hold, then $(|\X|,|\partial \X|)$ is also a $\sC$-Poincar\'e pair of $\baseTopos$-spaces.
\end{thm}
\begin{proof}
First, note that $(\X, \partial \X)$ is $\sC$-twisted ambidextrous by \cref{lem:unstraightening_twisted_ambidextrous}, using that finite posets are finite $\infty$-categories and in particular compact.

As a first step, we assume $\partial I=\emptyset$ and prove the first statement by induction on $d-\rho(j)\ge0$. For $d-\rho(j)=0$, i.e. $j\in I_{\max}$, the first statement is a direct consequence of \cref{lem:complements_in_combinatorial_manifolds,lem:pulling_back_complemented_pairs,prop:complementation_principle}. For $d-\rho(j)\ge1$ we may find $i\in I_{\max}$ with $i>j$; then the restriction of $\omega_{\X,\partial\X}$ on $\X_i$ is $\omega_{\X_i,\partial\X_i}$, which by \cref{prop:boundary_principle} further restricts to $\Omega\omega_{\partial\X_i}$ on $\partial\X_i$; by inductive hypothesis applied to the combinatorial manifold $I_{<i}$ of dimension $d-1$, we have that $\omega_{\partial\X_i}$ restricts to $\Omega^{d-1-\rho(j)}\omega_{\X_j,\partial\X_j}$ on $\X_j$, as desired.

For the second statement, the equivalence of (1) and (3) is a direct consequence of \cref{cor:local_to_global_principle}. It is clear that (2) implies (3); to prove that (3) implies (2) we observe again that for $j\in I\setminus I_{\max}$ we may pick $i\in I_{\max}$ with $j<i$, and by the first statement $\omega_{\X_i,\partial\X_i}$ restricts to $\Omega^{d-\rho(j)}\omega_{\X_j,\partial\X_j}$ on $\X_j$.

We now pass to the general case in which $\partial I$ may be non-empty. 
We first claim that $\interior{I}\coloneqq I\setminus \partial I$ is final in $I$: to see this, we pick  $j\in\partial I$ and argue via Quillen's Theorem A, by checking that the overcategory $\interior{I}_{j/}$ is weakly contractible: we have an isomorphism $\interior{I}_{j/}\simeq I_{>j}\setminus\partial I_{>j}$, and the geometric realisation of the latter is homotopy equivalent to the open disc $|I_{>j}|\setminus|\partial I_{>j}|\cong D^{d-\rho(j)-1}\setminus S^{d-\rho(j)-2}$, which is contractible. It follows again by \cref{obs:leftclosed}\ref{rmk:left_closed_fibration} that also $\interior{\X}\coloneqq\X\setminus\partial\X$ is final in $\X$.
Now we let $I'\coloneqq I\amalg_{\partial I}I$ denote the ``double''of $I$ along $\partial I$, which again is a combinatorial manifold, 
and let $X'$ denote the composite functor $I'\xrightarrow{\text{fold}} I\xrightarrow{X}\cat_\baseTopos$, and let $\X'\coloneqq\int_{I'}X'\simeq \X\amalg_{\partial\X}\X$. By \cref{ex:gluingXY}, the restriction of $\omega_{\X'}$ along either inclusion $\X\hookrightarrow\X'$ agrees with $\omega_{\X,\partial\X}$; on the other hand, since $I'$ is a combinatorial manifold without boundary, and since $I'_{\le j}\simeq I_{\le j}$ for any $j\in I$ and any of the two inclusions $I\hookrightarrow I'$, we have that the restriction of $\omega_{\X'}$ on $\X_j$, that is the restriction of $\omega_{\X,\partial\X}$ on $\X_j$, agrees with $\Omega^{d-\rho(j)}\omega_{\X_j,\partial\X_j}$.
For the second part, by \cref{prop:complementation_principle}, (1) is equivalent to requiring that $\X'$ is $\sC$-Poincar\'e, which by the case ``$\partial I=\emptyset$'' is equivalent to (2) and (3) for the poset $I'$, which are  equivalent to (2) and (3), respectively, for the poset $I$, by the preceding observations about how $\omega_{\X'}$ restricts along any of the inclusions $\X\subseteq \X'$.

The final statement directly follows from \cref{prop:barXdXpdpair} and \cref{obs:leftclosed} \ref{rmk:left_closed_fibration}.
\end{proof}

\begin{rmk}
In the case where $X \colon I \to \baseTopos$ lands in $\baseTopos$-groupoids, an argument analogous to the proof of \cref{prop:wall_vs_us_ads}, which we leave to the reader, shows that conditions (1)-(3) in \cref{prop:combinatorial_manifold_Poincare} are also equivalent to the following:
\begin{enumerate}[label=(\arabic*)]
\setcounter{enumi}{3}
\item for all $j\in I$, the map of $\baseTopos$-spaces $\colim_{I_{\le j}}X\to X(j)$ is a $\sC$-Poincar\'e pair.
\end{enumerate}
\end{rmk}

\begin{example}
\label{example:cobordism}
Let $k\ge0$ and let $I$ denote the face poset of the segment $[0,k]$ subdivided into $k$ unit segments; schematically, $I$ is the following poset:
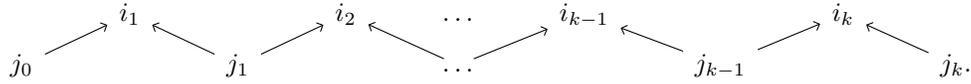
\[
\begin{tikzcd}[row sep=3pt]
&i_1& &i_2&\dots & i_{k-1}& &i_k\\
j_0\ar[ur]& &j_1\ar[ur]\ar[ul]& &\dots\ar[ur]\ar[ul]& &j_{k-1}\ar[ur]\ar[ul]& &j_k\ar[ul].
\end{tikzcd}
\]
Let $\partial I=\{j_0,j_k\}$. Then $(I,\partial I)$ is a combinatorial 1-manifold with boundary. Let $X\colon I\to\baseTopos$ be a functor, and assume that $(\X,\partial\X)\coloneqq(\int_{I}X,\int_{\partial I}X)$ is a $\sC$-Poincar\'e $\baseTopos$-category pair. Applying \cref{prop:combinatorial_manifold_Poincare}, we obtain the following:
\begin{itemize}
\item for $k=0$ we have $(I,\partial I)=(\ast,\emptyset)$ and hence we are requiring that $\X\simeq X(\ast)$ is a $\sC$-Poincar\'e $\baseTopos$-space;
\item for $k=1$ we may identify $(\X,\partial\X)\simeq(\int_{[1]}((X(j_0)\amalg X(j_1))\to X(i_1)),(X(j_0)\amalg X(j_1)))$ and hence we are asking that the map of $\baseTopos$-spaces $(X(j_0)\amalg X(j_1))\to X(i_1)$, considered as a functor $[1]\to\baseTopos$, is a $\sC$-Poincar\'e pair; we may think of $X(i_1)$ as a \textit{$\sC$-Poincar\'e cobordism} from $X(j_0)$ to $X(j_1)$.
\item for $k\ge2$, the local-to-global principle established in \cref{prop:combinatorial_manifold_Poincare} allows us to identify $X$ with the datum of $k$ composable $\sC$-Poincar\'e cobordisms.
\end{itemize}
Relying on generalisations of this idea, in future work we plan to construct and study higher categories of Poincar\'e cobordisms in a systematic way.
\end{example}

\subsection{More examples}
\label{subsec:more_examples}
We collect more applications with unstraightened $\baseTopos$-category pairs.
\begin{example}
\label{example:gluing_pairs}
Let $I\simeq[1]\amalg_{0}[1]$ be the poset with one minimum $0$ and two maximal elements $1,1'$. Let $X\colon I\to\baseTopos$ be a functor, and let $\X\coloneqq\int_IX$. Then by \cref{ex:gluingXY}, $\X$ is a $\sC$-Poincar\'e $\baseTopos$-category if and only if the restrictions of $X$ to both inclusions $[1]\hookrightarrow I$ are $\sC$-Poincar\'e duality pairs as in any of \cref{defn:pdpairs,defn:postclassical_poincare_pairs,defn:neoclassical_poincare_pairs}. If this is the case, we deduce from \cref{ex:absolute_realisation_principle} that $\vert\X\vert\simeq X(1)\amalg_{X(0)}X(1')$ is also a $\sC$-Poincar\'e duality $\baseTopos$-space. Our theory thus recovers and generalises the well-known principle according to which gluing two Poincar\'e duality pairs along their common boundary yields a Poincar\'e duality space.
\end{example}

\begin{example}
\label{example:gluing_ads}
Generalising \cref{example:gluing_pairs}, let $n\ge1$, consider the inclusion $\Box^{n-1} = \Box^{n-1} \times \{0\} \hookrightarrow \Box^n$, let $I = \Box^n \amalg_{\Box^{n-1}} \Box^n=\Box^{n-1}\times([1]\amalg_{[0]}[1])$, and let $\partial I=\partial\Box^{n-1}\times([1]\amalg_{[0]}[1])$. Let $X \colon I \to \baseTopos$ be a functor, and denote the associated $\baseTopos$-category pair by $(\X,\partial\X)\coloneqq(\int_IX,\int_{\partial I}X)$.
As the two copies of $(\Box^n, \partial \Box^n)$ form a complemented cover of $(I, \partial I)$, we have that $(\X,\partial\X)$ is $\sC$-Poincar\'e if and only if the two $(n+1)$-ads obtained by restricting $X$ to the two copies of $\Box^n$ are $\sC$-Poincar\'e in the sense of \cref{defn:poincare_ads}.
Now let $\W\subset I$ denote the wide subposet $(\Box^{n-1})^\simeq\times([1]\amalg_{[0]}[1])$. Then $\scL(I,\W)\simeq\Box^{n-1}$ and $\scL(\partial I,\W\cap \partial I) \simeq \partial \Delta^{n-1}$. By \cref{prop:localisation_principle} we have that the localisation of $(\X,\partial\X)$ at $p^{-1}(\W)$ is again a $\sC$-Poincar\'e $\baseTopos$-category pair. We thus recover the well-known principle according to which gluing two Poincar\'e $(n+1)$-ads along a common boundary $n$-ad results in a new Poincar\'e $n$-ad.
\end{example}

\begin{example}
The following example suggests that for a generic poset $I$ we can expect very few $\sC$-Poincar\'e $\baseTopos$-categories of the form $\int_IX$ as in \cref{cons:integral_categories}.
Let $k\ge1$ and let $I$ be the poset with $k$ maximal elements $i_1,\dots,i_k$ and a minimum $j$. Let $X\colon I\to\baseTopos$ be a functor and let $p\colon\X\coloneqq\int_IX\to I$ denote the unstraightening. Let $\Y\coloneqq p^{-1}(j)$, and let $i\colon\Y\to\X$ denote the (left closed) inclusion.

Assume that $\X$ is $\sC$-Poincar\'e with dualising system $D_\X=D(\omega_\X)\colon\X\to\picardSpaceTopos(\sC)$. Then the equivalence $\X_*\simeq\X_!(D_\X\otimes-)$ gives rise to a composite equivalence
\[
\X_*i_!\simeq\X_!(D_\X\otimes i_!(-))\simeq \Y_!(i^*D_\X\otimes-).
\]
On the other hand, recalling that $i_!\colon\presheafTopos(\Y;\sC)\to\presheafTopos(\X;\sC)$ is extension by zero, we have an equivalence $\X_*i_!\simeq(\Omega\Y_*)^{\bigoplus k-1}$.
We deduce an equivalence $i^*D_\X\simeq (\Omega D_\Y)^{\oplus k-1}$. If $\presheafTopos(\Y;\sC) \neq 0$, this is only possible for $k=2$. The case $k=2$ recovers \cref{example:gluing_pairs}.
\end{example}

\begin{example}
\label{example:spheres}
Let $\Y$ be a $\sC$-twisted ambidextrous $\baseTopos$-space, and let $t\colon\Y\to\ast$ denote the terminal map. Let $\unit\colon\ast\to\sC$ denote $\baseTopos$-functor giving the monoidal unit. We say that $\Y$ is:
\begin{itemize}
\item a \textit{$\sC$-homology sphere} if $(\ast,\Y,t)_*(\unit)\in\presheaf(\ast;\sC)\simeq\Gamma\sC$ is invertible;
\item a \textit{$\sC$-sphere} if $t$, considered as a functor $[1]\to\baseTopos$, is a $\sC$-Poincar\'e pair. 
\end{itemize}
If $\Y$ is a $\sC$-sphere, then we have a canonical identification $(\ast,\Y,t)_*(\unit)\simeq\ast_!(\omega_{\ast,\Y,t})\simeq\omega_{\ast,\Y,t}$; in particular $\Y$ is also a $\sC$-homology sphere. We will return to this very elementary example in \cref{example:Thom_isomorphism} when discussing the Thom isomorphism.
\end{example}

\begin{example}
\label{example:two_parallel_arrows}
Let $I$ be the category with two objects $a,b$ and two parallel arrows $a \rightrightarrows b$. Let $s\colon J\coloneqq I_{/b}\to I$ denote the source functor; note that $s$ is a right fibration, and that $J\simeq\twistedArrow([1])$ is a poset with one maximum and two minimal elements. Let $\partial I\coloneqq\{a\}$ and let $\partial J\coloneqq s^{-1}(\partial I)$; note that $s$ restricts to a (trivial) two-fold cover $\partial J\simeq \udl{2}\to\partial I\simeq\udl{1}$.

Now let $X\colon I\to\baseTopos$ be a functor taking values in $\sC$-twisted ambidextrous $\baseTopos$-spaces, and let $\X\coloneqq\int_IX$; similarly, let $(\tilde\X,\partial\tilde\X)\coloneqq(\int_JXs,\int_{\partial J}Xs)$. We claim that $\X$ is a $\sC$-Poincar\'e $\baseTopos$-category if and only if $(\tilde\X,\partial\tilde\X)$ is a $\sC$-Poincar\'e $\baseTopos$-category pair: in this case, in the light of \cref{example:cobordism}, we may think of $X$ as the datum of a Poincar\'e duality $\baseTopos$-space $X(a)$ together with a Poincar\'e self-cobordism $X(b)$. We next prove our claim.

Let $r\colon\tilde\X\to\X$ be the canonical map covering $s\colon J\to I$, let $\partial\X\coloneqq\int_{\partial I}X$, and let $\partial r\colon\partial\tilde\X\to\partial\X$ denote the restriction of $r$; again $\partial r$ is a trivial two-fold cover of $\baseTopos$-spaces. 
We start by observing that the fibre product of the cospan $\partial r_!\xrightarrow{\eta_{\partial r_*}} \partial r_*\partial r^*\partial r_!\overset{\eta_{\partial r_!}}{\longleftarrow}\partial r_*$ in $\func_{\sC,\baseTopos}^L(\presheafTopos(\partial\tilde\X;\sC),\presheafTopos(\partial\X;\sC))$ vanishes: this is immediate after identifying $\presheafTopos(\partial\tilde\X;\sC)\simeq\presheafTopos(\partial\X;\sC)^2$ and identifying both functors $r_!,r_*$ with the direct sum $\baseTopos$-functor. Let $i\colon\partial\X\to\X$ and $\tilde i\colon\partial\tilde\X\to\tilde\X$ denote the inclusions; then the definition of $(\tilde\X,\partial\tilde\X)_*$, together with the last remark, give a diagram of cartesian squares in $\func^L_{\sC,\baseTopos}(\presheafTopos(\tilde\X;\sC),\sC)$:
\[
\begin{tikzcd}[row sep=9pt, column sep=40pt]
(\tilde\X,\partial\tilde\X)_*\ar[r]\ar[d]\ar[dr,phantom,"\lrcorner"very near start]&0\ar[d]\ar[r]\ar[dr,phantom,"\lrcorner"near start]&\partial\X_*\partial r_!\tilde i^*\ar[d,"\eta_{\partial r_*}"]\\
\tilde\X_*\ar[r,"\eta_{\tilde i_*}"]&\partial\tilde\X_*\tilde i^*\ar[r,"\eta_{\partial r_!}"]&\partial\tilde\X_*\partial r^*\partial r_!\tilde i^*.
\end{tikzcd}
\]
On the other hand we have a pushout and pullback square of right fibrations in $\cat_\baseTopos$ as follows on left, giving rise to the middle pullback square in $\func_{\sC,\baseTopos}^L(\presheafTopos(\X;\sC),\sC)$ and, after precomposing by $r_!$ and using the Beck--Chevalley equivalences, to the commutative diagram on right, whose top right horizontal arrow is an equivalence because $r_!$ is fully faithful, see \cref{constr:stable_recollement_for_pairs}: 
\[
\begin{tikzcd}[column sep=9pt]
\partial\tilde\X\ar[d,"\tilde i"']\ar[r,"\partial r"]\ar[dr,phantom,"\lrcorner"very near start,"\ulcorner"very near end]&\partial\X\ar[d,"i"]\\
\tilde\X\ar[r,"r"]&\X;
\end{tikzcd}
\hspace{.5cm}
\begin{tikzcd}
\X_*\ar[d,"\eta_{i_*}"]\ar[r,"\eta_{r_*}"]\ar[dr,phantom,"\lrcorner"very near start]&\tilde\X_*r^*\ar[d,"\eta_{\tilde i_*}"]\\
\partial\X_*i^*\ar[r,"\eta_{\partial r_*}"]&\partial\tilde\X_*\partial r^*i^*;
\end{tikzcd}
\hspace{.5cm}
\begin{tikzcd}[row sep=9pt]
\X_*r_!\ar[d,"\eta_{i_*}"]\ar[r,"\eta_{r_*}"]\ar[dr,phantom,"\lrcorner"very near start]&\tilde\X_*r^*r_!\ar[d,"\eta_{\tilde i_*}"]&\tilde\X_*\ar[l,"\eta_{r_!}"',"\simeq"]\ar[d,"\eta_{\tilde i_*}"]\\
\partial\X_*i^*r_!\ar[r,"\eta_{\partial r_*}"]\ar[d,"\simeq"',"BC"]&\partial\tilde\X_*\partial r^*i^*r_!\ar[dr,"BC","\simeq"']&\partial\X_*\tilde i^*\ar[l,"\eta_{r_!}"']\ar[d,"\eta_{\partial r_!}"]\\
\partial\X_*\partial r_!\tilde i^*\ar[rr,"\eta_{\partial r_*}"]& &\partial\tilde\X_*\partial r^*\partial r_!\tilde i^*.
\end{tikzcd}
\]
It follows that we have an equivalence $\X_*r_!\simeq(\tilde\X,\partial\tilde\X)_*$, and by \cref{prop:classification_of_linear_functors} this corresponds to an equivalence
$\omega_{\tilde\X,\partial\tilde\X}\simeq r^*\omega_\X$. We conclude by observing that the pushout description of $\X$ implies that $\omega_\X$ factors through $\picardSpaceTopos(\sC)$ if and only both $r^*\omega_\X$ and $i^*\omega_\X$ do; but since $\partial r$ admits a section, it is also true that if $r^*\omega_\X$ factors through $\picardSpaceTopos\sC)$, then also $i^*\omega_\X$ automatically does.
\end{example}

\begin{rmk}
    Geometrically, we view \cref{example:two_parallel_arrows} as the situation of a Poincar\'e duality space $\X$ with a map to $S^1$, and a chosen transverse preimage of $1 \in S^1$. 
\end{rmk}

\begin{defn}[Formal dimension]
In the case $\baseTopos=\spc$, we consider the abelian group $\pi_0(\picardSpace(\sC))$. For a $\sC$-Poincar\'e $\infty$-category pair $(\X,\partial\X)$ with $|\X|$ connected we denote by $\mathrm{fdim}(\X,\partial\X)\in\pi_0(\picardSpace(\sC))$ the \textit{formal dimension} of $(\X,\partial\X)$, defined as \textit{the negative of} the component of $\picardSpace(\sC)$ in which $\omega_{\X,\partial\X}$ takes values.
\end{defn}

\begin{example}
\label{ex:two_copies_RP2}
Let $\baseTopos=\spc$ and $\sC=\spectra_\bbQ$; let $\partial\X=\bbR P^2\amalg\bbR P^2$ and let $\X$ be the left fibration over $[1]$ unstraightening the arrow of spaces $\partial\X\to\ast$. Then $\omega_{\X,\partial\X}$ is a functor restricting to the non-trivial coefficient system with fibre $\bbQ[-3]$ on each copy of $\bbR P^2$, and evaluating $\bbQ[-1]$ on $\ast\simeq\interior{\X}$. The first claim follows from \cref{prop:boundary_principle} and a standard computation of the orientation coefficient system of $\bbR P^2$; the second follows from the computation $\omega_{\X,\partial\X}(*)\simeq \X_!(\omega_{\X,\partial\X})\simeq(\X,\partial\X)_*(\unit)$ and a standard computation of the relative cohomology of the pair $(\ast,\partial\X)$ with \textit{constant} coefficients in $\bbQ$. In particular $\omega_{\X,\partial\X}$, even though it is pointwise valued in $\picardSpace(\spectra_\bbQ)$, does not factor through $\picardSpace(\spectra_\bbQ)$, so $(\X,\partial\X)$ is not $\spectra_\bbQ$-Poincar\'e.

We may now observe that if we replace $\sC$ by $\sC'=\module_{\bbQ[x,x^{-1}]}(\spectra)$, where $x$ has degree 2, then we have that $(\X,\partial\X)$ is $\sC'$-twisted ambidextrous and $\omega_{\X,\partial\X}$ takes pointwise value in a single component of $\pi_0(\picardSpace(\sC'))\simeq\bbZ/2$, and yet $(\X,\partial\X)$ is not $\sC'$-Poincar\'e because the connecting map $\delta\colon\Omega\omega_{\partial\X}\to i^*\omega_{\X,\partial\X}$, being basechanged from $\sC$, is equivalent to the zero map.

This example shows that it is hard to generalise \cite[Corollary C]{KleinQinSu} to an arbitrary stable $\sC\in\calg(\presentable^L)$.
\end{example}

\section{Fibred Poincar\'e duality pairs}
\label{sec:fibredpd}
The goal of this section is to prove \cref{thm:fibred_dualising_object_factorisation}. The main consequence of this theorem, \cref{cor:integration_for_unstraightenings}, provides a generalisation of Klein--Qin--Su's fibration theorem on Poincar\'e duality pairs \cite[Thm. G]{KleinQinSu} to all Poincar\'e ads and, in fact, to arbitrary shapes of diagrams of spaces satisfying Poincar\'e duality; also the category of coefficients will be an arbitrary $\sC\in\calg(\presentable^L_{\baseTopos,\stable})$. 

In this section we shall restrict to the setting $\baseTopos=\spc$ (cf. \cref{rmk:reason_for_topos_being_anima_fibredpd} below for an explanation).
In \cref{subsection:fibred_ambidex}, we introduce the concept of \textit{fibred ambidexterity} for functors between categories, generalising Cnossen's notion for maps of groupoids \cite{Cnossen2023}.
Cnossen defines the notion of \textit{$\sC$-twisted ambidextrous map of spaces}, for $\sC \in \calg(\presentable^L)$: writing $\paramFibred{E}$ for the object in the $\infty$-topos $\spc_{/\B}$ corresponding to a map of spaces $p\colon \E \rightarrow \B$, we say that $p$ is a $\sC$-twisted ambidextrous map if $\paramFibred{E}$ is $\pi^*_{\B}\sC$-twisted ambidextrous. In this case, the dualising system $D_\paramFibred{E} \colon \paramFibred{E} \rightarrow \pi_{\B}^* \sC$ corresponds, by adjunction, to a local system $D_p \colon \E \rightarrow \sC$.

The goal of this section is to generalise this line of thought to the case of a pair of functors $(p,q) \colon (\E,\partial\E) \rightarrow \B$ of $\infty$-categories. We will restrict our attention to the case in which $p$ and $q$ are cartesian fibrations, so that we may regard $\E$ and $\partial\E$ as  $\presheaf(\B)$-categories, that we shall denote by $\paramFibred{E}$ and $\partial\paramFibred{E}$ respectively. For $\sC \in \calg(\presentable^L)$, we construct in \cref{sec:cofree} a presentably symmetric monoidal $\B$-category $\cofree_{\B\op}\sC$, generalising the basechange $\pi_{\B}^*\sC$ in the case when $\B$ is a groupoid.
We will then say that $p \colon \E \rightarrow \B$ is $\sC$-twisted ambidextrous if $\paramFibred{E}$ is $\cofree_{\B\op}\sC$-twisted ambidextrous. Applying the Morita theory developed for arbitrary $\infty$-topoi in \cref{sec:morita} to the $\infty$-topos $\presheaf(\B)$, where $\B$ is an arbitrary category (not necessarily a space), we obtain a corresponding classifying system $\omega_{p,q} \in  \func_{\presheaf(\B)}(\paramFibred{E},\cofree_{\B\op}\sC)$. This again admits an interpretation as a $\sC$-valued functor - not defined on $\E$ itself, nor on $\E\op$, but over the ``horizontal opposite'' $\E\hop$ relative to $\B$ (cf \cref{cons:fibrewise_opposite}).

We do not expect that the notion of fibred ambidexterity may be characterised in a ``fibrewise'' fashion in general, but we prove in \cref{thm:fibrewise_characterisation_of_twisted_ambidexterity} in \cref{subsec:fibrewise_characterisation} that under suitable bicartesian assumptions this is indeed the case; this generalises the case of maps between groupoids, which are always bicartesian. With even more hypothesis on the bicartesian fibration, we also prove in \cref{subsec:special_computation_dualising_system} that the classifying system $\omega_{p,q}$ is ``horizontally groupoidal''. Combining all these materials, we then prove
\cref{thm:fibred_dualising_object_factorisation} in \cref{subsec:fibred_poincare_integration}.
We conclude the section with some  examples gathered in \cref{subsec:fibred_examples}, including specialising \cref{thm:fibred_dualising_object_factorisation} to recover a generalisation of Klein--Qin--Su's result.

\begin{rmk}\label{rmk:reason_for_topos_being_anima_fibredpd}
The main reason for the restriction $\baseTopos=\spc$ in this section is that, for an $\infty$-category $\B$, we will also need to consider parametrised homotopy theory over the alternative $\infty$-topos $\presheaf(\B)=\presheaf(\B;\spc)$, and exploit the interaction between the $\infty$-topoi $\presheaf(\B)$ and $\spc$.

In order for the arguments of this section to be valid over an arbitrary topos $\baseTopos$, we would need a well-developed parametrised homotopy theory over a $\baseTopos$-topos, i.e. a topos parametrised over $\baseTopos$; in particular, we would need to generalise the results from \cref{sec:morita} to categories parametrised over the $\baseTopos$-topos $\presheafTopos_\baseTopos(\B)=\presheafTopos_\baseTopos(\B;\animaTopos_\baseTopos)$. Currently, this extension of parametrised homotopy theory seems to be missing from the literature, and this forces our restriction to the case $\baseTopos=\spc$ in this section.
\end{rmk}

\subsection{Fibred ambidexterity and Poincar\'e pairs}\label{subsection:fibred_ambidex}
We fix a presentably symmetric monoidal $\infty$-category $\sC\in\calg(\presentable^L)$ throughout the section.
After addressing the requisite preliminaries, will reach the notion of fibred ambidexterity for a cartesian fibration $p \colon \E \rightarrow \B$ in  \cref{defn:fibred_ambidexterity} which will feed into the main \cref{defn:fibred_poincare_pairs}.
Similar to \cite{Cnossen2023}, the general theme here is to pass from working over $\spc$ to working over $\presheaf(\B)$.

\begin{constr}
\label{cons:fibrewise_opposite}
Let $\B$ be an $\infty$-category. The $\infty$-categories
\begin{equation}
\label{eq:hop_cocart_cart}
\cocartesianCategory_{/\B\op} \simeq \func(\B\op,\cat) \simeq \cartesianCategory_{/\B}
\end{equation}
inherit an involution from the involution $(-)\op \colon \cat \rightarrow \cat$, applied in the middle. We will denote this involution by $(-)\vop$ for both $\cocartesianCategory_{/\B\op}$ and $\cartesianCategory_{/\B}$, and refer to it as the \textit{vertical opposite construction}. For instance, $(-)\vop$ sends the cartesian fibration $p \colon \E \rightarrow \B$ to $p\vop \colon \E\vop \rightarrow \B$, which is intuitively obtained by switching the direction of all morphisms in the fibres of the cartesian fibration, but keeping the direction of the cartesian lifts of morphisms in $\B$. Recall that $p\colon \E \rightarrow \B$ is a cartesian fibration if and only if $p\op \colon \E\op \rightarrow \B\op$ is a cocartesian fibration; in other words, we have an equivalence of $\infty$-categories
\[ (-)\op \colon \cocartesianCategory_{/\B\op} \simeq \cartesianCategory_{/\B} \cocolon (-)\op. \]
We let $(-)\hop \coloneqq (-)\op \circ (-)\vop$ and refer to this functor as the \textit{horizontal opposite construction}. For instance, $(-)\hop$ sends $p\colon \E \rightarrow \B$ to $p\hop \colon \E\hop = (\E\vop)\op \rightarrow \B\op$, intuitively obtained by keeping the orientation of morphisms inside fibres, and switching the orientation of cartesian lifts of morphisms in $\B$, which result in cocartesian lifts of morphisms in $\B\op$. Note that the pair of inverse equivalences $(-)\hop\colon\cocartesianCategory_{/\B\op} \simeq \cartesianCategory_{/\B} \cocolon (-)\hop$ agrees with the pair of composite equivalences in \cref{eq:hop_cocart_cart}.
\end{constr}

\begin{nota}
\label{nota:cartesian_fibration_and_B_cat}
Given a cartesian fibration $(p \colon \E \rightarrow \B)\in\cartesianCategory_{/\B}$,
we denote the associated $\presheaf(\B)$-category by $\paramFibred{E} \in \cat_{\presheaf(\B)}$: this also corresponds to $(p\hop\colon \E\hop\rightarrow \B\op)\in\cocartesianCategory_{/\B\op}$ 
via the equivalences \cref{eq:hop_cocart_cart} and  $ \func(\B\op,\cat) \simeq \cat_{\presheaf(\B)}$.

We further denote by $\vert p\vert^v\colon\vert\E\vert^v\to\B$ and $\vert p\hop\vert^v\colon\vert\E\hop\vert^v\to\B\op$ the right, respectively left, fibration corresponding to $\vert\paramFibred{E}\vert\in\presheaf(\B)$. The corresponding total $\infty$-categories can be expressed as the following localisations, thanks to \cref{lem:localisation_of_fibration}:
\[
\vert\E\hop\vert^v\simeq\scL(\E\hop,\E\hop\times_{\B\op}\B^\simeq);\quad\quad 
\vert\E\vert^v\simeq\scL(\E,\E\times_{\B}\B^\simeq). 
\]
We say that a functor out of $\E$ or out of $\E\hop$ is \textit{fibrewise groupoidal} if it factors through $|\E|^v$, respectively through $|\E\hop|^v$.
We observe that the equivalence $|\paramFibred{E}|\simeq|\paramFibred{E}\op|\in\presheaf(\B)$ corresponds to an equivalence $|\E\hop|^v\simeq|\E\op|^v$ of left fibrations over $\B\op$.
\end{nota}

Given $\sC \in \calg(\presentable^L_\mathrm{st})$, the cofree construction from \cref{sec:cofree} provides a levelwise stable presentably symmetric monoidal $\presheaf(\B)$-category $\cofree_{\B\op}\sC$. The design criterion of the cofree construction $\cofree_{\B\op}\sC$ is that $\Gamma \presheafTopos_{\presheaf(\B)}(\paramFibred{E};\cofree_{\B\op}\sC)$ agrees with the $\infty$-category $\presheaf(\E;\sC)$ of $\sC$-valued presheaves, as explained in the following.
\begin{obs}[Cofree $\presheaf(\B)$-categories as coefficients]
\label{obs:cofree_coeff_cats}
Let $p \colon \E \rightarrow \B$ be a cartesian fibration.
Let $\D$ be any $\infty$-category. Then the definition of $\cofree_{\B\op}$ as a right adjoint of the cocartesian unstraightening $\int$, together with the identification $\E\hop\simeq\int_{\B\op}\paramFibred{E}$, yields
\begin{equation} \label{eqn:global_section_of_functors_into_cofree}
\Gamma \funTopos_{\presheaf(\B)}(\paramFibred{E}, \cofree_{\B\op}\D)
\simeq\func_{\presheaf(\B)}(\paramFibred{E}, \cofree_{\B\op}\D) \simeq \func(\E\hop,\D).
\end{equation}

In other words, $\cofree_{\B\op}\D$-valued systems on $\paramFibred{E} \in \cat_{\presheaf(\B)}$ are the same as $\D$-valued systems on $\E\hop$. On presheaves, by the equivalence $\int_{\B\op}(\paramFibred{E}\op)\simeq\E\op$, we have instead identifications
\[
\Gamma \presheafTopos_{\presheaf(\B)}(\paramFibred{E}; \cofree_{\B\op}\D)
\simeq \func_{\presheaf(\B)}(\paramFibred{E}\op,\cofree_{\B\op}\D) \simeq \func(\E\op,\D)\simeq\presheaf(\E;\D). 
\]
Given a $\presheaf(\B)$-functor $f \colon \paramFibred{E} \to\paramFibred{E}'$, the global section functor $\Gamma$ sends the $\presheaf(\B)$-functor $f^* \colon \presheafTopos_{\presheaf(\B)} (\paramFibred{E}';\cofree_{\B\op}\D) \rightarrow \presheafTopos_{\presheaf(\B)} (\paramFibred{E};\cofree_{\B\op}\D)$ to the functor $(\int(f))^* \colon \presheaf(\E';\D) \rightarrow \presheaf(\E;\D)$; since $\Gamma$ refines to an $(\infty,2)$-functor, we similarly have $\Gamma(f_!)\simeq(\int(f))_!$ and $\Gamma(f_*)\simeq(\int(f))_*$ if these left and right adjoints exist.

Furthermore, for $\sC\in\calg(\presentable^L)$, since $\Gamma\colon\presentable^L_{\presheaf(\B)}\to\presentable^L$ is lax symmetric monoidal, for $\xi \in \Gamma(\cofree_{\B\op}\sC)\simeq\presheaf(\B;\sC)$ the $\cofree_{\B\op}\sC$-linear $\presheaf(\B)$-endofunctor of $\presheafTopos_{\presheaf(\B)}(\paramFibred{E};\cofree_{\B\op}(\sC))$ given by tensoring with $\xi$ is sent along $\Gamma$ to the $\presheaf(\B;\sC)$-linear endofunctor of $\presheaf(\E;\sC)$ given by tensoring with $\xi$.
\end{obs}

Specialising  \cref{defn:ambidexterityforpairs} of twisted ambidexterity to the topos $\baseTopos=\presheaf(\B)$, we now obtain the following:

\begin{defn}
\label{defn:fibred_ambidexterity}
Let $\sC \in \calg(\presentable^L)$, and $p \colon \E \rightarrow \B$ a cartesian fibration. We say that $p$ is a \textit{$\sC$-twisted ambidextrous} cartesian fibration if the $\presheaf(\B)$-category $\paramFibred{E}\in\cat_{\presheaf(\B)}$ corresponding to $p$ is $\cofree_{\B\op}\sC$-twisted ambidextrous. We define similarly the notions of $p$ being \textit{groupoidally $\sC$-twisted ambidextrous} and \textit{$\sC$-Poincar\'e}.
\end{defn}

Let $\sC \in \calg(\presentable^L)$.
Unraveling the notion of a $\sC$-twisted ambidextrous map, we get that a cartesian fibration $p \colon \E \rightarrow \B$ is $\sC$-twisted ambidextrous if the $\presheaf(\B)$-functor
\[ 
\paramFibred{E}_* \colon \presheafTopos_{\presheaf(\B)}(\paramFibred{E};\cofree_{\B\op}\sC) \rightarrow \cofree_{\B\op}\sC
\]
is $\presheaf(\B)$-colimit preserving and its lax $\cofree_{\B\op}\sC$-linear structure as a right adjoint to the $\cofree_{\B\op}\sC$-linear functor $p^*$ is again $\cofree_{\B\op}\sC$-linear. We will unpack these properties more explicitly in \cref{subsec:fibrewise_characterisation} under the further hypothesis that $p$ is a bicartesian fibration. Next, we  generalise the notion of twisted ambidextrous maps to pairs.

\begin{defn}
\label{defn:B_pair}
Consider a morphism in $\cartesianCategory_{/\B}$ given by a commuting triangle
\[
\begin{tikzcd}[row sep=10pt]
\partial_v \E \ar[rr, "j_v"] \ar[dr, "q"'] && \E \ar[dl, "p"] \\
& \B.
\end{tikzcd}
\]
We denote by $\paramFibred{j} \colon \partial \paramFibred{E} \rightarrow \paramFibred{E}$ the corresponding functor of $\presheaf(\B)$-categories. We say that $(p,q)$ is a \textit{$\B$-pair} if 
$\paramFibred{j}$ is a left closed inclusion of $\presheaf(\B)$-categories. We let $(p,q)_*\coloneqq\Gamma((\paramFibred{E},\partial\paramFibred{E})_*)\simeq\fib(p_*\to q_*j_v^*)\colon\presheaf(\E;\sC)\to\sC$.
\end{defn}

\begin{obs}
\label{obs:B-pair_over_[1]}
A triangle as in \cref{defn:B_pair} is a $\B$-pair if and only if for each $b \in \B$, the fibre $\partial_v \E_b \subset \E_b$ is left closed, and if the cartesian fibre transport along each morphism $b \rightarrow c$ in $\B$ makes the diagram
\[
\begin{tikzcd}[row sep=10pt]
\E_b \ar[dr] && \E_{c} \ar[ll] \ar[dl]\\
& {[1]} &
\end{tikzcd}
\]
commute, where the downward pointing maps exhibit the left closedness of $\partial_v \E_b \subset \E_b$ and $\partial_v \E_c \subset \E_c$ as in
\cref{defn:categorypair}. Equivalently, $(p,q)$ is a $\B$-pair if there is a morphism $\E\to\B\times[1]$ in $\cartesianCategory_{/\B}$ exhibiting $\partial_v\E$ as the fibre product of the cospan of $\infty$-categories $\B\xrightarrow{-\times0}\B\times[1]\leftarrow\E$.
\end{obs}

\begin{defn}
A $\B$-pair $(p,q)$ is \textit{$\sC$-twisted ambidextrous} if both $p$ and $q$ are $\sC$-twisted ambidextrous in the sense of \cref{defn:fibred_ambidexterity}, i.e. if $(\paramFibred{E},\partial\paramFibred{E})$ is $\cofree_{\B\op}\sC$-twisted ambidextrous.
\end{defn}

By \cref{prop:classification_of_linear_functors}, if $(p,q)$ is a $\sC$-twisted ambidextrous $\B$-pair, the $\presheaf(\B)$-functor
\[ 
(\paramFibred{E},\partial \paramFibred{E})_*\colon \presheafTopos_{\presheaf(\B)}(\paramFibred{E};\cofree_{\B\op}\sC) \rightarrow \cofree_{\B\op}\sC
\]
corresponds to a classifying system $\omega_{p,q} \in \func_{\presheaf(\B)}(\paramFibred{E},\cofree_{\B\op}\sC)$, whose corresponding object in $\func(\E\hop,\sC)$ under the equivalence \cref{eqn:global_section_of_functors_into_cofree} shall be denoted $\omega_{p,q}$.

\begin{rmk}
\label{rmk:omegapq_and_paramFibredE_*}
We unravel the correspondence of $(\paramFibred{E},\partial \paramFibred{E})_*(-)$ to $\omega_{p,q}$ on global sections using \cref{obs:cofree_coeff_cats}. To this end, write $\frakTwistedArrow(\paramFibred{E})$ for the twisted arrow $\presheaf(\B)$-category of $\paramFibred{E}$. Write $\pi_\B\colon\twistedArrow_{\B}(\E)\to\B$ for the associated cartesian fibration over $\B$. We have source and target functors
$s_\B \colon \twistedArrow_{\B}(\E) \rightarrow \E\vop$ and $t_\B \colon \twistedArrow_\B(\E) \rightarrow \E$, which are maps of cartesian fibrations over $\B$.
We may regard $\omega_{p,q} \in \func(\E\hop;\sC) \simeq \presheaf(\E\vop;\sC)$, so that $s_\B^* (\omega_{p,q}) \in \presheaf(\twistedArrow_{\B}(\E);\sC)$. By \cref{prop:formula_with_twisted_arrow}, the functor $(p,q)_*\simeq\Gamma((\paramFibred{E},\partial \paramFibred{E})_*)\colon\presheaf_{\presheaf(\B)}(\paramFibred{E};\cofree_{\B\op}\sC)\to\Gamma\cofree_{\B\op}\sC$ may be identified with the composite functor
\[
\begin{tikzcd}[column sep=50pt]
\presheaf(\E;\sC) \ar[r,"t_\B^*\simeq","\Gamma(\paramFibred{t}^*)"']& \presheaf(\twistedArrow_{\B}(\E);\sC) \ar[r,"s_\B^* \omega_{p,q}\otimes-\simeq","\Gamma(\paramFibred{s}^*\omega_{p,q}\otimes-)"']& \presheaf(\twistedArrow_\B(\E);\sC) \ar[r,"(\pi_\B)_!\simeq","\Gamma(\frakTwistedArrow(\paramFibred{E})_!)"']&\presheaf(\B;\sC).
\end{tikzcd}
\]
\end{rmk}

\begin{defn}\label{defn:fibred_poincare_pairs}
Let $(p,q)$ be a $\sC$-twisted ambidextrous $\B$-pair. We say that $(p,q)$ is \textit{groupoidally $\sC$-twisted ambidextrous}, respectively \textit{$\sC$-Poincar\'e}, if $(\paramFibred{E},\partial\paramFibred{E})$ is groupoidally $\cofree_{\B\op}\sC$-twisted ambidextrous, respectively $\cofree_{\B\op}\sC$-Poincar\'e.
\end{defn}
\begin{rmk}
\label{rmk:Dpq_and_paramFibredE_*}
Similarly as in \cref{rmk:omegapq_and_paramFibredE_*}, if $(p,q)$ is a groupoidally $\sC$-twisted ambidextrous $\B$-pair, then we have an equivalence $(p,q)_*\simeq p_!(D_{p,q}\otimes-)$ of $\sC$-linear functors $\presheaf(\E;\sC)\to\presheaf(\B;\sC)$.
\end{rmk}

\begin{warning}
\label{warning:fibrewise_groupoidal}
If $(p,q)$ is a groupoidally $\sC$-twisted ambidextrous $\B$-pair, it is in general not true that 
$\omega_{p,q}$ factors through $\vert \E\hop\vert$;
the only conclusion that one can reach in general is that $\omega_{p,q}$ factors through the localisation $|\E\hop|^v$ from \cref{nota:cartesian_fibration_and_B_cat}, i.e. it is fibrewise groupoidal. Equivalently, $(p,q)$ is groupoidally $\sC$-twisted ambidextrous
if and only if the
restriction of $\omega_{p,q}$ on the fibre $\E_b\coloneqq\E\hop \times_{\B\op} \{b\}$ factors through $|\E_b|$ for each $b \in \B\op$. 

In this case, by \cref{prop:Dequivalence},
$\omega_{p,q}\colon|\E\hop|^v\to\sC$ and $D_{p,q}\colon|\E\op|^v\to\sC$ correspond to each other along the equivalence $|\E\hop|^v\simeq|\E\op|^v$ from \cref{nota:cartesian_fibration_and_B_cat}. Similarly, if $(p,q)$ is $\sC$-Poincar\'e, we can only deduce that $\omega_{p,q}$ is fibrewise groupoidal and its image is contained in the \textit{full $\infty$-subcategory} of $\sC$ spanned by objects in $\picardSpace(\sC)$.
\end{warning}

To end this subsection, we show that the classifying system interacts excellently with  basechange in the base category along right fibrations.
\begin{prop}[\'Etale locality for $\B$-pairs]
\label{prop:local_global_principle_for_functor_pairs}
Let $(p,q)$ be a $\B$-pair as in \cref{defn:B_pair}, and let $\pi\colon\B' \rightarrow \B$ be a right fibration. Let $p' \colon \E'\coloneqq\E \times_\B \B' \rightarrow \B'$ and $q'= \partial_v\E'\coloneqq\partial_v\E \times_\B \B' \rightarrow \B'$ denote the pullback cartesian fibrations. Then the following hold.
\begin{enumerate}
\item If $(p,q)$ is a $\sC$-twisted ambidextrous $\B$-pair, then $(p',q')$ is a $\sC$-twisted ambidextrous $\B'$-pair, and $\omega_{p',q'}$ is the restriction of $\omega_{p,q}$ under the canonical functor $\tilde\pi\colon(\E')\hop\to\E\hop$.
\item If $\pi$ is essentially surjective and $(p',q')$ is $\sC$-twisted ambidextrous, then $(p,q)$ is also $\sC$-twisted ambidextrous.
\end{enumerate}
\end{prop}
\begin{proof}
The right fibration $\pi\colon\B' \rightarrow \B$ gives rise to an \'etale geometric morphism $\pi_! \colon \presheaf(\B') \rightleftharpoons \presheaf(\B) \cocolon \pi^*$. Writing $(\paramFibred{E},\partial \paramFibred{E})$ for the $\presheaf(\B)$-category pair corresponding to $(p,q)$, the $\presheaf(\B')$-category pair corresponding to $(p',q')$ is exactly $(\pi^* \paramFibred{E},\pi^* \partial \paramFibred{E})$. Furthermore we may identify $\pi^* \cofree_{\B\op}\sC \simeq \cofree_{(\B')\op}(\sC)$, see \cref{rmk:compatibility_of_cofree_with_etale_basechange}.
Both claims on ambidexterity now follow from \cref{lem:etale_basechange_classifying_systems}. To identify $\omega_{p',q'}$ in (1), consider the following commuting diagram:
\[
\begin{tikzcd}[row sep=10pt]
\func_{\presheaf(\B)}(\paramFibred{E},\cofree_{\B\op}\sC) \ar[r, "\pi^*"] \ar[d, "\simeq"] 
&\func_{\presheaf(\B')}(\pi^* \paramFibred{E}, \pi^* \cofree_{\B\op}\sC) \ar[d, "\simeq"] \ar[r, "\simeq"] 
&\func_{\presheaf(\B')}(\pi^* \paramFibred{E}, \cofree_{(\B')\op} \sC) \ar[dl, "\simeq"]\\
\func(\E\hop,\sC) \ar[r,"\tilde\pi^*"] & \func((\E')\hop,\sC).
\end{tikzcd}
\]
Again by \cref{lem:etale_basechange_classifying_systems}, the top left horizontal functor $\pi^*$ sends
$\omega_{p,q}\mapsto\omega_{p',q'}$. The rest of the diagram then shows that this is compatible with the identification of functors into cofree coefficient $\presheaf(\B)$-categories from \cref{obs:cofree_coeff_cats}, in particular $\tilde\pi^*$ sends $\omega_{p,q}\mapsto\omega_{p',q'}$.
\end{proof}

\subsection{Fibrewise characterisation for bicartesian pairs}\label{subsec:fibrewise_characterisation}

Our next goal is to continue our study of $\sC$-twisted ambidextrous $\B$-pairs, and the behaviour of the classifying systems. More specifically, for a $\B$-pair $(p,q)$, we want to relate $\omega_{p,q}$ with the classifying systems for the pairs $(\E_b,\partial\E_b) \coloneqq (\E \times_\B \{b\}, \partial_v \E \times_\B \{b\})$, a notation that we will keep throughout the rest of the section. 

\begin{defn}
\label{defn:bicartesian_B_pair}
A $\B$-pair $(p,q)$ as in \cref{defn:B_pair}
is \textit{bicartesian} if both $p$ and $q$ are bicartesian fibrations, and the inclusion $\partial_v \E \subset \E$ is a map of bicartesian fibrations.
\end{defn}

Proving the following result is the main goal of this subsection.
\begin{thm}
\label{thm:fibrewise_characterisation_of_twisted_ambidexterity}
Let $(p,q)$ be a bicartesian $\B$-pair.
Then the following are equivalent:
\begin{enumerate}
\item the $\B$-pair is (groupoidally) $\sC$-twisted ambidextrous;
\item for each $b \in \B$, the $\infty$-category pair $(\E_b,\partial\E_b)$ is (groupoidally) $\sC$-twisted ambidextrous.
\end{enumerate}
In the $\sC$-twisted ambidextrous case, we  have equivalences $\omega_{\E_b,\partial\E_b} \simeq \omega_{p,q}\vert_{\E_b}$ for all $b\in \B$; and in the groupoidal $\sC$-twisted ambidextrous case we  also have equivalences $D_{\E_b,\partial\E_b}\simeq D_{p,q}\vert_{\E_b}$.
\end{thm}

\begin{rmk}
Let us comment on why the statement of \cref{thm:fibrewise_characterisation_of_twisted_ambidexterity} makes sense. On the one hand, $D_{p,q}\in\presheaf(\E,\sC)$ can be restricted to a presheaf over $\E_b$ for all $b\in\B$. On the other hand, $\omega_{p,q}$ is an object of $\func(\E\hop,\sC)$; the $\infty$-category $\E_b$, although defined as a fibre over $b$ of $p\colon \E \rightarrow \B$, is also the fibre over $b$ of $\E\hop \rightarrow \B\op$; this allows us to restrict $\omega_{p,q}$ to $\E_b$ as well.  
\end{rmk}

One direction of the proof of \cref{thm:fibrewise_characterisation_of_twisted_ambidexterity} depends on the following preliminary lemmas. 

\begin{lem}
\label{lem:colimits_fibrewise}
Let $(p,q)$ be a bicartesian pair, and assume that for each $b \in \B$, the functor $(\E_b,\partial \E_b)_* \colon \presheaf(\E_b;\sC) \rightarrow \sC$ preserves colimits. Then the  functor
$(\paramFibred{E},\partial \paramFibred{E})_* \colon \presheafTopos_{\presheaf(\B)}(\paramFibred{E};\cofree_{\B\op}\sC) \rightarrow \cofree_{\B\op}\sC$
preserves $\presheaf(\B)$-colimits.
\end{lem}

\begin{proof}
We first check that for $\tau \in \presheaf(\B)$, the functor $(\paramFibred{E},\partial\paramFibred{E})_*(\tau) \colon \presheaf_{\presheaf(\B)}(\paramFibred{E};\cofree_{\B\op}\sC)(\tau) \rightarrow \cofree_{\B\op}\sC(\tau)$ preserves colimits. Note that $\tau$ corresponds to a right fibration $\pi\colon \B' \rightarrow \B$. Define $p' \colon \E' \rightarrow \B'$ and $q'\colon\partial_v\E'\to\B'$ as in \cref{prop:local_global_principle_for_functor_pairs}; we are then left to show that
$(p',q')_* \colon \presheaf(\E';\sC) \rightarrow \presheaf(\B';\sC)$ preserves colimits, and this may be checked by evaluating at each object $b\in \B'$. The map $q'\rightarrow p'$ is a map of cocartesian fibrations, so $(p',q')_*$ is computed fibrewise; hence, it suffices to show that for $b \in \B'$ the functor $(\E'_b,\partial\E'_b)_* \colon \presheaf(\E'_b;\sC) \rightarrow \sC$ preserves colimits, and this follows from the hypothesis that both $\E'_b \simeq \E_{\pi(b)}$ and $\partial\E'_b \simeq \partial\E_{\pi(b)}$ are $\sC$-twisted ambidextrous.

To see that $(\paramFibred{E},\partial \paramFibred{E})_*$ preserves all $\presheaf(\B)$-colimits, it is left to check that for each diagram
\[ 
\begin{tikzcd}[row sep=10pt]
(\E'',\partial_v\E'') \ar[r, "\tilde \pi'"] \ar[d,"{(p'',q'')}"']\ar[dr,phantom,"\lrcorner"very near start]& (\E',\partial_v\E') \ar[r, "\tilde\pi"] \ar[d, "{(p',q')}"'] \ar[dr,phantom,"\lrcorner"very near start]& (\E,\partial_v\E) \ar[d,"{(p,q)}"'] \\
\B'' \ar[r, "\pi'"] & \B' \ar[r, "\pi"] & \B
\end{tikzcd}
\]
in which all horizontal maps are right fibrations, and the top row can be specialised in two ways, the Beck--Chevalley transformation 
$\pi'_!(p'',q'')_* \rightarrow (p',q')_* \tilde \pi'_!$ is an equivalence of functors $\presheaf(\E'';\sC) \rightarrow \presheaf(\B';\sC)$. It is sufficient to show this after postcomposition with restriction along a general map $b \colon * \rightarrow \B'$.
Since $\tilde \pi'$ and $\pi'$ are right fibrations, and $(p',q')$ as well as $(p'',q'')$ are bicartesian pairs, all functors $\pi'_!$, $\tilde \pi'_!$, $ (p',q')_*$ and $(p'',q'')_*$ behave well with basechange along $b$, i.e. \cref{prop:proper_smooth_base_change} applies; therefore the transformation $b^*\pi'_!(p'',q'')_* \rightarrow b^*(p',q')_* \tilde \pi'_!$ is an equivalence as soon as the corresponding Beck--Chevalley map for the back face of the following basechanged cube is an equivalence:
\begin{equation}\label{eqn:big_cube_basechange}
\begin{tikzcd}[row sep=10pt]
& \B''_b\times (\E'_b,\partial\E'_b)\ar[dr, phantom, very near start, "\lrcorner"] \ar[rr] \ar[dl]\ar[dd]&& (\E'_b,\partial\E'_b) \ar[dl]\ar[dd]\\
(\E'',\partial_v\E'') \ar[rr, "\tilde \pi'" xshift = -20, crossing over] \ar[dr, phantom, very near start, "\lrcorner"]\ar[dd, "{(p'',q'')}"']&& (\E',\partial_v\E')\\
& \B''_b \ar[rr] \ar[dl]&&  \ast\ar[dl, "b"]\\
\B'' \ar[rr, "\pi'"] && \B'.\ar[uu,"{(p',q')}"' yshift = 15, leftarrow, crossing over]
\end{tikzcd}
\end{equation}
By assumption $(\E_b',\partial \E_b')_*\simeq(\E_{\pi(b)},\E_{\pi(b)})_*$ commutes with colimits, and in particular with $\B''_b$-shaped colimits; this is exactly the statement that the following square commutes:
\begin{equation}\label{eqn:double_beck_chevalley_commutation}
\begin{tikzcd}
\presheaf(\B''_b \times \E'_b;\sC) \ar[r, "(\B''_b)_!"] \ar[d, "{(\E_b',\partial \E_b')_*}"'] & \presheaf(\E'_b;\sC) \ar[d, "{(\E_b',\partial \E_b')_*}"] \\
\presheaf(\B''_b;\sC) \ar[r, "(\B''_b)_!"] & \sC,
\end{tikzcd}
\end{equation}
as required.
\end{proof}

\begin{lem}
\label{lem:linearity_fibrewise}
Let $(p,q)$ be a bicartesian $\B$-pair, and assume that for each $b \in \B$, the lax $\sC$-linear functor $(\E_b,\partial \E_b)_* \colon \presheaf(\E_b;\sC) \rightarrow \sC$ is $\sC$-linear. Then the lax $\cofree_{\B\op}\sC$-linear functor
$(\paramFibred{E},\partial \paramFibred{E})_* \colon \presheafTopos_{\presheaf(\B)}(\paramFibred{E};\cofree_{\B\op}\sC) \rightarrow \cofree_{\B\op}\sC$ is
$\cofree_{\B\op}\sC$-linear.
\end{lem}

\begin{proof}
To unravel the notion of $\cofree_{\B\op}\sC$-linearity, let $\pi\colon \B' \rightarrow \B$ be a right fibration corresponding to $\tau\in\presheaf(\B)$. Let $p' \colon \E'\to\B'$ and $q'\colon\partial_v\E'\to\B'$ be as in \cref{prop:local_global_principle_for_functor_pairs}. We have to check that for every $\lambda \in \cofree_{\B\op}\sC(\tau) \simeq \presheaf(\B';\sC)$ and $\xi \in \presheaf(\paramFibred{E};\cofree_{\B\op}\sC)(\tau) \simeq \presheaf(\E';\sC)$, the following map is an equivalence:
\begin{equation}
\label{eq:linearity_equation}
\lambda \otimes (p',q')_*(\xi) \rightarrow (p',q')_*((p')^*(\lambda) \otimes \xi).
\end{equation}
It suffices to check this after evaluation at each $b \in \B'$, i.e. after postcomposition with $b^*$ for any $b\colon\ast\to\B'$. 
We have an identification $b^*(p',q')_*(\zeta)\simeq(\E'_b,\partial\E'_b)_*(\zeta\vert_{\E'_b})$, using that $(p',q')$ is a bicartesian $\B'$-pair and relying on basechange along $b$ as in the proof of \cref{lem:colimits_fibrewise}. Postcomposing $b^*$ to \cref{eq:linearity_equation} gives  the lax linearity structure on $(\E_b',\partial\E'_b)_*$. This is because the lax linearity structure maps \cref{eq:linearity_equation} is constructed out of the relevant (co)unit maps of adjunctions, and applying $b^*$ to the adjunction (co)units of $(p')^* \dashv p'_*$ and $(q')^* \dashv q'_*$  gives the adjunction (co)unit of $(\E_b')^* \dashv (\E_b')_*$ and $(\partial\E_b')^* \dashv (\partial\E_b')_*$ respectively. But by assumption, the lax linearity map associated to $(\E_b',\partial\E_b')_* \colon \presheaf(\E_b';\sC) \rightarrow \sC$ is an equivalence. This completes the proof. 
\end{proof}
We are now ready to prove \cref{thm:fibrewise_characterisation_of_twisted_ambidexterity}.
\begin{proof}[Proof of \cref{thm:fibrewise_characterisation_of_twisted_ambidexterity}] We first consider the equivalence between the two statements treating (possibly non-groupoidal) twisted ambidexterity. The implication (2) $\implies$ (1) is immediate by \cref{lem:colimits_fibrewise,lem:linearity_fibrewise}.
For the implication (1) $\implies $ (2), consider the commutative diagram
\begin{equation}
\label{eq:fibrewise_classifying_system_computation_diagram}
\begin{tikzcd}[row sep=10pt]
(\E_b,\partial\E_b) \ar[r, "j"] \ar[d] & (\E_{/b},\partial_v\E_{/b}) \ar[r] \ar[d, "{(p_b,q_b)}"']\ar[dr,phantom,"\lrcorner"very near start] & (\E,\partial_v\E) \ar[d, "{(p,q)}"]\\
\{b\} \ar[r, "k"] & \B_{/b} \ar[r] & \B
\end{tikzcd}
\end{equation}
in which $b\in\B$ and the right square is defined to be cartesian. Since $\B_{/b}\to\B$ is a right fibration, by \cref{prop:local_global_principle_for_functor_pairs} we obtain the equivalence $\omega_{p,q} \vert_{(\E_{/b})\hop} \simeq \omega_{p_{b},q_{b}}$.

Turning to the left half of the diagram \cref{eq:fibrewise_classifying_system_computation_diagram}, since $b$ is a final object in $\B/_{b}$, we have that $j \colon (\E_b,\partial\E_b) \subset (\E_{/b},\partial_v\E_{/b})$ is a pair of full $\infty$-subcategory inclusions. In particular, we get an equivalence $(\E_b,\partial\E_b)_* \simeq (\E_b,\partial\E_b)_* j^* j_!$ of lax $\sC$-linear functors $\presheaf(\E_b;\sC) \rightarrow \sC$.
As $(p_{b},q_b)$ is a bicartesian pair, by \cref{prop:proper_smooth_base_change} we further have $(\E_b,\partial\E_b)_*j^* \simeq k^* (p_{b},q_b)_*$. Since this writes $(\E_b,\partial\E_b)_*\simeq k^*(p_{b},q_b)_*j_!$ as a composite of $\sC$-linear functors, we learn that $(\E_b,\partial\E_b)$ is $\sC$-twisted ambidextrous. This completes the proof of (1) $\implies$ (2).

We next pass to the comparison of classifying systems, assuming (possibly non-groupoidal) twisted ambidexterity. For $b\in\B$ we have an equivalence of $\sC$-linear functors
\[ 
(p_{b},q_b)_*(-) \simeq (\pi_{\B_{/b}})_!(s_{\B_{/b}}^*(\omega_{p,q})\otimes t_{\B_{/b}}^*(-)) \colon \presheaf(\E_{/b};\sC) \rightarrow \presheaf(\B_{/b};\sC),
\]
using the notation from \cref{rmk:omegapq_and_paramFibredE_*}.
The fibre over $b \in \B_{/b}$ of the cartesian fibration $\pi_{\B_{/b}}\colon\twistedArrow_{\B_{/b}}(\E_{/b})\to\B_{/b}$ is given by $\twistedArrow(\E_b)$. Writing $u \colon \twistedArrow(\E_b) \rightarrow \twistedArrow_{\B_{/b}}(\E_{/b})$ for the  fibre inclusion, \cref{prop:proper_smooth_base_change} again allows us to identify
$k^* (\pi_{\B_{/b}})_!\simeq \twistedArrow(\E_b)_! u^*$ as $\sC$-linear functor $ \presheaf(\twistedArrow_{\B_{/b}}(\E_{/b});\sC) \rightarrow \sC$.
Putting all equivalences together, we obtain a $\sC$-linear equivalence
\[
(\E_b,\partial\E_b)_*(-) \simeq k^*(p_{b},q_b)_*j_!\simeq \twistedArrow(\E_b)_!(u^* s_{\B_{/b}}^*(\omega_{p,q})\otimes u^*t^*_{\B_{/b}} j_!(-)).
\]
Finally, by the following commutative squares we have equivalences $u^* t_{\B_/b}^* j_! \simeq t^* j^*j_! \simeq t^*$ and $u^* s_{\B_/b}^* \simeq s^* (j\vop)^*$:
\[
\begin{tikzcd}[row sep=10pt]
\twistedArrow(\E_b) \ar[r, "u"] \ar[d, "t"] & \twistedArrow_{\B_{/b}}(\E_{/b}) \ar[d, "t_{\B_{/b}}"]\\  
\E_b \ar[r, "j"] &\E_{/b}; \end{tikzcd}
\hspace{1cm}
\begin{tikzcd}[row sep=10pt]
\twistedArrow(\E_b) \ar[r, "u"] \ar[d, "s"] & \twistedArrow_{\B_{/b}}(\E_{/b}) \ar[d, "s_{\B_{/b}}"] \\
\E_b\op \ar[r,"j\vop"] & (\E_{/b})\vop
\end{tikzcd}
\]
We thus obtain the equivalence $(\E_b,\partial\E_b)_*(-) \simeq \twistedArrow(\E_b)_!(s^*\omega_{p,q}\vert_{\E_b}\otimes t^*(-))$, showing that $\omega_{p,q}\vert_{\E_b}$ classifies $(\E_b,\partial\E_b)_*$, and hence agrees with $\omega_{\E_b,\partial\E_b}$.

We next turn to the equivalence of (1) and (2) in the setting of groupoidal ambidexterity. Assuming $\sC$-twisted ambidexterity for both $(p,q)$ and all $(\E_b,\partial\E_b)$ for simplicity, we have that $(p,q)$ is \textit{groupoidal} $\sC$-twisted ambidextrous if and only if $\omega_{p,q}$ is fibrewise groupoidal in the sense of \cref{nota:cartesian_fibration_and_B_cat}; however, the identifications $\omega_{p,q}\vert_{\E_b}\simeq\omega_{\E_b,\partial\E_b}$ precisely tell us that $\omega_{p,q}$ is fibrewise groupoidal if and only if $\omega_{\E_b,\partial\E_b}$ is groupoidal for all $b\in\B$, establishing the equivalence between (1) and (2).

Finally, if $\omega_{p,q}$ is fibrewise groupoidal, then $D_{p,q}$ and $\omega_{p,q}$ are equivalent when considered as functors out of $|\E\hop|^v\simeq|\E\op|^v$, see \cref{warning:fibrewise_groupoidal}. In particular their restriction on $|\E_b|\simeq|\E_b\op|$ agrees with $\omega_{\E_b,\partial\E_b}$, corresponding to $D_{\E_b,\partial\E_b}$.
\end{proof}

\subsection{Fibrewise computation for special pairs}
\label{subsec:special_computation_dualising_system}
Recall from \cref{warning:fibrewise_groupoidal} that for a groupoidally $\sC$-twisted ambidextrous $\B$-pair $(p,q)$ the systems $\omega_{p,q}$ and $D_{p,q}$ are in general not groupoidal, i.e. they do not factor through $|\E|$. In this subsection, we prove as \cref{cor:horizontal_groupoidality_special_bicartesian} that this is indeed the case at least when $(p,q)$ is a ``special'' $\B$-pair, in the sense that we now introduce.

\begin{defn}
\label{defn:special_bicartesian}
A bicartesian fibration $p\colon\E\rightarrow\B$ is said to be \textit{special} if $p$ participates in a  pullback square with $\bar\B\in\spc$
\begin{equation}
\label{eq:special_bicartesian}
\begin{tikzcd}[row sep=10pt]
\E \rar["\ell"]\dar["p"'] \ar[dr, phantom, "\lrcorner"very near start]& \bar\E\dar["\bar{p}"]\\
\B\rar[] & \bar\B;
\end{tikzcd}
\end{equation}
or equivalently, if the cartesian fibre transport $\E_c\to\E_b$ associated with any morphism $b\to c$ in $\B$ is an equivalence of $\infty$-categories;
or again equivalently, if the wide subcategories of $\E$ spanned by $p$-cocartesian and by $p$-cartesian morphisms coincide with each other.
\end{defn}

\begin{obs}
\label{obs:special_bicartesian_over_weakly_contractible}
Let $p \colon \E \rightarrow \B$ be a special bicartesian fibration, and assume that $\B$ is weakly contractible; then $p$ is equivalent to a product projection. Indeed $p$ is the pullback of a functor $\bar\E\to|\B|\simeq*$ along $\B\to|\B|$; in particular $\E\simeq\bar\E\times\B$.
\end{obs}
\begin{obs}
\label{obs:special_B_pair}
Note that if $(p,q)$ is a $\B$-pair and $p$ is special bicartesian, then $q$ is also special bicartesian: the commutativity of the triangle in \cref{obs:B-pair_over_[1]} implies that the fibre transport $\E_c\times_{[1]}0\to\E_b\times_{[1]}0$ is an equivalence as well, if $\E_c\to\E_b$ is. More explicitly, this means that there exists a bicartesian $\bar\B$-pair $(\bar p,\bar q)\colon (\bar\E,\partial_v\bar\E)\rightarrow\bar\B$ such that there is a diagram of pullback squares as follows:
\[
\begin{tikzcd}[row sep=10pt]
\partial_v\E \ar[d,hook]\ar[r,"\partial_v\ell"] \ar[dr,phantom,"\lrcorner"very near start]& \partial_v\bar\E\ar[d,hook]\ar[dr,phantom,"\lrcorner"very near start]\ar[r]&\{0\}\ar[d,hook]\\
\E \ar[r,"\ell"]\ar[dr,phantom,"\lrcorner"very near start]\ar[d,"p"]& \bar\E\ar[d,"\bar p"]\ar[r]&{[1]}\\
\B\ar[r]&\bar B.
\end{tikzcd}
\]
\end{obs}
By virtue of \cref{obs:special_B_pair}, we may give the following definition.

\begin{defn}
\label{defn:special_pairs}
A \textit{special $\B$-pair} $(p,q)$ is a bicartesian $\B$-pair $(p,q)$ where $p$ (and hence also $q$) is a special bicartesian fibration.
\end{defn}

\begin{prop}
\label{prop:classifying_systems_of_product_projections}
Let $(\F,\partial\F)$ be a $\sC$-twisted ambidextrous $\infty$-category pair, and let $(p,q)$ denote the $\B$-pair consisting of the two product projections $p \colon \E\coloneqq\F \times \B \rightarrow \B$ and $q\colon\partial_v\E\coloneqq\partial\F\times\B\to\B$. Then $(p,q)$ is $\sC$-twisted ambidextrous and $\omega_{p,q}$ is the image of $\omega_{\F,\partial\F}$ under $\pi_\F^*\colon\func(\F;\sC) \rightarrow\func(\E\hop;\sC)$, where $\pi_\F\colon\E\hop\simeq\F\times\B\op\to\F$ is the projection.

If moreover $(\F,\partial\F)$ is groupoidally $\sC$-twisted ambidextrous, then $(p,q)$ is also groupoidally $\sC$-twisted ambidextrous and $D_{p,q}\simeq(\pi'_\F)^*D_{\F,\partial\F}$, where $(\pi'_\F)\colon\E\simeq\F\times\B\to\F$ is the projection.
\end{prop}
\begin{proof}
Identifying $\cat_{\presheaf(\B)}\simeq\func(\B\op,\cat)$, we have that $(\paramFibred{E},\partial\paramFibred{E})$ corresponds to the constant functor $\B\op\to\cat_{/[1]}$ at $(\F,\partial\F)$. Letting $i\colon\partial\F\to\F$ and $\paramFibred{i}\colon\partial\paramFibred{E}\to\paramFibred{E}$
denote the inclusions, we similarly have that $\paramFibred{i}$, considered as a natural transformation of functors $\B\op\to\cat$, is the constant natural transformation at $i$.

For $b\in \B\op$ we have an symmetric monoidal identification $\cofree_{\B\op}\sC(b)\simeq\func((\B_{/b})\op,\sC)$, naturally in $b\in\B$; this refines $\cofree_{\B\op}\sC$,  considered as a functor $\B\op\to\presentable^L$, to a functor $\B\op\to\calg(\module_\sC(\presentable^L))$.
Moreover, using that $\paramFibred{i}\colon\partial\paramFibred{E}\to\paramFibred{E}$ is a constant $\presheaf(\B)$-functor, we have further symmetric monoidal identifications as follows, naturally in $b\in\B$:
\[
\begin{tikzcd}[row sep=10pt]
\presheafTopos(\paramFibred{E};\cofree_{\B\op}\sC)(b)\ar[r,"\simeq"]\ar[d,"\paramFibred{i}^*(b)"]&\presheaf(\F;\cofree_{\B\op}\sC(b))\ar[r,"\simeq"]\ar[d,"i^*"]&\cofree_{\B\op}\sC(b)\otimes_\sC\presheaf(\F;\sC)\ar[d,"i^*\otimes\id"]\\
\presheafTopos(\partial\paramFibred{E};\cofree_{\B\op}\sC)(b)\ar[r,"\simeq"]&\presheaf(\partial\F;\cofree_{\B\op}\sC(b))\ar[r,"\simeq"]&\cofree_{\B\op}\sC(b)\otimes_\sC\presheaf(\partial\F;\sC).
\end{tikzcd}
\]
In particular also the $\presheaf(B)$-functor $\paramFibred{i}^*\colon\presheafTopos(\paramFibred{E};\cofree_{\B\op}\sC)\to\presheafTopos(\partial\paramFibred{E};\cofree_{\B\op}\sC)$, considered as a functor $\B\op\to\func([1],\presentable^L)$, refines to a functor $\B\op\to\func([1],\calg(\module_\sC(\presentable^L)))$. 
Moreover, the above equivalences, which are natural in $b\in\B$, identify $\presheafTopos(\paramFibred{E};\cofree_{\B\op}\sC)$
with the (pointwise in $\B\op$) tensor product over $\sC$ of the functor $\cofree_{\B\op}\sC\colon\B\op\to\calg(\module_\sC(\presentable^L))$ and the object $\presheaf(\F;\sC)\in\calg(\module_\sC(\presentable^L))$.

The symmetric monoidal $\presheaf(\B)$-functor $\paramFibred{E}^*$ is similarly given by tensoring over $\sC$ the functor $\cofree_{\B\op}\sC\colon\B\op\to\calg(\module_\sC(\presentable^L))$ with the morphism $\F^*\colon\sC\to\presheaf(\F;\sC)$ in $\calg(\module_\sC(\presentable^L))$.
The hypothesis that $\F$ is $\sC$-twisted ambidextrous implies that $\paramFibred{E}^*$ admits a 
right adjoint after forgetting to $\module_{\cofree_{\B\op}\sC}(\presentable^L_{\presheaf(\B)})$,  and these can be computed as well by tensoring over $\sC$ the functor $\cofree_{\B\op}\sC\colon\B\op\to\module_\sC(\presentable^L)$ with the $\sC$-linear morphisms $F_*\colon\presheaf(\F;\sC)\to\sC$. 

Similar remarks hold for expressing $\partial\paramFibred{E}^*$ and $\partial\paramFibred{E}_*$ as tensor products of $\cofree_{\B\op}\sC$ with $\partial\F^*$ and $\partial\F_*$ over $\sC$, respectively. Passing to twisted arrow categories, we have that $\frakTwistedArrow(\paramFibred{E})$ corresponds to the constant functor $\B\op\to\cat$ at $\twistedArrow(\F)$, and $\paramFibred{s},\paramFibred{t}\colon\frakTwistedArrow(\paramFibred{E})\to\paramFibred{E}\op,\paramFibred{E}$ correspond to the constant natural transformations of functors $\B\op\to\cat$ at $s,t$, respectively. Thus we can identify $\presheafTopos(\frakTwistedArrow(\paramFibred{E});\cofree_{\B\op}\sC)\colon\B\op\to\calg(\module_\sC(\presentable^L))$ with the tensor product over $\sC$ of the functor $\cofree_{\B\op}\sC\colon\B\op\to\calg(\module_\sC(\presentable^L))$ and $\presheaf(\twistedArrow(\F);\sC)\in\calg(\module_\sC(\presentable^L))$; and we can similarly as above identify $\paramFibred{s}^*$, $\paramFibred{t}^*$, $\frakTwistedArrow(\paramFibred{E})^*$ and hence also the left adjoint $\frakTwistedArrow(\paramFibred{E})_!$ with the tensor product over $\sC$ of the functor $\cofree_{\B\op}\sC\colon\B\op$ and the morphisms $s^*$, $t^*$, $\twistedArrow(F)^*$ and $\twistedArrow(F)_!$ in $\module_\sC(\presentable^L)$, respectively.

The equivalence of $\sC$-linear functors $(\F_,\partial\F)_*\simeq\twistedArrow(\F)_!(s^*\omega_\F\otimes t^*(-))$ gives therefore rise to an identification of $\cofree_{\B\op}\sC$-linear functors $(\paramFibred{E},\partial\paramFibred{E})_*\simeq\frakTwistedArrow(\paramFibred{E})_!(\paramFibred{s}^*(\pi_\F^*\omega_{\F,\partial\F})\otimes\paramFibred{t}^*(-))$: we simply tensor the functor $\cofree_{\B\op}\sC$ over $\sC$ with the former equivalence. Here we regard 
\[
\pi_\F^*\omega_{\F,\partial\F}\in\func(\E\hop,\sC)\simeq\func(\B\op,\sC)\otimes_\sC\func(\F,\sC)
\simeq\Gamma\cofree_{\B\op}\sC\otimes_\sC\func(\F,\sC)
\]
as the object corresponding to $\unit\otimes\omega_{\F,\partial\F}$. This simultaneously shows that $(\paramFibred{E},\partial\paramFibred{E})$ is $\cofree_{\B\op}\sC$-twisted ambidextrous, and it identifies $\omega_{p,q}$ with $\pi_\F^*\omega_{\F,\partial\F}$.

If $(\F,\partial\F)$ is groupoidally $\sC$-twisted ambidextrous, we argue in a similar manner: the equivalence $(\F,\partial\F)_*\simeq\F_!(D_{\F,\partial\F}\otimes-)$, when tensored over $\sC$ with the functor $\cofree_{\B\op}\sC$, yields an equivalence $(\paramFibred{E},\partial\paramFibred{E})_*\simeq\paramFibred{E}_!((\pi'_\F)^*D_{\F,\partial\F}\otimes-)$, where $(\pi'_\F)^*D_{\F,\partial\F}\in\presheaf(\E;\sC)$ is evidently fibrewise groupoidal and corresponds to the object $\unit\otimes D_{\F,\partial\F}\in\Gamma\cofree_{\B\op}\sC\otimes_\sC\presheaf(\F;\sC)$. In particular $(\paramFibred{E},\partial\paramFibred{E})$ is groupoidally $\sC$-twisted ambidextrous.
\end{proof}

\begin{cor}
\label{cor:horizontal_groupoidality_special_bicartesian}
Let $(p,q)$ be a special $\sC$-twisted ambidextrous $\B$-pair. Then $\omega_{p,q}\in\func(\E\hop,\sC)$ is \textnormal{horizontally groupoidal}, i.e. it sends to equivalences all cartesian (equivalently, all cocartesian) morphisms of the special bicartesian fibration $\E\hop\to\B\op$. If moreover $(p,q)$ is groupoidally $\sC$-twisted ambidextrous, then $D_{p,q}$ is groupoidal.
\end{cor}
\begin{proof}
For $b\in\B$ we may consider the right fibration $\pi\colon\B'\coloneqq\B_{/b}\to\B$, and define the $\B'$-pair $(p',q')$ by pullback as in \cref{prop:local_global_principle_for_functor_pairs}; the same proposition then ensures that $\omega_{p',q'}\simeq\pi^*\omega_{p,q}$. By \cref{obs:special_bicartesian_over_weakly_contractible} we are able to apply \cref{prop:classifying_systems_of_product_projections} and conclude in particular that $\omega_{p',q'}$ is horizontally groupoidal.
Since every (co)cartesian morphism in $\E\hop$ with target in $\E_b$ is the image of a (co)cartesian morphism in $(\E')\hop$, letting $b$ vary we obtain the first claim.

For the second claim, we similarly argue that $D_{p,q}\colon\E\op\to\sC$ is horizontally groupoidal; by \cref{thm:fibrewise_characterisation_of_twisted_ambidexterity} we have that $D_{p,q}$ is also fibrewise groupoidal, and the two conditions together imply that $D_{p,q}$ is groupoidal.
\end{proof}

\subsection{The fibration theorem}\label{subsec:fibred_poincare_integration} 
Suppose we have a $\B$-pair $(p,q)$ where $\B$ itself participates in an $\infty$-category pair $(\B,\partial\B)$. Out of this, we may construct an $\infty$-category pair $(\E,\partial\E)$ that combines the fibrewise boundaries coming from $(p,q)$ and the boundary $\partial\B$ of the base. The main goal of this subsection is to prove our main result in the theory of fibred ambidexterity, \cref{thm:fibred_dualising_object_factorisation}, which relates Poincar\'e duality on $(\E,\partial\E)$ to that on the fibres and the base.
To be able to give the precise statement, we will need a couple of definitions.

\begin{defn}
\label{defn:fibred_triad}
A \textit{fibred triad} is the datum of an $\infty$-category pair $(\B,\partial\B)$ and of a bicartesian
$\B$-pair $(p,q)$ as in \cref{defn:bicartesian_B_pair}.
Given a fibred triad $(\B,\partial\B,p,q)$
we usually write $\partial_h \E\coloneqq p^{-1}(\partial \B)$, $\partial\partial \E\coloneqq\partial_v \E \times_{\E} \partial_h \E\simeq q^{-1}(\partial\B)$, which also coincides with the intersection $\partial_v\E\cap\partial_h\E$ of left closed $\infty$-subcategories of $\E$; and we set $\partial \E \coloneqq \partial_v \E \cup_{\partial \partial \E} \partial_h \E$. 

We say that $(\B,\partial\B,p,q)$ is \textit{$\sC$-twisted ambidextrous} if $(\B,\partial \B)$ is a $\sC$-twisted ambidextrous $\infty$-category pair and $(p,q)$ is a $\sC$-twisted ambidextrous $\B$-pair. We say that $(\B,\partial\B,p,q)$ is \textit{tame} if  $(\B,\partial\B)$ is tame in the sense of \cref{defn:tame_pair}.
\end{defn}

\begin{rmk}
    By \cref{thm:fibrewise_characterisation_of_twisted_ambidexterity},  $\sC$-twisted ambidexterity of $(p,q)$ in the definition of $\sC$-twisted ambidexterity of $(\B,\partial\B,p,q)$ above may also be rephrased as saying that each fibre of $(p,q)$ is $\sC$-twisted ambidextrous.
\end{rmk}

With the language set up, we can now state the theorem.
\begin{thm}
\label{thm:fibred_dualising_object_factorisation}
Let $(\B,\partial\B,p,q)$ be a $\sC$-twisted ambidextrous tame fibred triad as in \cref{defn:fibred_triad} and assume  that $p$ is a special bicartesian fibration (i.e., $(p,q)$ is a special $\B$-pair as in \cref{defn:special_pairs}).
Consider the following statements:
\begin{enumerate}
\item $(\E,\partial\E)$ is a $\sC$-Poincar\'e $\infty$-category pair;
\item $(\B,\partial\B)$ and all the fibres of $(p,q)\colon (\E,\partial_v\E)\rightarrow\B$ are $\sC$-Poincar\'e $\infty$-category pairs.
\end{enumerate} 
Then (2) implies (1), and under these  conditions, we  have an equivalence $D_{\E,\partial\E}\simeq D_{p,q} \otimes p^* D_{\B,\partial\B}$. If we additionally assume that $p$ is essentially surjective, then we also have that (1) implies (2).
\end{thm}

The rest of this subsection is devoted to working towards the proof of this theorem, which will be given at the end of the subsection. Before proceeding, we first explain  how to reduce to proving the statement under the additional assumption $\partial\B=\emptyset$.
\begin{lem}[Doubling principle for fibred triads]
\label{lem:doubling_trick}
Recall \cref{ex:doubling_principle}.
Let $(\B,\partial\B,p,q)$
be a tame fibred triad, and consider the following three properties:
\begin{enumerate}
\item $(\E,\partial\E)$ is a $\sC$-Poincar\'e $\infty$-category pair;
\item $(p,q)$ is a $\sC$-Poincar\'e $\B$-pair;
\item $(\B,\partial \B)$ is a $\sC$-Poincar\'e $\infty$-category pair.
\end{enumerate}
Then there is a category $\B'$ and a bicartesian $\B'$-pair $(p',q')$ such that each of the properties above is equivalent to the following respective analogue:
\begin{itemize}
\item[(1)$^\prime$] $(\E',\partial_v\E')$ is a $\sC$-Poincar\'e $\infty$-category;
\item[(2)$^\prime$] $(p',q')$ is a $\sC$-Poincar\'e $\B'$-pair;
\item[(3)$^\prime$] $\B'$ is a $\sC$-Poincar\'e $\infty$-category.
\end{itemize}
Moreover, if $p$ is a special bicartesian fibration, then so is $p'$. A similar statement holds if we replace all instances of ``$\sC$-Poincar\'e'' above with ``groupoidally $\sC$-twisted ambidextrous'': for $i=1,2,3$, property ($i$) is equivalent to property ($i$)$^\prime$.
\end{lem}
\begin{proof}
We  construct a replacement fibred triad $(\B',\partial\B',p',q')$ satisfying $\partial\B'\simeq\emptyset$ but possibly $\partial_v\E'\neq\emptyset$. Recall by definition of fibred triads that $p$ and $q$ are both bicartesian fibrations.
We double the base $\infty$-category $\B$ along $\partial \B$, and accordingly $\E$ along $\partial_h \E$; letting $\B'\coloneqq\B\amalg_{\partial\B}\B$ and $\pi\colon \B'\to\B$ be the fold functor, we thus replace the fibred triad $(\B,\partial\B,p,q)$ by the fibred triad $(\B',\emptyset,p',q')$, where $(p',q')\coloneqq(\pi^*p,\pi^*q)$. Note that $\pi\colon \B'\to\B$ is a right fibration corresponding to the functor $\B\op\to\spc$ attaining value $*\amalg*$ on $\B\setminus\partial\B$ and $*$ on $\partial\B$: in particular $\pi$ is an effective epimorphism. Observe also that the $\B'$-pair $(p',q')$ is again tame: indeed the final inclusion $\presheaf(\B)$-functor $\interior{\paramFibred{E}}\subseteq\paramFibred{E}$ is sent along the \'etale basechange $\pi^*$ again to a final inclusion $\presheaf(\B')$-functor $\interior{\paramFibred{E}'}\subseteq\paramFibred{E}'$.

Now note that the $\infty$-category pairs $(\E,\partial_h\E)$ and $(\partial_v\E,\partial\partial\E)$ are tame by \cref{obs:leftclosed}\ref{rmk:left_closed_fibration}
and the hypothesis that $p$ and $q$ are also cocartesian. By \cref{ex:gluingXY}, we then have that $(\E,\partial\E)$ is $\sC$-Poincar\'e if and only if $(\E',\partial_v\E')$ is $\sC$-Poincar\'e. Moreover, by \cref{prop:local_global_principle_for_functor_pairs}, the $\B$-pair $(p,q)$ is $\sC$-Poincar\'e if and only if the $\B'$-pair $(p',q')$ is $\sC$-Poincar\'e. By \cref{ex:gluingXY}, the $\infty$-category pair $(\B,\partial \B)$ is $\sC$-Poincar\'e if and only if $\B'$ is $\sC$-Poincar\'e.
Hence, using these observations, we can reduce (1)-(3) to the analogous properties of the fibred triad $(\B',\emptyset,p',q')$. It is also clear by construction that if $p$ is special, then so is $p'$.
\end{proof}

We next record the following composition relationship which will serve as the basic connection between fibred ambidexterity and ordinary ambidexterity.
\begin{obs}\label{obs:compositionality_formula}
Let $(\B,\partial\B,p,q)$ be a fibred triad. Then there is a natural equivalence
\[(\E,\partial\E)_*\simeq (\B,\partial\B)_*(p,q)_*\] of lax $\sC$-linear functors $\presheaf(\E;\sC)\rightarrow \sC$. To wit, simply consider the chain of equivalences
\[
\begin{split}
(\E,\partial\E)_*&\simeq\fib(\fib(\E_*\to\partial_v\E_*j_v^*)\to\fib(\partial_h\E_*j_h^*\to\partial\partial\E_*j^*))\\
&\simeq\fib(\B_*\fib(p_*\to q_*j_v^*)\to \partial\B_*\fib(\partial_h p_*j_h^*\to\partial_hq_*j^*))\\
&\simeq\fib(\B_*\fib(p_*\to q_*j_v^*)\to \partial\B_*i^*\fib(p_*\to q_*j_v^*))\simeq(\B,\partial\B)_*(p,q)_*,
\end{split}
\]
where the first equivalence comes from the identification $\partial\E\simeq\partial_v\E\amalg_{\partial\partial\E}\partial_h\E$; the second uses that $\B_*$ and $\partial\B_*$ preserve fibres; the third uses that $p$ and $q$ are \textit{cocartesian} and relies on \cref{prop:proper_smooth_base_change}; and the fourth is immediate by definition.
\end{obs}

As we see next, under suitable hypotheses the composition formula from \cref{obs:compositionality_formula} relates the dualising systems of $(\B,\partial\B)$ and $(p,q)$ to the relative cohomology functor $(\E,\partial\E)_*$.

\begin{lem}
\label{lem:basic_fibred_dualising_module_factorisation}
Let $(\B,\partial\B,p,q)$ be a fibred triad as in \cref{defn:fibred_triad}. Assume that $(\B,\partial \B)$ is a groupoidally $\sC$-twisted ambidextrous $\infty$-category pair and that $(p,q)$ is a groupoidally $\sC$-twisted ambidextrous $\B$-pair.
Then there is an equivalence $(\E,\partial \E)_*(-)\simeq \E_!(-\otimes D_{p,q} \otimes p^* D_{\B,\partial\B})$  of $\sC$-linear functors $\presheaf(\E;\sC)\to\sC$; in particular $(\E,\partial\E)$ is groupoidally $\sC$-twisted ambidextrous.
\end{lem}
\begin{proof}
Let $i\colon\partial\B\hookrightarrow\B$ denote the inclusion, and let $j$, $j_h$ and $j_v$ denote the inclusions of $\partial\partial\E$, $\partial_h\E$ and $\partial_v\E$ in $\E$, respectively. Let also $\partial_hp$ and $\partial_hq$ denote the restrictions of $p$ and $q$ over $\partial\B$. 
Then we obtain equivalences of $\sC$-linear functors $\presheaf(\E;\sC)\rightarrow\sC$
\[
\begin{split}
(\E,\partial\E)_*&\simeq(\B,\partial\B)_*(p,q)_*\simeq\B_!(D_{\B,\partial\B}\otimes p_!(D_{p,q}\otimes-)) \simeq \B_!(p_!(p^*D_{\B,\partial\B}\otimes D_{p,q} \otimes -))\\
&\simeq \E_!(p^* D_{\B,\partial\B} \otimes D_{p.q} \otimes -).
\end{split}
\]
The first equivalence uses \cref{obs:compositionality_formula}. The second uses the groupoidal ambidexterity of $(\B,\partial\B)$ and $(p,q)$, and relies on \cref{rmk:Dpq_and_paramFibredE_*}. The third uses \cref{lem:projection_formula_for_presheaves}, which applies because $p$ is a cartesian fibration. The fourth is also immediate.
\end{proof}
\cref{lem:basic_fibred_dualising_module_factorisation} will handle the implication (2)$\implies$(1) in \cref{thm:fibred_dualising_object_factorisation}. The rest of this preparatory part leading up to the proof of the theorem concerns the other  implication. 

\vspace{1mm}

In broad strokes, the  strategy is first to prove that $(p,q)$ is $\sC$-Poincar\'e and then use this as an ingredient to prove also that $(\B,\partial\B)$ is $\sC$-Poincar\'e. In general, given a bicartesian $\B$-pair $(p,q)$, it is not clear how to obtain Poincar\'e duality of the fibres $(\E_b,\partial\E_b)$ just from knowing Poincar\'e duality on the total $\infty$-category pair $(\E,\partial_v\E)$. The next results will address this point under the assumption that $(p,q)$ is special.

\begin{lem}
\label{lem:horizontal_localisation}
Suppose we have a pullback square as in \cref{eq:special_bicartesian}, with bottom horizontal functor given by the localisation $L_\B\colon\B\to|\B|$.
Then the top horizontal functor $\ell$ is a localisation.
\end{lem}
\begin{proof}
Firstly, note that $L_\B$ has weakly contractible fibres: for a point $b\in |\B|$, we have that $|\{b\}\times_{|\B|}\B|\simeq \ast\times_{|\B|}|\B|\simeq \ast$, using that geometric realisations preserve pullbacks over groupoids. Thus $\ell$, which is a pullback of $L_\B$, is a bicartesian fibration with weakly contractible fibres; by \cref{lem:fibration_weakly_contractible_fibres} we conclude that $\ell$ is a localisation.
\end{proof}

\begin{cor}\label{cor:poincare_domination_of_localisations}
Suppose $(p,q)$ is a special $\B$-pair as in \cref{defn:special_pairs}. If $(\E,\partial_v\E)$ is $\sC$-Poincar\'e, then $(\bar\E,\partial_v\bar\E)$ is $\sC$-Poincar\'e and $D_{\E,\partial_v\E}\simeq\ell^*D_{\bar\E,\partial_v\bar\E}$.
\end{cor}
\begin{proof}
This is an immediate consequence of \cref{lem:horizontal_localisation} and \cref{prop:localisation_principle}.
\end{proof}

Next, in \cref{lem:special_case_of_converse_for_groupoid_base}, we aim at showing that for a (special) $\bar\B$-pair $(p,q)$ with $\bar\B$ a space, Poincar\'e duality of the total $\infty$-category pair \textit{does} guarantee Poincar\'e duality of the fibres. This result will in turn depend on the following identifications.
\begin{lem}\label{lem:twisted_arrow_over_groupoids}
Let $\bar p\colon \bar \E\rightarrow\bar\B$ be a functor, and assume $\bar\B\in\spc$. Let $\bar{\paramFibred{E}}\in\cat_{\presheaf(\bar\B)}$ correspond to $\bar p$, regarded as a cartesian fibration.
Then there is an equivalence $\twistedArrow(\bar\E)\simeq \twistedArrow_{\bar\B}(\bar\E)$, using the notation from \cref{rmk:omegapq_and_paramFibredE_*}, under which the functor $(s,t)\colon\twistedArrow(\bar\E)\to\bar\E\op\times\bar\E$ corresponds to the functor $(s_{\bar\B},t_{\bar\B})\colon\twistedArrow_{\bar\B}(\bar\E)\to\bar\E\vop\times\bar\E\simeq\bar\E\op\times\bar\E$; the equivalence is natural in $\bar{\paramFibred{E}}\in\cat_{\presheaf(\bar\B)}$, and restricts to the identification $\twistedArrow(\bar\B)\simeq\bar\B\simeq\twistedArrow_{\bar\B}(\bar\B)$ if we pick the terminal $\presheaf(\bar\B)$-category.
\end{lem}
\begin{proof}
We may regard $\bar{\paramFibred{E}}\in\cat_{\presheaf(\bar\B)}$ as an object in $\func(\bar\B\op,\cat)$, which we consider as the full subcategory of $\func(\simplex,\func(\bar\B\op,\spc))$ spanned by complete Segal objects.
The statement is then a consequence of the following diagram, whose right triangle commutes:
\[
\begin{tikzcd}[column sep=50pt, row sep=10pt]
\simplex\op\ar[r,"{(-)\ast(-)\op}"'{name=MD},""{name=MU}]\ar[r,"\id",bend left=40,""'{name=U}]\ar[r,"(-)\op"',""{name=D},bend right=70] & \simplex\op \ar[r,"\bar{\paramFibred{E}}"] \ar[dr, "\bar\E"]& \presheaf(\bar\B) \ar[d,"\int"]\\
& & \spc.
\ar[from=MU, to=U, Rightarrow]
\ar[from=MD, to=D, Rightarrow]
\end{tikzcd}
\]
If $\bar{\paramFibred{E}}$ is constant with value the terminal presheaf $*\in\presheaf(\bar\B)$, then postcomposition by $\bar{\paramFibred{E}}$ makes the two natural transformations into natural equivalences, and all of $\bar{\paramFibred{E}}$, $\bar{\paramFibred{E}}\circ((-)*(-)\op)$ and $\bar{\paramFibred{E}}\circ(-)\op$ are identified with the constant functor $\simplex\op\to\presheaf(\bar\B)$ at $*$.
\end{proof}

\begin{lem}\label{lem:special_case_of_converse_for_groupoid_base}
Let $(\bar{p},\bar{q})\colon(\bar\E,\partial_v\bar\E)\rightarrow\bar\B$ be a $\bar\B$-pair with $\bar\B\in\spc$. Assume that $(\bar\E,\partial_v\bar\E)$ is a $\sC$-Poincar\'e $\infty$-category pair, $\bar\B$ is a $\sC$-twisted ambidextrous space, and $(\bar p,\bar q)$ is a $\sC$-twisted ambidextrous $\bar\B$-pair. Then all fibres of $(\bar{p},\bar{q})$ are $\sC$-Poincar\'e pairs.
\end{lem}
\begin{proof}
We keep using the notation from \cref{rmk:omegapq_and_paramFibredE_*}, and compute:
\[        
\begin{split}
(\bar\E,\partial_v\bar\E)_*(-) &\simeq \bar\B_*(\bar{p},\bar{q})_*(-)\simeq \bar\B_!\big(D_{\bar\B}\otimes (\pi_{\bar\B})_!(s_{\bar\B}^*\omega_{\bar{p},\bar{q}}\otimes t_{\bar\B}^*(-))\big)\\
&\simeq \bar\B_!(\pi_{\bar\B})_!\big(\pi_{\bar\B}^*D_{\bar\B}\otimes s_{\bar\B}^*\omega_{\bar{p},\bar{q}}\otimes t_{\bar\B}^*(-)\big)\\
& \simeq \twistedArrow(\bar\E)_!\big(s^*\bar{p}^*D_{\bar\B}\otimes s^*\omega_{\bar{p},\bar{q}}\otimes t^*(-)\big)\simeq \twistedArrow(\bar\E)_!\big(s^*(\bar{p}^*D_{\bar\B}\otimes\omega_{\bar{p},\bar{q}})\otimes t^*(-)\big).
\end{split}
\]
Here the first equivalence is immediate. The second uses \cref{rmk:omegapq_and_paramFibredE_*} and the twisted ambidexterity of $\bar\B$. The third uses \cref{lem:projection_formula_for_presheaves}, which applies because $\pi_{\bar\B}$ is a cartesian fibration. The fourth relies on \cref{lem:twisted_arrow_over_groupoids}. The last uses symmetric monoidality of $s^*$. By \cref{prop:classification_of_linear_functors,prop:formula_with_twisted_arrow}, we obtain an equivalence $\omega_{\bar\E,\partial_v\bar\E}\simeq\bar{p}^*D_{\bar\B}\otimes\omega_{\bar{p},\bar{q}}$ of functors $\bar\E\rightarrow \sC$. Since $(\bar\E,\partial_v\bar\E)$ is  $\sC$-Poincar\'e, the left hand side is invertible, and thus so is the factor $\omega_{\bar{p},\bar{q}}$ in the right hand side. 
By \cref{thm:fibrewise_characterisation_of_twisted_ambidexterity}, for $b\in\B$ we have that $\omega_{\bar\E_b,\partial_v\bar\E_b}\simeq\omega_{\bar{p},\bar{q}}|_{\bar\E_b}$, whence $\omega_{\bar\E_b,\partial_v\bar\E_b}$ is also invertible. 
\end{proof}

The following is an immediate corollary of \cref{thm:fibrewise_characterisation_of_twisted_ambidexterity}.
\begin{cor}
\label{cor:special_localisation_again_groupoidally_ambidextrous}
Let $(p,q)$ be a special and groupoidally $\sC$-twisted ambidextrous $\B$-pair, and let $(\bar p,\bar q)$ be the essentially unique $\bar{\B}$-pair from which $(p,q)$ is pulled back, where we set $\bar{\B}\coloneqq|\B|$ and we use the notation from \cref{defn:special_bicartesian}.
Then also $(\bar p,\bar q)$ is a groupoidally $\sC$-twisted ambidextrous $\bar\B$-pair.
\end{cor}
\begin{proof}
Since $(p,q)$ is groupoidally $\sC$-twisted ambidextrous, all fibres $(\E_b,\partial\E_b)$ for $b \in \B$ are groupoidally $\sC$-twisted ambidextrous by \cref{thm:fibrewise_characterisation_of_twisted_ambidexterity}. Since $(\E_b,\partial\E_b)\simeq(\bar\E_{L_\B(b)},\partial\bar\E_{L_\B(b)})$ and since $L_\B\colon\B\to\bar\B$ is essentially surjective, we have that all fibres $(\bar\E_b,\partial\bar\E_b)$ for $b \in \bar\B$ are groupoidally $\sC$-twisted ambidextrous, whence again by \cref{thm:fibrewise_characterisation_of_twisted_ambidexterity} we conclude that $(\bar p,\bar q)$ is groupoidally $\sC$-twisted ambidextrous.
\end{proof}

The last ingredient we will need for the proof of \cref{thm:fibred_dualising_object_factorisation} is the following technical observation, which will be used to show that $(\B,\partial\B)$ is $\sC$-Poincar\'e given that $(\E,\partial\E)$ and $(p,q)$ are $\sC$-Poincar\'e.
\begin{lem}\label{lem:twisted_arrow_cofinality}
Let $p\colon \E\rightarrow\B$ be a cartesian fibration, and consider the pullback
\begin{equation}\label{eqn:twisted_arrow_projection_pullback}
\begin{tikzcd}[row sep=10pt]
\mathcal{P} \rar["\tilde{t}"]\dar["\tilde{p}"']\ar[dr,phantom,"\lrcorner"very near start] & \E\dar["p"]\\
\twistedArrow(\B)\rar["t"] & \B.
\end{tikzcd}
\end{equation}
Then the functor $\varphi\colon \twistedArrow(\E)\rightarrow\mathcal{P}$ induced by $t\colon \twistedArrow(\E)\rightarrow\E$ and $\twistedArrow( p)\colon \twistedArrow(\E)\rightarrow\twistedArrow(\B)$ is initial. In particular, the natural transformation $\twistedArrow(p)_!\varphi^*\rightarrow \tilde{p}_!$ of functors $\presheaf(\mathcal{P};\sC)\rightarrow\presheaf(\twistedArrow(\B);\sC)$ is an equivalence.
\end{lem}
\begin{proof}
By Quillen's Theorem A, it suffices to prove that $\varphi$ is a cocartesian fibration with weakly contractible fibres. To prove that $\varphi$ is a cocartesian fibration, we consider the following commutative diagram, in which $\mathcal{P}'$ is defined as the pullback $\mathcal{P}\times_{\B\op\times\E}\E\op\times\E$:
\[
\begin{tikzcd}[column sep=50pt]
\twistedArrow(\E)\ar[d]\ar[dd,bend right=30, "\varphi"']\ar[dr,"\mathrm{left}"]\\
\mathcal{P}'\ar[d,"\cocartesianCategory"]\ar[r,"\mathrm{left}"']\ar[dr,phantom,"\lrcorner"very near start]&\E\op\times\E\ar[d,"p\op\times\id"',"\cocartesianCategory"]\\
\mathcal{P}\ar[d,"\tilde p"]\ar[r,"{(s\circ\tilde p,\tilde t)}","\mathrm{left}"']\ar[dr,phantom,"\lrcorner"very near start]&\B\op\times\E\ar[d,"\id\times p"]\ar[r]\ar[dr,phantom,"\lrcorner"very near start]&\E\ar[d,"p"]\\
\twistedArrow(B)\ar[r,"{(s,t)}","\mathrm{left}"']&\B\op\times\B\ar[r]&\B.
\end{tikzcd}
\]
The functor $\twistedArrow(\E)\to\mathcal{P}'$, being a map between left fibrations over $\E\op\times\E$, is itself a left fibration; it follows that $\varphi$ is the composite of a left fibration and a cocartesian fibration.

To prove that $\varphi$ has weakly contractible fibres, we observe that for an object $x\in\mathcal{P}$ corresponding to the datum of an object $e\in\E$ and a morphism $f\colon b'\to b\coloneqq p(e)$, we may identify $\varphi^{-1}(x)$ with the opposite of $(\E_{b'})_{/e}\coloneqq\E_{b'}\times_\E\E_{/e}$. If $e'\to e$ is a cartesian lift at $e$ of $f$, then $(\E_{b'})_{/e}\simeq(\E_{b'})_{/e'}$, which is weakly contractible because it has a terminal object. Hence also $\varphi^{-1}(x)\simeq((\E_{b'})_{/e'})\op$ is weakly contractible.
\end{proof}

With all auxiliary results set in place, we may now present the proof of \cref{thm:fibred_dualising_object_factorisation}.
\begin{proof}[Proof of \cref{thm:fibred_dualising_object_factorisation}.]
First, we prove that (2)$\implies$(1). By \cref{lem:basic_fibred_dualising_module_factorisation}, which is applicable thanks to \cref{thm:fibrewise_characterisation_of_twisted_ambidexterity}, we have an equivalence $(\E,\partial \E)_*(-)\simeq \E_!(p^* D_{\B,\partial\B} \otimes D_{p.q} \otimes -)$ of $\sC$-linear functors; by \cref{prop:Dequivalence} it then suffices to prove that $p^* D_{\B,\partial\B} \otimes D_{p.q}\in\presheaf(\E;\sC)$ is invertible. By hypothesis $D_{\B,\partial\B}$, and hence $p^*D_{\B,\partial\B}$, are invertible. Moreover, the hypothesis and again \cref{thm:fibrewise_characterisation_of_twisted_ambidexterity} imply that $D_{p,q}$ is pointwise invertible, whereas \cref{cor:horizontal_groupoidality_special_bicartesian} implies that $D_{p,q}$ is groupoidal; the two observation imply that $D_{p,q}$ is invertible, and all together we obtain that $D_{p,q} \otimes p^* D_{\B,\partial\B}$ as desired.

Next, we turn to the proof of (1)$\implies$(2). By \cref{lem:doubling_trick}, we may assume that $\partial\B=\emptyset$.
We let $\bar\B=|\B|$ and denote by $(\bar p,\bar q)$ the essentially unique (special) $\bar\B$-pair such that $(p,q)$ is the pullback of $(\bar p,\bar q)$ along $L_{\B}\colon\B\to\bar\B$. By \cref{cor:poincare_domination_of_localisations} and the hypothesis we then have that $(\bar\E,\partial_v\bar\E)$ is a $\sC$-Poincar\'e $\infty$-category pair. By \cref{lem:special_case_of_converse_for_groupoid_base}, all fibres $(\bar\E_b,\partial\bar\E_b)$ of $(\bar p,\bar q)$ for $b\in\bar\B$ are $\sC$-Poincar\'e $\infty$-category pairs as well. Since for $b\in\B$ we have $(\E_b,\partial\E_b)\simeq(\bar\E_{L_\B(b)},\partial\bar\E_{L_\B(b)})$, we conclude that all fibres of $(p,q)$ are $\sC$-Poincar\'e $\infty$-category pairs as well. This, together with \cref{thm:fibrewise_characterisation_of_twisted_ambidexterity} and \cref{cor:horizontal_groupoidality_special_bicartesian}, implies that $(\paramFibred{E},\partial\paramFibred{E})$ is groupoidally $\sC$-twisted ambidextrous and in fact that $D_{p,q}$ is invertible.

We let $\bar D_{p,q}$ be the image of $D_{p,q}$ along the functor $\bar D\colon\presheaf(\E;\sC)\to\presheaf(\E\op;\sC)$ from \cref{defn:DandbarD}; in particular we have $t^*D_{p,q}\simeq s^*\bar D_{p,q}\in\presheaf(\twistedArrow(E);\sC)$.
Using the notation from \cref{lem:twisted_arrow_cofinality}, we may now compute
\[
\begin{split}
(\E,\partial\E)_*(-) &\simeq \B_*(p,q)_*(-)\simeq \twistedArrow(\B)_!\big(s^*\omega_{\B}\otimes t^*p_!(D_{p,q}\otimes -)\big)\\
&\simeq \twistedArrow(\B)_!\big(s^*\omega_{\B}\otimes \tilde{p}_!\tilde{t}^*(D_{p,q}\otimes-)\big)\simeq \twistedArrow(\B)_!\tilde{p}_! \big(\tilde{p}^*s^*\omega_{\B}\otimes\tilde{t}^*(D_{p,q}\otimes-)\big)\\
&\simeq\twistedArrow(\B)_!\twistedArrow(p)_!\big( \varphi^*\tilde{p}^*s^*\omega_{\B}\otimes \varphi^*\tilde{t}^*D_{p,q}\otimes\varphi^*\tilde{t}^*(-)\big)\\
&\simeq \twistedArrow(\E)_!\big( s^*p^*\omega_{\B}\otimes t^*D_{p,q}\otimes t^*(-)\big)\simeq \twistedArrow\E_!\big(s^*p^*\omega_{\B}\otimes s^*\bar{D}_{p,q}\otimes t^*(-)\big)\\
&\simeq \twistedArrow(\E)_!\big(s^*(p^*\omega_{\B}\otimes \bar{D}_{p,q})\otimes t^*(-)\big)
\end{split}
\]
where the third equivalence is by basechange from the pullback \cref{eqn:twisted_arrow_projection_pullback}, the fourth is by \cref{lem:projection_formula_for_presheaves}, which applies because $\tilde{p}\colon \mathcal{P}\rightarrow\twistedArrow\B$ is a cartesian fibration, and the fifth  is by \cref{lem:twisted_arrow_cofinality}. By \cref{prop:classification_of_linear_functors} obtain the equivalence $\omega_{\E,\partial\E}\simeq u^*\bar{D}_{p,q}\otimes p^*\omega_{\B,\partial\B}$. By hypothesis, the left hand side is invertible, hence also $p^*\omega_{\B,\partial\B}$. Using that $p$ is essentially surjective, we deduce that $\omega_{\B,\partial\B}$ is pointwise invertible; and since $p$ is a cartesian fibration, we have that each morphism in $\B$ is the image of some morphism in $\E$, so we can deduce that $\omega_{\B,\partial\B}$ is also groupoidal.

Lastly, the  formula for $D_{\E,\partial\E}$ is a combination of \cref{cor:formula_in_terms_of_dualising_system_for_dualish} and \cref{lem:basic_fibred_dualising_module_factorisation}.
\end{proof}

\subsection{Fibred examples}
\label{subsec:generalising_KQS}\label{subsec:fibred_examples}
In this last subsection, we gather various naturally interesting examples of fibred ambidexterity. To begin with, let us note the following immediate consequence of the theory of fibred ambidexterity.

\begin{example}[Thom isomorphism]
\label{example:Thom_isomorphism}
Recall \cref{example:spheres}, and let $q\colon\partial\E\to\B$ be a map of spaces whose fibres are $\sC$-spheres. Denote $\interior{\E}\coloneqq\B$ and let $t=q\colon\partial\E\to\interior{\E}$ denote the terminal map in $\spc_{/\B}$. Let $\E=\int_{[1]}(\partial\E\xrightarrow{t}\B)$; the localisation functor $p\colon\E\to|\E|\simeq\B$ allows us to regard the $\B$-pair $(p,q)$ as a $\presheaf(\B)$-category pair $(\paramFibred{E},\partial\paramFibred{E})$, which is $\cofree_{\B\op}\sC$-twisted ambidextrous.

Combining \cref{thm:fibrewise_characterisation_of_twisted_ambidexterity} and \cref{rmk:Dpq_and_paramFibredE_*} we have obtain an equivalence $(p,q)_*\simeq p_!(D_{p,q}\otimes-)$ of $\sC$-linear functors $\presheaf(\E;\sC)\to\sC$, and moreover we have an equivalence $D_{p,q}=p^*(p,q)_*(\unit)\simeq p^*(p,q)_!(\unit)^\vee\in\presheaf(\E,\sC)$.\end{example}

Next, we deduce from \cref{thm:fibred_dualising_object_factorisation} a generalisation of \cite[Thm. G]{KleinQinSu} to all ads (in fact, to arbitrary diagram shapes) and to all $\infty$-categories of coefficients $\sC\in\calg(\presentable^L_{\stable})$.

\begin{setting}\label{setting:integration_bicartesian}
Let $(H,\partial H)$ and $(V,\partial V)$ be $\infty$-category pairs, let 
$B\in\func(H,\spc)$ and write $(\B,\partial\B)\coloneqq (\int_{H}B,\int_{\partial H}B)$.
Finally, let 
$F\in \func(|\B|\times V,\spc)\subset\func(\B\times V,\spc)$.
From these input data, we may construct a special $\B$-pair $(p,q)$ using the following specifications:
\begin{equation}
\label{eq:setting_KQS}
\begin{tikzcd}[row sep=10pt]
\E \rar\dar["\tilde p"]\ar[dd,bend right=60,"p"']\ar[dr,phantom,"\lrcorner"very near start] & \int_{|\B|\times V}F\dar\\
\B\times V \rar \dar \ar[dr,phantom,"\lrcorner"very near start] & {|\B|}\times V\dar\\
\B\ar[r]& {|\B|};
\end{tikzcd}
\hspace{1cm}
\begin{tikzcd}[row sep=10pt]
\partial_v\E \rar\dar\ar[dd,bend right=60,"q"']\ar[dr,phantom,"\lrcorner"very near start] & \int_{|\B|\times \partial V}F\dar \\
\B\times \partial V \rar \dar \ar[dr,phantom,"\lrcorner"very near start] & {|\B|}\times \partial V\dar\\
\B \rar & {|\B|}.
\end{tikzcd}
\end{equation}
We obtain a fibred triad $(\B,\partial\B,p,q)$ as in \cref{defn:fibred_triad}, and a corresponding $\infty$-category pair $(\E,\partial \E)$, where $\partial_h \E\coloneqq p^{-1}(\partial \B)$, $\partial \partial \E\coloneqq\partial_v \E \times_{\E} \partial_h \E\simeq\partial_v\E\cap\partial_h\E$, and $\partial \E \coloneqq \partial_v \E \cup_{\partial \partial \E} \partial_h \E$. 
\end{setting}

The $\infty$-category pairs $(H,\partial H)$ and $(V,\partial V)$ should be thought of as the ``horizontal'' and ``vertical'' diagrams which index the base and the fibres, respectively.  

\begin{cor}[Generalised {\cite[Thm. G]{KleinQinSu}}]\label{cor:integration_for_unstraightenings}
In \cref{setting:integration_bicartesian}, suppose that $H$ and $V$ are finite posets, and suppose moreover that both $B$ and $F$ attain $\sC$-twisted ambidextrous spaces as values. Then $(\E,\partial\E)$ is a $\sC$-Poincar\'e $\infty$-category pair if  the following hold: $(\B,\partial\B)$ is a $\sC$-Poincar\'e $\infty$-category pair, and $(\int_{\{b\}\times V}F,\int_{\{b\}\times\partial V}F)$ is a $\sC$-Poincar\'e $\infty$-category pair for all $b\in\B$. The converse holds too if we additionally assume that $\int_{\{b\}\times V}F\neq\emptyset$ for all $b\in\B$.
\end{cor}
\begin{proof}
This is an immediate consequence of \cref{thm:fibred_dualising_object_factorisation}, using that $V$ and $H$ were finite posets to ensure that the twisted ambidexterity hypotheses of said theorem are satisfied. The nonemptiness assumption is to ensure that $p
\colon \E\rightarrow \B$ is essentially surjective.
\end{proof}

\begin{rmk}
The nonemptiness assumption in the converse implication in \cref{cor:integration_for_unstraightenings} is necessary. For example, consider the case in which $(H,\partial H) = (V,\partial V) = (\ast,\emptyset)$ and $F\colon |\B|\rightarrow \spc$ is the constant functor at $\emptyset$. In this case we get $\E=\emptyset$, which is Poincar\'e; however this does not imply that $|\B|\simeq B(*)$ is Poincar\'e.
\end{rmk}

As an immediate corollary, we obtain a generalisation of \cite[Thm. G]{KleinQinSu} to all ads. This is achieved  by specialising the preceding proposition to $(H,\partial H)=(\Box^m,\partial\Box^m)$ and $(V,\partial V)=(\Box^n,\partial\Box^n)$.

\begin{cor}\label{cor:kleinQinSu_integration}
Let $B\colon \Box^m\rightarrow \spc$ be an $(m+1)$-ad of $\sC$-twisted ambidextrous (e.g. compact) spaces. Writing $\bar\B\coloneqq B(\udl1)$, suppose we have a family $F\colon \bar\B\times\Box^n\rightarrow \spc$ of $(n+1)$-ads of $\sC$-twisted ambidextrous spaces parametrised by $\bar\B$. Let $E\in\func(\Box^m\times \Box^n,\spc)$ be the straightening of the composite left fibration $\E\to\B\times\Box^n\to\Box^m\times\Box^n$, where the first arrow is the functor $\tilde p$ from \cref{eq:setting_KQS}, and the second is the product of the left fibration $\B\to\Box^m$ and the identity of $\Box^n$. 

Then $E$ is a $\sC$-Poincar\'e $(m+n+1)$-ad if the following hold: $B\colon \Box^m\rightarrow \spc$ and $F|_{\{b\}\times\Box^n}\colon \Box^n\rightarrow\spc$ are $\sC$-Poincar\'e ads for all $b\in\bar\B$. The converse also holds if we additionally assume that, for all $b\in\bar\B$, the restriction $F|_{\{b\}\times\Box^n}\colon \Box^n\rightarrow\spc$ is not the empty functor. 

In particular, given two  functors $B\colon \Box^m\rightarrow\spc$ and $F\colon\Box^n\rightarrow \spc$ with values in $\sC$-twisted ambidextrous spaces, the product $B\times F\colon \Box^m\times\Box^n\rightarrow \spc$ is a $\sC$-Poincar\'e $(m+n+1)$-ad if  $B$ and $F$ are $\sC$-Poincar\'e as an $(n+1)$-ad and an $(m+1)$-ad, respectively. The converse also holds if $B$ and $F$ were nonempty functors.
\end{cor}

\begin{example}[Product of Poincar\'e duality pairs]
Let $g\colon\partial X\to X$ and $g'\colon\partial Y\to Y$ be pairs of compact spaces. Then the square of spaces
\[
\begin{tikzcd}[column sep=50pt]
\partial X\times\partial Y\ar[r,"\id_{\partial X}\times g'"]\ar[d,"g\times\id_{\partial Y}"]&\partial X\times Y\ar[d,"g\times\id_Y"]\\
X\times\partial Y\ar[r,"\id_X\times g'"]& X\times Y
\end{tikzcd}
\]
is a $\sC$-Poincar\'e duality triad if and only if both $g$ and $g'$ are $\sC$-Poincar\'e duality pairs of spaces.
\end{example}

\appendix
\part{Appendices}
\label{part:appendices}
\counterwithin{thm}{section}
\section{Some basechange results}
In this section we prove a (co)cartesian basechange result, \cref{prop:proper_smooth_base_change}, as well as a result on basechange of functor categories along geometric morphisms of topoi, \cref{lem:base_change_formula_functor_category}. 

\begin{prop}[Basechange for (co)cartesian fibrations]\label{prop:proper_smooth_base_change}
Consider a pullback
\[
\begin{tikzcd}[row sep=10pt]
\X' \ar[r, "f"] \ar[d, "q"'] \ar[dr, phantom, very near start, "\lrcorner"]
& \X \ar[d, "p"]\\
\Y' \ar[r, "g"]& \Y
\end{tikzcd}
\]
of $\baseTopos$-categories and consider for $\D\in\cat_\baseTopos$ the commutative square
\begin{equation}\label{diag:base_change_square}
\begin{tikzcd}[row sep=10pt]
\funTopos_{\baseTopos}(\X', \D) & \funTopos_{\baseTopos}(\X, \D) \ar[l, "f^*"]\\
\funTopos_{\baseTopos}(\Y', \D) \ar[u, "q^*"] & \funTopos_{\baseTopos}(\Y, \D). \ar[l, "g^*"] \ar[u, "p^*"]
\end{tikzcd}
\end{equation}
\begin{enumerate}
\item Suppose that $p$ is a cocartesian fibration.
If $\D$ is cocomplete, then the square is vertically left adjointable\footnote{Recall that a commutative square in an $(\infty,2)$-category is called horizontally (vertically) left (right) adjointable if the square obtained by passing to horizontal (vertical) left (right) adjoints commutes via the canonical Beck--Chevalley transformation.}.
  If $\D$ is complete and cocomplete, then the square is horizontally right adjointable.
  \item Suppose that $p$ is a cartesian fibration.
  If $\D$ is complete, then the square is vertically right adjointable.
  If $\D$ is complete and cocomplete, then the square is horizontally left adjointable.
  \end{enumerate}
\end{prop}

\begin{rmk}
According to \cite[Proposition 5.5.3 and Example 5.5.5]{MartiniWolf2024Colimits}, \cref{prop:proper_smooth_base_change} in the case $\D = \spc_{\baseTopos}$ together with the straightening equivalence $\lfib(\X) \simeq \func_{\baseTopos}(\X, \spc_{\baseTopos})$ show that cocartesian fibrations are $\lfib$-smooth and cartesian fibrations are $\lfib$-proper.
\end{rmk}

Before coming to the proof of \cref{prop:proper_smooth_base_change}, let us us begin by recalling the definition of cocartesian fibrations in the parametrised setting from \cite[Definition 3.1.1]{Martini2022Cocartesian}.
\begin{recollect}[Parametrised cocartesian fibrations]
    \label{rec:parametrised_cocartesian_fibrations}
    For a $\baseTopos$-functor $p \colon \E \to J$, consider the comma $\baseTopos$-category $\slice{\E}{J}{J}$
    defined by the right pullback square in the commutative diagram
    \begin{equation}\label{diag:def_cocoart_fibration}
    \begin{tikzcd}[row sep=10pt]
        \E^{[1]} \ar[r, "\res_p"'] \ar[rr, bend left=15, "p"] \ar[d]
        & \slice{\E}{J}{J} \ar[d] \ar[r] \ar[dr, phantom, very near start, "\lrcorner"]
        & J^{[1]} \ar[d] \\
        \E \times \E \ar[r, "\id \times p"]
        & \E \times J \ar[r, "p \times \id"]
        & J \times J.
    \end{tikzcd}
    \end{equation}
    The $\baseTopos$-functor $p$ is called a \textit{$\baseTopos$-cocartesian fibration} if the induced $\baseTopos$-functor $\res_p \colon \E^{[1]} \to \slice{\E}{J}{J}$ admits a fully faithful left adjoint $\lift_p \colon \slice{\E}{J}{J} \to \E^{[1]}$. 
    A commutative square
    \begin{equation}\label{diag:commutative_square_cocartesian_fibrations}
    \begin{tikzcd}[row sep=10pt]
        \E \ar[r, "g"] \ar[d, "p"]
        & \E' \ar[d, "p'"] 
        \\
        J \ar[r, "f"]
        &J'
    \end{tikzcd}
    \end{equation}
    whose vertical arrows $p$ and $p'$ are $\baseTopos$-cocartesian fibrations is called a \textit{cocartesian $\baseTopos$-functor} $p \to p'$ if the commutative square
    \begin{equation}\label{diag:defining_square_cocartesian_functor}
    \begin{tikzcd}[row sep=10pt]
        \E^{[1]} \ar[r, "g"] \ar[d, "\res_p"]
        & \E'^{[1]} \ar[d, "{\res_{p'}}"]
        \\
        \slice{\E}{J}{J} \ar[r, "f"]
        & \slice{\E'}{J'}{J'}
    \end{tikzcd}
    \end{equation}
    is vertically left adjointable. We denote by $\cocartesianCategory^{\baseTopos} \subset \func([1], \cat_{\baseTopos})$ the $\infty$-subcategory spanned by cocartesian fibrations and cocartesian functors.
    These assemble into a $\baseTopos$-category $\cocartesianCatTopos^{\baseTopos} \subset \funTopos([1], \catTopos_{\baseTopos})$.
    We will also just write $\cocartesianCatTopos = \cocartesianCatTopos^{\baseTopos}$ if $\baseTopos$ is clear from the context.
    For a fixed $\baseTopos$-category $J$, there are also the ``local'' versions $\cocartesianCategory_{/J} = \cocartesianCategory \times_{\cat_{\baseTopos}} \{J\}$ and $\cocartesianCatTopos_{/J} = \cocartesianCatTopos \times_{\catTopos_{\baseTopos}} \{J\}$
    where the map $\cocartesianCatTopos\rightarrow\catTopos$ is given by evaluating at the target.
    
    The dual notion of a $\baseTopos$-\textit{cartesian fibration} is introduced similarly: in fact $p\colon\E\to J$ is a $\baseTopos$-cartesian fibration if and only if $p\op\colon\E\op\to J\op$ is a $\baseTopos$-cocartesian fibration. The $\baseTopos$-subcategory $\cartesianCatTopos\subset\funTopos([1],\catTopos_\baseTopos)$ is also defined similarly.
\end{recollect}

For the proof of the basechange theorem, and throughout the rest of this article, we will need a parametrised version of the pointwise formula for left Kan extensions.
\begin{recollect}\label{rec:pointwise_formula_kan_extension}
Consider a functor $f \colon \X \to \Y$ of $\baseTopos$-categories and and a coefficient $\baseTopos$-category $\D$.
By \cite[Theorem 6.3.5]{MartiniWolf2024Colimits}, a sufficient criterion for the existence of a left adjoint $f_!$ to the restriction functor
\[
f^* \colon \funTopos_{\baseTopos}(\Y, \D) \to \funTopos_{\baseTopos}(\X, \D)
\]
is the existence of all colimits of the form $\colim_{\X_{/y}} F$ in $\pi_\tau^*\D$ for all $\tau\in \baseTopos$ and $y \in \Y(\tau)$.
Here, $\X_{/y} =\pi_\tau^* \X \times_{\pi_\tau^* \Y} \pi_\tau^*\Y_{/y} \in \cat_{\baseTopos_{/\tau}}$ is the slice $\baseTopos_{/\tau}$-category.
By \cite[Remark 6.3.6]{MartiniWolf2024Colimits}, the Beck--Chevalley transformation obtained by passing to horizontal left adjoints in the commutative square
    \begin{equation*}
    \begin{tikzcd}[row sep=10pt]
        \func_{\baseTopos}(\Y, \D)  \ar[r, "f^*"] \ar[d, "y^*"]
        &
        \func_{\baseTopos}(\X, \D) \ar[d, "\res"]
        \\
        \D(\tau) \ar[r, "{\X_{/y}^*}"]
        &
        \func_{\baseTopos_{/t}}(\X_{/y}, \pi_\tau^* \D)
    \end{tikzcd}
    \end{equation*}
    is invertible, which gives an explicit formula for $f_!$.
\end{recollect}

\begin{proof}[Proof of \cref{prop:proper_smooth_base_change}.]
    Note that the part about horizontal adjointability follows from the part on vertical adjointability if all adjoints exist, which is guaranteed by $\D$ being complete and cocomplete.
    Also observe that (1) and (2) are equivalent as the equivalence $\funTopos_{\baseTopos}(\X, \D)\op \simeq \funTopos_{\baseTopos}(\X\op, \D\op)$ identifies $f_!$ with $(f\op)_*$ and $f$ is a cocartesian fibration if and only if $f\op$ is a cartesian fibration.

    By restricting along $\baseTopos \to \baseTopos_{/s}$ for varying $s \in \baseTopos$ it suffices to prove the statement on global sections, as cocartesian fibrations and Kan extensions are stable under \'etale basechange.
    Let us explain how to reduce (1) to the case $\Y' = t \in \baseTopos$.
    For this, consider for any object $y \in \Y'(t)$ the commutative diagram
    \begin{equation}\label{diag:basechange_reduction_fiber}
    \begin{tikzcd}[row sep=10pt]
        \func_{\baseTopos}(\X' \times_{\Y'} t, \D)
        &
        \func_{\baseTopos}(\X', \D) \ar[l, "i^*"]
        & 
        \func_{\baseTopos}(\X, \D) \ar[l, "f^*"]
        \\
        \func_{\baseTopos}(t, \D) \ar[u, "a^*"] 
        &
        \func_{\baseTopos}(\Y', \D) \ar[u, "q^*"] \ar[l, "y^*"]
        &
        \func_{\baseTopos}(\Y, \D). \ar[l, "g^*"] \ar[u, "p^*"]
    \end{tikzcd}
    \end{equation}
    where $\X'_y = \X' \times_\Y t$ and $a \colon \X' \times_{\Y'} t \to t$ is the projection.
    The case $\Y' = t$ implies that the left and outer square in \cref{diag:basechange_reduction_fiber} are vertically left adjointable.
    From this it follows that the basechange transformation $q_! f^* \to g^* p_!$ becomes an equivalence after postcomposing with $y^*$.
    As the functors $y^* \colon \func_{\baseTopos}(\Y', \D)  \to \func_{\baseTopos}(t, \D)$ are jointly conservative for varying $t \in \baseTopos$ and $y \in \Y'(t)$ this shows that the right square in \cref{diag:basechange_reduction_fiber} is vertically left adjointable.

    It remains to prove the statement in the case $\Y' = t$ and $g = y \colon t \to \Y$.
    We next explain how to reduce this to $t = *$.
    For this, consider the commutative cube
    \begin{equation*}
    \begin{tikzcd}[row sep=10pt]
        & \func_{\baseTopos_{/t}}(t^*\X_{y}, t^*\D) \ar[dd, leftarrow]
        &
        & \func_{\baseTopos_{/t}}(t^*\X, t^*\D) \ar[ll]
        \\
        \func_{\baseTopos}(\X_{y}, \D) \ar[ur, "t^*"]
        &
        &\func_{\baseTopos}(\X, \D) \ar[ll, crossing over] \ar[ur, "t^*"]
        &
        \\
        & \func_{\baseTopos_{/t}}(t^* t, t^*\D)
        &
        & \func_{\baseTopos_{/t}}(t^* \Y, t^*\D) \ar[ll] \ar[uu, "p^*"]
        \\
        \func_{\baseTopos}(t, \D) \ar[ur, "t^*", "\simeq"'] \ar[uu]
        &
        &\func_{\baseTopos}(\Y, \D) \ar[ll] \ar[ur, "t^*"] \ar[uu, "p^*", near start, crossing over]
        &
    \end{tikzcd}.
    \end{equation*}
    Parametrised Kan extensions are compatible with restriction along $\baseTopos \to \baseTopos_{/t}$ so the left and right face are vertically left adjointable.
    As $\func_{\baseTopos}(t, \D) \simeq \func_{\baseTopos_{/t}}(t^* \Y, t^*\D)$ this shows that vertical left adjointability of the back face implies vertical left adjointability of the front face.

    Now consider the case $\Y' = *$ and $f = y \in \Y(*)$.
    From the pointwise formula for Kan extensions in \cref{rec:pointwise_formula_kan_extension} we get the commutative diagram
    \begin{equation*}
    \begin{tikzcd}[row sep=10pt]
        \func_{\baseTopos}(\X_{y}, \D)
        &
        \func_{\baseTopos}(\X_{/y}, \D) \ar[l, "j^*"]
        &
        \func_{\baseTopos}(\X, \D) \ar[l, "\res"]
        \\
        \func_{\baseTopos}(*, \D) \ar[u, "{\X_{y}^*}"] \ar[r, equal]
        &
        \func_{\baseTopos}(*, \D) \ar[u, "{\X_{/y}^*}"]
        &
        \func_{\baseTopos}(\Y, \D)  \ar[u, "f^*"] \ar[l, "y^*"]
    \end{tikzcd}
    \end{equation*}
    in which the right square is vertically left adjointable.
    We only have to show that the left square is vertically left adjointable.
    For this, we claim that $j \colon \X_{y} \to  \X_{/y}$ admits a left adjoint $L$, which clearly implies vertical left adjointability:
    Note that $L^* \simeq j_!$, so that adjointability follows from the equivalence $L^* \X_{y}^* \simeq \X_{/y}^*$.
    To obtain the adjoint $L$, observe that that $\eval_1 \colon \X^{[1]} \to \X$ admits the right adjoint $\constant$ from which we see that the composite $s \colon \slice{\X}{\Y}{\Y} \xrightarrow{\lift_p} \X^{[1]} \xrightarrow{\eval_1} X$ is left adjoint to $l \colon X \xrightarrow{\constant} X^{[1]} \xrightarrow{\res_p} \slice{\X}{\Y}{\Y}$.
    We obtain the commutative diagram
    \begin{equation*}
    \begin{tikzcd}[row sep=10pt]
        \X \ar[r, "l"', shift right=1.5] \ar[d, "p"]
        & \slice{\X}{\Y}{\Y} \ar[dl, bend left] \ar[l, "s"', shift right=1.5]
        \\
        \Y
    \end{tikzcd}.
    \end{equation*}
    Passing to vertical fibres over $y$ yields the adjunction
    \begin{equation*}
    \begin{tikzcd}
        \X_y \ar[r, "j"', shift right=1.5]
        & \X_{/y} \ar[l, "s_y"', shift right=1.5]
    \end{tikzcd}
    \end{equation*}
    as claimed.
    Additionally, we have $sl \simeq \id$ as constant maps in $\X$ are cocartesian, so $j$ is even fully faithful.
\end{proof}

Another base change result that is used in this article is the following generalisation of \cite[Lemma 2.1.19]{PD1}.

\begin{lem}\label{lem:base_change_formula_functor_category}
    Let $f^* \colon \baseTopos \rightleftharpoons \baseTopos' \cocolon f_*$ be a geometric morphism of topoi, $\X$ a $\baseTopos$-category and $\D$ a $\baseTopos'$-category.
    Then there is a natural equivalence $\funTopos_{\baseTopos}(\X, f_* \D) \simeq f_* \funTopos_{\baseTopos'}(f^* \X, \D)$ of $\baseTopos$-categories.
    Moreover, if $\D$ admits $f^* \X$-shaped limits and colimits, then this equivalence induces an identification of adjoint triples
    \begin{equation}
    \label{eq:identification_of_projection_under_base_change}
    \begin{tikzcd}
        f_* \funTopos_{\baseTopos'}(f^* \X, \D) \ar[d, equal] \ar[rr, "f_*(f^* \X)_!", shift left=4] \ar[rr, "f_*(f^* \X)_*", shift right=4]
        && f_* \D \ar[d, equal] \ar[ll, "f_*(f^* \X)^*"']
        \\
        \funTopos_{\baseTopos}(\X, f_* \D) \ar[rr, "\X_!", shift left=4] \ar[rr, "\X_*", shift right=4]
        && f_* \D.  \ar[ll, "\X^*"']
    \end{tikzcd}
    \end{equation}
\end{lem}
\begin{proof}
    We can construct a map $\funTopos_{\baseTopos}(\X, f_* \D) \to f_* \funTopos_{\baseTopos'}(f^* \X, \D)$ adjoint to the composite
    \begin{equation*}
        f^* \funTopos_{\baseTopos}(\X, f_* \D) \times f^* \X 
        \simeq f^*(\funTopos_{\baseTopos}(\X, f_* \D) \times \X) 
        \to f^* f_* \D 
        \to \D
    \end{equation*}
    and a map $f_* \funTopos_{\baseTopos'}(f^* \X, \D) \to \funTopos_{\baseTopos}(\X, f_* \D)$ adjoint to the composite
    \begin{equation*}
        f_* \funTopos_{\baseTopos'}(f^* \X, \D) \times \X 
        \to f_* \funTopos_{\baseTopos'}(f^* \X, \D) \times f_* f^* \X 
        \simeq f_* (\funTopos_{\baseTopos'}(f^* \X, \D) \times f^* \X) 
        \to f_* \D.
    \end{equation*}
    A routine exercise in adjunction (co)units proves that these are inverse to each other.    
    Consider the following commuting diagram, in which the maps from the middle row to the upper row are induced by the map $\X \rightarrow *$.
    \begin{equation*}
        \begin{tikzcd}[row sep=10pt]
            \funTopos_\baseTopos(\X,f_* \D) \ar[r] & f_* \funTopos_{\baseTopos'}(f^* \X, f^*f_* \D) \ar[r] & f_* \funTopos_{\baseTopos'}(f^*\X,\D) \\
            \funTopos_\baseTopos(*,f_* \D) \ar[u] \ar[r] & f_* \funTopos_{\baseTopos'}(f^* * , f^*f_* \D) \ar[r] \ar[u] & f_* \funTopos_{\baseTopos'}(f^* *, \D) \ar[u]\\
            f_* \D \ar[r, "\eta_{f_*}"]  \ar[u] \ar[uu, bend left = 80, "\X^*"] & f_* f^* f_* \D  \ar[r,"\epsilon_{f_*}"] \ar[u] & f_*\D  \ar[u] \ar[uu, bend right = 80, "f_*(f^*\X)^*"']
        \end{tikzcd}
    \end{equation*}
    The lower horizontal composite is the identity, while the upper composite is exactly the left horizontal equivalence.
    Hence, in \cref{eq:identification_of_projection_under_base_change} the subdiagram with left pointing arrows commutes.
    
    To finish the proof, note that $f_*$, as a $2$-functor, preserves adjunctions so that $f_*(f^*\X)_!$ is the left adjoint to $f_*(f^*\X)^*$, and $f_*(f^*\X)_*$ is its right adjoint. Hence, also the two subdiagrams with the upper/lower right pointing arrows commute.
\end{proof}

\section{The cofree construction}
\label{sec:cofree}
Some of our results, most notably those contained in \cref{sec:fibredpd}, will rely on a technical supplement in internal higher category theory, that we develop in this subsection: the existence of cofree cocartesian fibrations in \cref{lem:omnibus_cofree}.

\begin{thm}[Cofree cocartesian fibration]
\label{lem:omnibus_cofree}
Consider a $\baseTopos$-category $J \in \cat_{\baseTopos}$. Then the $\baseTopos$-functor $\int \colon \funTopos_{\baseTopos}(J, \catTopos_{\baseTopos}) \simeq \cocartesianCatTopos^{\baseTopos}_{/ J} \to\catTopos_{\baseTopos},
$, sending a $\baseTopos$-cocartesian fibration $p \colon \D \to J$ to its total $\baseTopos$-category $\D$, admits a right adjoint $\cofree_{J}$.
\end{thm}

\begin{rmk}
\label{rmk:compatibility_of_cofree_with_etale_basechange}
Note that for an \'etale geometric morphism $\pi_\tau^* \colon \baseTopos \rightarrow \baseTopos_{/\tau}$ the identifications  $\pi_\tau^* \catTopos_{\baseTopos} \simeq \catTopos_{\baseTopos_{/\tau}} $ and $\pi_\tau^* \funTopos_{\baseTopos}(J, \catTopos_{\baseTopos}) \simeq \funTopos_{\baseTopos_{/\tau}}(\pi_\tau^* J,\cat_{\baseTopos_{/\tau}})$ carry the unstraightening $\baseTopos$-functor $\int$ to the analogous unstraightening $\baseTopos_{/\tau}$-functor $\int \colon \funTopos_{\baseTopos_{/\tau}}(\pi_\tau^*J, \catTopos_{\baseTopos_{/\tau}}) \simeq \cocartesianCatTopos^{\baseTopos}_{/ J} \to \catTopos_{\baseTopos_{/\tau}}$.
Since $\pi_\tau^* \colon \widehat{\cat}_\baseTopos \rightarrow \widehat{\cat}_{\baseTopos_{/\tau}}$ refines to a functor of $(\infty,2)$-categories, this also shows that the cofree construction is compatible with \'etale basechange.
\end{rmk}

Our strategy is to reduce \cref{lem:omnibus_cofree} to the case where $\baseTopos = \presheaf(\B)$ is a presheaf $\infty$-topos, which is covered by the following unparametrised version of \cref{lem:omnibus_cofree}.
\begin{lem}\label{lem:cofree_unparametrised}
Let $p \colon I \to J$ be a cocartesian fibration of $\infty$-categories. The the functor
    \begin{equation*}
        p_\sharp \colon \cocartesianCategory_{/I} \to \cocartesianCategory_{/J}, \quad  (q \colon \D \to I) \mapsto (pq \colon \D \to J)
    \end{equation*}
    given by postcomposition with $p$ admits a right adjoint.\footnote{We use the notation $p_\sharp$ to avoid confusion with left Kan extension along $p$, which agree only if $p$ is a left fibration, see \cref{lem:left_kan_extension_left_fibration}.}
\end{lem}
\begin{proof}
    Consider the commutative diagram
    \begin{equation*}
    \begin{tikzcd}
        \cocartesianCategory_{/I} \ar[r, "p_\sharp"] \ar[d, "\int_I"]
        & \cocartesianCategory_{/J} \ar[dl, "\int_J"]
        \\
        \cat
    \end{tikzcd}.
    \end{equation*}
    It is shown in \cite[Proposition A.1]{ramziMonoidalGrothendieck} that both vertical maps preserve colimits.
    As the functor $\int$ is also conservative, this implies that $p_\sharp$ preserves colimits.
    The statement now follows from the adjoint functor theorem, using that $\cocartesianCategory_{/I} \simeq \func(I, \cat)$ is presentable.
\end{proof}
For the reduction to the presheaf $\infty$-topos case, we will need the to study the behaviour of cocartesian fibrations under basechanges along geometric (or algebraic) morphisms. The occurence of the cotensor $(-)^{[1]}$ in the definition of cocartesian fibrations forces us to first investigate the behavior of cotensors under geometric morphisms.

\begin{lem}\label{obs:base_change_cotensor}
    Let $I \in \cat$ be an $\infty$-category and let $r^* \colon \baseTopos \rightleftharpoons \baseTopos' \cocolon r_*$ be a geometric morphism of $\infty$-topoi. 
    Then for $\sC \in \cat_{\baseTopos}$ and $\D \in \cat_{\baseTopos'}$ there are maps 
    \begin{equation*}
        r^* \funTopos_{\baseTopos}(\constant_I, \sC) \to \funTopos_{\baseTopos'}(\constant_I, r^* \sC) \quad \text{and} \quad r_* \funTopos_{\baseTopos'}(\constant_I, \D) \to \funTopos_{\baseTopos}(\constant_I, r_* \D).
    \end{equation*}
    The second map is always an equivalence, while the first map is an equivalence if $I$ is compact.
\end{lem}
\begin{proof}
    For the first claim, recall that $r_* \colon \cat_{\baseTopos'} \to \cat_{\baseTopos}$ is given by restriction along $(r^*)\op$ in the sheaf model $\cat_{\baseTopos'} = \func^{\lim}({\baseTopos'}\op, \cat)$.
    Thus, for $\tau\in\baseTopos$ we have $r_* \funTopos_{\baseTopos'}(\constant_I, \D) (\tau) = \func(I, \D(\tau)) = \funTopos_{\baseTopos}(\constant_I, r_* \D)(\tau)$.

    Next we show that $r^* \funTopos_{\baseTopos}(\constant_I, \sC) \to \funTopos_{\baseTopos'}(\constant_I, r^* \sC)$ is an equivalence for compact $I$.
    Let us first argue the slightly weaker claim that the induced map
    \[ r^* \funTopos_{\baseTopos}(\constant_I, \sC)^{\simeq} \to \funTopos_{\baseTopos'}(\constant_I, r^* \sC)^{\simeq} \]
    is an equivalence in $\baseTopos$ for every compact $\infty$-category $I$.
    As both sides are stable under finite limits, and since compact $\infty$-categories are generated by $[0]$ and $[1]$ through retracts and finite colimits, it suffices to check this claim for $I = [n]$ for $n \ge 0$.
    In this case, it follows from the identification of $\funTopos_{\baseTopos'}(\constant_{[n]}, \sC)^\simeq$ with the value of $\sC$ at $[n]$, viewing $\baseTopos$-categories as complete Segal objects, together with the observation that $r^*\colon \cat_\baseTopos \rightarrow \cat_{\baseTopos'}$ is obtained by restricting the functor
    $\func(\simplex\op,\baseTopos) \rightarrow \func(\simplex\op,\baseTopos')$ given by postcomposition with $r^*$ to complete Segal objects.
    
    For the general statement, recall that a map $\X \to \Y$ of $\baseTopos$-categories is an equivalence if and only if the maps of $\baseTopos$-groupoids $\X^\simeq \to \Y^\simeq$ and $\funTopos(\constant_{[1]}, \X)^\simeq \to \funTopos(\constant_{[1]}, \Y)^\simeq$ are equivalences.
    From the core-part above we obtain the chain of equivalences
    \begin{align*}
        \funTopos_{\baseTopos'}(\constant_{[1]}, r^* \funTopos_{\baseTopos}(\constant_I, \sC))^\simeq
        &\simeq r^* \funTopos_{\baseTopos}(\constant_{[1]}, \funTopos_{\baseTopos'}(\constant_I, \sC))^\simeq \\
        &\simeq r^* \funTopos_{\baseTopos}(\constant_{[1] \times I}, \sC))^\simeq \\
        &\simeq \funTopos_{\baseTopos'}(\constant_{[1] \times I}, r^* \sC))^\simeq \\
        &\simeq \funTopos_{\baseTopos'}(\constant_{[1]}, \funTopos_{\baseTopos'}(\constant_I, r^* \sC))^\simeq,
    \end{align*}
    where we used that also $[1] \times I$ is a compact $\infty$-category.
\end{proof}

\begin{cor}\label{cor:cocartesian_preserved_under_basechange}
    Let $r^* \colon \baseTopos \rightleftharpoons \baseTopos' \cocolon r_*$ be a geometric morphism. 
    Then both functors $r_* \colon \func([1], \cat_{\baseTopos'}) \to \func([1], \cat_{\baseTopos})$ and $r^* \colon \func([1], \cat_{\baseTopos}) \to \func([1], \cat_{\baseTopos'})$ preserve cocartesian fibrations and cocartesian functors between those.
    Furthermore, the adjunction unit $\id \to r_* r^*$ and counit $r^*r_* \to \id$ evaluate to cocartesian functors between cocartesian fibrations.
\end{cor}
\begin{proof}
Both functors $r^*$ and $r_*$ refine to $(\infty,2)$-functors, hence they both preserve the property of being a right adjoint with fully faithful left adjoint. 
The statement now follows from the following two identifications for functors $p \colon \D \to J$ in $\cat_{\baseTopos'}$ and $q \colon \E \to I$ in $\cat_{\baseTopos}$.
    \begin{align*}
        &r_*( \res_p \colon \D^{[1]} \rightarrow \slice{\D}{J}{J} ) \simeq (\res_{r_* p} \colon (r_*\D)^{[1]} \rightarrow  \slice{r_* \D}{r_* J}{r_* J}), \\
        &r^*(\res_q \colon \E^{[1]} \rightarrow \slice{\E}{I}{I}) \simeq (\res_{r^*q} \colon (r^* \E)^{[1]} \rightarrow \slice{r^* \E}{r^*I}{r^*I})
    \end{align*}
    This uses that both $r_*$ and $r^*$ preserve finite limits and cotensors by $[1]$ as shown in \cref{obs:base_change_cotensor}.

    The $(\infty,2)$-functors $r^*$ and $r_*$ also preserve vertically left adjointable squares.
    Applying this to the square \cref{diag:defining_square_cocartesian_functor} appearing in the definition of cocartesian functors, we see that $r^*$ and $r_*$ preserve those.

    Finally, let us argue that the unit and counit restrict to cocartesian functors between cocartesian fibrations.
    This follows from the general $(\infty,2)$-categorical fact that for a transformation $h \colon F \to G$ of $(\infty,2)$-functors $F, G \colon \mathsf{C} \to \mathsf{D}$ and a right adjoint morphism $r \colon x \to y$ in $\mathsf{C}$, the following commutative square is vertically left adjointable:
    \begin{equation*}
    \begin{tikzcd}[row sep=10pt]
        F(x) \ar[d, "F(r)"] \ar[r, "{h_x}"]
        & G(x) \ar[d, "G(r)"]
        \\
        F(y) \ar[r, "{h_y}"]
        & G(y).
    \end{tikzcd}
    \end{equation*}
    For this we have to show that the upper composite in the commutative diagram
    \begin{equation*}
    \begin{tikzcd}[row sep=10pt]
        G(l) h_y \ar[r, "F(\eta)"] \ar[dr, "G(\eta)"']
        & G(l) h_y F(r) F(l) \ar[r, "\simeq"] \ar[d, "\simeq"]
        & G(l) G(r) h_x F(l) \ar[r, "G(\epsilon)"]
        & h_x F(l)
        \\
        & G(l) G(r) G(l) h_y \ar[r, "G(\epsilon)"] \ar[ur, "\simeq"]
        & G(l) h_y \ar[ur, "\simeq"]
    \end{tikzcd}
    \end{equation*}
    is an equivalence, where $l \colon y \to x$ is left adjoint to $r$ and the left commuting triangle comes from the structure of $F$ being a 2-functor.
    But the bottom composite is an equivalence by the triangle identities, which proves the claim.
\end{proof}

For the proof of \cref{lem:omnibus_cofree} we need one further ingredient, which identifies left Kan extension along left fibrations with the cocartesian pushforward functor.
For this, given a cocartesian fibration $p \colon I \to J$ of $\baseTopos$-categories, denote by $p_\sharp$ the functor
\begin{equation*}
    p_\sharp \colon \cocartesianCategory^\baseTopos_{/I} \to  \cocartesianCategory^\baseTopos_{/J}, \quad (q \colon \D \to J) \mapsto (pq \colon \D \to I)
\end{equation*}
obtained by postcomposition with $p$.
In the case $I = *$, this recovers the unstraightening functor $\int_J \colon \cocartesianCategory^\baseTopos_{/J} \to \cat_{\baseTopos}$.
Also denote by $p^* \colon \cocartesianCategory^\baseTopos_{/J} \to \cocartesianCategory^\baseTopos_{/I}$ the pullback functor.
It corresponds to restriction along $p$ under the unstraightening equivalence $\cocartesianCategory^\baseTopos_{/I} \simeq \func_{\baseTopos}(I, \catTopos_{\baseTopos})$ and admits a left adjoint $p_!$ by left Kan extension.
The cocartesian functor
\begin{equation*}
\begin{tikzcd}[row sep=10pt]
    \D \times_{J} I \ar[d] \ar[r]
    & \D \ar[d, "q"]
    \\
    I \ar[r, "p"]
    & J.
\end{tikzcd}
\end{equation*}
for $q \in \cocartesianCategory^{\baseTopos}_{/J}$ induces a transformation $p_\sharp p^* \to \id$.
\begin{lem}\label{lem:left_kan_extension_left_fibration}
    If $p \colon I \to J$ is a left fibration of $\baseTopos$-categories, the transformation $p_\sharp p^* \to \id$ exhibits $p_\sharp$ as a left adjoint to $p^*$, and induces an equivalence $p_\sharp \simeq p_!$.
\end{lem}

\begin{proof}
    First of all, note that we have an adjunction $p_{\sharp}\colon \cat^{\baseTopos}_{/I}\rightleftharpoons \cat^{\baseTopos}_{/J} \cocolon p^*$ as a general fact about slice $\infty$-categories. We would like to argue that this adjunction restricts to the nonfull subcategories of cocartesian fibrations. In other words, we have commuting squares
    \[
    \begin{tikzcd}[row sep=10pt]
        \cocartesianCategory^{\baseTopos}_{/I} \rar["p_{\sharp}", shift left = 1]\dar[hook]& \cocartesianCategory^{\baseTopos}_{/J}\lar["p^*", shift left = 1]\dar[hook]\\
        \cat^{\baseTopos}_{/I} \rar["p_{\sharp}", shift left = 1]& \cat^{\baseTopos}_{/J}, \lar["p^*", shift left = 1]
    \end{tikzcd}
    \] and to argue that the bottom adjunction restricts to one on the top, we just have to argue that the adjunction (co)units restrict to the nonfull subcategories. That is, for $X\coloneqq (q \colon \X\rightarrow I)\in\cocartesianCategory^{\baseTopos}_{/I}$ and $Y\coloneqq (\Y\rightarrow J)\in\cocartesianCategory^{\baseTopos}_{/J}$, the morphisms $X\rightarrow p^*p_{\sharp}X$ and $p_{\sharp}p^*Y\rightarrow Y$ in $\cat^{\baseTopos}_{/I}$ and $\cat^{\baseTopos}_{/J}$ are in fact morphisms in $\cocartesianCategory^{\baseTopos}_{/I}$ and $\cocartesianCategory^{\baseTopos}_{/J}$, respectively. 
    For $p_{\sharp}p^*Y\rightarrow Y$ we explained this before the statement of the lemma.
    For the map $X\rightarrow p^*p_{\sharp}X$, consider the pullback square defining $p^*p_{\sharp}X$:
\[
\begin{tikzcd}[row sep=10pt]
\X \times_J I\ar[dr,phantom,"\lrcorner"very near start] \ar[rr] \ar[d]
& & I \ar[d, "p"]\\
\X \ar[r, "q"]& I \ar[r, "p"] &J.
\end{tikzcd}
\]
Note that a morphism $(f, g)$ in $\X \times_J I$ is $p^*p_{\sharp}X$-cocartesian if and only if the morphism $f$ in $\X$ is $qp$-cocartesian and the morphism $g$ in $I$ is $p$-cocartesian.
As $p$ is a left fibration, any morphism in $I$ is $p$-cocartesian.
This shows that $g$ is $qp$-cocartesian if and only if it is $q$-cartesian, from which we see that the map $X\rightarrow p^*p_{\sharp}X$ preserves cocartesian morphisms as claimed.
\end{proof}

\begin{proof}[Proof of \cref{lem:omnibus_cofree}.]
    The proof consists of multiple reduction steps, the first three showing the existence of levelwise adjoints and the last showing that Beck--Chevalley conditions hold for them to assemble to an adjoint as $\baseTopos$-functor according to \cite[Prop. 3.2.9.]{MartiniWolf2024Colimits}:
    
    \textit{1. The presheaf $\infty$-topos case:}
    Let us first argue that the right adjoint exists after applying global sections in the case $\baseTopos = \presheaf(\B)$.
    Under the equivalence $\cat_\baseTopos = \func(\B\op, \cat)$ and $\func_\baseTopos(J, \catTopos_\baseTopos) = \func((\int_\B J)\op, \cat)$, the unstraightening functor identifies with the cocartesian pushforward 
    \begin{equation*}
        p_\sharp \colon \cocartesianCategory_{/(\int_\B J)\op} \to \cocartesianCategory_{/\B\op}
    \end{equation*}
    where $p \colon \int_\B J \to \B$ is the cartesian fibration classifying $J \in \func(\B\op, \cat)$.
    This admits a right adjoint by \cref{lem:cofree_unparametrised}.

    \textit{2. Reduction to the presheaf $\infty$-topos case:}
    Next, let us argue that the right adjoint exists on global sections for general $\baseTopos$.
    For this, we can write $\baseTopos$ as a left exact accessible localisation of a presheaf $\infty$-topos and obtain a geometric morphism $r^* \colon \presheaf(\B) \rightleftharpoons \baseTopos \cocolon r_*$ for a $\infty$-category $\B$ such that the right adjoint $r_*$ is fully faithful.
    We claim that the adjunction
    \begin{equation*}
    \begin{tikzcd}
        (\cat_{\presheaf(\B)})_{/r_* J} \ar[r, "r^*", shift left=1.5]
        &  (\cat_\baseTopos)_{/J} \ar[l, hook, "r_*", shift left=1.5]
    \end{tikzcd}
    \end{equation*}
    restricts to the following adjunction of nonfull subcategories
    \begin{equation*}
    \begin{tikzcd}
        \cocartesianCategory^{\presheaf(\B)}_{/r_* J} \ar[r, "r^*", shift left=1.5]
        & \cocartesianCategory^{\baseTopos}_{/J}. \ar[l, hook, "r_*", shift left=1.5]
    \end{tikzcd}
    \end{equation*}
    This follows from \cref{cor:cocartesian_preserved_under_basechange}, where we showed that both $r^*$ and $r_*$ restrict to functors between those subcategories and the adjunction unit and counit restrict to those subcategories as well.
    Now consider the diagram
    \begin{equation*}
    \begin{tikzcd}
        \cocartesianCategory^{\presheaf(\B)}_{/r_* J} \ar[r, "r^*", shift left=1.5]\ar[d, "\int_{r_* J}"]
        & \cocartesianCategory^{\baseTopos}_{/J} \ar[l, hook, "r_*", shift left=1.5] \ar[d, "\int_J"]
        \\
        \cat_{\presheaf(\B)}\ar[r, "r^*", shift left=1.5]
        & \cat_{\baseTopos}, \ar[l, hook, "r_*", shift left=1.5] 
    \end{tikzcd}
    \end{equation*}
    which by construction commutes with both leftwards and both rightwards pointing arrows.
    This formally implies that the right vertical arrow preserves colimits as the left vertical arrow does so:
    For a diagram $F \colon L \to \cocartesianCategory_{J}^{\baseTopos}$ one has $\colim_L F \simeq r^* \colim_L r_* F$, which gives
    \begin{equation*}
        \colim_L \int_J F \simeq r^* \colim_L r_* \int_J F \simeq \int_J r^* \colim_L r_* F \simeq \int_J \colim_L F.
    \end{equation*}
    Again using the adjoint functor theorem, we see that $\int_J$ admits a right adjoint.

    \textit{3. Existence of levelwise adjoints:}
    The previous part of the proof showed that the functor $\int_J \colon \cocartesianCategory^\baseTopos_{/J} \rightarrow \cat_\baseTopos$ preserves colimits.
    So, the proof of \cref{lem:cofree_unparametrised} applies to show that for any cocartesian fibration $p \colon I \to J$ of $\baseTopos$-categories, the functor
    \begin{equation*}
        p_\sharp \colon \cocartesianCategory^\baseTopos_{/I} \to  \cocartesianCategory^\baseTopos_{/J}
    \end{equation*}
    obtained by postcomposition with $p$ admits a right adjoint.
    Under the identification $\cocartesianCategory^\baseTopos_{/ J \times \tau} = \cocartesianCatTopos^\baseTopos_{/J}(\tau)$ for an object $\tau \in \baseTopos$, the functor $\int \colon \cocartesianCatTopos^{\baseTopos}_{/J} \to \catTopos_{\baseTopos}$ at level $\tau$ identifies with $p_\sharp \colon \cocartesianCategory^\baseTopos_{/J \times \tau} \to  \cocartesianCategory^\baseTopos_{/\tau}$ for the projection $p \colon J \times \tau \to \tau$. 
    This proves that the $\baseTopos$-functor $\int_J \colon \cocartesianCatTopos^\baseTopos_{/J} \rightarrow \catTopos_\baseTopos$ admits levelwise right adjoints.

    \textit{4. Beck--Chevalley condition for levelwise adjoints:}
    Finally, consider a map $f \colon \tau' \to \tau$ in $\baseTopos$.
    For the levelwise adjoints to $\int_J$ to assemble to a $\baseTopos$-adjoint, we have to show that for any such $f$ the following commutative diagram with downwards pointing arrows is horizontally right adjointable.
    \begin{equation}\label{diag:base_change_cocartesian_pushforward}
    \begin{tikzcd}
        \cocartesianCategory^\baseTopos_{/J \times \tau'} \ar[r, "J_\sharp"] \ar[d, "(J \times f)^*",  shift left=1.5]
        &
        \cocartesianCategory^\baseTopos_{/\tau'} \ar[d, "f^*", shift left=1.5] 
        \\
        \cocartesianCategory^\baseTopos_{/J \times \tau} \ar[r, "J_\sharp"] \ar[u, "(J \times f)_\sharp" , shift left=1.5]
        &
        \cocartesianCategory^\baseTopos_{/\tau} \ar[u, "f_\sharp" , shift left=1.5]
    \end{tikzcd}
    \end{equation}
    Equivalently, we have to show that it is vertically left adjointable.
    Applying \cref{lem:left_kan_extension_left_fibration} to the left fibrations $J \times f$ and $f$, the lax commutative square obtained by passing to vertical left adjoints is given by the commutative diagram with upwards pointing arrows, which proves the claim.
\end{proof}

\begin{prop}
\label{prop:cofree_of_presentable}
    The functor $\cofree_{\B\op}\colon \widehat{\cat}\rightarrow \widehat{\cat}_{\presheaf(\B)}\simeq \func(\B\op,\widehat{\cat})$ refines to a symmetric monoidal functor between cartesian symmetric monoidal $\infty$-categories, and it restricts to a lax symmetric monoidal functor $\cofree_{\B\op}\colon \presentable^L\rightarrow \presentable^L_{\presheaf(\B)}$.
\end{prop}
\begin{proof}
Since $\cofree_{\B\op}\colon \widehat{\cat}\rightarrow \func(\B\op,\widehat{\cat})$ is a right adjoint, it preserves products and so  promotes to a symmetric monoidal functor as claimed. We next argue that $\cofree_{\B\op}$ restricts to a functor $\presentable^L\rightarrow \presentable^L_{\presheaf(\B)}$. We first argue on the level of objects: let $\sC\in\presentable^L$; we need to check that 
$\cofree_{\B\op}\sC$ is presentable, according to the criteria stated in \cref{subsec:recollections}.
First, note that $\presheaf(\B)$-groupoids $\paramFibred{X}$ correspond to right fibrations $\X \rightarrow \B$, and that $\cofree_{\B\op}(\sC)(\paramFibred{X}) \simeq \presheaf(\X;\sC)$ (see \cref{obs:cofree_coeff_cats}) which is presentable. Given a morphism $f \colon \paramFibred{X} \rightarrow \paramFibred{Y}$, corresponding to a map $f \colon \X \rightarrow \Y$ of right fibrations over $\B$, the map $f^* \colon \presheaf(\paramFibred{Y};\cofree_{\B\op}(\sC)) \rightarrow \presheaf(\paramFibred{X};\cofree_{\B\op}(\sC))$ identifies with $f^* \colon \presheaf(\Y;\sC) \rightarrow \presheaf(\X;\sC)$, which indeed admits a left adjoint. 

To verify the second requirement for presentability, 
let $p\colon \X\rightarrow \Y$, $q\colon \Z\rightarrow \Y$ be maps of right fibrations over $\B$. In particular, $p$ and $q$ are themselves right fibrations. Since right fibrations are stable under compositions, the inclusion $\rfib(\B)\subset \cat_{/\B}$ preserves pullbacks. Consider the leftmost pullback in the following:
\[
\begin{tikzcd}[row sep=10pt]
\W \rar["\bar{q}"]\dar["\bar{p}"'] \ar[dr,phantom,"\lrcorner"very near start] & \X\dar["p"]\\
\Z\rar["q"] & \Y;
\end{tikzcd}
\hspace{1cm}
\begin{tikzcd}[row sep=10pt]
\presheaf(\W,\sC) \dar["\bar{p}_!"'] \ar[dr,phantom,"\Rightarrow"] & \presheaf(\X,\sC)\dar["p_!"]\lar["\bar{q}^*"']\\
\presheaf(\Z,\sC) & \presheaf(\Y,\sC);\lar["q^*"'] 
\end{tikzcd}
\hspace{1cm}
\begin{tikzcd}[row sep=10pt]
\presheaf(\X,\sC)\dar["p_!"']\rar["{\presheaf(\X,F)}"]\ar[dr,phantom,"\Rightarrow"]  & \presheaf(\X,\D)\dar["p_!"]\\
\presheaf(\Y,\sC)\rar["{\presheaf(\Y,F)}"] & \presheaf(\Y,\D).
\end{tikzcd}
\]
Since $p$ is a right fibration, by proper basechange the Beck--Chevalley lax square in the middle commutes. This proves $\cofree_{\B\op}\sC\in\presentable^L_{\presheaf(\B)}$; in particular $\cofree_{\B\op}$ restricts to a functor between full subcategories $\presentable\to\presentable_{\presheaf(\B)}$.

On the level of morphisms, if $F\colon \sC\rightarrow \D$ is a morphism in $\presentable^L$, then $F\colon \cofree_{\B\op}\sC\rightarrow\cofree_{\B\op}\D$ is fibrewise colimit preserving and for any right fibration $p\colon \X\rightarrow \Y$, the rightmost Beck--Chevalley lax square above commutes since $\presheaf(\X,G)p^*\simeq p^*\presheaf(\Y,G)$, where $G\colon\D\to\sC$ is the right adjoint of $F$. This concludes the proof that $\cofree_{\B\op}$ restricts to a functor $\presentable^L\to\presentable^L_{\presheaf(\B)}$.

Finally, to see that $\cofree_{\B\op}\colon\presentable^L\to\presentable^L_{\presheaf(\B)}$ refines to a lax symmetric monoidal functor, we need to check that the restricted functor $\cofree_{\B\op}\colon\presentable\to\presentable_{\presheaf(\B)}$ between full subcategories of $\widehat{\cat}$ and $\widehat{\cat}_{\presheaf(\B)}$ sends multilinear functors to parametrised multilinear functors.
That is, if $F\colon\sC\coloneqq \prod_{i=1}^n\sC_i \rightarrow \D$ is a functor in $\presentable$ which is colimit preserving in each variable, then $\cofree_{\B\op}F\colon \cofree_{\B\op}\sC\simeq\prod^n_{i=1}\cofree_{\B\op}\sC_i\rightarrow\cofree_{\B\op}\D$ preserves $\B$-parametrised colimits in each variable. It follows from \cref{obs:cofree_coeff_cats} 
that $\cofree_{\B\op}F$ at level $\paramFibred{X}$ identifies with $\presheaf(\paramFibred{X}, \prod_i \sC_i) \xrightarrow{F} \presheaf(\paramFibred{X}, \D)$ and thus preserves fibrewise colimits in each variable.
To see that it also preserves $\B$-groupoidal colimits in each variable, it suffices to argue in the case $n=2$. Let $p\colon \X\rightarrow \Y$ be a map of right fibrations over $\B$ and $\xi\in \presheaf(\Y,\sC_2)=(\cofree_{\B\op}\sC_2)(\Y)$. We need to show that the following lax square on the left commutes:
\[
\begin{tikzcd}[row sep=10pt]
\presheaf(\X;\sC_1) \rar["{F(-,p^*\xi)}"]\dar["p_!"'] \ar[dr, phantom , "\Leftarrow"]& \presheaf(\X;\D)\dar["p_!"]\\
\presheaf(\Y;\sC_1) \rar["{F(-,\xi)}"] & \presheaf(\Y;\D);
\end{tikzcd}
\hspace{2cm}
\begin{tikzcd}[row sep=10pt]
\W\rar["\bar{y}"] \dar["\W"'] \ar[dr,phantom,"\lrcorner"very near start]& \X\dar["p"]\\ 
\ast\rar["y"] & \Y.
\end{tikzcd}
\]
This may be checked pointwise on each $y\in \Y$, so consider the pullback square on the right. Since $p$ was a right fibration, basechange gives us that $y^*p_!F(-,p^*\xi)\simeq \W_!\bar{y}^*F(-,p^*\xi)\simeq \W_!F(\bar{y}^*(-),\W^*y^*\xi)$ and $y^*F(p_!(-),\xi)\simeq F(y^*p_!(-),y^*\xi)\simeq F(\W_!\bar{y}^*-,y^*\xi)$. But then the Beck--Chevalley map $\W_!F(-,\W^*y^*\xi)\rightarrow F(\W_!-,y^*\xi)\colon \presheaf(\W;\sC)\rightarrow \D$ is an equivalence since $F\colon\sC_1\times\sC_2\rightarrow\D$ preserves colimits in the first variable, and so we reach the desired conclusion. 

Therefore, because the restricted functor $\cofree_{\B\op}\colon\presentable^L\to\presentable^L_{\presheaf(B)}$ preserves multilinearity as shown above, the definition of the symmetric monoidal structure on $\presentable^L_{\presheaf(\B)}$ \cite[Section 2.6]{MartiniWolf2022Presentable}
yields that $\cofree_{\B\op}\colon\presentable^L\rightarrow\presentable^L_{\presheaf(\B)}$ refines to a lax symmetric monoidal functor, as wanted.
\end{proof}

\section{Localisation of \texorpdfstring{$\baseTopos$}{T}-categories}
In this subsection we review the theory of localisations of $\baseTopos$-categories, by which we shall mean the direct generalisation of Dwyer-Kan localisations to the parametrised setting. 
Recall that a wide $\baseTopos$-subcategory of $\X \in \cat_\baseTopos$ is a $\baseTopos$-functor $\W \rightarrow \X$ such that for each $\tau\in\baseTopos$, the functor $\W(\tau) \rightarrow \X(\tau)$ is a wide subcategory inclusion of $\infty$-categories.
\begin{defn}
\label{defn:relcat}
We denote by $\relcat_{\baseTopos}$ the full subcategory of $\func([1],\cat_{\baseTopos})$ spanned by inclusions of wide $\baseTopos$-subcategories. An object $(\D,\W)\in\relcat_{\baseTopos}$ is called a \textit{relative $\baseTopos$-category}.
The fully faithful inclusion functor $\cat_{\baseTopos}\to\relcat_{\baseTopos}$, sending $\D\mapsto(\D,\D^\simeq)$, admits a left adjoint $\scL\colon\relcat_{\baseTopos}\to\cat_{\baseTopos}$, which we call the localisation functor.

We can assemble the $\infty$-categories $\relcat_{\baseTopos_{/\tau}}$ for varying $\tau\in\baseTopos$ into a large $\baseTopos$-category $\relCatTopos_\baseTopos\in\widehat{\cat}_\baseTopos$; we can moreover assemble the inclusion functors $\cat_{\baseTopos_{/\tau}}\to\relcat_{\baseTopos_{/\tau}}$ into a fully faithful $\baseTopos$-functor $\catTopos_\baseTopos\to\relCatTopos_\baseTopos$, whose left adjoint $\baseTopos$-functor shall be denoted $\scL\colon\relCatTopos_\baseTopos\to\catTopos_\baseTopos$.
\end{defn}
For $(\D,\W)\in\relcat_{\baseTopos}$, the $\baseTopos$-category $\scL(\D,\W)$ is the initial example of a $\baseTopos$-category receiving a $\baseTopos$-functor from $\D$ along which ``morphisms in $\W$ are inverted'', i.e. such that the composite $\baseTopos$-functor $\W\to\D\to \scL(\D,\W)$ factors through $\scL(\D,\W)^\simeq$.
\begin{obs}
\label{obs:localisation_op}
We may consider $(-)\op$ as an involution on both $\cat_{\baseTopos}$ and $\relcat_{\baseTopos}$, and the inclusion $\cat_{\baseTopos}\to\relcat_{\baseTopos}$ is $C_2$-equivariant; similarly, the $\baseTopos$-inclusion $\catTopos_\baseTopos\to\relCatTopos_\baseTopos$ is $C_2$-equivariant. Passing to the left adjoint and $\baseTopos$-left adjoint, respectively, we obtain that also $\scL$ is $C_2$-equivarint. In particular, for $(\D,\W)\in\relcat_\baseTopos$ we have $\scL(\D,\W)\op\simeq \scL(\D\op,\W\op)$.
\end{obs}

In the next lemma we shall use that a $\baseTopos$-functor $f \colon I \to \relCatTopos_{\baseTopos}$ can be described as a $\baseTopos$-natural transformation $\theta' \to \theta$ of two $\baseTopos$-functors $I \to \catTopos_{\baseTopos}$ with the property that, for all $\tau\in\baseTopos$ and $i\in I(\tau)$, the $\baseTopos_{/\tau}$-functor $\theta'(i) \to \theta(i)$ is a wide $\baseTopos_{/\tau}$-subcategory inclusion.

\begin{lem}
\label{lem:localisation_of_fibration}
Let $(\theta'\to\theta)\colon I\to\relCatTopos_{\baseTopos}$ be a $\baseTopos$-functor, let $p\colon \X\to I$ be the $\baseTopos$-cocartesian fibration corresponding to $\theta\colon I\to\catTopos_{\baseTopos}$, let $p'\colon  \widehat{\W} \to I$ be the $\baseTopos$-cocartesian fibration corresponding to $\theta'$, and denote $\W:=I^\simeq\times_I \widehat{\W}$.
Then the pair $(\X,\W)$ is again a relative $\baseTopos$-category and the $\baseTopos$-functor $\scL(\X,\W)\to I$ induced on localisation is the $\baseTopos$-cocartesian fibration corresponding to the composite $\baseTopos$-functor $\scL\circ (q'\to q)\colon I\to\catTopos_{\baseTopos}$.
\end{lem}
\begin{proof}
It is easy to see that $\widehat{W} \to \W$ is a wide $\baseTopos$-subcategory.
Furthermore, $\W = \widehat{\W} \times_I I^\simeq \to \widehat{\W}$ is a wide $\baseTopos$-subcategory.
This shows that $(\X,\W)$ is a relative $\baseTopos$-category.

For $r,r'\in\func_{\baseTopos}(I,\catTopos_{\baseTopos})$ we let $\nattrans(r,r')\coloneqq \map_{\func_{\baseTopos}(I,\catTopos_{\baseTopos})}(r,r')\in\spc$ denote the space of $\baseTopos$-natural transformations from $r$ to $r'$. For $r\in\func(I,\catTopos_{\baseTopos})$ we further let $\nattrans(\theta,r)_{\theta'\mathrm{-loc}}\subseteq\nattrans(\theta,r)$ denote the subspace comprising all natural transformations that, for all $\tau\in\baseTopos$ and all $i\in I(\tau)$, restrict to a $\theta'(i)$-local functor $\theta(i)\to r(i)$; in other words, $\nattrans(\theta,r)_{\theta'\mathrm{-loc}}$ is the image of the map $\nattrans(\scL\circ(\theta'\to \theta),r) \to \nattrans(\theta, r)$ induced by the localisation $\baseTopos$-natural transformation $\theta\to \scL\circ(\theta'\to \theta)$.

Recall the $\baseTopos$-functor $\int_I\colon\func_\baseTopos(I,\catTopos_{\baseTopos})\to\catTopos_{\baseTopos}$ from \cref{subsec:recollections}, sending $\bar\theta\colon I\to\catTopos_{\baseTopos}$ to the source of a cocartesian fibration corresponding to $\bar\theta$, admits a $\baseTopos$-right adjoint $\cofree_I\colon\catTopos_{\baseTopos}\to\func_\baseTopos(I,\catTopos_{\baseTopos})$ by \cref{lem:omnibus_cofree}. 
Now let $\Y\to I$ be the  $\baseTopos$-cocartesian fibration corresponding to $\scL\circ (\theta'\to \theta)$.
For $\D\in\cat_{\baseTopos}$ we compute the following chain of equivalences of spaces
\[
\map_{\cat_{\baseTopos}}(\Y,\D)\simeq\nattrans(\scL\circ (\theta'\to \theta),\cofree_I\D)\simeq\nattrans(\theta,\cofree_I\D)_{\theta'\mathrm{-loc.}}\simeq\map_{\cat_{\baseTopos}}(\X,\D)_{\W\mathrm{-loc}}.
\]
In the last step we use that $\theta'$-local natural transformations correspond to $\W$-local functors along the identification $\nattrans(\theta,\cofree_I\D)\simeq\map_{\cat_{\baseTopos}}(\X,\D)$. This exhibits $\Y\simeq \scL(\X,\W)$.
\end{proof}

\section{Lax linearity}
\label{subsec:lax_linearity}
In the following, we want to study the notion of lax $\sC$-linear $\baseTopos$-functors between $\sC$-linear $\baseTopos$-categories, where $\sC \in \calg(\presentable^L_\baseTopos)$ is a presentably symmetric monoidal $\baseTopos$-category. The purpose of this study is that while Morita theory classifies $\sC$-linear functors,  many of the functors we wish to apply it to will naturally only carry a slightly weaker structure, namely \textit{lax $\sC$-linearity}. A typical example of a lax $\sC$-linear functor is the right adjoint to a $\sC$-linear functor, which will be the main result (\cref{prop:right_adjoint_to_C_linear_functor_is_lax_C_linear}) of this section. However, actual $\sC$-linearity is just a \textit{property} for lax $\sC$-linear functors. 

To provide some intuition, let us first consider the case $\baseTopos=\spc$. Given a presentably symmetric monoidal $\infty$-category $\sC$, the datum of a $\sC$-linear functor $f \colon \mathcal{M} \rightarrow \mathcal{N}$ provides in particular, for all $c \in \sC$ and $m\in \mathcal{M}$, an equivalence
\[  c \otimes_\sC f(m) \xrightarrow{\sim} f( c \otimes_\sC m).  \]
Now suppose that $f$ admits a right adjoint as a plain functor. Then, for $c \in \sC$ and $n \in \mathcal{N}$ there still is a comparison map $c \otimes_\sC g(n) \rightarrow g(c \otimes_\sC m)$, namely the adjoint of the composite
\[ f(c \otimes_\sC g(n)) \xleftarrow{\simeq} c \otimes_\sC fg(n) \xrightarrow{\mathrm{counit}} c \otimes_\sC n. \]
The collection of these morphisms together with higher coherence data equip $g$ with the structure of a \textit{lax $\sC$-linear functor}. At this level of detail, it is easy to believe that lax $\sC$-linear functors can be composed to lax $\sC$-linear functors, and that it is a property among lax $\sC$-linear functors to be actually $\sC$-linear, namely the property that all the comparison maps $c \otimes_\sC g(n) \rightarrow g(c \otimes_\sC n)$ are actually equivalences. The following will give a more precise recollection on lax linearity, in the presentation given by Cnossen in \cite{Cnossen2023}.

Recall that for an ($\infty$-)operad $\calO$, an \textit{$\calO$-monoidal $\infty$-category} is a cocartesian fibration $\sC^{\otimes} \rightarrow \calO^{\otimes}$ such that the composite $\sC^{\otimes} \rightarrow \calO^{\otimes} \rightarrow \mathrm{Fin}_*$ is an operad. An $\calO$-monoidal functor is a morphism of operads over $\calO^{\otimes}$, preserving all cocartesian lifts of morphisms in $\calO^{\otimes}$. We write $\cat^{\calO} \subset \cocartesianCategory(\calO^\otimes)$ for the full $\infty$-subcategory spanned by $\calO$-monoidal $\infty$-categories. The main result of \cite[Appendix A]{Cnossen2023} we will use is that one may regard $\calO$-algebras in $\cat_\baseTopos$ as $\baseTopos^{\mathrm{op}, \calO}$-monoidal $\infty$-categories, where $\baseTopos^{\mathrm{op}, \calO}$ is the following (ordinary) operad, constructed out of $\calO$; up to enlarging the universe, we may similarly describe $\calO$-algebras in $\presentable^L_{\baseTopos}$ as large $\baseTopos^{\mathrm{op}, \calO}$-monoidal $\infty$-categories.
\begin{defn}
We consider the coproduct symmetric monoidal structure on $\baseTopos\op$, encoded by a cocartesian fibration $\baseTopos^{\mathrm{op},\sqcup} \rightarrow \mathrm{Fin}_*$. We set
    \[ (\baseTopos^{\mathrm{op}, \calO})^\otimes = \calO^\otimes \times_{\mathrm{Fin}_*} \baseTopos^{\mathrm{op},\sqcup}. \]
\end{defn}

\begin{prop}\label{prop:O_algebras_in_T-categories}
    There is an equivalence
    \[ \alg_\calO(\func(\baseTopos\op,\cat)) \simeq \cat^{\baseTopos^{\mathrm{op},\calO}}. \]
    In particular, the $\infty$-category of $\calO$-algebras in $\cat_\baseTopos$ is a full $\infty$-subcategory of the $\infty$-category of $\baseTopos^{\mathrm{op},\calO}$-monoidal $\infty$-categories.
\end{prop}

Given an $\calO$-algebra $\sC$ in $\cat_\baseTopos$, let us write $\sC^{\otimes, \baseTopos,\calO}$ for the corresponging $\baseTopos^{\mathrm{op},\calO}$-monoidal category. 
Specialising to the case of left modules, recall that there is an operad $\leftmod$ coming with an operad map $E_1 \rightarrow \leftmod$. Intuitively, an $\leftmod$-algebra consists of a pair $(A,M)$, where $A$ is an $E_1$-algebra and $M$ a left $A$-module; in fact $A$ is extracted from $(A,M)$ by restricting along the aforementioned operad map $E_1 \rightarrow \leftmod$.

By definition of the left side, for $\sC \in \alg_{E_1}(\cat_\baseTopos)$ we have an equivalence
\[
\lmodule_{\sC}(\cat_\baseTopos) \simeq \alg_{\leftmod}(\cat_\baseTopos) \times_{\alg_{E_1}(\cat_{\baseTopos})} \{ \sC\}.  
\]
In particular, to a left $\sC$-module we can assign a $\baseTopos^{\mathrm{op},\leftmod}$-monoidal $\infty$-category.

\begin{defn}
\label{defn:lax_C_linear_functor}
    Let $\sC \in \alg_{E_1}(\cat_\baseTopos)$ be a monoidal $\baseTopos$-category, and let $\mathcal{M}$ as well as $\mathcal{N}$ be left $\sC$-modules. A \textit{lax $\sC$-linear $\baseTopos$-functor} $f \colon \mathcal{M} \rightarrow \mathcal{N}$ is the datum of
    \begin{itemize}
    \item a morphism $\mathcal{M}^{\otimes,\baseTopos,\leftmod} \rightarrow \mathcal{N}^{\otimes,\baseTopos,\leftmod}$ in $\cat_{/(\baseTopos^{\mathrm{op},\leftmod})^\otimes}$
    which is a map of operads;
    \item an identification with the identity of $\mathcal{\sC}^{\otimes,\baseTopos,E_1}$ of the pullback morphism along the operad map $\baseTopos^{\mathrm{op},E_1} \rightarrow \baseTopos^{\mathrm{op},\leftmod}$.
    \end{itemize}
\end{defn}

\begin{prop}
\label{prop:right_adjoint_to_C_linear_functor_is_lax_C_linear}
Let $\sC \in \alg_{E_1}(\cat_\baseTopos)$ and let $f \colon \mathcal{M} \rightarrow \mathcal{N}$ be a $\sC$-linear $\baseTopos$-functor which admits a right adjoint $g$ as a plain functor. Then $g$ admits a lax $\sC$-linear structure.
\end{prop}

\begin{proof}
Let $a,m\in\leftmod^\otimes\left<1\right>$ denote the colours for the algebra and the module, respectively.
The functor $f$ induces a functor $f^{\otimes,\baseTopos,\leftmod}\colon\mathcal{M}^{\otimes,\baseTopos,\leftmod} \rightarrow \mathcal{N}^{\otimes,\baseTopos,\leftmod}$, i.e. a morphism in $\cocartesianCategory((\baseTopos^{\mathrm{op},\leftmod})^\otimes)$. This functor admits a ``right adjoint extraction'' in the sense of \cite[Cor. C]{HHLNCalculusOfMates}: to apply this result we need to check that $f^{\otimes,\baseTopos,\leftmod}$ admits objectwise in $(\baseTopos^{\mathrm{op},\leftmod})^\otimes$ a right adjoint, and the Segal condition implies that it suffices to check the last claim on colours, i.e. on objects of $(\baseTopos^{\mathrm{op},\leftmod})^\otimes\left<1\right>$; for objects projecting to $a$, this is immediate as an identity functor admits itself as right adjoint; for objects projecting to $m$ this is a consequence of the hypothesis.
Thus we obtain a right adjoint
\[
g^{\mathrm{lax},\baseTopos,\leftmod}\colon\mathcal{N}^{\otimes,\baseTopos,\leftmod} \rightarrow \mathcal{N}^{\otimes,\baseTopos,\leftmod}.
\]
to $f^{\otimes,\baseTopos,\leftmod}$, which is a map of operads but in general does not preserve all cocartesian morphisms over $\baseTopos^{\mathrm{op},\leftmod}$.
Since the extraction of adjoints is computed fibrewise, $g^{\mathrm{lax},\baseTopos,\leftmod}$ restricts to the right adjoint $g\colon\mathcal{N}\to\mathcal{M}$ over $m\times\baseTopos\op$, and it restricts to the (fibrewise right adjoint of the) identity of $\sC^{\otimes,\baseTopos,E_1}$ along pullback along $\baseTopos^{\mathrm{op},E_1} \rightarrow \baseTopos^{\mathrm{op},\leftmod}$.
\end{proof}

\begin{defn}
Recall \cref{defn:lax_C_linear_functor}.
We let $\func_{\sC-\mathrm{lax}}(\mathcal{M},\mathcal{N})$ denote the full subcategory spanned by operad maps $\mathcal{M}\to\mathcal{N}$ of the fibre at the identity of $\sC^{\otimes,\baseTopos,E_1}$ of the functor
\[
\func_{/\baseTopos^{\mathrm{op},\leftmod}}(\mathcal{M}^{\otimes,\baseTopos,\leftmod},\mathcal{N}^{\otimes,\baseTopos,\leftmod}) \rightarrow 
\func_{/\baseTopos^{\mathrm{op},E_1}}(\sC^{\otimes,\baseTopos,E_1},\sC^{\otimes,\baseTopos,E_1})
\]
we obtain a $\infty$-category of lax $\sC$-linear functors from $\mathcal{M}$ to $\mathcal{N}$.
\end{defn}

\begin{lem}
    \label{lem:adjunction_unit_lax_linear}
    Let $f \colon \mathcal{M} \rightarrow \mathcal{N}$ be a $\sC$-linear map, which has a lax $\sC$-linear right adjoint. Then the adjunction unit and counit refine to morphisms in $\func_{\sC-\mathrm{lax}}(\mathcal{M},\mathcal{M})$ and $\func_{\sC-\mathrm{lax}}(\mathcal{N},\mathcal{N})$ respectively.
\end{lem}

\begin{proof}
    Recall from the proof of \cref{prop:right_adjoint_to_C_linear_functor_is_lax_C_linear} that the adjunction $f \dashv g$ induces an adjunction $f^{\otimes,\baseTopos,\leftmod} \dashv g^{\mathrm{lax},\baseTopos,\leftmod}$ in $\cat_{/(\baseTopos^{\mathrm{op},\leftmod})^\otimes}$
    which restricts to the adjunction $\id\colon\sC^{\otimes,\baseTopos,E_1}\rightleftharpoons\sC^{\otimes,\baseTopos,E_1}\cocolon\id$ 
and consists of operad maps, so the claim follows.
\end{proof}

\printbibliography
\end{document}